\documentclass[jsl]{asl}

\usepackage{latexsym}
\usepackage{url}
\usepackage{hyperref}
\usepackage{tikz}
\tikzset{>=latex, 
	point/.style = {circle,draw,thick,minimum size=2mm,inner sep=0pt},
	point1/.style = {circle,draw,thick,minimum size=6mm,inner sep=0pt},
	hm/.style = {dotted,semithick},
	role/.style = {thick},
	tree/.style = {rounded corners=10pt, dashed, fill opacity=0.5, fill=nullscolour},
	wiggly/.style={thick,%decorate,decoration={snake,amplitude=0.3mm,segment length=2mm,post length=1mm}
	},
	query/.style={thick},
	itria/.style={
  draw,dashed,shape border uses incircle,
  isosceles triangle,shape border rotate=90,yshift=-1.45cm},
  square/.style={regular polygon,regular polygon sides=4}
}

\usepackage{rotating}
\usepackage{pdflscape}
\usepackage{graphicx}

\theoremstyle{plain}
\newtheorem{lemma}{Lemma}[section]
\newtheorem{theorem}[lemma]{Theorem}
\newtheorem{example}[lemma]{Example}
\newtheorem{definition}[lemma]{Definition}
\newtheorem{remark}[lemma]{Remark}
\newcommand{\clusterki}{\mbox{\begin{picture}(13,10)
			\put(0,0){\scalebox{1.3}{$\bigcirc$}}
			\put(2.9,0.35){\mbox{{\scriptsize $k_i$}}}
			\end{picture}}}
\newcommand{\clusterkj}{\mbox{\begin{picture}(13,10)
			\put(0,0){\scalebox{1.3}{$\bigcirc$}}
			\put(2.9,0.19){\mbox{{\scriptsize $k^j$}}}
			\end{picture}}}			
\newcommand{\clusterkn}{\mbox{\begin{picture}(13,10)
			\put(0,0){\scalebox{1.5}{$\bigcirc$}}
			\put(2.9,0.19){\mbox{{\scriptsize $k^\partN$}}}
			\end{picture}}}	
\newcommand{\clusterkz}{\mbox{\begin{picture}(13,10)
			\put(0,0){\scalebox{1.35}{$\bigcirc$}}
			\put(2.9,0.19){\mbox{{\scriptsize $k^0$}}}
			\end{picture}}}						
			
\newcommand{\tight}{tig}
\newcommand{\minmax}{minmax}
\newcommand{\atomic}{atomic}

\newcommand{\simple}{simple}
\newcommand{\qpbmp}{quasi-polysize bisimilar model property}
\newcommand{\qfbmp}{quasi-finite bisimilar model property}
\newcommand{\bmp}[1]{$({#1})$}
\newcommand{\bounded}{$L$-bounded}
\newcommand{\yy}{b}
\newcommand{\iso}{f}
\newcommand{\sC}[1]{C_{#1}^-}
\newcommand{\eC}[1]{C_{#1}^+}

\newcommand{\Bt}{\Box_\mathsf{X}}

\usepackage{stmaryrd}
\usepackage{wasysym}
\usepackage{rotating}
\usepackage{pifont}
\usepackage{booktabs}
\usepackage{array}
\usepackage{multirow}
\usepackage{tikz}
\usepackage{color}
\usepackage{xcolor}
\usepackage{enumitem}
\usepackage{xspace}
\usepackage{thm-restate}
\usepackage{bm}

\usepackage{url}
\usepackage{hyperref}
\newcommand{\sig}{\textit{sig}}
\newcommand{\sub}{\textit{sub}}

\newcommand{\bis}{\boldsymbol{\beta}}
\newcommand{\INT}{\mathcal{P}}

\newcommand{\at}{\textit{at}}
\newcommand{\att}{\textit{a}}
\newcommand{\md}{\textit{md}}

\newcommand{\bl}{{\boldsymbol{b}}}
\newcommand{\FC}{\mathfrak F_c}

\newcommand{\KF}{\mathsf{K4}}
\newcommand{\SF}{\mathsf{S4}}
\newcommand{\SFi}{\mathsf{S5}}
\newcommand{\KFT}{\mathsf{K4.3}}
\newcommand{\GLT}{\mathsf{GL.3}}
\newcommand{\K}{\mathsf{K}}
\newcommand{\Lin}{\mathsf{Lin}}
\newcommand{\coNP}{\textup{\textsc{coNP}}}

\newcommand{\clusterk}{\mbox{\begin{picture}(10,10)
			\put(0,0){$\bigcirc$}
			\put(2.9,0.19){\mbox{{\scriptsize $k$}}}
\end{picture}}}

\newcommand{\clustern}{\mbox{\begin{picture}(10,10)
			\put(0,0){$\bigcirc$}
			\put(2.5,0.4){\mbox{{\scriptsize $n$}}}
\end{picture}}}

\newcommand{\clusterm}{\mbox{\begin{picture}(10,10)
			\put(0,0){$\bigcirc$}
			\put(1.5,0.3){\mbox{{\scriptsize $m$}}}
\end{picture}}}

\newcommand{\clusterl}{\mbox{\begin{picture}(10,10)
			\put(0,0){$\bigcirc$}
			\put(3.9,0.19){\mbox{{\scriptsize $l$}}}
\end{picture}}}
\newcommand{\clustero}
{\mbox{\begin{picture}(10,10)
			\put(0,0){$\bigcirc$}
			\put(3,0.19){\mbox{{\scriptsize $1$}}}
\end{picture}}}
\newcommand{\clustert}{\mbox{\begin{picture}(10,10)
			\put(0,0){$\bigcirc$}
			\put(3.1,0.19){\mbox{{\scriptsize 2}}}
\end{picture}}}

\newcommand{\ExpTime}{\textnormal{\sc ExpTime}\xspace}

\newcommand{\cluster}[1]{\mbox{\begin{picture}(10,10)
			\put(0,0){$\bigcirc$}
			\put(2.9,0.19){\mbox{{\scriptsize ${#1}$}}}
\end{picture}}}

\newcommand{\chain}[2]{{\mathfrak C}(\cluster{#1},{#2})}
\newcommand{\Log}{\mathsf{Log}}

\newcommand{\intord}[1]{\prec_{#1}}
\newcommand{\atom}{atom}
\newcommand{\move}{move}
\newcommand{\globalb}{global}
\newcommand{\globally}{globally}
\newcommand{\tail}{tail}
\newcommand{\tails}{tails}
\newcommand{\source}{head}

\newcommand{\rest}[2]{{#1}\mathop{\restriction}_{#2}}
\newcommand{\parent}{\mathsf{h}}

\newcommand{\partN}{N}

\newcommand{\Dt}{\Diamond_\mathsf{X}}

\newcommand{\Df}{\Diamond_\mathsf{F}}
\newcommand{\Dp}{\Diamond_\mathsf{P}}
\newcommand{\Bf}{\Box_\mathsf{F}}
\newcommand{\Bp}{\Box_\mathsf{P}}
\newcommand{\Linf}{\Lin_{<\omega}}
\newcommand{\LinZ}{\Lin_{\mathbb{Z}}}
\newcommand{\LinQ}{\Lin_{\mathbb{Q}}}
\newcommand{\LinR}{\Lin_{\mathbb{R}}}

\newcommand{\mset}[1]{\boldsymbol{M}_{#1}}
\newcommand{\sset}[1]{\boldsymbol{S}_{#1}}
\newcommand{\size}[1]{\|{#1}\|}
\newcommand{\kbound}{\boldsymbol{k}(\varphi_1,\varphi_2)}
\newcommand{\pbound}{\boldsymbol{p}_L(\varphi_1,\varphi_2)}
\newcommand{\cbound}{\boldsymbol{c}_L}
\newcommand{\step}[1]{{\bf{(s{$_#1$})}}}
\newcommand{\rn}{\boldsymbol{r}}
\newcommand{\nn}{\boldsymbol{n}}

\newcommand{\match}{$\sigma$-matching}
\newcommand{\qfin}{quasi-finite}
\newcommand{\starr}{\star}
\newcommand{\nice}{nice}

\title[A Non-Uniform View of Craig Interpolation in Modal Logics with Linear Frames]{A Non-Uniform View of Craig Interpolation in Modal Logics with Linear Frames}

\author[A.~Kurucz]{Agi~Kurucz}\revauthor{Kurucz, Agi}
\author[F.~Wolter]{Frank~Wolter}\revauthor{Wolter, Frank}
\author[M.~Zakharyaschev]{Michael~Zakharyaschev}\revauthor{Zakharyaschev, Michael}

\address{Department of Informatics\\
King's College London\\
Strand Campus, Bush House, 30 Aldwych, London  WC2B 4BG, U.K.}
\email{agi.kurucz@kcl.ac.uk}

\address{School of Computer Science and Informatics\\
University of Liverpool\\
Ashton Building,
Liverpool L69 3DR, U.K.}
\email{wolter@liverpool.ac.uk}

\address{School of Computing and Mathematical Sciences\\
Birkbeck, University of London\\
Malet Street, London WC1E 7HX, U.K.}
\email{m.zakharyaschev@bbk.ac.uk}

\date{ }

\begin{document}

\begin{abstract}
Normal modal logics extending the logic $\KFT$ of linear transitive frames are known to lack the Craig interpolation property, except some logics of bounded depth such as $\SFi$. We turn this `negative' fact into a research question and pursue a non-uniform approach to Craig interpolation by investigating the following interpolant existence problem: decide whether there exists a Craig interpolant between two given formulas in any fixed logic above $\KFT$. Using a bisimulation-based characterisation of interpolant existence for descriptive frames, we show that this problem is decidable and \coNP-complete for all finitely axiomatisable normal modal logics containing $\KFT$. It is thus not harder than entailment in these logics, which is in sharp contrast to other recent non-uniform interpolation results. 
We also extend our approach to Priorean temporal logics (with both past and future modalities) over the standard time flows---the integers, rationals, reals, and finite strict linear orders---none of which is blessed with the Craig interpolation property.
\end{abstract}

\maketitle

%************

\section{Introduction}	

Unlike classical and intuitionistic first-order and propositional logics, numerous modal logics, $L$, do not enjoy the Craig interpolation property (CIP): they contain valid implications $\varphi \to \psi$ without an interpolant in $L$---a formula $\iota$ in the shared signature of $\varphi$ and $\psi$ such that both $\varphi \to \iota$ and $\iota \to \psi$ are also valid in $L$. Typical examples of such $L$ are first-order modal logics with constant domains between $\mathsf{K}$ and $\SFi$~\cite{DBLP:journals/jsyml/Fine79} and propositional modal logics with linear transitive Kripke frames of unbounded depth~\cite{MGabbay2005-MGAIAD,DBLP:journals/jphil/Wolter97}. There have been various attempts to classify propositional modal logics with the CIP, successful for extensions of $\mathsf{S4}$~\cite[Section~8]{MGabbay2005-MGAIAD} and unsuccessful for extensions of $\KF$ or $\mathsf{GL}$, where the CIP turned out to be undecidable~\cite[Sections~14, 17]{DBLP:books/daglib/0030819}. 

While establishing the CIP of a logic $L$ typically gives rise to further research problems---develop proof systems that admit efficient/elegant interpolant computation~\cite{DBLP:journals/apal/Kuznets18,DBLP:journals/tocl/BenediktCB16}, investigate the complexity of computing interpolants from proofs~\cite[Sections 17, 18]{Krajicek}, consider restrictions on the shape of interpolants such as in, say, Lyndon's interpolation~\cite{Lyndon59}, or employ the CIP to investigate related properties such as Beth definability~\cite{craig_1957,DBLP:journals/synthese/Feferman08}---a counterexample to the CIP has usually  terminated further research of Craig interpolants and their applications for the unfortunate logic in question.  

In this article, we take a different, non-uniform view of Craig interpolation and aim to understand interpolants also for logics $L$ without the CIP. We consider the following \emph{interpolant existence problem} (IEP) for $L$: given formulas $\varphi$ and $\psi$,  
decide whether $\varphi \to \psi$ has an interpolant in $L$. 
For $L$ without the CIP, the existence of an interpolant for $\varphi$ and $\psi$ does not follow from the validity of $\varphi\rightarrow\psi$ in $L$, and so the IEP does not reduce to validity checking (which is reducible to the IEP). A first question then is whether the former problem is 
genuinely harder than the latter one.
In fact, when the IEP was introduced~\cite{DBLP:conf/lics/JungW21,DBLP:journals/corr/abs-2007-02736}, this was shown to be the case  for modal logics with nominals and for the two-variable and guarded fragments of first-order logic. Since then, this has also been confirmed for the one-variable fragment of first-order modal logic $\mathsf{S5}$ and weak $\KF$~\cite{ourLMCS25,DBLP:conf/aiml/KuruczWZ24}.  

Here, we show that the opposite is true of propositional modal logics containing $\KFT$, the logic of linear transitive frames: while none of these logics with frames of unbounded depth has the CIP~\cite{MGabbay2005-MGAIAD,DBLP:journals/jphil/Wolter97}, interpolant  existence is nevertheless decidable in \coNP{} for finitely axiomatisable logics, and so is as hard as validity~\cite{DBLP:journals/sLogica/LitakW05}. This is the first general result on Craig interpolant existence covering a large family of modal logics and, potentially, a step towards a classification of modal logics according to the complexity of the IEP.

We proceed as follows. To begin with, we give a `folklore' characterisation of interpolant existence via bisimulations between models based on descriptive frames: $\varphi \to \psi$ does not have an  interpolant in $L$ iff $\varphi$ and $\neg\psi$ can be satisfied in 
$\sig(\varphi)\cap\sig(\psi)$-bisimilar models based on descriptive frames for $L$. If $L$ had the CIP, we could merge these two models into a single one satisfying $\varphi \land \neg \psi$ (using, say, bisimulation products~\cite{DBLP:conf/amast/Marx98}) or amalgamate the induced modal algebras~\cite{MGabbay2005-MGAIAD}, which is impossible in our case. Instead, we aim to understand the fine-grained structure of the required bisimilar models and use it to decide their existence. We show that, for some logics (such as first-order definable cofinal subframe logics), any pair of bisimilar models can be transformed into  bisimilar models of polynomial size; in other words, such logics enjoy the polysize bisimilar model property. However, for other logics like $\GLT$, not even models based on infinite Kripke frames are enough despite $\GLT$ having the finite model property. 
%
%We prove that, nevertheless, one can always find models based on sufficiently transparent infinite descriptive frames whose existence can be checked in NP in the size of the input formulas $\varphi$ and $\psi$, for any finitely axiomatisable $L\supseteq \KFT$. As the underlying frames of these models are built from polynomially-many descriptive frames of some specified forms, we say that all $L\supseteq \KFT$ have the quasi-polysize bisimilar model property.

We prove, nevertheless, that every pair of bisimilar models satisfying $\varphi$ and $\neg\psi$ and based on descriptive frames for a finitely axiomatisable $L \supseteq \KFT$ can be converted to a pair of such models with an understandable structure. 
In a nutshell, their underlying frames look like a polynomial-size chain of polynomial-size clusters and tadpole-like descriptive frames that comprise a non-degenerate cluster $\{a_0,\dots, a_{k-1}\}$, for some polynomial-size $k>0$, followed by an infinite descending chain of points $\yy_n$, $n < \omega$, which are all irreflexive or all reflexive, with the internal sets (restricting possible valuations) generated as a modal algebra by the singletons $\{\yy_n\}$ and the $k$-many pairwise disjoint infinite sets \mbox{$X_i = \{a_i\} \cup \{\yy_n \mid n \equiv i \ (\text{mod}\ k) \}$}. The picture below illustrates the underlying Kripke frame and the generators of the tadpole descriptive frame with $k=2$.\\ 
\centerline{
\begin{tikzpicture}[>=latex,line width=0.4pt,xscale = 1.2,yscale = 1]
\node[point,scale = 0.7,label=above:{\footnotesize $a_0$}] (a0) at (0,0) {};
\node[point,scale = 0.7,label=below:{\footnotesize $a_{1}$}] (a1) at (1,0) {};
\draw[] (.5,0) ellipse (1 and .6);
\node[]  at (2,0) {$\dots$};
\node[]  at (2,1) {$\dots$};
\node[scale = 0.9,label=above:{\footnotesize $\yy_{2n}$}] (bn) at (3,0) {$\ast$};
\node[scale = 0.9,label=below:{\footnotesize $\yy_{2n-1}$}] (bnn) at (4,0) {$\ast$};
\node[]  at (5,0) {$\dots$};
\node[scale = 0.9,label=above:{\footnotesize $\yy_{2}$}] (b2) at (6,0) {$\ast$};
\node[scale = 0.9,label=below:{\footnotesize $\yy_{1}$}] (b1) at (7,0) {$\ast$};
\node[scale = 0.9,label=above:{\footnotesize $\yy_{0}$}] (b0) at (8,0) {$\ast$};
\draw[->] (2.5,0) to (bn);
\draw[->] (bn) to (bnn);
\draw[->] (bnn) to (4.5,0);
\draw[->] (5.5,0) to (b2);
\draw[->] (b2) to (b1);
\draw[->] (b1) to (b0);
\draw[gray,dashed,rounded corners=5] (1.5,.7) -- (.3,.7) -- (.3,-.2) -- (-.3,-.2) -- (-.3,1) -- (1.5,1);
\node[gray]  at (2,1) {$\dots$};
\node[gray]  at (2,.7) {$\dots$};
\draw[gray,dashed] (2.4,1) -- (4.6,1);
\node[gray]  at (5,1) {$\dots$};
\node[gray]  at (5,.7) {$\dots$};
\draw[gray,dashed,rounded corners=5] (2.4,.7) -- (2.7,.7) -- (2.7,-.3) -- (3.3,-.3) -- (3.3,.7) -- (4.6,.7);
\draw[gray,dashed,rounded corners=5] (5.4,.7) -- (5.7,.7) -- (5.7,-.3) -- (6.3,-.3) -- (6.3,.7) -- (7.7,.7) -- (7.7,-.3) -- (8.3,-.3) -- (8.3,1) -- (5.3,1);
\draw[gray,dashed,rounded corners=5] (2.4,-.5) -- (3.5,-.5) -- (3.5,.3) -- (4.4,.3) -- (4.4,-.5) -- (4.6,-.5);
\draw[gray,dashed,rounded corners=5] (5.4,-.5) -- (6.6,-.5) -- (6.6,.3) -- (7.5,.3) -- (7.5,-.8) -- (5.4,-.8);
\draw[gray,dashed] (2.4,-.8) -- (4.6,-.8);
\draw[gray,dashed,rounded corners=5] (1.5,-.8) -- (.7,-.8) -- (.7,.3) -- (1.2,.3) -- (1.2,-.5) -- (1.5,-.5);
\node[gray]  at (2,-.5) {$\dots$};
\node[gray]  at (2,-.8) {$\dots$};
\node[gray]  at (5,-.5) {$\dots$};
\node[gray]  at (5,-.8) {$\dots$};
\node[gray]  at (8.6,0.8) {\small $X_0$};
\node[gray]  at (7.8,-0.67) {\small $X_1$};
\end{tikzpicture}
}
%
%%caption{The infinite generators $X_0$ and $X_1$ in the tadpole descriptive frame for $k=1$.}\label{tad2}
%\caption{The tadpole descriptive frame for $k=1$.}\label{tad2}
%\end{figure}

Because of this, we say that all finitely axiomatisable $L\supseteq \KFT$ have the quasi-polysize bisimilar model property. We show that the existence of such quasi-polysize bisimilar models can be checked in NP in the size of $\varphi$ and $\psi$, for any finitely axiomatisable $L$.

Finally, we extend the developed techniques to analyse the IEP for a few Priorean temporal logics with past and future modal operators: the logic $\Lin$ of all linear frames, the logic $\Linf$ of all finite strict linear orders, and the logics $\LinQ$ of the rationals, $\LinR$ of the reals, and $\LinZ$ of the integers. We prove that $\Lin$, $\LinQ$, and $\LinR$ have the polysize bisimilar model property, while $\Linf$ and $\LinZ$ have the quasi-polysize one, with the IEP being  \coNP-complete.  
The proofs can be regarded as applications of the general method, which works for all extensions of $\KFT$, to a few concrete logics with transparent frames. 
% Kripke frames. 
In fact, one could read the Priorean case in parallel with the full general proof, using the former as an illustration of the latter. 

The remainder of the article is organised as follows. The introduction is concluded with a brief discussion of related work. \S\ref{prelims} contains the necessary modal logic preliminaries. \S\ref{sec:warming} gives the bisimulation-based criterion of interpolant existence and applies it to first-order definable cofinal subframe logics above $\KFT$. It also provides illustrative examples explaining why the same method does not work in general and what kind of descriptive frames might be needed. \S\ref{section:K4.3} establishes the quasi-finite  bisimilar model property of all logics above $\KFT$ and the quasi-polysize bisimilar model property of all finitely axiomatisable ones; for the latter, it gives a \coNP{}-algorithm for deciding the IEP. \S\ref{sec:temporal} extends the developed techniques to the Priorean temporal logics mentioned above.

\subsection{Related work} 
The IEP for some logics of linear frames turns out to be closely related to separability of  regular languages by first-order definable languages. Formally, the separability problem is to decide whether two input regular languages $L_1$ and $L_2$ can be separated by some language $L$ in a given class $\mathcal{L}$ in the sense that $L_{1} \subseteq L$ and $L \cap L_{2} = \emptyset$. If $\mathcal{L}$ is the class of first-order definable languages over finite words, the separability problem is equivalent to the IEP for the linear temporal logic $\mathsf{LTL}$ extending modal logic with the operators `next' and `until' over finite strict linear orders.
For regular languages of infinite words, the separability problem is equivalent to the IEP for $\mathsf{LTL}$ over the natural numbers; see~\cite{chapter:separation} for details. It was shown in~\cite{henkell1,henkell2,DBLP:journals/corr/PlaceZ14} that both of these separability problems are decidable in 2\ExpTime{} in the size of NFAs defining $L_1$ and $L_2$. It follows that the corresponding IEPs are decidable in 3\ExpTime{} in the size of $\mathsf{LTL}$-formulas. (Separability by other language classes $\mathcal{L}$ are discussed in~\cite{DBLP:journals/lmcs/Place18,DBLP:journals/lmcs/PlaceZ21}.) 
These separability results have been obtained using algebraic machinery from semigroup theory, which seems to be orthogonal to our model-theoretic approach to the IEP developed to deal with all modal logics of linear orders.
However, for finite strict linear orders and the natural numbers, the algebraic approach also provides an upper bound for the size of interpolants.

It is also worth mentioning that, for these two frame classes, the smallest modal logic with the CIP is $\mathsf{LTL}$ extended with fixed-point operators or, equivalently, monadic second-order logic (under very mild conditions on the definition of what a logic is)~\cite{DBLP:conf/csl/GheerbrantC09}. Thus, to `repair' the CIP by extending the expressive power of the logic, we require the addition of second-order features.

%********************************************************************************************

\section{Preliminaries}\label{prelims}

This section provides the basic definitions and facts that will be used later on in the article; consult~\cite{Goldblatt76a,Goldblatt76b,DBLP:books/daglib/0030819,DBLP:books/cu/BlackburnRV01,DBLP:books/el/07/BBW2007} for more details.

%************

\subsection{Descriptive frames for normal modal logics.}\label{basics}
The formulas, $\varphi$, of propositional unimodal logics are built from propositional variables $p_i \in \mathcal{V}$, for some countably-infinite set $\mathcal{V} = \{p_i \mid i < \omega\}$, and constants $\top$, $\bot$ using the Boolean connectives $\neg$, $\land$, and the unary possibility operator $\Diamond$. The other Booleans and the necessity operator $\Box$ dual to $\Diamond$ are defined as standard abbreviations. We also use $\Diamond^+\varphi= \varphi\lor\Diamond\varphi$, $\Box^+\varphi=\varphi\land\Box\varphi$, and $\Diamond\Gamma=\{\Diamond\varphi\mid\varphi\in\Gamma\}$, for a set $\Gamma$ of formulas.
By a \emph{signature} we mean any 
%\emph{finite} 
set $\sigma \subseteq \mathcal{V}$, denoting by $\sig(\varphi)$ the (finite) set of variables in a formula $\varphi$. If $\sig(\varphi) \subseteq \sigma$, we call $\varphi$ a $\sigma$-\emph{formula\/}. We denote by $\sub(\varphi)$ the set of subformulas of $\varphi$ together with their negations, and let $|\varphi|=|\sub(\varphi)|$.

A (\emph{normal}) \emph{modal logic}, $L$, is any set of formulas that contains all Boolean tautologies, the modal axiom
$
\Box (p_0 \to p_1) \to (\Box p_0 \to \Box p_1) ,
$
and is closed under the rules of \emph{modus ponens}, uniform substitution of formulas in place of variables, and necessitation $\varphi/\Box\varphi$. The smallest such logic goes by the moniker $\K$.
Given a set $\Gamma$ of formulas and a modal logic $L$, the smallest modal logic to contain $L$ and $\Gamma$ is denoted by $L\oplus\Gamma$. We write $L\oplus\varphi$ for $L\oplus\{\varphi\}$. For example, 
\begin{align*}
\KF & = \K\oplus \Box p_0 \to \Box\Box p_0, \\
\KFT & = \KF \oplus \Box (\Box^+ p_0 \to p_1) \lor \Box (\Box^+ p_1 \to p_0),\\
\GLT & = \KFT \oplus \Box( \Box p_0 \to p_0) \to \Box p_0,\\
\Log\{(\mathbb N,<)\} & = \KFT \oplus \Diamond\top \oplus \Box (\Box p_0 \to p_0) \to (\Diamond\Box p_0 \to \Box p_0). 
\end{align*}
All logics considered in this article are extensions of $\KFT$. 

We interpret formulas in (\emph{general}) \emph{frames} $\mathfrak{F} = (W,R,\INT)$, where $R$ is a binary (accessibility) relation on a nonempty set $W$ (of worlds or, more neutrally, points) and $\INT \subseteq 2^W$ contains $\emptyset$, $W$ and is closed under $\cap$, $\neg$, and the operator
\[
\Diamond^{\mathfrak F} X = \{ x \in W \mid \exists y\in X \, x R y  \}. %\quad \Box^{\mathfrak F} X = \{ x \in W \mid \forall y\in W \, (x R y \to y\in X)\}.
\]
The structure $\mathfrak F^+ = (\INT, \cap,\neg,\emptyset, W, \Diamond^{\mathfrak F})$ is a Boolean algebra $(\INT, \cap,\neg,\emptyset, W)$ with a normal and additive operator $\Diamond^{\mathfrak F}$ (BAO, for short). 
If $\mathfrak F^+$ is generated by a set $\mathcal X\subseteq\INT$ as a BAO, we say that the frame $\mathfrak F$ 
(or the set $\INT$) is \emph{generated by} $\mathcal X$. If $|\mathcal X|=n$, for some $n<\omega$, we call $\mathfrak F$ $n$-\emph{generated} or \emph{finitely generated\/}.
The elements of $\INT$ are called \emph{internal sets} in $\mathfrak F$. If $\INT = 2^W$, $\mathfrak{F}$ is known as a \emph{Kripke frame}; in this case, we drop $\INT$ and write $\mathfrak{F}=(W,R)$. 
A frame $\mathfrak{F}=(W,R,\INT)$ is \emph{descriptive} if the following conditions hold, for any $x,y \in W$
and any $\mathcal{X} \subseteq \mathcal{P}$:
%and any ultrafilter $\U$ over $\mathfrak F^+$:

\smallskip
\noindent
\textbf{(dif)} $x = y$ iff $\forall X \in \INT\, (x \in X \leftrightarrow y \in X)$, 

\noindent
\textbf{(\tight)} $xRy$ iff $\forall X\in \INT \, (y \in X \to x \in \Diamond^{\mathfrak F} X)$,  

\noindent
\textbf{(com)} if $\mathcal{X} \subseteq \mathcal{P}$ has the \emph{finite intersection property} (\emph{fip\/}, for short)---that is,\\
\mbox{} \hspace*{0.9cm} $\bigcap\mathcal{X}' \ne \emptyset$, for every finite $\mathcal{X}' \subseteq \mathcal{X}$---then $\bigcap\mathcal{X} \ne \emptyset$.
%\item[(com)] if $\bigcap\mathcal{X}' \ne \emptyset$ for every finite $\mathcal{X}' \subseteq \mathcal{X}$, then $\bigcap\mathcal{X} \ne \emptyset$.

\smallskip
\noindent
(Frames with {\bf (dif)} are called \emph{differentiated\/}, with {\bf (\tight)} \emph{tight\/}, and with {\bf (com)} \emph{compact\/}.) Every BAO is isomorphic to $\mathfrak F^+$, for some descriptive frame $\mathfrak F$.
A finite frame is descriptive iff it is a Kripke frame~\cite[Section~8]{DBLP:books/daglib/0030819}. 

Given a signature $\sigma$,
a $\sigma$-\emph{model}
based on a frame $\mathfrak{F}=(W,R,\INT)$ is a pair $\mathfrak{M}=(\mathfrak{F},\mathfrak{v})$ with a \emph{valuation} 
%$\mathfrak{v} \colon \mathcal{V} \to \INT$. 
 $\mathfrak{v} \colon \sigma \to \INT$. 
%
%
%Given a signature $\sigma$,
The \emph{atomic $\sigma$-type} of $x\in W$ in $\mathfrak{M}$ is

\[
\at^\sigma_{\mathfrak{M}}(x) = \{ p_i\mid p_i \in\sigma,\ x\in\mathfrak v(p_i)\}\cup \{ \neg p_i \mid p_i \in\sigma,\ x\notin\mathfrak v(p_i)\}.
\]
We omit $\sigma=\mathcal V$, saying simply \emph{model} and writing $\at_{\mathfrak{M}}(x)$. 
The \emph{value} of a formula $\varphi$ in $\mathfrak M$ is the set $\mathfrak{v}(\varphi) \in \INT$ computed inductively in the obvious way starting from $\mathfrak v(p_i)$, $\mathfrak v(\top) = W$ and $\mathfrak v(\bot) = \emptyset$. 
A set $X\subseteq W$ is \emph{definable} in $\mathfrak{M}$ if $X=\mathfrak{v}(\varphi)$, for some
formula $\varphi$, in which case $X \in \INT$. If every internal set $X\in\INT$ is definable in $\mathfrak M$, we say that $\mathfrak{F}$ is $\mathfrak{M}$-\emph{generated\/}. 
Every $\mathfrak F$ with countable $\INT$ is clearly $\mathfrak M$-generated, for some model $\mathfrak M$.

%\nb{name ok? We do use this in proving L.\ref{lem:descr'}}
%\textcolor{green}{Observe that if $\mathfrak{M}$ is a model based on a finitely generated frame $\mathfrak F$, then 
%$\mathfrak F$ is $\mathfrak M$-generated.}\nz{??}
%$\{\mathfrak{v}(p_i)\mid i<\omega\}\cup\{\mathfrak{v}(\bot),\mathfrak{v}(\top)\}$).
%
%Without loss of generality, we always assume that, in every model $\mathfrak{M}=(\mathfrak{F},\mathfrak{v})$ based on a descriptive frame $\mathfrak{F}=(W,R,\INT)$, the internal sets in $\INT$ are \emph{generated} as a BAO by the $\mathfrak v(p_i)$, $i < \omega$.\nz{call $\mathfrak M$ descriptive or smg else?} 
%The \emph{type} of a point $x$ in $\mathfrak{M}$ is the set $t_{\mathfrak{M}}(x)$ of all formulas that are true at $x$ in $\mathfrak M$. 
%(So we have $\varphi\in t_{\mathfrak{M}}(x)$ iff $x\in\mathfrak{v}(\varphi)$.)
%For a set $\Gamma$ of formulas, we write $\mathfrak M, x \models \Diamond \Gamma$ if there is $y$ with $xRy$ and $\mathfrak M,y \models \Gamma$. 

A formula $\varphi$ is \emph{true} at $x$ in $\mathfrak M$, written $\mathfrak M,x \models \varphi$, if $x \in \mathfrak{v}(\varphi)$. 
The \emph{$\sigma$-type of $x$ in} $\mathfrak{M}$ is the set $t^\sigma_{\mathfrak{M}}(x)$ of all $\sigma$-formulas that are true at $x$ in $\mathfrak M$. For a set $X$ of points in $\mathfrak M$, we let
$t_{\mathfrak{M}}^{\sigma}(X)=\bigl\{t_{\mathfrak{M}}^{\sigma}(x)\mid x\in X\bigr\}$. As before, we drop $\sigma=\mathcal V$. 

%(So we have $\varphi\in t_{\mathfrak{M}}(x)$ iff $x\in\mathfrak{v}(\varphi)$.)
%For a set $\Gamma$ of formulas, we write $\mathfrak M, x \models \Diamond \Gamma$ if there is $y$ with $xRy$ and $\mathfrak M,y \models \Gamma$. 
A set $\Gamma$ of formulas is \emph{finitely satisfiable} in $\mathfrak M$ if, for every finite subset $\Gamma' \subseteq \Gamma$, there is $x' \in W$ such that $\Gamma' \subseteq t_{\mathfrak{M}}(x')$; $\Gamma$ is \emph{satisfiable} in $\mathfrak M$ if $\Gamma \subseteq t_{\mathfrak{M}}(x)$, for some $x \in W$. 
 Using these definitions and notations, we can equivalently reformulate
%\nb{do we really need this? or we just need Lemmas~\ref{maxpoints} and \ref{lem:descr0} to hold?} 
conditions {\bf (dif)}, {\bf (\tight)}, and {\bf (com)} for $\mathfrak M$-generated frames as follows: for any $x,y \in W$ and any set $\Gamma$ of formulas,

\smallskip
\noindent
\textbf{(dif)} $x = y$ iff $t_{\mathfrak{M}}(x) = t_{\mathfrak{M}}(y)$,

\noindent
\textbf{(\tight)} $xRy$ iff $\Diamond t_{\mathfrak{M}}(y) \subseteq t_{\mathfrak{M}}(x)$ iff 
$\{\varphi\mid\Box\varphi\in t_{\mathfrak{M}}(x)\}\subseteq t_{\mathfrak{M}}(y)$,

%\item[(com)] $\Gamma$ is satisfiable in $\mathfrak M$ iff it is finitely satisfiable in $\mathfrak M$.\nz{FO connection?}

\noindent
\textbf{(com)} if $\Gamma$ is finitely satisfiable in $\mathfrak M$, then $\Gamma$ is satisfiable in $\mathfrak M$.

\smallskip

%For a set $\Gamma$ of formulas, we write $\mathfrak M,x \models \Gamma$ if $\mathfrak M,x \models \varphi$ for every $\varphi\in\Gamma$, in which case we also say that $\mathfrak M$ \emph{satisfies} $\Gamma$.
A frame $\mathfrak F$ \emph{satisfies} $\Gamma$ if there is a model $\mathfrak M$ based on
 $\mathfrak F$ satisfying $\Gamma$.
 %$\mathfrak M,x \models \Gamma$ for some model based on $\mathfrak F$ and point $x$ in $\mathfrak M$.
Further, $\varphi$ is \emph{valid} in $\mathfrak F$, written $\mathfrak F \models \varphi$, if $\mathfrak M,x \models \varphi$, for any model $\mathfrak M$ based on $\mathfrak F$ and any $x \in W$. We call $\mathfrak F$ a \emph{frame for} a logic $L$ and write $\mathfrak F \models L$ if $\mathfrak F \models \varphi$, for all $\varphi \in L$. 
Conversely, any class $\mathcal{S}$ of general frames \emph{determines} the modal logic $\Log\, \mathcal{S} = \{\varphi \mid \forall \mathfrak F \in \mathcal{S} \ \mathfrak F \models \varphi\}$. We write $\Log(\mathfrak F)$ for $\Log(\{\mathfrak F\})$. 
%
%{\color{red}
%If $\mathfrak F$ is a transitive rooted\nz{descriptive? this is true for K4.3, no clue about K4}\nb{we will see, if it is more complex, we can put this bit in front of L.~\ref{l:logs}} frame, then we let
%
%\[
%\rLog(\mathfrak F) = \{\varphi\mid \mathfrak M,x\models\varphi\mbox{ for all (any) roots $x$ of $\mathfrak F$ and all models $\mathfrak M=(\mathfrak F,\mathfrak v)$}\}.
%\]
%
%Note that, in general, $\rLog(\mathfrak F)$ is not a normal modal logic because it is not necessarily closed under necessitation (it is known as a \emph{quasi-normal} modal logic).}

%If $x$ is a point in $\mathfrak F$, then $\Log(\mathfrak F,x) = \{\varphi\mid \mathfrak M,x\models\varphi\mbox{ for all models $\mathfrak M$ based on $\mathfrak F$}\}$.\nz{Note that $\Log(\mathfrak F,x)$ is a \emph{quasi-normal} modal logic as it is not necessarily closed under necessitation.} 

A set $\Gamma$ of formulas is $L$-\emph{consistent} if
$(\bigwedge\Gamma'\to\bot)\notin L$, for any finite $\Gamma'\subseteq\Gamma$.
%
%such $\Gamma$
%is $L$-\emph{maximal} if $\varphi \in \Gamma$ or $\neg\varphi \in \Gamma$, for any formula $\varphi$. The \emph{canonical frame} for consistent $L$ is $\mathfrak F_L=(W_L,R_L,\INT_L)$, where $W_L$ is the set of all $L$-maximal sets of formulas, $\Gamma R_L\Delta$ iff $\Diamond\Delta\subseteq\Gamma$, and $\INT_L=\bigl\{\{\Gamma\in W_L\mid p_i\in\Gamma\} \mid i<\omega\bigr\}$. The \emph{canonical model for} $L$ is $\mathfrak M_L=(\mathfrak F_L,\mathfrak{v}_L)$, where $\mathfrak v_L(p_i)=\{\Gamma\in W_L\mid p_i\in\Gamma\}$, for $i<\omega$.
%
We require the following well-known fact (see,  e.g.,~\cite[Section 8.6]{DBLP:books/daglib/0030819}): 
%\cite[Theorems 8.3, 8.4]{DBLP:books/daglib/0030819}): 

\begin{lemma}\label{univframe}
%For any consistent modal logic $L$ and any signature $\sigma$,
%
%\begin{itemize}
%\item[$(a)$]
%$\mathfrak F_L$ is descriptive and generated by $\{\mathfrak v_L(p_i)\mid i<\omega\}$\textup{;}
%
%\item[$(b)$]
%for every formula $\varphi$ and every $\Gamma\in W_L$, we have
%$\mathfrak M_L,\Gamma\models\varphi$ iff $\varphi\in \Gamma$\textup{;}
%
%\item[$(c)$]
%$\mathfrak F_L\models L$\textup{;}
%
%\item[$(d)$]
%any $L$-consistent set $\Sigma$ of $\sigma$-formulas is satisfied in a model $\mathfrak M$ based on a finitely $\mathfrak M$-generated descriptive frame for $L$.
%\end{itemize}
For any modal logic $L$ and any finite signature $\sigma$, if $\Sigma$ is an $L$-consistent set of $\sigma$-formulas, then $\Sigma$ is satisfiable in a $\sigma$-model $\mathfrak M$ based on a finitely $\mathfrak M$-generated descriptive frame for $L$. 
\end{lemma}
%
%\begin{proof}
%We give a sketch of the proof of $(d)$ only: As $\Sigma$ is $L$-consistent, it is contained in some $\Gamma\in W_L$, and so
%$\mathfrak M_L,\Gamma\models\Sigma$ by (b). Take the subalgebra $\mathfrak A$ of $\mathfrak F_L^+$ that is finitely generated
%by the $\mathfrak v_L(p_i)$, for $p_i\in\sigma$. Then the ultrafilter frame $\mathfrak A_+$ of $\mathfrak A$ is a finitely generated descriptive frame for $L$ satisfying $\Sigma$. Moreover, $\mathfrak A_+$ is $\mathfrak M_+$-generated
%for the model $\mathfrak M_+=(\mathfrak A_+,\mathfrak{v}_+)$, where
%$\mathfrak v_+(p_i)=\{u\mid \mathfrak v_L(p_i)\in u\}$ for $p_i\in\sigma$, and $\mathfrak v_+(p_i)=\emptyset$ otherwise.
%\end{proof}

%Denote by $\DFr\, L$ and $\KFr\, L$ the classes of all descriptive and Kripke frames for $L$, respectively.
%
By Lemma~\ref{univframe}, every modal logic $L$ is determined by the class of all descriptive frames for $L$.
%$L = \Log\,\DFr\, L$,  for every modal logic $L$. 
A logic $L$ is \emph{Kripke complete} if $L$ is determined by the class of all Kripke frames for $L$. 
%$L = \Log\,\KFr\, L$. 
$L$ is \emph{d-persistent} (aka \emph{canonical}) if $(W,R,\INT) \models L$ implies $(W,R) \models L$, for any descriptive frame $(W,R,\INT)$. $L$ has the \emph{finite model property} (\emph{fmp}) if it is determined by its finite (Kripke) frames. 

The smallest logic $\KFT$ we are interested in is \mbox{d-persistent}; its descriptive and Kripke frames $\mathfrak F=(W,R,\INT)$ are \emph{transitive}  and \emph{weakly connected\/}, that is,
\begin{align*}
& \forall x,y,z \in W\, (xRy \land yRz \to xRz),\\
& \forall x,y,z \in W\, (xRy\land xRz \to y = z \lor yRz \lor zRy).
\end{align*}
$\GLT$, on the contrary, is not d-persistent yet has the fmp. In fact, all extensions of $\KFT$ are Kripke complete~\cite{DBLP:journals/jsyml/Fine74}.

%Given frames $\mathfrak F = (W,R,\INT)$ and $\mathfrak F' = (W',R',\INT')$, a surjection $f \colon W \to W'$ is a \emph{p-morphism} from $\mathfrak F$ onto $\mathfrak F'$ if, for all $x,y \in W$ and $X' \in \INT'$, 
%%
%\begin{itemize}
%\item[--] $xRy$ implies $f(x) R' f(y)$,
%
%\item[--] $f(x) R' y$ implies that there is $z \in W$ with $xRz$ and $f(z) = y$,
%
%\item[--] $f^{-1}(X') \in \INT$.
%\end{itemize}
%%
%%
%If $\mathfrak F$ and $\mathfrak F'$ are rooted and $f$ is such that, for every root $r'$ in $\mathfrak F'$, there is a root $r$ in $\mathfrak F$
%with $f(r)=r'$, we call $f$ \emph{\rp\/}. If there is a \rp{} p-morphism from $\mathfrak F$ onto $\mathfrak F'$,
%we write $\mathfrak F \rightarrowtail \mathfrak F'$.  
%The following is well-known~\cite[Section 8.5]{DBLP:books/daglib/0030819}:
%
%\begin{lemma}\label{pmorphism}
%If $\mathfrak F \rightarrowtail \mathfrak F'$, then
%$\Log(\mathfrak F)\subseteq\Log(\mathfrak F')$ and $\rLog(\mathfrak F)\subseteq\rLog(\mathfrak F')$. 
%\end{lemma}

%****** rooted 

A frame $\mathfrak F' = (W',R',\INT')$ is a \emph{subframe} of a frame $\mathfrak F = (W,R,\INT)$ if $W' \subseteq W$, $R' = \rest{R}{W'}=R\cap(W'\times W')$, and $\INT' \subseteq \INT$. 
For every internal set $V\in\INT$, the frame
$\rest{\mathfrak F}{V}=\bigl(V,\rest{R}{V},\rest{\INT}{V}\bigr)$ with $\rest{\INT}{V}=\{V\cap X\mid X\in\INT\}$ is a subframe of $\mathfrak F$.
For a model $\mathfrak M=(\mathfrak F,\mathfrak v)$, we let 
$\rest{\mathfrak M}{V}=(\rest{\mathfrak F}{V},\rest{\mathfrak v}{V})$, where
$\rest{\mathfrak v}{V}(p)=V\cap\mathfrak v(p)$. 
Given a frame $\mathfrak F = (W,R,\INT)$ with transitive $R$ and a point $x\in W$, we define the frame 
$\mathfrak F_x = (W_x,R_x,\INT_x)$ by taking 
$W_x = \{y \in W \mid xR^+y\}$, where $R^+$ is the \emph{reflexive closure} of $R$ (that is, $R^+=R\cup\{(y,y) \mid y\in W\}$),
$R_x = \rest{R}{W_x}$, and $\INT_x = \rest{\INT}{W_x}$. 
We call $\mathfrak{F}$ \emph{rooted} if $\mathfrak F = \mathfrak F_x$, for some $x \in W$, in which case $x$ is called a \emph{root} of $\mathfrak F$.
Note that $\mathfrak F_x$ is not necessarily a subframe of $\mathfrak F$, but we have:
\begin{equation}\label{rootdescr}
\mbox{if $\mathfrak F$ is descriptive and transitive, then $\mathfrak F_x$ is descriptive as well.}
\end{equation}
Indeed, suppose $\mathfrak F = (W,R,\INT)$ is descriptive and $x\in W$.
Conditions \textbf{(dif)} and \textbf{(\tight)} for $\mathfrak F_x$ are straightforward and left to the reader. 
To establish  \textbf{(com)},
consider any $\mathcal{X}_x \subseteq \mathcal{P}_x$ with the fip. Then 
\begin{equation*}
\mathcal{X} = \{V \in \mathcal{P} \mid V \cap W_x \in \mathcal{X}_x\} \cup \{V \in \mathcal{P} \mid W_x \subseteq V\}
\end{equation*}
also has the fip, and so $\bigcap\mathcal{X} \ne \emptyset$. 
To prove that $\bigcap\mathcal{X}_x \ne \emptyset$, it suffices to show that $\bigcap \{V \in \mathcal{P} \mid W_x \subseteq V\} \subseteq W_x$.
To this end,
suppose on the contrary that $y \in \bigcap \{V \in \mathcal{P} \mid W_x \subseteq V\}$ and $y \notin W_x$. Then \textbf{(dif)} and \textbf{(\tight)} give  $Z,Y \in \INT$ such that $x\in Z$, $y\notin Z$, $y \in Y$, and $x \in \Box\neg Y$. It follows that $Z\cup\neg Y \in \INT$, $W_x \subseteq Z\cup\neg Y$, and
so $y \in Z\cup\neg Y$, which is a contradiction.

\subsection{The structure of linear finitely-generated descriptive frames}\label{ss:structure} 
From now on, all frames $\mathfrak F=(W,R,\INT)$ are assumed to be  rooted frames for $\KFT$, so their relation $R$ is always transitive and 
%weakly connected. Every rooted $\mathfrak F$ for $\KFT$ is 
\emph{connected}: 
\begin{equation}\label{connected}
\forall x,y \in W \, \big(xRy \lor x=y \lor yRx \big). 
\end{equation}
A \emph{cluster} in $\mathfrak F$ is any set of the form $C(x) = \{x\} \cup \{ y \in W \mid xRy \land y R x\}$ with $x\in W$. If $x$ is \emph{irreflexive}, i.e., $xRx$ does not hold, $C(x)$ is called a \emph{degenerate cluster} and depicted as $\bullet$; a  \emph{reflexive} $x$ (for which $xRx$) is depicted as $\circ$. A non-degenerate cluster with $k \ge 1$ (reflexive) points is depicted as $\clusterk$. 
The next  example will be used many times in what follows. 

\begin{example}\label{k-omega}\em
Consider the frame $\mathfrak F=(W_k,R_{k\bullet},\INT_k)$, where $0 < k < \omega$,
%$k > 0$, 
%
\begin{align*}
& W_k=A_k\cup\{\yy_n\mid  n < \omega \},\qquad A_k= \{a_0,\dots, a_{k-1} \},\\
& xR_{k\bullet} y\ \text{ iff }\ \text{either } x = a_i \text{ or } x=\yy_n,\ y=\yy_m\text{ and }m<n,
\end{align*}
%$$%\begin{equation*}
%W_k = \{a_0,\dots, a_{k-1} \} \cup \{n \mid n < \omega \},\qquad 
%W = \{a_0,\dots, a_{k-1} \} \cup\omega,\qquad
%xR_{k\bullet} y \text{ iff either } x = a_i \text{ or } y<x<\omega,
%$$%\end{equation*}
%
and $\INT_k$ is generated by the sets 
$X_i = \{a_i\} \cup \{\yy_n\mid n < \omega,\ n \equiv i \ (\text{mod}\ k)  \}$, for $i <k$, and $\{\yy_n\}$, for $n < \omega$. 
(For instance, $\INT_1$ consists of all finite subsets of $\{\yy_n\mid n<\omega\}$ and their complements in $W_1$.)
The underlying Kripke frame $(W_k,R_{k\bullet})$ is shown in the picture below, where all $\ast$ are $\bullet$. 
%Fig.~\ref{tad},
%
%
%\begin{figure}[h]
%\centering
%   \includegraphics[scale=0.7]{../Pics/tad}

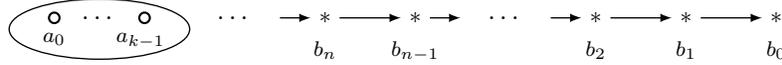
\begin{center}
\begin{tikzpicture}[>=latex,line width=0.5pt,xscale = 1.2,yscale = 1]
\node[point,scale = 0.7,label=below:{\footnotesize $a_0$}] (a0) at (0,0) {};
\node[]  at (.5,0) {$\dots$};
\node[point,scale = 0.7,label=below:{\footnotesize $a_{k-1}\ $}] (akk) at (1,0) {};
\draw[] (.5,-.15) ellipse (1 and .4);
\node[]  at (2,0) {$\dots$};
\node[scale = 0.9,label=below:{\footnotesize $\yy_n$}] (bn) at (3,0) {$\ast$};
\node[scale = 0.9,label=below:{\footnotesize $\yy_{n-1}$}] (bnn) at (4,0) {$\ast$};
\node[]  at (5,0) {$\dots$};
\node[scale = 0.9,label=below:{\footnotesize $\yy_{2}$}] (b2) at (6,0) {$\ast$};
\node[scale = 0.9,label=below:{\footnotesize $\yy_{1}$}] (b1) at (7,0) {$\ast$};
\node[scale = 0.9,label=below:{\footnotesize $\yy_{0}$}] (b0) at (8,0) {$\ast$};
\draw[->] (2.5,0) to (bn);
\draw[->] (bn) to (bnn);
\draw[->] (bnn) to (4.5,0);
\draw[->] (5.5,0) to (b2);
\draw[->] (b2) to (b1);
\draw[->] (b1) to (b0);
\end{tikzpicture}
\end{center}
%
 %\caption{Kripke frame underlying a `tadpole' descriptive frame.}\label{tad}
%\end{figure}
 %$\{0\}=\Box^{\mathfrak F}\emptyset$ and 
%$\{n+1\}=\Diamond^{\mathfrak F}\{n\}\cup\Box^{\mathfrak F}(\{0\}\cup\dots\cup\{n\})$, we have $\{n\}\in\INT_k$, for all $n<\omega$.
It is not hard to see that
\begin{equation}\label{infiniteincl}
\mbox{
for any $X\in\INT_k$, $X$ is infinite iff $A_k\cap X\ne\emptyset$,}
%$a_i\in X$ for some $i<k$, }
\end{equation}
and so  $A_k\notin\INT_k$.
%$\{a_0,\dots, a_{k-1} \}\notin\INT_k$.
%(For instance, $\INT_1$ consists of finite subsets of $\{\yy_n\mid n<\omega\}$ and their complements in $W_1$.)
For every nonempty $X\in\INT_k$, the set $\Diamond^{\mathfrak F}X$ is cofinite in $W_k$. Using these observations, it is readily checked that $\mathfrak F$ is a descriptive frame;
%for $\GLT$; 
we denote it by $\chain{k}{\bullet}$.
Clearly, $\chain{k}{\bullet}$ is $\mathfrak M$-generated for $\mathfrak M$ with $\mathfrak{v}(p_i)=X_i$ if $i<k$, and $\mathfrak{v}(p_i)=\emptyset$ otherwise. 
The descriptive frame $(W_k,R_{k\circ},\INT_k)$ with $R_{k\circ}=R_{k\bullet}\cup\{ (\yy_n,\yy_n)\mid n<\omega\}$ is denoted  by $\chain{k}{\circ}$; $(W_k,R_{k\circ})$ looks like in the picture above,
%in Fig.~\ref{tad}, 
with all $\ast=\circ$. Note that $\chain{k}{\bullet}\models \GLT$ but $\chain{k}{\circ} \not\models \GLT$, cf.~Example~\ref{e:canform}~$(a)$.
\hfill $\dashv$
\end{example}

The next lemma, originating in~\cite{DBLP:journals/jsyml/Fine74}, will play a key role in our subsequent constructions. 
Let $\mathfrak{M}$ be a model based on a 
rooted frame $\mathfrak{F} = (W,R,\INT)$ for $\KFT$, 
and let $\Gamma$ be a set of formulas. 
A point $x\in W$ is called 
\emph{$\Gamma$-maximal in} $\mathfrak{M}$ if $\mathfrak{M},x \models \Gamma$, and whenever $xRy$ and $\mathfrak{M},y\models \Gamma$, then $yRx$. We denote by $\max_{\mathfrak{M}} \Gamma$ the set of all $\Gamma$-maximal points in $\mathfrak{M}$. 
%\emph{maximal for $\Gamma$ at} $\mathfrak{M},x$ if $xRy$, $\mathfrak{M},y \models \Gamma$ and whenever $yRy'$ and $\mathfrak{M},y'\models \Gamma$, then $y'Ry$. We denote by $\max_{\mathfrak{M},x} \Gamma$ the set of all maximal points for $\Gamma$ at $\mathfrak{M}, x$.\nz{assume connected frames, maximal for $\Gamma$, $\max_{\mathfrak{M}} \Gamma$, remove $x$?} 

\begin{lemma}\label{maxpoints}
Suppose $\Gamma$ is a set of formulas and 
%$\mathfrak{M} =(\mathfrak{F},\mathfrak{v})$ 
$\mathfrak M$ a model based on a 
rooted descriptive frame $\mathfrak{F}= (W,R,\INT)$ for $\KFT$. 
%descriptive transitive frame $\mathfrak{F}= (W,R,\INT)$. 
Then the following hold\textup{:}
%
%\item[$(a)$] there is $x\in W$ with $\mathfrak{M},x \models \Gamma$ iff there is $x\in W$ with $\mathfrak{M},x \models \Gamma'$ for every finite $\Gamma' \subseteq \Gamma$ \textup{(}in other words, if $\Gamma$ is finitely satisfiable in $\mathfrak M$, then $\Gamma$ is satisfiable in $\mathfrak M$\textup{)}\textup{;}
%\item[$(a)$] If $\Gamma$ is finitely satisfiable in $\mathfrak M$, then $\Gamma$ is satisfiable in $\mathfrak M$.

\smallskip
\noindent
{\rm\bf (modal saturation)}
if $\mathfrak M, x\models \Diamond \bigwedge\Gamma'$ for every finite $\Gamma' \subseteq \Gamma$, then there is $y$ with $xRy$ and $\mathfrak M, y \models \Gamma$\textup{;} 
%$\mathfrak M, x \models \Diamond \Gamma$, \\
%\mbox{}\hspace*{0.7cm}in the sense that there is $y$ with $xRy$ and $\mathfrak M, y \models \Gamma$\textup{;}

\smallskip
\noindent
{\rm\bf (maximal points)}
if there is $x$ with $\mathfrak M, x \models \Gamma$, then $\max_{\mathfrak{M}}\Gamma \ne \emptyset$.
%if there is $y$ with $xRy$ and $\mathfrak M, y \models \Gamma$, then $\max_{\mathfrak{M},x}\Gamma \ne \emptyset$.
\end{lemma}

Given a rooted frame $\mathfrak F = (W, R,\INT)$ for $\KFT$, let 
%\nb{changed, see also start of Sec.~\ref{ss:structure}, and proof of L.~\ref{int-prop}} 
$R^s = \{(x,y) \in R \mid (y,x) \notin R\}$ be the \emph{strict} $R$-accessibility in $\mathfrak F$.  
Sometimes it will be convenient to view $(W,R)$ as a strict linear order
%
%\footnote{An irreflexive and transitive relation $<$ is a \emph{strict linear order} if, for all $x\ne y$, we have either $x<y$ or $y<x$.}
%
$\FC = (W_c, <_R)$ of clusters, where $W_c = \{C(x) \mid x \in W\}$ and $C(x) <_R C(y)$ iff $xR^sy$. 
A cluster $C$ is \emph{final} in $\mathfrak F$ if 
there is no cluster $C'$ with $C<_R C'$. 
A cluster $C$ is a \emph{root cluster} if there is no cluster $C'$ with $C'<_R C$, in which case 
$C<_R C'$ for every $C'\ne C$ in $\mathfrak F$; the root cluster in $\mathfrak F$ is unique. A cluster $C'$ is an \emph{immediate successor}  of a cluster $C$ in $\mathfrak F$ if $C<_R C'$
and there is no $C''$ with $C<_R C''<_R C'$, in which case  
$C$ is an \emph{immediate predecessor} of $C'$.
%We require the following four types of intervals:\nb{non-closed only used in the proof of L.\ref{int-prop} Can we rewrite that?}
%
%\begin{align*}
%& (C,C')=\bigcup\{ D\mid C<_R D <_R C'\},\qquad [C,C')=(C,C')\cup C,\\
%& (C,C']=(C,C')\cup C',\qquad [C,C']=(C,C']\cup C.
%\end{align*} 
%
%Intervals of the form $[C,C']$ are called \emph{closed\/}. 
%
%Until the end of Section~\ref{section:K4.3}, every frame $\mathfrak{F} = (W,R,\INT)$ is assumed to be a rooted frame for $\KFT$.  
%Recall from Section~\ref{modular} that $(W,R)$ can be regarded as the strict linear order $\FC = (W_c, <_R)$ of clusters, where $W_c = \{C(x) \mid x \in W\}$ and $C(x) <_R C(y)$ iff $xR^sy$. 
% transitive and connected. 
%A\nz{moved up} cluster $C'$ is an \emph{immediate successor}  of a cluster $C$ in $\mathfrak F$ if $C<_R C'$
%and there is no cluster $C''$ with $C<_R C''<_R C'$, in which case  
%$C$ is an \emph{immediate predecessor} of $C'$.
% if $C'<_R C$
%and there is no cluster $C''$ with $C'<_R C''<_R C$. 
%$C R C'$ and, for any cluster $C''$, whenever $C R C'' R C'$ then $C''=C$ or $C''=C'$. 
A sequence $C_n$, $n<\omega$, of clusters in $\FC$ is an
\emph{infinite ascending chain} if $C_n <_R C_{n+1}$, for all $n < \omega$. $\FC$ is \emph{converse well-founded} if it has no infinite ascending chain of clusters. 
%If $C_n$, $n<\omega$, are such that $C_{n+1} <_R C_{n}$ for all $n < \omega$, they form an \emph{infinite descending chain} of clusters in $\FC$. 

The next lemma follows from, e.g., the more general~\cite[Theorems~10.34, 10.35]{DBLP:books/daglib/0030819}:

\begin{lemma}\label{lem:descr0}
If $\mathfrak F$ is a rooted $n$-generated descriptive frame for $\KFT$, for some $n<\omega$, then 
\begin{itemize}
\item[$(a)$] $\FC$ is converse well-founded, and so
the strict linear order $\FC^{-1} = (W_c, >_R)$ is isomorphic to some ordinal\textup{;}

\item[$(b)$] every cluster in $\mathfrak F$ has at most $2^n$ points.%\nb{$(c)$ moved later}
%\item[$(c)$] $\FC$ is countable, and so
%the strict linear order $\FC^{-1} = (W_c, >_R)$ is isomorphic to some countable ordinal.
\end{itemize}
\end{lemma}
\begin{proof}
Let $\mathfrak F=(W,R,\INT)$, let $\le_R$ be the reflexive closure of $<_R$, and let $\mathcal G$ be a finite set generating $\INT$ with $|\mathcal G| = n$. For $x,y\in W$, we write $x\sim_{\mathcal G} y$ in case $x\in G$ iff $y\in G$,
for all $G\in\mathcal G$, and denote by $[x]_{\mathcal G}$ the $\sim_{\mathcal G}$-class of $x$.
Clearly, $\bigl|\bigl\{[x]_{\mathcal G}\mid x\in W\bigr\}\bigr|\le 2^{|{\mathcal G}|}=2^n$.

$(a)$
Suppose on the contrary that $C(x_i)$, $i<\omega$, is an infinite ascending chain in $\FC$.
Call $x \in W$ a \emph{middle-point} if $C(x_0)\le_R C(x)\le_R C(x_i)$, for some $i<\omega$.
Let $V_x=\{[y]_{\mathcal G}\mid \mbox{$y$ a middle-point with $xR y$}\}$. Since $V_x\supseteq V_y$ whenever $xR y$ and each $V_x$ is
finite, there is $m<\omega$ such that $V_y=V_{x_m}$, for every middle-point $y$ with $C(x_m)\le_R C(y)$.
By induction on the construction of $X\in\INT$ from the generators in $\mathcal G$, it is readily seen that
\begin{multline}\label{noasc}
\mbox{if $y,z$ are middle-points, $C(x_m)\le_RC(y)$, $C(x_m)\le_RC(z)$, and $y\sim_{\mathcal G}z$,}\\
\mbox{then $y\in X$ iff $z\in X$, for all $X\in\INT$.}
\end{multline}
(Indeed, the only non-trivial case is when $X=\Diamond^{\mathfrak F}Y$, $yRz$ and $y\in\Diamond^{\mathfrak F}Y$.
Then there is $x\in Y$ with $yRx$. If $zRx$, we are done. Otherwise, $x$ is a middle-point.
As $V_y=V_z$, there is a middle-point $x'$ with $zRx'$ and $x\sim_{\mathcal G} x'$. By IH, $x'\in Y$.)
As there are finitely many $\sim_{\mathcal G}$-classes, 
there exist $k\ne\ell\geq m$ such that $x_k\sim_{\mathcal G}x_\ell$, and so
$x_k\in X$ iff $x_\ell\in X$, for all $X\in\INT$, by \eqref{noasc}. But this contradicts {\bf (dif)}.

$(b)$
It is straightforward to show that if $C(x)=C(y)$ and $x\sim_{\mathcal G}y$, then $x\in X$ iff $y\in X$, for all $X\in\INT$.
So by {\bf (dif)},
every cluster in $\mathfrak F$ has $\le 2^{|{\mathcal G}|}$ points.
\end{proof}

%$(c)$
%\textcolor{red}{Suppose\nb{next try} $\gamma$ is the order type of $\FC^{-1}$, that is, $W_c=\{C_\alpha\mid \alpha<\gamma\}$ and $C_\alpha<_R C_\beta$ iff $\beta<\alpha<\gamma$. Let $Z$ be a set of $<\gamma$ successor ordinals with $|Z|=|\gamma|$. For every $\alpha\in Z$, choose some $y_\alpha\in C_\alpha$. By {\bf (ref)}, for every $\alpha\in Z$, there is $X_\alpha\in\INT$ with $y_\alpha\in X_\alpha$ but there is no $x\in C_{\alpha-1}$ with $x\in\Diamond^{\mathfrak F} X_\alpha$, and so there is no $\beta\in Z$, $\beta<\alpha$ with $y_\beta\in X_\alpha$. (If $\alpha-1\in Z$ and $C_{\alpha-1}=\{x\}$ is degenerate, by {\bf (diff)} we can choose $X_\alpha$ such that $x\notin X_\alpha\cup\Diamond^{\mathfrak F} X_\alpha$.) As $\mathfrak F$ is finitely generated, $\INT$ is countable. Therefore, $Z$ is countable, and so $\FC$ is countable.}
%
%The types of $x\ne y$ should be different, and if $xRy$ then every $\Diamond \varphi$ true at $y$ is true at $x$;   so we can't have more than countably-many types?
%As $\mathfrak F$ is finitely generated, $\INT$ is countable.
%By {\bf (ref)}, for every $y\in W$, there is $X_y\in\INT$ with $y\in X_y$ but there is no $x\in X_y$ with $yR^s x$.\nz{???}
%So if $y_\alpha\in W$, $\alpha<\kappa$, are such that $y_\alpha R^s y_\beta$ if $\alpha <\beta<\kappa$, then
%$\kappa$ is countable. Thus, $\FC$ is countable. 

%
%\textcolor{red}{Note that rooted frames for $\KFT$ has a unique root cluster.  Also,}\nb{added}

Note 
that the existence of maximal points (Lemma~\ref{maxpoints}) in models based on rooted finitely generated descriptive frames 
for $\KFT$ also follows from Lemma~\ref{lem:descr0}. 
Another consequence is that
such a frame $\mathfrak F$ contains a unique final cluster, and any non-root cluster in $\mathfrak F$ has an immediate predecessor. 
%We refer to clusters that are images by the above isomorphism of a non-zero limit ordinal as \emph{limit clusters\/}.
%
If \mbox{$\FC^{-1}= (W_c, >_R)$} is isomorphic to an ordinal $\gamma$ and $\alpha\le\gamma$, we denote by $C_\alpha^{\mathfrak F}$ the cluster that is the image of $\alpha$  under this isomorphism. If $\alpha$ is a non-zero limit ordinal, we call $C_\alpha^{\mathfrak F}$ a \emph{limit cluster\/}.
A non-final cluster is a limit cluster iff it does not have an immediate successor.
By \textbf{(dif)} and Lemma~\ref{lem:descr0}~$(b)$, we also have:
\begin{lemma}\label{l:defpoints}
If $\mathfrak F = (W,R,\INT)$ is a rooted finitely generated descriptive frame for $\KFT$ and $C\in\INT$, for some cluster $C$, then $\{x\}\in\INT$, for every $x\in C$.
\end{lemma}
%
%\begin{proof}
%By \textbf{(dif)}, for all $x,y\in C$, $x\ne y$, there exist $A_{xy}\in\INT$ such that $x\in A_{xy}$ and $y\notin A_{xy}$. Then $\{x\}=C\cap\bigcap_{y\in C,\,y\ne x}A_{xy}$, for every $x\in C$. As $C$ is finite by Lemma~\ref{lem:descr0}~$(c)$, $\{x\}\in\INT$ follows.
%\end{proof}

%A cluster $C$ is the \emph{limit of} an infinite descending chain  of clusters $C_n$,  $n<\omega$, in $\FC$ if  $C<_R C_n$ for all $n<\omega$, and there is no cluster $C'$ with $C<_R C'$ and $C'<_R C_n$ for all $n<\omega$. The next property follows from the fact that the cofinality of any countable limit ordinal is the cardinality $\aleph_0$ of $\omega$ \cite{Jech03}:
%
%\smallskip
%\noindent
%{\bf (limit)} 
%each limit cluster is the limit of some infinite descending chain of clusters.

%\smallskip

Now, suppose $\mathfrak M$ is a model based on a rooted finitely $\mathfrak M$-generated descriptive frame $\mathfrak{F}= (W,R,\INT)$ for $L\supseteq\KFT$.
Given a formula $\mu$, a cluster $C$ is called \emph{$\mu$-maximal in $\mathfrak{M}$} if
%if there is $x\in C$ such that $\mathfrak{M},x\models \mu$ but $\mathfrak{M},y\models \neg\mu$ whenever $xR^s y$
%(that is, 
there is a point in $C$ that is $\{\mu\}$-maximal in $\mathfrak M$. 
%$C\cap\max_{\mathfrak M}\{\mu\}\ne\emptyset$). 
Further, $C$ is \emph{maximal in} $\mathfrak{M}$ if it is $\mu$-maximal in $\mathfrak M$,  for some $\mu$, and $C$ is \emph{$\sigma$-maximal in} $\mathfrak{M}$, for a signature $\sigma$, if there is such a $\sigma$-formula $\mu$. 
%A cluster $C$ is \emph{definable} in $\mathfrak M$ if $C = \mathfrak v(\varphi)$, for some formula $\varphi$. 
Every definable in $\mathfrak M$ cluster is clearly maximal in $\mathfrak M$.
The next lemma says that the converse is also true:

\begin{lemma}\label{lem:descr'}
Suppose $\mathfrak{M}$ is a model based on a rooted  finitely $\mathfrak M$-generated descriptive frame $\mathfrak F = (W, R,\INT)$ for $\KFT$. Then
\begin{itemize}
%\item[$(a)$] every cluster in $\mathfrak F$ has at most $2^n$ points\textup{;}
%
%\item[$(b)$] $\mathfrak F$ does not contain strictly ascending chains of clusters\textup{;} in particular, there is a final cluster in $\mathfrak F$ and every cluster in $\mathfrak F$ has an immediate predecessor\textup{;} 
%
\item[$(a)$] every degenerate cluster in $\mathfrak F$ is maximal in $\mathfrak{M}$\textup{;}

\item[$(b)$] a cluster is maximal in $\mathfrak{M}$ iff either it is final or has an immediate successor\textup{;}

\item[$(c)$] a cluster is definable in $\mathfrak{M}$ iff it is maximal in $\mathfrak{M}$.
\end{itemize}
So limit clusters are not definable and not degenerate, while every other cluster
%\nb{definable points moved to L.\ref{l:defpoints}; item $(d)$ moved to L.\ref{l:closedintdef}} 
is definable in $\mathfrak M$.
%Also, 
%%
%\begin{itemize}
%\item[$(d)$]
%every interval $[C,C']$ in $\mathfrak F$ with a non-limit cluster $C'$ is definable in $\mathfrak M$. 
%%\nb{moved here, was separate next lemma}
%\end{itemize}
%
\end{lemma}
\begin{proof}
%As $\mathfrak F$ is finitely generated, $\mathfrak F$ is $\mathfrak M$-generated.
$(a)$ 
If $C(x)$ is degenerate, then $\Diamond t_{\mathfrak M}(x) \not\subseteq t_{\mathfrak M}(x)$ by {\bf (\tight)}.
So there is a formula $\mu$ with $\mathfrak M,x\models\mu$ but $\mathfrak M,x\not\models\Diamond\mu$.

$(b,\Rightarrow)$ Let $C(x)$ be maximal in $\mathfrak M$ 
with $\mathfrak M,x \models \mu$ and $\mathfrak M,y \not\models \mu$ whenever $xR^sy$. 
Suppose $C(x)$ is a limit cluster.
%Then, by {\bf (limit)}, there is an infinite sequence $y_i$, $i<\omega$ of points such that $y_{i+1}R^s y_i$ and $xR^s y_i$, for $i<\omega$, and there is no $y$ with $xR^s yR^s y_i$ for all $i<\omega$.
%Consider 
%
%\[
%\Gamma = \bigcup_{i<\omega}\Diamond t_{\mathfrak M}(y_i) \cup \{\psi \mid \Box\psi \in t_{\mathfrak M}(x)\} \cup \{ \Box\neg\mu \}.
%\]
%
Let $S=\{C\in W_c\mid C(x)<_R C\}$ with $y_C\in C$, for $C\in S$.
Consider
\[
\Gamma = \bigcup_{C\in S}\Diamond t_{\mathfrak M}(y_C) \cup \{\psi \mid \Box\psi \in t_{\mathfrak M}(x)\} \cup \{ \Box\neg\mu \}.
\]
Clearly, $\Gamma$ is finitely satisfiable in $\mathfrak M$, and so, by {\bf (com)}, $\Gamma \subseteq t_{\mathfrak M}(y)$, for some $y$. 
Thus, by {\bf (\tight)},  $x R y R y_C$ for all $C\in S$, and so $yR^s y_C$ for all $C\in S$ and $yRx$. 
%Thus, by {\bf (ref)},  $x R y R y_i$ for all $i < \omega$, and so $yR^s y_i$ for all $i<\omega$ and $yRx$. 
But we also have $\mathfrak M,y\models\Box\neg\mu$, contrary to $\mathfrak M,x\models\mu$. 
%which is a contradiction.

$(b,\Leftarrow)$ 
The (unique) final cluster is maximal in $\mathfrak M$ for $\top$.
%Let $C(x)$ be final. If $C(x)$ is degenerate, then it is maximal for $\Box\bot$. Otherwise, assuming that $C(x) = \{x_0,\dots,x_m\}$, we take the atomic types $\chi_i$ of $x_i$ in $\mathfrak M$, and set $\varphi = \Box \bigvee_{i \le m} \chi_i$. This formula cannot be true at any point in the immediate predecessor $C(y)$ of $C(x)$ in $\mathfrak F$ because of \textbf{(dif)}.
%
Suppose $C(y)$ is an immediate successor of $C(x)$. If $C(y)$ is degenerate, then $C(y)$ is maximal 
in $\mathfrak M$ by $(a)$, and so there is $\mu$ with $\mathfrak M,y \models \mu \land \neg \Diamond \mu$. 
It follows that $C(x)$ is $\Diamond (\mu \land \neg \Diamond \mu)$-maximal in $\mathfrak M$.
%$\tau = (\mu \land \neg \Diamond \mu) \land \neg \Diamond (\mu\land \neg \Diamond \mu)$ defines $y$ in $\mathfrak M$, and so $\neg \psi \land \Diamond \psi$ is true at $x$ but not it any of its $R^s$-successors. 
If $C(y)$ is non-degenerate and $C(x)$ is not maximal in $\mathfrak M$, then $\Diamond t_{\mathfrak M}(x) \subseteq t_{\mathfrak M}(y)$, and so $yRx$ by {\bf (\tight)}, contrary to $xR^s y$.

$(c,\Leftarrow)$ Let $C(x)$ be $\mu$-maximal in $\mathfrak M$. If $C(x)$ is degenerate, it is defined by 
$\mu \land \neg \Diamond \mu$. 
If $C(x)$ is the non-degenerate root cluster, then $\Diamond \mu$ defines $C(x)$.
Otherwise, take the immediate predecessor $C(y)$ of $C(x)$.
By $(b)$, $C(y)$ is $\tau$-maximal in $\mathfrak M$, for some $\tau$, so  $\Box^+ \neg \tau \land \Diamond \mu$ defines $C(x)$. 
$(c,\Rightarrow)$ is obvious.
%
%$(d)$ By $(a)$--$(c)$, the non-limit $C'$ is defined in $\mathfrak M$ by some $\gamma$. Let $\delta = \bot$ if $C$ is the root cluster, and let $\delta$ define the immediate predecessor of $C$ in $\mathfrak M$ otherwise (which exists by Lemma~\ref{lem:descr0}~$(a)$ and is definable by $(a)$--$(c)$). Then $[C,C']$ is defined in $\mathfrak M$ by $\neg \Diamond^+ \delta \land \Diamond^+ \gamma$. 
\end{proof}

We require a few important consequences of Lemmas~\ref{lem:descr0} and \ref{lem:descr'}.

\begin{lemma}\label{l:countable}
If $\mathfrak F = (W,R,\INT)$ is a rooted finitely generated descriptive frame for $\KFT$, then 
$W$ is countable.
\end{lemma}
\begin{proof}
By Lemma~\ref{lem:descr0}, it suffices to show that the ordinal $\gamma$ isomorphic to $\FC^{-1}= (W_c, >_R)$ is countable. 
%So suppose $W_c=\{C_\alpha\mid \alpha<\gamma\}$ and $C_\alpha<_R C_\beta$ iff $\beta<\alpha<\gamma$. 
Let $Z=\{\alpha+1\mid \alpha<\gamma,\ \alpha+1\ne\gamma\}$ be the set of successor ordinals $<\gamma$. Then $|Z|=|\gamma|$ and $C_\beta^{\mathfrak F}\in\INT$, for any $\beta\in Z$, by Lemma~\ref{lem:descr'}.
%~$(a)$--$(c)$.
As $\mathfrak F$ is finitely generated, $\INT$ is countable, and so are $Z$ and $W$.
\end{proof}

Given a rooted  finitely $\mathfrak M$-generated descriptive frame $\mathfrak F = (W, R,\INT)$
for $\KFT$, let $m_{\mathfrak F}$ be the 
%smallest 
largest ordinal $\le\omega$ with degenerate  
$C_n^{\mathfrak F}$ for all $n<m_{\mathfrak F}$.
We call the (possibly empty) interval $Z=\bigcup_{n<m_{\mathfrak F}}C_n^{\mathfrak F}$
the \emph{\tail{} of} $\mathfrak F$.
%\nb{moved here from \S\ref{ss:modular}, changed a bit, check (tails only used in \S\ref{ss:small} and L.\ref{l:tails} is only used in L.\ref{l:irrelevant})}
%is the smallest (possibly empty) interval $Z$ in $\mathfrak F$ such that
%
%\begin{itemize}
%\item[--]
%if the final cluster of $\mathfrak F$ is degenerate, then it is included in $Z$\textup{;}
%
%\item[--]
%if $z\in Z$ and a degenerate cluster $\{z'\}$ is the immediate predecessor of $\{z\}$, then $z'\in Z$.
%\end{itemize}
%
We may assume that
$Z=\{z_n\mid n< m_{\mathfrak F}\}$, 
%for some $m_Z\leq\omega$, 
where all $z_n$ are irreflexive and $z_{n}R z_{n-1}$,  $0<n<m_{\mathfrak F}$.
If $Z$ is infinite, then $Z\ne W$ (as $\mathfrak F$ is rooted).
%Note also that
%$\{z_0\}=\Box^{\mathfrak F}\emptyset$ and $\{z_{n}\}=\Diamond^{\mathfrak F}\{z_{n-1}\}\cup\Box^{\mathfrak F}(\{z_0\}\cup\dots\cup\{z_{n-1}\})$, for $0<n<m_{\mathfrak F}$, and so
%
%\begin{equation}\label{fincofinZ}
%\mbox{every finite subset of $Z$ and its complement in $W$ are in $\INT$.}
%\end{equation}
%
If $Z\ne W$, we call $C_{m_{\mathfrak F}}^{\mathfrak F}$ the \emph{\source{} of} $Z$.
In particular, if $Z=\emptyset$, its \source{} is the final (non-degenerate) cluster 
$C_{0}^{\mathfrak F}$;
if $Z\ne W$ and $Z\ne\emptyset$ is finite, its \source{} is the immediate predecessor of 
$C_{m_{\mathfrak F}-1}^{\mathfrak F}=\{z_{m_{\mathfrak F}-1}\}$; and if $Z$ is infinite, its \source{} is the limit 
cluster $C_{\omega}^{\mathfrak F}$.
Thus, by Lemma~\ref{lem:descr'}, 
\begin{equation}\label{facenondeg}
\mbox{the \source{} of a \tail{} is always non-degenerate.}
\end{equation}

\subsection{Building linear models from pieces}\label{modular} %\nb{interval stuff moved to this subsection}

\begin{definition}\label{d:orderedsum}\em
The \emph{ordered sum} $\mathfrak F_{0} \lhd \dots \lhd \mathfrak F_{n-1} = (W, R,\INT)$ of rooted frames $\mathfrak F_{i}= ( W_i,R_i,\INT_i)$, $i<n$, for $\KFT$ with pairwise disjoint $W_i$ is defined by \label{orderedsum}
$$
W = \bigcup_{i<n}W_{i},\ \ 
R = \bigcup_{i<n}R_{i} \cup \hspace*{-1mm} \bigcup_{i<j<n} \hspace*{-1mm} (W_{i}\times
W_{j}),\ \ \INT = \{ X_{0}\cup \dots \cup X_{n-1} \mid X_{i}\in \INT_{i}\}.
$$
It is not hard to see that 
if the $\mathfrak F_i$ are descriptive, then $\mathfrak F_{0} \lhd \dots \lhd \mathfrak F_{n-1}$ is also descriptive.  
If $\mathfrak M_i = (\mathfrak F_i, \mathfrak v_i)$, then $\mathfrak M=\mathfrak M_{0} \lhd \dots \lhd \mathfrak M_{n-1}$ is the model based on $\mathfrak F_{0} \lhd \dots \lhd \mathfrak F_{n-1}$ with the valuation $\mathfrak v(p) = \bigcup_{ i < n} \mathfrak v_i(p)$, for any $p \in \mathcal{V}$.
We call the $\mathfrak M_i$ $\lhd$-\emph{components of} $\mathfrak M$. 
\end{definition}

%\begin{lemma}\label{pmorphism}
%If\nb{root to root TBA} $\mathfrak{F}_{i} \rightarrowtail \mathfrak{G}_{i}$, for $i<n$, then $\mathfrak{F}_{0}\lhd \cdots \lhd \mathfrak{F}_{n-1} \rightarrowtail \mathfrak{G}_{0}\lhd \cdots \lhd \mathfrak{G}_{n-1}$, and so
%
%$$
%\Log(\mathfrak{F}_{0}\lhd \cdots \lhd \mathfrak{F}_{n-1}) \subseteq \Log(\mathfrak{G}_{0}\lhd \cdots \lhd \mathfrak{G}_{n-1}).
%$$ 
%\end{lemma} 
%
%\begin{proof}
%It suffices to notice that if $f_i$ is a p-morphism from $\mathfrak F_{i}$ onto $\mathfrak G_{i}$, for $i < n$, then $f$ defined by taking $f(x) = f_i(x)$ if $x \in W_i$ is clearly a p-morphism from $\mathfrak F$ onto $\mathfrak G$. 
%Moreover, if $f_0$ is \rp, then $f$ is \rp{} as well.
%\end{proof}

%\begin{lemma}
%If\nb{what else? TBA}$\mathfrak{M}_{i}\sim_{\sigma}\mathfrak{M}_{i}'$, $1 \le i\leq n$, then $\mathfrak{M}_{1} \lhd \cdots \lhd \mathfrak{M}_{n} \sim_{\sigma}\mathfrak{M}_{1}' \lhd \cdots \lhd \mathfrak{M}_{n}'$.
%\end{lemma} 

Now, let $\mathfrak F = (W, R,\INT)$ be a rooted frame for $\KFT$.
An \emph{interval in} $\mathfrak{F}$ is any subset $I\subseteq W$ such that $xRyRz$ and $x,z \in I$ imply $y \in I$, for all $x,y,z \in W$. If $I\cap C\ne\emptyset$, for a cluster $C$, then clearly $C\subseteq I$. 
An %\nb{only closed intervals are used, proof of L.\ref{int-prop} re-written}
interval $I$ is \emph{closed} if there are clusters $C,C'$ such that
$I=C\cup C'\cup \bigcup\{ D\mid C<_R D <_R C'\}$, in which case we write $I=[C,C']$.
%
%We require the following four types of intervals:
%%
%\begin{align*}
%& (C,C')=\bigcup\{ D\mid C<_R D <_R C'\},\qquad [C,C')=(C,C')\cup C,\\
%& (C,C']=(C,C')\cup C',\qquad [C,C']=(C,C']\cup C.
%\end{align*} 
%%
%Intervals of the form $[C,C']$ are called \emph{closed\/}. 
%
Given two closed intervals $I,I'$ in $\mathfrak F$, we write $I\intord{\mathfrak F}I'$ if $I$ and $I'$ are disjoint and $xRx'$, for all $x\in I$, $x'\in I'$.
Notice that if $I$ is a closed internal interval in $\mathfrak F$, then $\rest{\mathfrak F}{I}$ is also a rooted frame for $\KFT$. Also, 
if $\mathfrak F$ is descriptive, then $\rest{\mathfrak F}{I}$ is descriptive as well.
And if $\mathfrak F$ is finitely $\mathfrak M$-generated, for some model $\mathfrak M$, then $\rest{\mathfrak F}{I}$ 
is finitely $\rest{\mathfrak M}{I}$-generated. 
%
%--such that $I \in \INT$. In this case, $\mathfrak F\restriction I = (I, R \restriction I, \INT \restriction I)$, where $\INT \restriction I = \{X \cap I \mid X \in \INT\}$, is also a general frame, which is descriptive if $\mathfrak F$ is descriptive. If $\mathfrak F$ can be partitioned into disjoint internal intervals $I_1, \dots, I_n$ such that $i < j$, $x \in I_i$ and $y \in I_j$ imply $xRy$, then $\mathfrak F$ is isomorphic to $\mathfrak F_1\restriction I_1 \lhd \dots \lhd \mathfrak F_n\restriction I_n$. 
%Let $\mathfrak{M}$ be a model based on a descriptive $\mathfrak{F}= (W, R,\INT)$. By an \emph{interval} in $\mathfrak{M}$ we mean any \emph{convex} subset $I\subseteq W$, in which $[x,y]\subseteq I$, for all $x,y\in I$ with $xR^ry$. An interval $I$ is \emph{closed} if $I = [x,y]$, for some $x,y \in W$. An interval $I$ is \emph{definable} if there is a formula $\varphi$\nz{$\delta$?} such that $\mathfrak v(\varphi) = I$. 
%
We clearly have the following:
\begin{lemma}\label{intpart}
Suppose $\mathfrak F= (W, R,\INT)$ is a rooted frame for $\KFT$ and $W$ 
is partitioned as $\{I_j\mid j<n\}$, $n<\omega$, with closed intervals $I_j\in\INT$ and $I_j\intord{\mathfrak F}I_k$ iff $j<k$. Then 
\begin{itemize}
\item[$(a)$]
$\mathfrak F=\rest{\mathfrak F}{I_0}\lhd\dots\lhd\rest{\mathfrak F}{I_{n-1}}$\textup{;}

\item[$(b)$]
if $\mathfrak M$ is a model based on $\mathfrak F$, then
$\mathfrak M=\rest{\mathfrak M}{I_0}\lhd\dots\lhd\rest{\mathfrak M}{I_{n-1}}$.
\end{itemize}
%
%Let $I$ be an interval  in a frame $\mathfrak F$ partitioned as $\{I_j\mid j<n\}$, $n<\omega$, with internal closed intervals $I_j$ in $\mathfrak F$ and $I_j\intord{\mathfrak F}I_k$ iff $j<k$. Then 
%%
%\begin{itemize}
%\item[$(a)$]
%$\rest{\mathfrak F}{I}=\rest{\mathfrak F}{I_0}\lhd\dots\lhd\rest{\mathfrak F}{I_{n-1}}$\textup{;}
%
%\item[$(b)$]
%if $\mathfrak M$ is a model based on $\mathfrak F$, then
%$\rest{\mathfrak M}{I}=\rest{\mathfrak M}{I_0}\lhd\dots\lhd\rest{\mathfrak M}{I_{n-1}}$.
%\end{itemize}
\end{lemma}

%************

\subsection{Canonical formulas}\label{can-for}
To check whether a frame validates a given finitely axiomatisable logic,
%The main complexity result of this article  
%%(more precisely, Theorem~\ref{dperscofinal} and Lemma~\ref{l:fpoly}) 
%relies upon a polynomial-time algorithm deciding whether a given frame of a certain form validates a fixed logic $\KFT \oplus \varphi$. To design this algorithm, 
we use the canonical formulas of~\cite{DBLP:journals/mlq/ZakharyaschevA95,DBLP:journals/mlq/Wolter96,DBLP:books/daglib/0030819,DBLP:journals/sLogica/BezhanishviliB11} whose basic properties are  summarised below in the context of $\KFT$; for more details  consult~\cite[Section 16.3]{DBLP:books/daglib/0030819}.  
%
%Every logic $L = \KFT \oplus \varphi$ can be efficiently represented in the form
%%
%\begin{equation}\label{canon}
%L = \KFT \oplus \{\alpha(\mathfrak G_j,\mathfrak D_j,\bot) \mid j < m_L\}, \quad \text{for some $m_L < \omega$},
%\end{equation}
%%
%
Every logic $L \supseteq\KFT$ can be represented in the form
\begin{equation}\label{canon}
L = \KFT \oplus \{\alpha(\mathfrak G_j,\mathfrak D_j,\bot) \mid j \in J_L\}, \quad \text{for some index set $J_L$},
\end{equation}
where each $\alpha(\mathfrak G_j, \mathfrak D_j,\bot)$ is a  \emph{canonical formula} based on a finite rooted Kripke frame $\mathfrak G_j = (V_j, S_j)$ for $\KFT$ and a (possibly empty) set $\mathfrak D_j\subseteq V_j$ of \emph{irreflexive} non-root points in $\mathfrak G_j$. 
If $L$ is finitely axiomatisable, its canonical axiomatisation \eqref{canon} with finite $J_L$ can be constructed effectively, given any finite set of axioms.

Let $\mathfrak F = (W,R,\INT)$ be any rooted finitely generated descriptive frame for $\KFT$. By Theorem~\ref{lem:descr0}, $\mathfrak F$ contains a unique final cluster, and any non-root cluster in $\mathfrak F$ has an immediate predecessor. 
The formulas $\alpha(\mathfrak G_j, \mathfrak D_j,\bot)$ are defined so that
%for any $w\in W$, we have 
$\mathfrak F\not \models \alpha(\mathfrak G_j, \mathfrak D_j,\bot)$ iff 
there is an injection $f \colon V_j \to W$ such that
the following conditions hold, for all $x,y \in V_j$: 

%\smallskip
%\noindent
%\textbf{(cf$_1$)} there is an injective function $f \colon V_j \to W$ such that, for all $x,y \in V_i$, 

%\smallskip
%\noindent
%\textbf{(cf0)} either the root cluster $C(r)$ of $\mathfrak G_j$ is degenerate and $f(r)=w$,
%or $C(r)$ is\\
%\mbox{}\hspace*{0.82cm} non-degenerate and $wRf(r)$;

\smallskip
\noindent
%\textbf{(cf$_2$)}
\textbf{(cf$_1$)}
$x S_j y$ iff $f(x) R f(y)$ (so $x$ is irreflexive iff $f(x)$ is);
%for all $x,y \in V_i$;

\smallskip
\noindent
%\textbf{(cf$_3$)}
\textbf{(cf$_2$)}
if $C(x)$ is the final cluster in $\mathfrak G_j$, then $C(f(x))$ is the final cluster in $\mathfrak F$;

\smallskip
\noindent
\textbf{(cf$_3$)}
%\textbf{(cf$_4$)}
if $x \in \mathfrak D_j$ and $C(y)$ is the immediate predecessor of $C(x) = \{x\}$ in $\mathfrak G_j$,\\ 
\mbox{}\hspace*{0.8cm} then $C(f(y))$ is the immediate predecessor of $C(f(x))=\{f(x)\}$ in $\mathfrak F$;

\smallskip
\noindent
\textbf{(cf$_4$)}
%\textbf{(cf5)}
$\{f(x)\} \in \INT$.
%for all $x \in V_i$. 

\smallskip
\noindent
%
%\{Further,\nz{new} for $w \in W$, we have $\mathfrak F, w \not \models \alpha(\mathfrak G_j, \mathfrak D_j,\bot)$ iff either a root  $x$ of $\mathfrak G_j$ is irreflexive and $f(x) = w$ or $x$ is reflexive and $w R f(x)$.}
%
%
%Observe that if an $f$ satisfies \textbf{(cf$_1$)}, then it satisfies \textbf{(cf0)} with $w=f(r)$, for any root $r$ of $\mathfrak G_j$. Thus,
%It follows that $\mathfrak F \not \models \alpha(\mathfrak G_j, \mathfrak D_j,\bot)$ iff 
%there is $w$ with $\mathfrak F,w \not \models \alpha(\mathfrak G_j, \mathfrak D_j,\bot)$ iff 
%there is $f$ satisfying \textbf{(cf$_1$)}--\textbf{(cf$_4$)}.}\nb{I think the red text was easier to follow}
%
Intuitively, every frame $\mathfrak F$ with $\mathfrak F \not \models \alpha(\mathfrak G_j, \mathfrak D_j,\bot)$ can be obtained by inserting certain chains of clusters immediately before some clusters $C(x)$ in $\mathfrak G_j$, provided that $x \notin \mathfrak D_j$, and by enlarging some non-degenerate clusters in $\mathfrak G_j$. 

%To illustrate, let $\mathfrak F$ be any finite frame for $\KFT$. Then $\mathfrak F \not \models \alpha (\circ, \emptyset, \bot)$ iff the final cluster in $\mathfrak F$ is non-degenerate; $\mathfrak F \not \models \alpha (\circ \lhd \bullet, \emptyset, \bot)$ iff the final cluster in $\mathfrak F$ is degenerate and there is a non-degenerate cluster in $\mathfrak F$. To validate both of these formulas, $\mathfrak F$ must be a finite chain of irreflexive points, that is, a frame for $\GLT$. 

%Kripke frames $\mathfrak F = (W,R)$ for $\GLT$ do not contain infinite ascending $R$-chains $x_0Rx_1Rx_2R \dots$ of not necessarily distinct points; in other words, rooted Kripke frames for $\GLT$ are finite chains of $\bullet$. 
%

Canonical formulas of the form $\alpha(\mathfrak G, \emptyset,\bot)$ axiomatise exactly the \emph{cofinal subframe logics} whose frames are closed under taking cofinal subframes. We remind the reader~\cite{DBLP:books/daglib/0030819} that a subframe $\mathfrak F' = (W',R',\INT')$ of a frame $\mathfrak F = (W,R,\INT)$ is called  \emph{cofinal} if $W'$ is cofinal in $\mathfrak F$ in the sense that, for any $x \in W'$ and $y \in W$, whenever $x R y$ then either $y \in W'$ or  there is $z \in W'$ with $y R z$. Cofinal subframe logics enjoy the fmp, and so are decidable if finitely axiomatisable~\cite{DBLP:journals/jsyml/Zakharyaschev96}. Example~\ref{e:canform} shows the canonical axioms of some extensions of $\KFT$.

\begin{example}\label{e:canform}\em
$(a)$
We prove that
%
%\begin{equation}\label{GLTcanform}
\[
\GLT = \KFT \oplus \Box( \Box p_0 \to p_0) \to \Box p_0 = \KFT \oplus \alpha (\circ, \emptyset, \bot) \oplus \alpha (\circ \lhd \bullet, \emptyset, \bot).
\]
%\end{equation}
%
Let $\mathfrak F = (W,R,\INT)$ be a rooted finitely generated descriptive frame for 
$\KFT$. By Lemma~\ref{l:countable}, $W$ is countable, and so $\mathfrak F$ is $\mathfrak M$-generated, for some model
$\mathfrak M=(\mathfrak{F},\mathfrak{v})$.
We claim that the following are equivalent:
\begin{enumerate}
\item
$\mathfrak M\not\models\GLT$;

\item
there is a formula $\psi$ with a non-degenerate $\psi$-maximal cluster in $\mathfrak M$;

\item
there is a non-degenerate non-limit cluster in $\mathfrak F$;

\item
$\mathfrak F\not\models\alpha (\circ, \emptyset, \bot) \land \alpha (\circ \lhd \bullet, \emptyset, \bot)$.
\end{enumerate}

1.\ $\Rightarrow$\ 2.\
Suppose $\mathfrak M,x\not\models\Box(\Box\varphi \to\varphi)\to\Box\varphi$, for some formula $\varphi$. Then the $\neg\bigl(\Box(\Box\varphi \to\varphi)\to\Box\varphi\bigr)$-maximal cluster $C$ in $\mathfrak M$    is non-degenerate.

2.\ $\Leftrightarrow$\ 3.\ by Lemma~\ref{lem:descr'}.

2.\ $\Rightarrow$\ 4.\
Suppose the $\psi$-maximal cluster $C_\psi$ in $\mathfrak M$ is non-degenerate, for some $\psi$. If $C_\psi$ is the final cluster of $\mathfrak F$, then the injection $f$ mapping $\circ$ to a point in $C_\psi$ satisfies
\textbf{(cf$_1$)}--\textbf{(cf$_4$)}, and so $\mathfrak F\not\models\alpha (\circ, \emptyset, \bot)$.
If $C_\psi$ is not the final cluster, then the $\neg\psi$-maximal cluster $C_{\neg\psi}$ is the final cluster in $\mathfrak F$. If $C_{\neg\psi}$  is non-degenerate, then again $\mathfrak F\not\models\alpha (\circ, \emptyset, \bot)$; 
otherwise $\mathfrak F\not\models\alpha (\circ \lhd \bullet, \emptyset, \bot)$ as witnessed by $f$ sending $\bullet$ to the point in the final cluster and $\circ$ to a point in $C_\psi$.

4.\ $\Rightarrow$\ 1.\
If $\mathfrak F\not\models\alpha (\circ \lhd \bullet, \emptyset, \bot)$, then take an injection $f$ from $\circ \lhd \bullet$ to $\mathfrak F$ satisfying \textbf{(cf$_1$)}--\textbf{(cf$_4$)}.
By \textbf{(cf$_4$)} and Lemma~\ref{lem:descr'}, $\{f(\circ)\}=\mathfrak{v}(\varphi)$, for some $\varphi$.
As $f(\circ)R f(\circ)$ by \textbf{(cf$_1$)}, it is easy to see that
$\mathfrak M,f(\circ)\not\models\Box(\Box\neg\varphi\to\neg\varphi)\to\Box\neg\varphi$. 
The case when $\mathfrak F\not\models\alpha (\circ, \emptyset, \bot)$ is similar.

%\smallskip
$(b)$
Similarly, we can prove that
\begin{align*}
\Log\{(\mathbb N,<)\} & = \KFT \oplus \Diamond\top \oplus \Box (\Box p_0 \to p_0) \to (\Diamond\Box p_0 \to \Box p_0)\\%\label{omegacanform}
& = \KFT \oplus \alpha (\bullet, \emptyset, \bot) \oplus \alpha (\circ \lhd \circ, \emptyset, \bot)
\end{align*}
by showing that, for every $\mathfrak M$ and $\mathfrak F$ as above,
the following are equivalent:
\begin{itemize}
\item[--]
$\mathfrak M\not\models\Log\{(\mathbb N,<)\}$;

%\item[--]
%\textcolor{red}{either the final cluster in $\mathfrak F$ is degenerate or}
%there is $\psi$ such that the $\psi$-maximal cluster in $\mathfrak M$ is a non-degenerate non-final cluster in $\mathfrak F$;}\nz{omit this or next item?}

\item[--]
either the final cluster in $\mathfrak F$ is degenerate or
there is a non-degenerate non-limit cluster different from the final cluster in $\mathfrak F$;

\item[--]
$\mathfrak F\not\models\alpha (\bullet, \emptyset, \bot) \land\alpha (\circ \lhd \circ, \emptyset, \bot)$. 
%\hfill $\dashv$
\end{itemize}
%

%\smallskip
$(c)$ A %\nb{other example, commented out in the proof of L.\ref{l:fpoly}?} 
prominent example of a non-cofinal subframe logic is $\KFT \oplus \Diamond p \to \Diamond \Diamond p$ with \emph{dense} frames, whose canonical axioms 
\[
\KFT \oplus \alpha (\bullet \lhd \bullet^a, \{a\}, \bot) \oplus \alpha (\bullet \lhd \bullet^a \lhd \circ, \{a\}, \bot) \oplus \alpha (\bullet \lhd \bullet^a \lhd \bullet, \{a\}, \bot)
\]
forbid any two consecutive 
degenerate clusters in finitely generated descriptive
%irreflexive points in 
frames for the logic (see also Lemma~\ref{l:tempdframes}). 
\hfill $\dashv$
\end{example}

%************

\section{Craig interpolant existence: warming up}\label{sec:warming}

In this section, we first give a model-theoretic, bisimulation-based criterion of interpolant non-existence, then apply it to design a \coNP-algorithm deciding the IEP in any finitely axiomatisable d-persistent cofinal subframe logic containing $\KFT$. 
% and finally, explain by examples what is needed for arbitrary finitely axiomatisable extensions of $\KFT$.}
Finally, we illustrate by examples that a way more involved  approach is needed to tackle arbitrary finitely axiomatisable extensions of $\KFT$.

A formula $\iota$ is called a \emph{Craig interpolant} of formulas $\varphi_1$ and $\varphi_2$ in a logic $L$ if $\sig(\iota) \subseteq \sig(\varphi_1) \cap \sig(\varphi_2)$ and both  $\varphi_1 \to \iota$ and $\iota \to \varphi_2$ are in $L$. We say that $L$ has the \emph{Craig interpolation property} (\emph{CIP}) if an interpolant for $\varphi_1$ and $\varphi_2$ exists whenever $(\varphi_1\rightarrow \varphi_2) \in L$.

Many standard modal logics have the CIP, including $\K$, $\KF$, $\SF$. In fact, there are a continuum of logics containing $\KF$ with the CIP.  However, none of the continuum-many extensions of $\KFT$ with frames of unbounded depth has the CIP, and very few---not more than 37---out of the continuum-many logics containing $\SF$ enjoy the CIP (deciding whether a finitely axiomatisable logic above $\SF$ has the CIP is in \textsc{coNExpTime} and \textsc{PSpace}-hard). The reader can find proofs of these results and further references in~\cite{MGabbay2005-MGAIAD,DBLP:books/daglib/0030819}; see also Example~\ref{ex:GL.3}. 
%\nb{mention that examples Ex.\ref{ex:GL.3} actually show that $\KFT$ itself has no CIP?}

We now introduce the model-theoretic notions and tools that are needed in our non-uniform approach to deciding interpolant existence in modal logics.

%Given a signature $\sigma$, a model $\mathfrak M$, and a point $x$ in $\mathfrak M$,  we define 
%the \emph{atomic $\sigma$-type} and (full) \emph{$\sigma$-type} of $x$ in $\mathfrak{M}$ as the respective  restrictions of $\at_{\mathfrak{M}}(x)$ and $t_{\mathfrak{M}}(x)$ to $\sigma$-formulas, denoted $\at_{\mathfrak{M}}^{\sigma}(x)$ and $t_{\mathfrak{M}}^\sigma(x)$.
%%
%%\begin{equation*}
%%at_{\mathfrak{M}}^{\sigma}(x) = \{ p_i \mid \mathfrak{M},x\models p_i,\ p_i\in \sigma\}\cup \{ \neg p_i \mid \mathfrak{M},x\models \neg p_i,\ p_i\in\sigma\},
%%\end{equation*}
%%%
%%and the  \emph{$\sigma$-type of $x$ in $\mathfrak{M}$} as
%%%
%%\begin{equation*}
%%t_{\mathfrak{M}}^{\sigma}(x) = \{ \varphi \mid \mathfrak{M},x\models \varphi, \ \sig(\varphi)\subseteq \sigma\}.
%%\end{equation*}
%%
%For a set $X$ of points in $\mathfrak M$, 
%%For any $X\subseteq W$, 
%we let
%%
%$t_{\mathfrak{M}}^{\sigma}(X)=\bigl\{t_{\mathfrak{M}}^{\sigma}(x)\mid x\in X\bigr\}$.

Given two models $\mathfrak{M}_i$, $i=1,2$, based on $\mathfrak F_i = (W_i,R_i,\INT_i)$ with $x_i \in W_i$, we write $\mathfrak{M}_1,x_1 \equiv_{\sigma} \mathfrak{M}_2,x_2$, for a signature $\sigma$, if  
 $t_{\mathfrak{M}_1}^{\sigma}(x_1)=t_{\mathfrak{M}_2}^{\sigma}(x_2)$.
 %$\mathfrak{M}_1, x_1 \models \varphi$ iff $\mathfrak{M}_2, x_2 \models \varphi$, for every $\sigma$-formula $\varphi$. 
 %
The equivalence relation $\equiv_{\sigma}\subseteq W_1\times W_2$ can be characterised in terms of bisimulations. 
Namely, a relation $\bis \subseteq W_1 \times W_2$ is called a $\sigma$-\emph{bisimulation} between $\mathfrak M_1$ and $\mathfrak M_2$ if the following conditions hold whenever $x_1 \bis x_2$:

\smallskip
\noindent
{\bf (\atom)} 
$\at_{\mathfrak{M}_1}^{\sigma}(x_1) =\at_{\mathfrak{M}_2}^{\sigma}(x_2)$\textup{;}

\smallskip
\noindent
{\bf (\move)} if $x_1R_1y_1$, then there is $y_2$ such that $x_2R_2y_2$ and $y_1 \bis y_2$; and, conversely,\\
\mbox{}\hspace*{0.7cm} if $x_2R_2 y_2$, then there is $y_1$ with $x_1R_1 y_1$ and $y_1 \bis y_2$.

\smallskip
\noindent
If there is such $\bis$ with $z_1\bis z_2$, we write $\mathfrak{M}_1,z_1 \sim_{\sigma} \mathfrak{M}_2,z_2$.
We call $\bis$ \emph{\globalb} if, for every $x_1\in W_1$, there is $x_2\in W_2$ with 
$x_1\bis x_2$, and, 
for every $x_2\in W_2$, there is $x_1\in W_1$ with 
$x_1\bis x_2$. In this case, we say that $\mathfrak{M}_1$ and $\mathfrak{M}_2$ are \emph{\globally} $\sigma$-\emph{bisimilar} 
and write $\mathfrak{M}_1\sim_{\sigma} \mathfrak{M}_2$. 

%\bigskip
%\textcolor{red}{I think we need `global' because otherwise it can happen that we have a bisimulation between some $\mathfrak N_1$ and $\mathfrak N_2$ because we have the same atomic types in their final clusters (but otherwise they are very different). And we need `global' to prove this:
%
We employ the following characterisation of $\equiv_{\sigma}$ (see~\cite{goranko20075} for a further discussion of the relationship between bisimulations and modal equivalence): 

\begin{lemma}\label{bisim-lemma}
For any signature $\sigma$, any models $\mathfrak{M}_i$, $i = 1,2$, based on descriptive frames $\mathfrak F_i = (W_i,R_i,\INT_i)$, and any $x_i \in W_i$,
$$
\mathfrak{M}_1,x_1 \equiv_{\sigma} \mathfrak{M}_2,x_2 \quad \text{ iff } \quad \mathfrak{M}_1,x_1  \sim_{\sigma} \mathfrak{M}_2,x_2.
$$
The implication $(\Leftarrow)$ holds for arbitrary models.
\end{lemma}
\begin{proof}
$(\Rightarrow)$ We show that
$\{(y_1,y_2)\in W_1\times W_2\mid t_{\mathfrak{M}_1}^{\sigma}(y_1)=t_{\mathfrak{M}_2}^{\sigma}(y_2) \}$ 
is a $\sigma$-bisimulation between $\mathfrak M_1$ and $\mathfrak M_2$.  
Condition {\bf (\atom)} is obvious. For {\bf (\move)}, suppose $y_1R_1z_1$ and $t_{\mathfrak{M}_1}^{\sigma}(y_1)=t_{\mathfrak{M}_2}^{\sigma}(y_2)$. Let $\Gamma=t_{\mathfrak{M}_1}^{\sigma}(z_1)$. Then, for every finite $\Gamma' \subseteq \Gamma$, we have $\mathfrak M_1,y_1\models \Diamond \bigwedge\Gamma'$, and so 
$\mathfrak M_2,y_2\models \Diamond \bigwedge\Gamma'$ as well. Since $\mathfrak F_2$ is descriptive, 
Lemma~\ref{maxpoints} gives us $z_2$ with $y_2R_2z_2$ and $\mathfrak M_2, z_2 \models \Gamma$. It follows that $t_{\mathfrak{M}_1}^{\sigma}(z_1)=t_{\mathfrak{M}_2}^{\sigma}(z_2)$, as required. 
%The proof of the converse condition is similar, using that $\mathfrak F_1$ is descriptive. 
The implication $(\Leftarrow)$ is straightforward.
\end{proof}

Note that if $\mathcal{B}$ is a set of $\sigma$-bisimulations between $\mathfrak M_1$ and $\mathfrak M_2$, then $\bigcup_{\bis\in \mathcal{B}} \bis$ is also a $\sigma$-bisimulation between $\mathfrak M_1$ and $\mathfrak M_2$. It follows that there is always a \emph{largest} $\sigma$-bisimulation between $\mathfrak M_1$ and $\mathfrak M_2$ 
(which is $\equiv_\sigma$ if both $\mathfrak M_i$ are based on descriptive frames).

%$\bis$ defined by taking $(x_{1},x_{2})\in \bis$ iff $x_1 \sim_{\sigma} x_2$ is the \emph{maximal} $\sigma$-bisimulation between $\mathfrak M_1$ and $\mathfrak M_2$.\nb{we use $\equiv_\sigma$ everywhere anyway?}

Variations of the following criterion of interpolant (non-)existence are implicit in various (dis-)proofs of the CIP in modal logics~\cite{DBLP:conf/amast/Marx98,goranko20075}. 

\begin{theorem}\label{criterion}
Formulas $\varphi_1$ and $\varphi_2$ do not have an interpolant in a modal logic $L$
%\nz{$L \supseteq \KF$?} 
iff there are models $\mathfrak M_i$, $i=1,2$, based on finitely $\mathfrak M_i$-generated 
descriptive frames $\mathfrak{F}_i = (W_{i},R_{i},\INT_{i})$ for $L$ with 
points
%roots 
$x_i \in W_i$ such that  
\[
\mathfrak{M}_{1},x_{1} \models \varphi_1, \ \ 
\mathfrak{M}_{2},x_{2} \models \neg\varphi_2, \ \ 
\mathfrak{M}_{1},x_{1} \sim_{\sigma} \mathfrak{M}_{2},x_{2}, \text{ for $\sigma = \sig(\varphi_1) \cap \sig(\varphi_2)$}.
\]
%\begin{itemize}
%\item[--] $\mathfrak{M}_{1},x_{1} \models \varphi_1$,\nz{all three in one line?}
%
%\item[--] $\mathfrak{M}_{2},x_{2} \models \neg\varphi_2$,
%
%\item[--] $\mathfrak{M}_{1},x_{1} \sim_{\sigma} \mathfrak{M}_{2},x_{2}$, where $\sigma = \sig(\varphi_1) \cap \sig(\varphi_2)$.
%\end{itemize}
%
If $L \supseteq \KF$, we may assume that $x_i$ is the root of the descriptive frame $\mathfrak F_i$, $i=1,2$.
\end{theorem}
\begin{proof}
$(\Leftarrow)$ is straightforward (and holds for arbitrary frames for $L$). For $(\Rightarrow)$, 
consider the signature $\delta = \sig(\varphi_1) \cup \sig(\varphi_2)$ and the set
\[
\Sigma=\{\chi\mid\chi\mbox{ is a $\sigma$-formula and }(\varphi_1\to\chi)\in L\}\cup\{\neg\varphi_2\}
\]
of $\delta$-formulas.
As $\varphi_1$ and $\varphi_2$ have no interpolant in $L$, $\Sigma$ is $L$-consistent, and so, by 
Lemma~\ref{univframe}, there exist a $\delta$-model $\mathfrak M_2$ based on a finitely $\mathfrak M_2$-generated descriptive frame $\mathfrak F_2$ and a point $x_2$ with $\mathfrak M_2,x_2\models\Sigma$. 
Let $\Sigma'=t_{\mathfrak M_2}^\sigma(x_2)\cup\{\varphi_1\}$. As $\Sigma'$ is an $L$-consistent set of $\delta$-formulas, 
Lemma~\ref{univframe} gives a $\delta$-model $\mathfrak M_1$ based on a finitely $\mathfrak M_1$-generated descriptive frame $\mathfrak F_1$ and an $x_1$ in $\mathfrak M_1$ such that $\mathfrak M_1,x_1\models\Sigma'$.  
We clearly have
$t_{\mathfrak M_1}^\sigma(x_1)=t_{\mathfrak M_2}^\sigma(x_2)$, and so 
 $\mathfrak{M}_{1},x_{1} \sim_{\sigma} \mathfrak{M}_{2},x_{2}$ by Lemma~\ref{bisim-lemma}.
In case $L \supseteq \KF$, \eqref{rootdescr} allows us to make $x_i$ the root of $\mathfrak F_i$.
\end{proof}

%************* limit root lemma

The next lemma refines Theorem~\ref{criterion}; it is used in the proof of Lemma~\ref{l:replacement}. %\nz{moved here from \S \ref{ss:small}}
\begin{lemma}\label{l:rootnolimit}
If $\varphi_1$ and $\varphi_2$ do not have an interpolant in a logic $L \supseteq \KFT$, then there are rooted 
models $\mathfrak M_i,x_i$, $i=1,2$, satisfying the criterion Theorem~\ref{criterion} such that
$C(x_i)$ is not a limit cluster in $\mathfrak M_i$, for $i=1,2$.
\end{lemma}
\begin{proof}
Suppose $\mathfrak M,x$ is a rooted $\delta$-model, for some finite signature $\delta$,  that is based on a finitely $\mathfrak M$-generated descriptive frame $\mathfrak F=(W,R,\INT)$ such that $\mathfrak M,x\models\varphi$, for some $\varphi$ with $\sig(\varphi)\subseteq\delta$, and $C(x)$ is a root limit cluster in $\mathfrak F$.
Pick a fresh variable $q\notin\delta$. For $\ast \in \{\bullet, \circ\}$, take the frames $\mathfrak F^\ast = \ast \lhd \mathfrak F$, denote the root point of $\mathfrak F^\ast$ by $x^\ast$, and consider the $\delta\cup\{q\}$-models $\mathfrak M^\ast$ based on $\mathfrak F^\ast$, which coincide with $\mathfrak M$ on $\mathfrak F$ and have $\mathfrak M^\ast, x^\ast \models p$ iff $\mathfrak M, x\models p$, for $p\in\delta$, and $\mathfrak M^\ast, x^\ast \models q$.
To prove the lemma,
it suffices to show that there is $\ast \in \{\bullet, \circ\}$ with
$(i)$ $\mathfrak M^\ast,x^\ast\models\varphi$,
$(ii)$ $\mathfrak M^\ast,x^\ast\sim_\sigma\mathfrak M,x$, for any $\sigma\subseteq\sig(\varphi)$, and
$(iii)$ $\Log(\mathfrak F)\subseteq\Log(\mathfrak F^\ast)$.

As the limit cluster $C(x)$ is non-degenerate by Lemma~\ref{lem:descr'}, we have $(i)$ and $(ii)$. To show $(iii)$, suppose on the contrary that, for each $\ast \in \{\bullet, \circ\}$, there is a canonical formula $\alpha(\mathfrak G^\ast, \mathfrak D^\ast, \bot)$ with  
\mbox{$\mathfrak F\models \alpha(\mathfrak G^\ast, \mathfrak D^\ast, \bot)$} and 
$\mathfrak F^\ast\not\models \alpha(\mathfrak G^\ast, \mathfrak D^\ast, \bot)$. 
Let $f^\ast$ be an injection from $\mathfrak G^\ast$ to $\mathfrak F^\ast$ satisfying \textbf{(cf$_1$)}--\textbf{(cf$_4$)} for $\alpha(\mathfrak G^\ast, \mathfrak D^\ast, \bot)$, and
let $C(r^\ast)$ be the root-cluster in $\mathfrak G^\ast$ and $C(y^\ast)$ its immediate successor in $\mathfrak G^\ast$. By assumption, $f^\ast$ is not an injection from $\mathfrak G^\ast$ to $\mathfrak F$ satisfying \textbf{(cf$_1$)}--\textbf{(cf$_4$)}, so  $f^\ast(r^\ast)=x^\ast$ and
$f^\ast(y^\ast) \in W$.
As $\{f^\ast(y^\ast)\}\in\INT$ by \textbf{(cf$_4$)} and $C(x)$ is a limit cluster, it follows from Lemma~\ref{lem:descr'} that $f^\ast(y^\ast) \notin C(x)$, and so $y^\ast \notin \mathfrak D^\ast$. 
Suppose, for definiteness, that $f^\circ(y^\circ)R f^\bullet(y^\bullet)$ or $f^\circ(y^\circ)=f^\bullet(y^\bullet)$. Let $C$ be the immediate predecessor of $C(f^\circ(y^\circ))$
in $\mathfrak F$. Then $C$ is a non-limit cluster. By Lemma~\ref{lem:descr'}, $C\in\INT$ 
and, by Lemma~\ref{l:defpoints}, $\{z\}\in\INT$, for every $z\in C$. If $C$ is non-degenerate, then 
we modify $f^\circ$ by taking $f^\circ(r^\circ)\in C$; otherwise, we modify $f^\bullet$ by taking $f^\bullet(r^\bullet) \in C$. In either case, the modified $f^\ast$ is an injection 
from $\mathfrak G^\ast$ to $\mathfrak F$ satisfying \textbf{(cf$_1$)}--\textbf{(cf$_4$)}, a contradiction.
\end{proof}

%***********

We begin our study of the interpolant existence problem (IEP) by showing how the criterion of Theorem~\ref{criterion} can be used to decide whether given formulas have an interpolant in any fixed d-persistent cofinal subframe logic $L \supseteq \KFT$ (defined in \S\ref{can-for}). 
%We remind the reader that a frame $\mathfrak F' = (W',R',\INT')$ is called a \emph{subframe} of a frame $\mathfrak F = (W,R,\INT)$ if $W' \subseteq W$, $R' = R \restriction W'$, and $\INT' \subseteq \INT$. A subframe $\mathfrak F'$ of $\mathfrak F$ is called \emph{cofinal} if $W'$ is cofinal in $\mathfrak F$ in the sense that, for any $x \in W'$ and $y \in W$, whenever $x R y$ then either $y \in W'$ or  there is $z \in W'$ with $y R z$. 
%We remind the reader~\cite{DBLP:books/daglib/0030819} that a subframe $\mathfrak F' = (W',R',\INT')$ of a frame $\mathfrak F = (W,R,\INT)$ 
%is called  \emph{cofinal} if $W'$ is cofinal in $\mathfrak F$ in the sense that, for any $x \in W'$ and $y \in W$, whenever $x R y$ then either $y \in W'$ or  there is $z \in W'$ with $y R z$. 
%A modal logic containing $\KF$ is a \emph{cofinal} \emph{subframe logic} if it is determined by a class of frames that is closed under taking cofinal subframes. For example, both $\KFT$ and $\GLT$ are cofinal subframe logics, while $\KFT \oplus \Diamond p \to \Diamond \Diamond p$ is not. Cofinal subframe logics enjoy the fmp, and so are decidable if finitely axiomatisable~\cite{DBLP:journals/jsyml/Zakharyaschev96}. 
%
%$L \supseteq \KFT$ is a d-persistent cofinal subframe logic and 
%
Suppose that $\varphi_1$ and $\varphi_2$ do not have an interpolant in $L$. Let $\sigma = \sig(\varphi_1) \cap \sig(\varphi_2)$. By Theorem~\ref{criterion}, there exist models $\mathfrak M_i$, $i=1,2$, based on descriptive frames \mbox{$\mathfrak{F}_i = (W_{i},R_{i},\INT_{i})$} for $L$ with roots $x_i \in W_i$ such that $\mathfrak{M}_{1},x_{1} \sim_{\sigma} \mathfrak{M}_{2},x_{2}$, \mbox{$\mathfrak{M}_{1},x_{1} \models \varphi_1$} and $\mathfrak{M}_{2}, x_{2} \models \neg \varphi_2$. 
%Let $\bis$ be the \emph{maximal} 
We may assume that $\bis$ is the largest
$\sigma$-bisimulation $\equiv_\sigma$ between $\mathfrak M_1$ and $\mathfrak M_2$ 
%(which includes $\bis(x_1,x_2)$ 
(for which $x_1 \bis x_2$, 
of course). 
We show how to extract from the $\mathfrak M_i$ polynomial-size models $\mathfrak M'_i$ that still witness that $\varphi_1$ and $\varphi_2$ lack an interpolant in $L$. 
We proceed in two steps.
\begin{description}
\item[Step 1] For each $i=1,2$ and each $\tau\in \sub(\varphi_i)$ satisfied in $\mathfrak{M}_i$, we take a $\{\tau\}$-maximal point $y_{\tau}\in W_i$ 
%maximal for $\{\tau\}$ at $\mathfrak M_i,x_i$\nz{just maximal}
%a maximal $y_{\tau}\in W_1$ with $\mathfrak M_1, y_\tau\models \tau$, 
(which exists by Lemma~\ref{maxpoints}), and denote the 
%resulting set 
set of all these $y_\tau$ by $\mset{i} \subseteq W_i$. 
%We do the same with $\sub(\psi)$ and $\mathfrak M_2$ obtaining $M_2 \subseteq W_2$. 
Note that $\mset{i}$ is cofinal in $\mathfrak F_i$ because each point in $W_i \setminus \mset{i}$ has a $\{\varphi_i\}$- or $\{\neg\varphi_i\}$-maximal $R_i$-successor. Set
\begin{equation}\label{typesT}
T=\bigl\{ t_{\mathfrak{M}_{1}}^{\sigma}(x) \mid x\in \{x_{1}\}\cup \mset{1}\bigr\} \cup\bigl\{ t_{\mathfrak{M}_{2}}^{\sigma}(x) \mid x\in \{x_{2}\}\cup \mset{2}\bigr\}.
%T=\{ t_{\mathfrak{M}_{1}}^{\sigma}(x) \mid x\in M_{1}\} \cup\{ t_{\mathfrak{M}_{2}}^{\sigma}(x) \mid x\in M_{2}\}.
\end{equation}

\item[Step 2]  As $\mathfrak{M}_{1},x_{1} \sim_{\sigma} \mathfrak{M}_{2},x_{2}$ and $\bis$ is the largest $\sigma$-bisimulation,
 each $t\in T$ is satisfied in both $\mathfrak M_i$.
For $i=1,2$ 
%\textcolor{blue}{we take a smallest set $\sset{i} \subseteq W_i$ containing a $t$-maximal point $z_t$ in $\mathfrak M_i$ (which exists by Lemma~\ref{maxpoints}), for each $t\in T$.}
and every $t\in T$, we take a $t$-maximal point $z_t^i$ in $\mathfrak M_i$ (which exists by Lemma~\ref{maxpoints}) such that if $\{x_i\}\cup \mset{i}$ already contained some $t$-maximal points, then we choose $z_t^i$ from those, always taking $z_t^i=x_i$ when $x_i$ is $t=t_{\mathfrak M_i}^\sigma(x_i)$-maximal. We let $\sset{i}=\{z_t^i\mid t\in T\}$. 
Then $|\sset{1}|=|\sset{2}|=|T|$. Note, however, that $\{x_i\}\cup \mset{i}$ may contain more than one $t$-maximal point, for some $t\in T$ (but we only choose one of these to be included in $\sset{i}$).
\end{description}
%
%Then we clearly have that
%
%\begin{equation}\label{allinT}
%\mbox{$t^\sigma_{\mathfrak M_i}(y)\in T$, for all $y\in M_i\cup S_i$.}
%\end{equation}

Now, let 
%$W_{i}'=M_{i} \cup S_{i}$, 
$W_{i}'=\{x_{i}\}\cup \mset{i} \cup \sset{i}$, 
$R'_i = \rest{R_i}{W'_i}$, $\mathfrak F'_i = (W'_i,R'_i)$, and let $\mathfrak{M}_{i}'$ be the restriction of $\mathfrak{M}_{i}$ to $\mathfrak{F}_{i}'$. 
We let

\begin{equation}\label{kbound}
\kbound=3+3\max(|\varphi_1|,|\varphi_2|).
\end{equation}
Clearly,  $|W_i'|\leq \kbound$, 
%\label{kbound} 
so the size of $\mathfrak M_i$ is $\mathcal{O}\bigl(\max(|\varphi_1|,|\varphi_2|)\bigr)$.  
As $L$ is d-per\-sistent, $(W_i,R_i) \models L$. By construction, $\mathfrak F'_i$ is a cofinal subframe of $(W_i, R_i)$, and so $\mathfrak F'_i \models L$ as $L$ is a cofinal subframe logic. Finally, we define $\bis'$ as the restriction of $\bis$ to $W_{1}'\times W_{2}'$: $x_1' \bis' x_2'$
iff $t_{\mathfrak{M}_1}^{\sigma}(x_1')=t_{\mathfrak{M}_2}^{\sigma}(x_2')$, for all $x_1'\in W_1'$, $x_2'\in W_2'$.

\begin{lemma}\label{cofinsubfr}\
$(a)$ $\mathfrak{M}_{1}',x_{1}\models\varphi_1$, $\mathfrak{M}_{2}',x_{2}\models\neg \varphi_2$,  and 
$(b)$ 
$\bis'$ is a $\sigma$-bisimulation between $\mathfrak{M}_{1}'$ and $\mathfrak{M}_{2}'$ with $x_1 \bis' x_2$.
\end{lemma}
\begin{proof}
$(a)$ follows from the fact that, for any $\tau\in\sub(\varphi_i)$ and $x\in W'_i$,
$\mathfrak M_i,x\models\tau$ iff $\mathfrak M'_i,x\models\tau$, which can be established by a straightforward induction on the construction of $\varphi_1$ and $\varphi_2$.
We only show $(\Rightarrow)$ for $\tau=\Diamond\psi$. 
If $\mathfrak M_i,x\models\Diamond\psi$, then there is $y\in W_i$ with $xR_iy$ and 
$\mathfrak M_i,y\models\psi$. Take $y_\psi\in \mset{i}\subseteq W_i'$. By the $\{\psi\}$-maximality of $y_\psi$, 
either $y=y_\psi$ or $yR_iy_\psi$, and so $xR_i'y_\psi$ and $\mathfrak M'_i,x\models\Diamond\psi$.

%Next, we prove the third statement:
% We prove the third one. 
$(b)$
 Condition {\bf (\atom)} follows from the definition. To establish {\bf (\move)}, assume $x \bis' x'$ and $x R'_1 y$. 
 Let $t=t^\sigma_{\mathfrak M_1}(y)$. Then $t\in T$,
 %by \eqref{allinT}, 
 and so  there is a $t$-maximal $z_t\in \sset{2}\subseteq W_2'$ in $\mathfrak M_2$. 
% maximal for $t$ at $\mathfrak M_2,x_2$.\nz{just maximal}
 In particular, $t^\sigma_{\mathfrak M_2}(z_t)=t$, and so $y \bis' z_t$. As $x \bis x'$ and $\bis$ is the largest $\sigma$-bisimulation, there is $z\in W_2$ with $x'R_2z$ and $t^\sigma_{\mathfrak M_2}(z)=t$.
  It follows from the $t$-maximality of $z_t$ that $z=z_t$ or $zR_2 z_t$, and so $x'R_2'z_t$, as required.
% 
%Suppose first $y\in S_{1}$, so $y$ is maximal in $\mathfrak{M}_{1}$ for some $\sigma$-type $t$. Then $\mathfrak{M}_{1},x \models \Diamond t$, and so  $\mathfrak{M}_{2},x'\models \Diamond t$ by Lemma~\ref{bisim-lemma}. Let $y'\in W_{2}'$ be a maximal point satisfying $t$ in $\mathfrak{M}_{2}'$. Using again Lemma~\ref{bisim-lemma} and assumption that $\bis$ is the maximal $\sigma$-bisimulation, we obtain $(y,y')\in \bis$, and so $(y,y')\in \bis'$,  as required. Suppose next that $y$ was selected in Step 1 and $t = t^\sigma_{\mathfrak M_1}(y)$. Then either $y$ is maximal satisfying $t$ and we are done, or $y$ is not maximal satisfying $t$. In the latter case, take $z$ with $yR_1 z$ such that $z$ is maximal satisfying $t$ in $\mathfrak{M}_{1}$ and selected in Step 2. As above, we find $z'$ with $x'R_2 z'$ in $\mathfrak{M}_{2}'$ such that $(z,z')\in \bis$. But then we also have $(y,z')\in\bis$ as the $\sigma$-types of $y$ and $z$ coincide and $\bis$ is the maximal $\sigma$-bisimulation.	
\end{proof}

Thus, the fact that $\varphi_1$ and $\varphi_2$ have no interpolant in $L$ can always be witnessed (in the sense of Theorem~\ref{criterion}) by models $\mathfrak M_i$ of size polynomial in 
%$|\varphi_1|+|\varphi_2|$, 
$\max(|\varphi_1|,|\varphi_2|)$, 
and so we can say that $L$ has the \emph{polysize bisimilar model property}. This gives the first claim of the following theorem:  
%As a consequence, recalling that all consistent finitely axiomatisable logics containing $\KFT$ are coNP-complete~\cite{DBLP:journals/sLogica/LitakW05}, we obtain the following:

\begin{theorem}\label{dperscofinal}
$(a)$ All d-persistent cofinal subframe logics $L \supseteq \KFT$ have the polysize bisimilar model property. $(b)$ If such an $L$ is consistent and finitely axiomatisable, then the IEP for $L$ is \coNP-complete. 
\end{theorem}
%
%Claim\nz{proof not needed here}\nb{I don't know: then we should add sg to Sec.~\ref{section:K4.3} on why this is special case. Perhaps better to write here sg on how we use Sec.\ref{can-for} here?} $(b)$ follows from more general Theorem~\ref{t:iepcoNP} to be proved in Section~\ref{section:K4.3}.
%
\begin{proof}
We show that $(a) \Rightarrow (b)$; cf.\ Theorem~\ref{t:iepcoNP} in \S\ref{section:K4.3}. Indeed, suppose $L$ is given by~\eqref{canon} (with 
$\mathfrak G_j=(V_j,S_j)$ and $\mathfrak D_j= \emptyset$, for all 
$j$ in the finite index set $J_L$). To decide whether formulas $\varphi_1$ and $\varphi_2$ do not have an interpolant in $L$, we guess polynomial-size pointed models $\mathfrak M_i,x_i$ based on 
Kripke frames $\mathfrak F_i=(W_i,R_i)$ for $\KFT$ and restricted to the variables in $\varphi_1$ and $\varphi_2$. 
The conditions $\mathfrak M_1,x_1\models\varphi_1$ and  $\mathfrak M_2,x_2\models\neg\varphi_2$ are clearly polynomially checkable; 
%\textcolor{blue}{Checking the first two conditions in the criterion of  Theorem~\ref{criterion} 
%clearly needs polynomial time; 
%the third one 
that $\mathfrak{M}_{1},x_{1} \sim_{\sigma} \mathfrak{M}_{2},x_{2}$, for $\sigma = \sig(\varphi_1) \cap \sig(\varphi_2)$, can be established in polynomial time using a standard technique from~\cite[Chapter 7]{DBLP:books/daglib/0020348}.
Finally, to check whether $\mathfrak F_i \models \alpha(\mathfrak G_j, \emptyset,\bot)$, for   
%all $j<m_L$, 
each $j\in J_L$, we simply enumerate all injective functions from $\mathfrak G_j$ to $\mathfrak F_i$, whose number does not exceed $|W_i|^{|V_j|}$, and verify that at least one of them satisfies \textbf{(cf$_1$)}--\textbf{(cf$_2$)}, which can obviously be done in time polynomial in $|W_i|$.
(Condition  \textbf{(cf$_3$)} holds vacuously, and  \textbf{(cf$_4$)} always holds as $\mathfrak F_i$ is a Kripke frame.) 
%
%\Finally, to check $\mathfrak F_i \models L$, we assume that $L$ is given by~\eqref{canon}. Now, to decide whether $\mathfrak F_i \models \alpha(\mathfrak G_j, \mathfrak D_j,\bot)$, we simply enumerate all injective functions from $\mathfrak G_j$ to $\mathfrak F_i$, whose number does not exceed $|\mathfrak F_i|^{|\mathfrak G_j|}$, and verify that at least one of them satisfies \textbf{(cf$_2$)}--\textbf{(cf5)}, which can obviously be done in time polynomial in $|\mathfrak F_i|$.\nz{notation?} }
\end{proof}

We now give two examples illustrating that the construction above does not work for logics that are not d-persistent,
even for logics with the fmp.
Prominent examples of such logics are $\GLT$ and $\Log\{(\mathbb N,<)\}$; see Example~\ref{e:canform}.
%(whose finite frames were described in Section~\ref{can-for}).
We show that, for these logics, establishing model-theoretically (using Theorem~\ref{criterion}) that some formulas do not have an interpolant requires a pair of models that are based on infinite descriptive (non-Kripke) frames.

%
%We now give examples explaining why the construction above does not work for logics that are not d-persistent. 
%Prominent specimens of such logics are $\GLT$ and $\Log\{(\mathbb N,<)\}$ whose finite frames were described in Section~\ref{can-for}.
%}

%\begin{description}
%\item[$\GLT = \KFT \oplus \Box( \Box p_0 \to p_0) \to \Box p_0$] whose Kripke frames $\mathfrak F = (W,R)$ do not contain infinite ascending $R$-chains $x_0Rx_1Rx_2R \dots$ of not necessarily distinct points; in other words, rooted Kripke frames for $\GLT$ are finite chains of $\bullet$. $\GLT$ is not d-persistent but is Kripke complete and has the fmp.
%
%\item[$\Log\{(\mathbb N,<)\} = \KFT \oplus \Diamond\top \oplus \Box (\Box p \to p) \to (\Diamond\Box p \to \Box p)$]
%%
%is determined by the class of finite \emph{lassos} (aka balloons)---finite strict linear orders followed by some nondegenerate cluster $\clusterk$, for $k \ge 1$.
%\end{description}

\begin{example}\label{ex:GL.3}\em 
$(a)$ Consider the following formulas $\varphi_1$ and $\varphi_2$:
%where $\Diamond^+\chi = \chi \lor \Diamond \chi$: 
%
\begin{align}
\nonumber
& \mbox{$\varphi_1 = \Diamond( p_{1} \wedge  \Diamond^+ \neg q_{1}) \wedge \Box (p_{2} \rightarrow \Box^+ q_{1})$,}\\ \label{psi}
%\wedge  \Box(p_{1} \rightarrow \neg p_{2})
%
& \varphi_2 = \neg [ \Diamond( p_{2} \wedge \Diamond^+\neg q_{2}) \land \Box (p_{1} \rightarrow \Box^+ q_{2}) ].
\end{align}
To show that $(\varphi_1 \rightarrow \varphi_2) \in \KFT \subseteq \GLT$, 
% = \KFT \oplus \alpha(\circ)$. 
suppose otherwise. Then there exist a model $\mathfrak M$ based on a frame $\mathfrak F = (W,R)$ for $\KFT$ and $z \in W$ with $\mathfrak M, z \models \varphi_1 \land \neg \varphi_2$. 
So we have $x,x',y,y' \in W$ with $z R x R^+ x'$, $z R y R^+ y'$,
$\mathfrak M, x \models p_1$, $\mathfrak M, x' \models \neg q_1$, $\mathfrak M, y \models p_2$, and $\mathfrak M, y' \models \neg q_2$. 
%\mathfrak M, x \models p_1 \land \neg p_2, \quad \mathfrak M, x' \models \neg q_1 \quad \text{and} \quad \mathfrak M, y \models p_2 \land \neg p_1, \quad \mathfrak M, y' \models \neg q_2.
%
Since $\mathfrak F$ is 
%linear, 
a frame for $\KFT$, either $x'=y'$ or $x'Ry'$ or $y'Rx'$. 
%weakly connected, 
%$xRy$ or $yRx$. 
However, none of these is possible because of the boxed conjuncts of $\varphi_1$ and $\neg\varphi_2$. 
%subformulas of $\varphi_1$ and $\varphi_2$ according to which $xRy$ implies $\mathfrak M, y' \models q_2$, and $yRx$ implies $\mathfrak M, x' \models q_1$.

We now use Theorem~\ref{criterion} to show that $\varphi_1$ and $\varphi_2$ do not have an interpolant in $\GLT$. Let $\sigma = \sig(\varphi_1) \cap \sig(\varphi_2) = \{p_1,p_2\}$. Observe that any models $\mathfrak M_i$ meeting the conditions of Theorem~\ref{criterion} cannot be based on a Kripke frame $\mathfrak F_i=(W_i,R_i)$ for $\GLT$. 
 Indeed, let $\bis$ be the corresponding bisimulation. Then $\mathfrak{M}_{1},x_{1} \models \varphi_1$ implies that there is $x^1_1 \in W_1$ with $x_1 R_1 x^1_1$ and $\mathfrak{M}_{1},x^1_{1} \models p_1$; we must also have \mbox{$\mathfrak{M}_{1},y_{1} \models \neg q_1$}, for some $y_1$ with $x^1_1 R^+_1 y_1$. 
Similarly, $\mathfrak{M}_{2},x_{2} \models \neg\varphi_2$ implies that there is $x^1_2 \in W_2$ with $x_2 R_1 x^1_2$ and $\mathfrak{M}_{2},x^1_{2} \models p_2$, and we also have $\mathfrak{M}_{2},y_{2} \models\neg q_2$, for some $y_2$ with $x^1_2 R^+_2 y_2$. As $x_{1} \bis x_{2}$ and $x_1 R_1 x^1_1$, {\bf (\move)} gives $x^2_2$ with $x_2 R_2 x^2_2$ and $x^1_{1} \bis x^2_{2}$. But then $\mathfrak{M}_{2},x^2_{2} \models p_1$, and so $x_2 R_2 x^1_2 R_2^+ y_2R_2 x^2_2$ since $\mathfrak F_2$ is a frame for $\KFT$
%linear 
and in view of $\neg\varphi_2$'s second conjunct. Symmetrically, we find 
$x^2_1$ with $x_1 R_1 x^1_1 R_1^+y_1 R_1 x^2_1$ and $x^2_1 \bis x^1_2$. 
%$x_1 R_1 x^1_1 R_1 x^2_1$ with $x^2_1 \bis x^1_2$. 
Using {\bf (\move)}, we construct infinite ascending chains of not necessarily distinct
%(non-necessarily distinct) 
points as shown in the picture below.\\
%Fig.~\ref{GL3}.\nz{can make without Fig environment} 
%
%\centerline{\includegraphics[scale=0.7]{../Pics/GL3}}\\
%\begin{figure}[thbp]
% \centering
%  \includegraphics[scale=0.7]{../Pics/GL3}
%
\centerline{
\begin{tikzpicture}[>=latex,line width=0.5pt,xscale = 1,yscale = .65]
\node[]  at (-1,0) {$\mathfrak M_2$};
\node[scale = 0.9,label=below:{\footnotesize $\neg\varphi_2$},label=above:{\footnotesize $\qquad x_2$}] (x2) at (0,0) {$\ast$};
\node[scale = 0.9,label=below:{\footnotesize $p_2$},label=above:{\footnotesize $x_2^1$}] (x21) at (2,0) {$\ast$};
\node[scale = 0.9,label=below:{\footnotesize $\neg q_2$},label=above:{\footnotesize $y_2$}] (y2) at (3,0) {$\ast$};
\node[scale = 0.9,label=below:{\footnotesize $p_1$},label=above:{\footnotesize $x_2^2$}] (x22) at (4,0) {$\ast$};
\node[scale = 0.9,label=below:{\footnotesize $p_2$},label=above:{\footnotesize $x_2^3$}] (x23) at (6,0) {$\ast$};
\node[]  at (7,0) {$\dots$};
\draw[->] (x2) to (x21);
\draw[->] (x21) to (y2);
\draw[->] (y2) to (x22);
\draw[->] (x22) to (x23);
\node[]  at (-1,3) {$\mathfrak M_1$};
\node[scale = 0.9,label=above:{\footnotesize $\varphi_1$},label=below:{\footnotesize $\qquad x_1$}] (x1) at (0,3) {$\ast$};
\node[scale = 0.9,label=above:{\footnotesize $p_1$},label=below:{\footnotesize $x_1^1$}] (x11) at (2,3) {$\ast$};
\node[scale = 0.9,label=above:{\footnotesize $\neg q_1$},label=below:{\footnotesize $y_1$}] (y1) at (3,3) {$\ast$};
\node[scale = 0.9,label=above:{\footnotesize $p_2$},label=below:{\footnotesize $x_1^2$}] (x12) at (4,3) {$\ast$};
\node[scale = 0.9,label=above:{\footnotesize $p_1$},label=below:{\footnotesize $x_1^3$}] (x13) at (6,3) {$\ast$};
\node[]  at (7,3) {$\dots$};
\draw[->] (x1) to (x11);
\draw[->] (x11) to (y1);
\draw[->] (y1) to (x12);
\draw[->] (x12) to (x13);
\node[gray]  at (-.5,1.5) {$\bis$};
\draw[gray,thick,dotted] (x1) to (x2);
\draw[gray,thick,dotted] (x11) to (x22);
\draw[gray,thick,dotted] (x12) to (x23);
\draw[gray,thick,dotted] (x12) to (x21);
\draw[gray,thick,dotted] (x13) to (x22);
\end{tikzpicture}
}
%
% \caption{Infinite ascending chains in $\sigma$-bisimilar models $\mathfrak M_1$, $\mathfrak M_2$.}
%\caption{Infinite ascending chains in $\sigma$-bisimilar models $\mathfrak M_i$.} \label{GL3}
%\end{figure}
\noindent
It follows that the $\mathfrak F_i$ are not frames for $\GLT$; see any of \cite{Gol,DBLP:books/daglib/0030819,DBLP:books/cu/BlackburnRV01} for details.

We now give a descriptive frame for $\GLT$ that can be used to show that $\varphi_1$ and $\varphi_2$ do not have an interpolant in $\GLT$. Take the descriptive frame $\mathfrak{C}(\clustert, \bullet)$  defined in Example~\ref{k-omega} and construct $\mathfrak F=\bullet \lhd \bullet \lhd \mathfrak C(\clustert, \bullet)$ (see Definition~\ref{d:orderedsum}),
%add two irreflexive points before it, and define as internal all possible unions of internal sets in the three constituent frames (see the definition on page~\pageref{orderedsum}). We denote the resulting (descriptive) frame by $\bullet \lhd \bullet \lhd \mathfrak C(\clustert, \bullet)$.  
%
%The only reflexive points in $\mathfrak F$ are $a_0$ and $a_1$ in the cluster $\cluster{2}$,so for any function $f$ satisfying \textbf{(cf$_1$)} in \S\ref{can-for}, $f(\circ)$ must be $a_0$ or $a_1$,5and $\{f(\circ)\}$ must be an internal set in $\mathfrak F$ by \textbf{(cf$_4$)}, which is not the case by \eqref{infiniteincl}. Thus, $\mathfrak F\models\alpha (\circ, \emptyset, \bot) \land \alpha (\circ \lhd \bullet, \emptyset, \bot)$, and so 
which is a frame for $\GLT$ by property
$(iii)$ in Example~\ref{e:canform}~$(a)$.
%
%\eqref{GLTcanform}. 
%It is readily seen that 
%%$\bullet \lhd \bullet \lhd \mathfrak C(\clustert, \bullet)$ 
%$\mathfrak F$ is a frame for $\GLT$. Indeed, suppose $\LA = \Box (\Box p \to p) \to \Box p$ and there is a model $\mathfrak M=(\mathfrak F,\mathfrak v)$ for which $\mathfrak v(\neg\LA)\ne\emptyset$. As $\mathfrak v(\neg\LA)$ is an internal set in $\mathfrak F$, every $x\in \mathfrak v(\neg\LA)$ has a successor in $\mathfrak v(\neg\LA)$. 
%$\mathfrak M$ based on this frame with internal sets $\INT$, for which $X = \{x \mid \mathfrak M, x \not\models\LA\} \ne \emptyset$. 
%Clearly, $x \in X\in \INT$ has a successor in $X$. On the other hand, by Lemma~\ref{maxpoints}, 
%$X$ $\mathfrak v(\neg\LA)$ has a maximal point, which can only be $a_0$ or $a_1$, contrary to the definition of $\mathfrak F$.
%
Consider the rooted models $\mathfrak M_i, x_i$, $i=1,2$, shown in Fig.~\ref{GL3gf}, both of which are based on a frame isomorphic to  $\mathfrak F$. 
It is readily checked that $\mathfrak{M}_{1},x_1 \models \varphi_1$, $\mathfrak{M}_{2},x_2 \models \neg\varphi_2$, and the depicted relation $\bis$ is a $\sigma$-bisimulation 
between $\mathfrak M_1$ and $\mathfrak M_2$ with $x_1\bis x_2$.
%$\bullet \lhd \bullet \lhd \mathfrak C(\clustert, \bullet)$. 
%with root $r$.

In fact, the argument above shows that none of the logics $L$ in the interval $\KFT \subseteq L \subseteq \GLT$ has the CIP.

\begin{figure}[htbp]
\centering
\begin{tikzpicture}[>=latex,line width=0.5pt,xscale = 1.15,yscale = .65]
\node[]  at (-1,0) {$\mathfrak M_2$};
\node[point,fill=black,scale = 0.7,label=below:{\footnotesize $\neg\varphi_2$},label=above:{\footnotesize $\qquad x_2$}] (x2) at (0,0) {};
\node[point,fill=black,scale = 0.7,label=below:{\footnotesize $p_2,\neg q_2$},label=above:{\footnotesize $y_2$}] (y2) at (1,0) {};
\node[point,scale = 0.7,label=above:{\footnotesize $a_2^0$},label=below:{\footnotesize $p_2$}] (a20) at (2.5,0) {};
\node[point,scale = 0.7,label=above:{\footnotesize $a_2^1$},label=below:{\footnotesize $p_1$}] (a21) at (3.5,0) {};
\draw[] (3,.1) ellipse (1 and .95);
\node[scale = 0.9]  at (4.1,-.5) {$q_2$};
\node[]  at (4.5,0) {$\dots$};
\node[point,fill=black,scale = 0.7,label=above right:{\footnotesize $\yy_2^3$},label=below:{\footnotesize $p_1,q_2$}] (b23) at (5,0) {};
\node[point,fill=black,scale = 0.7,label=above right:{\footnotesize $\yy_2^2$},label=below:{\footnotesize $p_2,q_2$}] (b22) at (6,0) {};
\node[point,fill=black,scale = 0.7,label=above right:{\footnotesize $\yy_2^1$},label=below:{\footnotesize $p_1,q_2$}] (b21) at (7,0) {};
\node[point,fill=black,scale = 0.7,label=above right:{\footnotesize $\yy_2^0$},label=below:{\footnotesize $p_2,q_2$}] (b20) at (8,0) {};
\draw[->] (x2) to (y2);
\draw[->] (y2) to (2,0);
\draw[->] (b23) to (b22);
\draw[->] (b22) to (b21);
\draw[->] (b21) to (b20);
\node[]  at (-1,3) {$\mathfrak M_1$};
\node[point,fill=black,scale = 0.7,label=above:{\footnotesize $\varphi_1$},label=below:{\footnotesize $\qquad x_1$}] (x1) at (0,3) {};
\node[point,fill=black,scale = 0.7,label=above:{\footnotesize $p_1,\neg q_1$},label=below:{\footnotesize $y_1$}] (y1) at (1,3) {};
\node[point,scale = 0.7,label=below:{\footnotesize $a_1^0$},label=above:{\footnotesize $p_2$}] (a10) at (2.5,3) {};
\node[point,scale = 0.7,label=below:{\footnotesize $a_1^1$},label=above:{\footnotesize $p_1$}] (a11) at (3.5,3) {};
\draw[] (3,2.9) ellipse (1 and .95);
%\node[scale = 0.8]  at (3,3.9) {$q_1$};
\node[scale = 0.9]  at (4.1,3.5) {$q_1$};
\node[]  at (4.5,3) {$\dots$};
\node[point,fill=black,scale = 0.7,label=below right:{\footnotesize $\yy_1^3$},label=above:{\footnotesize $p_1,q_1$}] (b13) at (5,3) {};
\node[point,fill=black,scale = 0.7,label=below right:{\footnotesize $\yy_1^2$},label=above:{\footnotesize $p_2,q_1$}] (b12) at (6,3) {};
\node[point,fill=black,scale = 0.7,label=below right:{\footnotesize $\yy_1^1$},label=above:{\footnotesize $p_1,q_1$}] (b11) at (7,3) {};
\node[point,fill=black,scale = 0.7,label=below right:{\footnotesize $\yy_1^0$},label=above:{\footnotesize $p_2,q_1$}] (b10) at (8,3) {};
\draw[->] (x1) to (y1);
\draw[->] (y1) to (2,3);
\draw[->] (b13) to (b12);
\draw[->] (b12) to (b11);
\draw[->] (b11) to (b10);
\node[gray]  at (-.5,1.5) {$\bis$};
\draw[gray,thick,dotted] (x1) to (x2);
\draw[gray,thick,dotted] (y1) to (a21);
\draw[gray,thick,dotted] (y2) to (a10);
\draw[gray,thick,dotted] (a10) to (a20);
\draw[gray,thick,dotted] (a11) to (a21);
\draw[gray,thick,dotted] (b13) to (b23);
\draw[gray,thick,dotted] (b12) to (b22);
\draw[gray,thick,dotted] (b11) to (b21);
\draw[gray,thick,dotted] (b10) to (b20);
\end{tikzpicture}
\caption{$\sigma$-bisimilar models based on a descriptive frame for $\GLT$.}\label{GL3gf}
\end{figure}
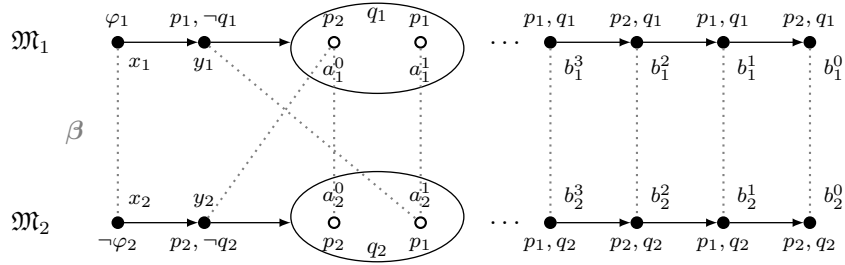

%
%
%witnessing $\mathfrak{M}_{1},r \sim_{\sigma} \mathfrak{M}_{2},r$.
%\hfill $\dashv$
%\end{example}
%
%\begin{example}\label{ex:Log(N)}\em

%\begin{example}\label{ex:Log(N)}\em

$(b)$ Consider next the logic $\Log\{(\mathbb N,<)\}$ and show that the formulas 
\[
%\begin{equation}\label{phi'}
\varphi_1' = \Diamond( p_{1} \wedge  \Diamond^+ \neg q_{1}) \wedge \Box (p_{2} \rightarrow \Box^+ q_{1}) \land \Diamond r \wedge \neg \Diamond (r \wedge \Diamond p_{1}) 
%\wedge  \Box(p_{1} \rightarrow \neg p_{2}) 
%\end{equation}
\]
and $\varphi_2$ given by~\eqref{psi} do not have an interpolant in it, though $(\varphi_1 \rightarrow \varphi_2) \in \KFT$,
and so $(\varphi_1' \rightarrow \varphi_2) \in \KFT \subseteq \Log\{(\mathbb N,<)\}$. 

%It is easy to see that $\varphi_1' \to \varphi_2$ is valid in all finite lassos, and so $(\varphi_1' \to \varphi_2) \in \Log\{(\mathbb N,<)\}$.

As in $(a)$ above, any models $\mathfrak M_i$, $i=1,2$, satisfying the conditions of Theorem~\ref{criterion} for $\varphi_1'$ and $\varphi_2$ cannot be based on Kripke frames, however the reason for this is slightly different. Suppose $\bis$ is a bisimulation witnessing these conditions. Then the models $\mathfrak{M}_i$ must contain infinite ascending chains such as those in Example~\ref{ex:GL.3}~$(a)$. 
%Fig.~\ref{GL3}. 
Also, the model $\mathfrak{M}_1$ with $\mathfrak{M}_1,x_1 \models \varphi_1'$ must contain a point $z$ such that $x_1R_1z$ and $\mathfrak{M}_1,z \models r \land \Box \neg p_1$, which means that $z$ is located after all of the $x_1^j$, $j < \omega$. But then the Kripke frame $\mathfrak F_1$ underlying $\mathfrak{M}_1$ is not a frame for $\Log\{(\mathbb N,<)\}$, as it refutes its axiom $\Box (\Box p \to p) \to (\Diamond\Box p \to \Box p)$ if we make $p$ true everywhere after the initial ascending chain in $\mathfrak F_1$ and false elsewhere. 

%To\nz{TBD} define the required models based on a frame for $\Log\{(\mathbb N,<)\}$, we can use the descriptive frame ${\mathfrak  C}(0,\clustert,0) \lhd {\mathfrak  C}(0,\clustero)$. Here, ${\mathfrak  C}(0,\clustero,0)$ is the same as in Example~\ref{ex:GL.3} and ${\mathfrak  C}(0,\clustero)$ is based on $\omega^<(0) \lhd \clustero$ shown below whose internal sets comprise all finite subsets of $\omega^<(0) \lhd \clustero$ without $a$ and the complements thereof.\\
%%
%\centerline{\includegraphics[scale=0.6]{../Pics/N0}}\\ 
%%
%Finally, the internal sets in ${\mathfrak  C}(0,\clustert,0) \lhd {\mathfrak  C}(0,\clustero)$ 
%take the form $Y_1 \cup Y_2$, where $Y_1$ is internal in ${\mathfrak  C}(0,\clustert,0)$ and $Y_2$ in ${\mathfrak  C}(0,\clustero)$. To show that the axiom $\Box (\Box p \to p) \to (\Diamond\Box p \to \Box p)$ is valid in ${\mathfrak  C}(0,\clustert,0) \lhd {\mathfrak  C}(0,\clustero)$, suppose otherwise, that is, some model $\mathfrak M$ based on this frame makes this axiom false. Let $x$ be the maximal point of the set $\{y \mid \mathfrak M,y \models \neg p \}$. Clearly, $y \ne a$ because of the premise $\Diamond\Box p$. Also, $y$ cannot be any of the irreflexive points because of the premise $\Box (\Box p \to p)$. And $y$ cannot be $a_i$ since then it is not maximal.

The picture below shows models $\mathfrak M_1$ and $\mathfrak M_2$ based on $\bullet \lhd \bullet \lhd {\mathfrak  C}(\clustert, \bullet) \lhd \circ$ and satisfying the conditions of Theorem~\ref{criterion} for $\varphi_1'$ and $\varphi_2$. That this frame is a frame for $\Log\{(\mathbb N,<)\}$ follows from 
Example~\ref{e:canform}~$(b)$.
% \eqref{omegacanform}. 
\hfill $\dashv$
\end{example} 
%
%\begin{figure}[htbp]
%\centering
%\centerline{\includegraphics[scale=0.6]{../Pics/N10}}
%
\centerline{
\begin{tikzpicture}[>=latex,line width=0.5pt,xscale = 1.1,yscale = .65]
\node[]  at (-1,0) {$\mathfrak M_2$};
\node[point,fill=black,scale = 0.7,label=below:{\footnotesize $\neg\varphi_2$},label=above:{\footnotesize $\qquad x_2$}] (x2) at (0,0) {};
\node[point,fill=black,scale = 0.7,label=below:{\footnotesize $p_2,\neg q_2$},label=above:{\footnotesize $y_2$}] (y2) at (1,0) {};
\node[point,scale = 0.7,label=above:{\footnotesize $a_2^0$},label=below:{\footnotesize $p_2$}] (a20) at (2.5,0) {};
\node[point,scale = 0.7,label=above:{\footnotesize $a_2^1$},label=below:{\footnotesize $p_1$}] (a21) at (3.5,0) {};
\draw[] (3,.1) ellipse (1 and .9);
\node[scale = 0.8]  at (4.1,-.5) {$q_2$};
\node[]  at (4.5,0) {$\dots$};
\node[point,fill=black,scale = 0.7,label=above right:{\footnotesize $\yy_2^3$},label=below:{\footnotesize $p_1,q_2$}] (b23) at (5,0) {};
\node[point,fill=black,scale = 0.7,label=above right:{\footnotesize $\yy_2^2$},label=below:{\footnotesize $p_2,q_2$}] (b22) at (6,0) {};
\node[point,fill=black,scale = 0.7,label=above right:{\footnotesize $\yy_2^1$},label=below:{\footnotesize $p_1,q_2$}] (b21) at (7,0) {};
\node[point,fill=black,scale = 0.7,label=above right:{\footnotesize $\yy_2^0$},label=below:{\footnotesize $p_2,q_2$}] (b20) at (8,0) {};
\node[point,scale = 0.7,label=below:{\footnotesize $r,q_2$}] (z2) at (9,0) {};
\draw[->] (x2) to (y2);
\draw[->] (y2) to (2,0);
\draw[->] (b23) to (b22);
\draw[->] (b22) to (b21);
\draw[->] (b21) to (b20);
\draw[->] (b20) to (z2);
\node[]  at (-1,3) {$\mathfrak M_1$};
\node[point,fill=black,scale = 0.7,label=above:{\footnotesize $\varphi_1'$},label=below:{\footnotesize $\qquad x_1$}] (x1) at (0,3) {};
\node[point,fill=black,scale = 0.7,label=above:{\footnotesize $p_1,\neg q_1$},label=below:{\footnotesize $y_1$}] (y1) at (1,3) {};
\node[point,scale = 0.7,label=below:{\footnotesize $a_1^0$},label=above:{\footnotesize $p_2$}] (a10) at (2.5,3) {};
\node[point,scale = 0.7,label=below:{\footnotesize $a_1^1$},label=above:{\footnotesize $p_1$}] (a11) at (3.5,3) {};
\draw[] (3,2.9) ellipse (1 and .9);
\node[scale = 0.8]  at (4.1,3.5) {$q_1$};
\node[]  at (4.5,3) {$\dots$};
\node[point,fill=black,scale = 0.7,label=below right:{\footnotesize $\yy_1^3$},label=above:{\footnotesize $p_1,q_1$}] (b13) at (5,3) {};
\node[point,fill=black,scale = 0.7,label=below right:{\footnotesize $\yy_1^2$},label=above:{\footnotesize $p_2,q_1$}] (b12) at (6,3) {};
\node[point,fill=black,scale = 0.7,label=below right:{\footnotesize $\yy_1^1$},label=above:{\footnotesize $p_1,q_1$}] (b11) at (7,3) {};
\node[point,fill=black,scale = 0.7,label=below right:{\footnotesize $\yy_1^0$},label=above:{\footnotesize $p_2,q_1$}] (b10) at (8,3) {};
\node[point,scale = 0.7,label=above:{\footnotesize $r,q_1$}] (z1) at (9,3) {};
\draw[->] (x1) to (y1);
\draw[->] (y1) to (2,3);
\draw[->] (b13) to (b12);
\draw[->] (b12) to (b11);
\draw[->] (b11) to (b10);
\draw[->] (b10) to (z1);
\node[gray]  at (-.5,1.5) {$\bis$};
\draw[gray,thick,dotted] (x1) to (x2);
\draw[gray,thick,dotted] (y1) to (a21);
\draw[gray,thick,dotted] (y2) to (a10);
\draw[gray,thick,dotted] (a10) to (a20);
\draw[gray,thick,dotted] (a11) to (a21);
\draw[gray,thick,dotted] (b13) to (b23);
\draw[gray,thick,dotted] (b12) to (b22);
\draw[gray,thick,dotted] (b11) to (b21);
\draw[gray,thick,dotted] (b10) to (b20);
\draw[gray,thick,dotted] (z1) to (z2);
\end{tikzpicture}
}

\section{Interpolant existence in logics above $\KFT$}\label{section:K4.3}

We now generalise Theorem~\ref{dperscofinal} to all finitely axiomatisable logics containing $\KFT$. 
It turns out that, even though these logics do not have the polysize bisimilar model property in general, the structure of the models required in Theorem~\ref{criterion} is perfectly understandable. 
We show that
one can assemble a pair of bisimilar models witnessing the absence of an interpolant for $\varphi_1$ and $\varphi_2$ in any $L \supseteq \KFT$ 
as the ordered sum of finitely-many `nice' models, which are either finite or infinite but finitely `presentable'\!.
Hence, we say that all $L\supseteq \KFT$ have the `\qfbmp'\!.
Moreover, if $L$ is finitely axiomatisable, we can replace `finite' by `polynomial in $\cbound$ and $\max(|\varphi_1|,|\varphi_2|)$', for some constant $\cbound$ depending on $L$ only.  In this case, we say that $L$ has the `\qpbmp'\!.

\S\ref{section:K4.3} is organised as follows. 
In \S\ref{ss:qpbmp}, we formulate our main results 
(Theorems~\ref{t:structmodel}--\ref{t:finax} and \ref{t:iepcoNP}), 
%\ref{t:smallmodel}, \ref{t:finax} 
and show how Theorem~\ref{t:finaxstruct} implies Theorem~\ref{t:iepcoNP}.
In \S\ref{ss:modular} and \S\ref{ss:small}, we prove  Theorem~\ref{t:structmodel}.
Then, in \S\ref{ss:finax}, we show how to fine-tune the proof of Theorem~\ref{t:structmodel} and obtain proofs of Theorem~\ref{t:finaxstruct} and \ref{t:finax}.
Finally, in \S~\ref{ss:cofinalsfr}, we formulate and prove an interesting consequence 
of our methods for cofinal subframe logics (Theorem~\ref{t:cofinalsfr}). 
%\nb{theorem numbering?}
%The remainder of \S\ref{section:K4.3} contains the proofs of Theorems~\ref{t:structmodel} and \ref{t:finaxstruct}.

%**********************************************

\subsection{The \qpbmp}\label{ss:qpbmp}

Given a finite signature $\delta$, a $\delta$-model $\mathfrak M=(\mathfrak F,\mathfrak w)$
is called \emph{simple} if either $\mathfrak F$ is finite or 
$\mathfrak F = \chain{k}{\ast}$, for $0<k<\omega$ and $\ast\in\{\bullet,\circ\}$, and, 
 for every $p \in \delta$, there is $A_p\subseteq\{0,\dots,k-1\}$ with $\mathfrak{w}(p)=\bigcup_{i\in A_p}X_i$, where the $X_i$ are the infinite generators of the internal sets in
$\chain{k}{\ast}$ defined in Example~\ref{k-omega}.
Thus, even though $\chain{k}{\ast}$ is infinite, 
%for the  finite signature $\delta=\sig(\varphi_1) \cup \sig(\varphi_2)$,
 any \simple{} $\delta$-model based on it is fully determined by the  \emph{finitary} 
information provided by the sets $A_p$, $p\in\delta$, that is, by the atomic $\delta$-types of the points in
the $\cluster{k}$-cluster.
A $\delta$-model is called \emph{\qfin} if it is the ordered sum of finitely-many \simple{} models.

\begin{definition}\label{d:bmp}
\rm
A logic $L \supseteq \KFT$ is said to have the \emph{\qfbmp} if, for any formulas $\varphi_1$, $\varphi_2$ without an interpolant in $L$,  
there are rooted \qfin{} $\delta$-models $\mathfrak N_1,x_1$ and $\mathfrak N_2,x_2$ 
satisfying conditions $(a)$--$(c)$ below, for $\delta = \sig(\varphi_1) \cup \sig(\varphi_2)$ and $\sigma = \sig(\varphi_1) \cap \sig(\varphi_2)$:
\begin{itemize}
\item[$(a)$] 
$\mathfrak N_1,x_1\models\varphi_1$ and $\mathfrak N_2,x_2\models\neg\varphi_2$;

\item[$(b)$] 
$\mathfrak N_1$ and $\mathfrak N_2$ are based on frames for $L$;

\item[$(c)$] 
$\mathfrak{N}_{1},x_{1} \sim_{\sigma} \mathfrak{N}_{2},x_{2}$.
\end{itemize}
\end{definition}

Our first result is:

\begin{theorem}\label{t:smallmodel}
All $L \supseteq\KFT$ have the \qfbmp.
\end{theorem}

We actually prove a stronger Theorem~\ref{t:structmodel} that prescribes more structure for the pair of \qfin{} models witnessing the lack of an interpolant, which makes it easy to deduce the existence of a $\sigma$-bisimulation between the models. The prescribed structure is easily checkable, which is used in the proof of the main Theorem~\ref{t:iepcoNP}.
To formulate our `structural' theorem, we require a few definitions.
%
%Suppose $L \supseteq \KFT$ and $\varphi_1$, $\varphi_2$ are formulas without an interpolant in $L$. 

For $0 < m < \omega$, let 
$m^{<} = \underbrace{\bullet \lhd \dots \lhd \bullet}_m$. An \emph{\atomic{} frame}
%for $\varphi_1$, $\varphi_2$ 
takes one of the forms
%\vspace*{-4mm}
%
\begin{multline}\label{atomicF}
m^{<}, \quad \cluster{1} \lhd m^{<}, \quad \cluster{k}, \quad \chain{k}{\bullet}, \quad \chain{k}{\circ},\\
\mbox{where $0<m<\omega$ and $0<k\le 2^{|\delta|}$}.
%\kbound =3+3\max(|\varphi_1|,|\varphi_2|)$.}
\end{multline}
%
%If $m = 0$,  $m^{<}$ is deemed to be empty.
%We call a model $\mathfrak N$ based on $\chain{k}{\bullet}\lhd\cluster{1}$ \emph{\simple}if $\mathfrak N=\mathfrak N_1\lhd\mathfrak N_2$ for some \simple{} model $\mathfrak N_1$ based on $\chain{k}{\bullet}$.
%
The \emph{size} $\size{\mathfrak{F}}$
of an \atomic{} $\mathfrak{F}$ is defined by taking $\size{m^{<}}=m$, $\size{\cluster{1} \lhd m^{<}}=1+m$,  and
$\size{\cluster{k}}=\size{\chain{k}{\bullet}}=\size{\chain{k}{\circ}}=k$. 
%and $\size{\chain{k}{\bullet}\lhd\cluster{1}}=k+1$
If $\mathfrak M=\mathfrak M_0\lhd\dots\lhd\mathfrak M_{n-1}$, for some $0<n<\omega$ and
\simple{} $\delta$-models $\mathfrak M_j$ based on \atomic{} frames $\mathfrak F_j$, $j<n$, then
%For a \simple{} $\delta$-model $\mathfrak{M}$, 
we set $\size{\mathfrak{M}}=\size{\mathfrak F_0\lhd\dots\lhd\mathfrak F_{n-1}}=
\size{\mathfrak F_0}+\dots +\size{\mathfrak F_{n-1}}$.
%\sum_{j<n}\size{\mathfrak F_j}$.

\begin{definition}\label{d:match}\rm
Suppose $\mathfrak N_i$, $i=1,2$, is the ordered sum of finitely-many \simple{} $\delta$-models based on \atomic{} frames.
The pair $(\mathfrak N_1,\mathfrak N_2)$ is called \emph{\match}
if it satisfies one of the following conditions \bmp{a}--\bmp{c}:
\begin{itemize}
\item[\bmp{a}] 
$\mathfrak N_1$ and $\mathfrak N_2$ are \simple{} models based on the same \atomic{} frame $\mathfrak H$ with
$\at^\sigma_{\mathfrak N_1}(y)=\at^\sigma_{\mathfrak N_2}(y)$, for every point $y$ in $\mathfrak H$;
 
 \item[\bmp{b}] 
the final clusters $C_i$ of  $\mathfrak N_i$, $i=1,2$, are non-degenerate and,
for every point $y_1$ in $\mathfrak N_1$, there is $y_2 \in C_2$
%the final cluster of $\mathfrak N_2$
with $\at^\sigma_{\mathfrak N_1}(y_1)=\at^\sigma_{\mathfrak N_2}(y_2)$, and, 
for every point $y_2$ in $\mathfrak N_2$, there is $y_1 \in C_1$
%the final cluster of $\mathfrak N_1$
with $\at^\sigma_{\mathfrak N_2}(y_2) = \at^\sigma_{\mathfrak N_1}(y_1)$;

 \item[\bmp{c}] 
 \begin{enumerate}
%\item
%$\mathfrak N_i$ is an ordered sum of $\rn_i$-many simple models based on \atomic{} frames,
%for $i=1,2$ and $0<\rn_i\le 2\kbound$;
%
\item
the last $\lhd$-components of the $\mathfrak N_1$ and $\mathfrak N_2$ are 
based on the same \atomic{} frame $\mathfrak G$ of the form $\chain{k}{\bullet}$ or $\chain{k}{\circ}$,
with $0<k\le 2^{|\delta|}$;

\item
$\at^\sigma_{\mathfrak N_1}(y)=\at^\sigma_{\mathfrak N_2}(y)$,
for every point $y$ in the root $\cluster{k}$-cluster $A_k$ of $\mathfrak G$; 
%and if $\mathfrak G$ is of the form $\chain{k}{\bullet}\lhd\cluster{1}$,  then  $\at^\sigma_{\mathfrak N_1}(u)=\at^\sigma_{\mathfrak N_2}(u)$ for the point $u$ in its final $\cluster{1}$-cluster;

\item
for every point $y_1$ in any non-last $\lhd$-component of $\mathfrak N_1$, there is $y_2 \in A_k$
with $\at^\sigma_{\mathfrak N_1}(y_1)=\at^\sigma_{\mathfrak N_2}(y_2)$ and, 
for every point $y_2$ in any non-last $\lhd$-component of $\mathfrak N_2$, there is $y_1\in A_k$
with $\at^\sigma_{\mathfrak N_2}(y_2)=\at^\sigma_{\mathfrak N_1}(y_1)$. 
\end{enumerate}
\end{itemize}
If $(\mathfrak N_1,\mathfrak N_2)$ satisfies condition \bmp{x}, for $x=a,b,c$, we say that 
it \emph{is of type} \bmp{x}. 
\end{definition}

%*****

\begin{center}
\begin{tikzpicture}[>=latex,line width=0.5pt,xscale = .65,yscale = 1]
\node[]  at (-.6,1.25) {{\small $\sim_\sigma$}};
\node[]  at (-.6,2.8) {{\small $(a)$}};
\node[]  at (-.5,.25) {{\small $\mathfrak N_2$}};
\draw [thin] (0,0) rectangle ++(1.5,.5);
\node[]  at (-.5,2.25) {{\small $\mathfrak N_1$}};
\draw [thin] (0,2) rectangle ++(1.5,.5);
\draw [thick,dotted,gray] (.15,.6) -- (.15,1.9);
\draw [thick,dotted,gray] (.45,.6) -- (.45,1.9);
\draw [thick,dotted,gray] (.75,.6) -- (.75,1.9);
\draw [thick,dotted,gray] (1.05,.6) -- (1.05,1.9);
\draw [thick,dotted,gray] (1.35,.6) -- (1.35,1.9);
\draw [thick,dashed] (1.75,-.5) -- (1.7,3);
\node[]  at (2.2,2.8) {{\small $(b)$}};
\draw[thin] (2,0) rectangle ++(1.2,.5);
\node[]  at (3.5,.25) {$\lhd$};
\draw[thin] (3.8,0) rectangle ++(.6,.5);
\node[]  at (4.7,.25) {$\lhd$};
\draw[thin] (5,0) rectangle ++(.6,.5);
\node[]  at (5.9,.25) {$\lhd$};
\draw[thin] (6.2,0) rectangle ++(.6,.5);
\node[]  at (7.1,.25) {$\lhd$};
\draw[thin] (7.4,0) rectangle ++(1.2,.5);
\draw [thin] (8.35,.25) ellipse (.17 and .2);
\node[gray]  at (5.3,.7) {$\overbrace{\hspace*{4.2cm}}$};
\draw[thin] (2,2) rectangle ++(1.4,.5);
\node[]  at (3.7,2.25) {$\lhd$};
\draw[thin] (4,2) rectangle ++(.6,.5);
\node[]  at (4.9,2.25) {$\lhd$};
\draw[thin] (5.2,2) rectangle ++(1.4,.5);
\draw [thin] (6.35,2.25) ellipse (.17 and .2);
\node[gray]  at (4.3,1.8) {$\underbrace{\hspace*{3cm}}$};
\draw [thick,dotted,gray] (8.35,.6) -- (4.35,1.65);
\draw [thick,dotted,gray] (5.45,.9) -- (6.35,1.95);
\draw [thick,dashed] (8.9,-.5) -- (8.9,3);
\node[]  at (9.4,2.8) {{\small $(c)$}};
\draw[thin] (9.2,0) rectangle ++(1.2,.5);
\node[]  at (10.7,.25) {$\lhd$};
\draw[thin] (11,0) rectangle ++(.6,.5);
\node[]  at (11.9,.25) {$\lhd$};
\draw[thin] (12.2,0) rectangle ++(.6,.5);
\node[]  at (13.1,.25) {$\lhd$};
\draw[thin] (13.4,0) rectangle ++(.6,.5);
\node[]  at (14.3,.25) {$\lhd$};
\draw[thin] (14.6,0) rectangle ++(1.8,.5);
\draw [thin] (14.85,.25) ellipse (.17 and .2);
\node[]  at (15.7,.25) {{\tiny $\cdots\ast\ast$}};
\node[gray]  at (12.1,.7) {$\overbrace{\hspace*{3.7cm}}$};
\draw[thin] (9.2,2) rectangle ++(1.4,.5);
\node[]  at (10.9,2.25) {$\lhd$};
\draw[thin] (11.2,2) rectangle ++(.6,.5);
\node[]  at (12.1,2.25) {$\lhd$};
\draw[thin] (12.4,2) rectangle ++(1.8,.5);
\draw [thin] (12.65,2.25) ellipse (.17 and .2);
\node[]  at (13.55,2.25) {{\tiny $\cdots\ast\ast$}};
\node[gray]  at (11,1.8) {$\underbrace{\hspace*{2.3cm}}$};
\draw [thick,dotted,gray] (14.8,.6) -- (10.9,1.65);
\draw [thick,dotted,gray] (12.1,.9) -- (12.65,1.95);
\draw [thick,dotted,gray] (15.75,.4) -- (13.7,2.2);
\draw [thick,dotted,gray] (16.05,.4) -- (14,2.2);

\end{tikzpicture}
\end{center}

%****

The following lemma justifies this definition:
\begin{lemma}\label{l:gbisall}
Suppose $\mathfrak N_i=\mathfrak N_i^{0}\lhd\dots\lhd\mathfrak N_i^{N-1}$, for $i=1,2$ and $0<N<\omega$,  and $x_i$ is a root of $\mathfrak N_i^0$. 
If $\at^\sigma_{\mathfrak N_1}(x_1) = \at^\sigma_{\mathfrak N_2}(x_2)$ and, for every $\ell<N$, the pair $(\mathfrak N_1^\ell,\mathfrak N_2^\ell)$ is \match, then 
$\mathfrak N_1,x_1\sim_\sigma\mathfrak N_2,x_2$.
\end{lemma}
\begin{proof}
First, we show that, for every $\ell<N$, there is a \globalb{} $\sigma$-bisimulation between $\mathfrak N_1^\ell$ and $\mathfrak N_2^\ell$. 
This is clear for $(\mathfrak N_1^\ell,\mathfrak N_2^\ell)$ of type \bmp{a}, in which case  
the identity function on $\mathfrak H$ is a $\sigma$-bisimulation. 

If $(\mathfrak N_1^\ell,\mathfrak N_2^\ell)$ is of type \bmp{b},
then the final clusters $C_i$ of $\mathfrak N_i^\ell$, $i=1,2$, are  non-degenerate. Thus, $\bis_1\cup\bis_2$
is a \globalb{} $\sigma$-bisimulation between $\mathfrak N_1^\ell$ and $\mathfrak N_2^\ell$,
\begin{align*}
& \bis_1=\bigl\{(y_1,y_2) \mid \mbox{$y_1$ in $\mathfrak N_1^\ell$, $y_2$ in $C_2$,  $\at^\sigma_{\mathfrak N_1^\ell}(y_1)=\at^\sigma_{\mathfrak N_2^\ell}(y_2)$} \bigr\},\\
& \bis_2=\bigl\{(y_1,y_2) \mid \mbox{$y_2$ in $\mathfrak N_2^\ell$, $y_1$ in $C_1$,
$\at^\sigma_{\mathfrak N_1^\ell}(y_1)=\at^\sigma_{\mathfrak N_2^\ell}(y_2)$} \bigr\}.
\end{align*}

If $(\mathfrak N_1^\ell,\mathfrak N_2^\ell)$ is of type \bmp{c},
suppose  
$\mathfrak N_i^\ell=\mathfrak N_i^{0,\ell}\lhd\dots\lhd\mathfrak N_i^{\nn_i-1,\ell}$, for 
$0<\nn_i<\omega$ and $i=1,2$. 
By \bmp{c}.1, $\mathfrak N_1^{\nn_1-1,\ell}$ and $\mathfrak N_2^{\nn_2-1,\ell}$ are \simple{} models based on the same \atomic{} frame of the form $\chain{k}{\bullet}$ or $\chain{k}{\circ}$.
% or $\chain{k}{\bullet}\lhd\cluster{1}$.
As in  Example~\ref{k-omega}, let $A_k=\{a_s\mid s<k\}$ and $W_k=A_k\cup\{\yy_n\mid n<\omega\}$ (containing all the points of $\chain{k}{\ast}$).
% and assume that $W_k\cup\{u\}$ are the points of $\chain{k}{\bullet}\lhd\cluster{1}$.
We claim that
\begin{equation}\label{tailsame}
\at^\sigma_{\mathfrak N_1^\ell}(\yy_n)=\at^\sigma_{\mathfrak N_2^\ell}(\yy_n),\quad\mbox{for all $n<\omega$}.
\end{equation}
Indeed, suppose $n<\omega$ and let $s<k$ be such that $n\equiv s$ (mod $k$).
As $\mathfrak N_i^{\nn_i-1,\ell}$ is a \simple{} model, we have 
\[
\at^\sigma_{\mathfrak N_i^\ell}(\yy_n)=\at^\sigma_{\mathfrak N_i^{\nn_i-1,\ell}}(\yy_n)=
\at^\sigma_{\mathfrak N_i^{\nn_i-1,\ell}}(a_s)=\at^\sigma_{\mathfrak N_i^\ell}(a_s),\quad\mbox{for $i=1,2$,}
\]
and so  \eqref{tailsame} follows from \bmp{c}.2.
Now let 
\begin{align*}
& \bis_1=\bigl\{(y_1,y_2) \mid \mbox{$y_1$ in $\mathfrak N_1^{0,\ell}\lhd\dots\lhd\mathfrak N_1^{\nn_1-2,\ell}$, 
$y_2\in A_k$, $\at^\sigma_{\mathfrak N_1^\ell}(y_1)=\at^\sigma_{\mathfrak N_2^\ell}(y_2)$} \bigr\},\\
& \bis_2=\bigl\{(y_1,y_2) \mid \mbox{$y_2$ in $\mathfrak N_2^{0,\ell}\lhd\dots\lhd\mathfrak N_2^{\nn_2-2,\ell}$, $y_1\in A_k$, $\at^\sigma_{\mathfrak N_1^\ell}(y_1)=\at^\sigma_{\mathfrak N_2^\ell}(y_2)$} \bigr\}.
\end{align*}
By \eqref{tailsame} and \bmp{c}.2--3, 
%if the $\mathfrak N_i^{\rn_i-1}$ take the form $\chain{k}{\ast}$ or $\chain{k}{\bullet}\lhd\cluster{1}$, then 
$\bis_1\cup\bis_2\cup \bigl\{(\yy_n,\yy_n)\mid n<\omega\bigr\}$ 
%or, respectively, $\bis\cup\{(u,u)\}$ 
is a \globalb{} $\sigma$-bisimulation between $\mathfrak N_1^\ell$ and $\mathfrak N_2^\ell$. 
Finally, if $\bis^0$ is a \globalb{} $\sigma$-bisimulation between $\mathfrak N_1^0$ and $\mathfrak N_2^0$, then $\bis^0\cup\{(x_1,x_2)\}$ is also a \globalb{} $\sigma$-bisimulation between $\mathfrak N_1^0$ and $\mathfrak N_2^0$ because  
$\at^\sigma_{\mathfrak N_1}(x_1) = \at^\sigma_{\mathfrak N_2}(x_2)$.
%~$(c)$ in Definition~\ref{d:bmp}. 
The union of the constructed \globalb{} bisimulations is a (\globalb)  bisimulation $\bis$ between $\mathfrak N_1$ and $\mathfrak N_2$ with $x_1\bis x_2$, 
%and so $\mathfrak N_1,x_1\sim_\sigma \mathfrak N_2,x_2$ 
as required.
\end{proof}

The following strengthening of Theorem~\ref{t:smallmodel}
will be proved in \S\ref{ss:modular}--\ref{ss:small}:

\begin{theorem}\label{t:structmodel}
For any logic $L \supseteq \KFT$ and formulas $\varphi_1$, $\varphi_2$ without an interpolant in $L$,  
there are rooted $\delta$-models $\mathfrak N_1,x_1$ and $\mathfrak N_2,x_2$ 
satisfying $(a)$--$(d)$ below, for $\delta = \sig(\varphi_1) \cup \sig(\varphi_2)$ and $\sigma = \sig(\varphi_1) \cap \sig(\varphi_2)$\textup{:}
\begin{itemize}
\item[$(a)$] 
$\mathfrak N_1,x_1\models\varphi_1$ and $\mathfrak N_2,x_2\models\neg\varphi_2$\textup{;}

\item[$(b)$] 
each $\mathfrak N_i$, $i=1,2$, is based on a frame for $L$\textup{;}

\item[$(c)$] 
$\at^\sigma_{\mathfrak N_1}(x_1) = \at^\sigma_{\mathfrak N_2}(x_2)$\textup{;}

\item[$(d)$]
there is $N=\mathcal{O}\bigl(\max(|\varphi_1|,|\varphi_2|)\bigr)$ such that 
$\mathfrak N_i=\mathfrak N_i^0\lhd\dots\lhd\mathfrak N_i^{N-1}$, $i=1,2$, and, for any $\ell<N$,
%there exist $N$, $0<N\leq 2\kbound-1$, and
%$\nn_i^\ell>0$, $\ell<N$, $i=1,2$, such that $\sum_{\ell<N}\nn_i^\ell\le 3\kbound -1$, 
%$\mathfrak N_i=\mathfrak N_i^0\lhd\dots\lhd\mathfrak N_i^{N-1}$ and, for each $\ell<N$, 
%
\begin{enumerate}
\item
each $\mathfrak N_i^\ell$ is the ordered sum of 
%$\nn_i^\ell$
$\mathcal{O}\bigl(\max(|\varphi_1|,|\varphi_2|)\bigr)$-many \simple{} $\delta$-models based on \atomic{} frames\textup{;} 
%for some $\nn_i^\ell$ with $0<\nn_i^\ell\le 2\kbound+1$\textup{;}
\item
the pair $(\mathfrak N_1^\ell,\mathfrak N_2^\ell)$ is \match.
\end{enumerate}
\end{itemize}
\end{theorem}

Observe that the models provided by Theorem~\ref{t:structmodel}
%the \qfbmp{} 
are ordered sums of poly\-nomially-many simple models.
However, the sizes of these \simple{} models are not necessarily polynomial in $\max(|\varphi_1|,|\varphi_2|)$. 
%, for some $\varphi_1$ and $\varphi_2$ without an interpolant in $L \supseteq \KFT$.}
%
Our second main result shows that all \emph{finitely axiomatisable} logics $L \supseteq \KFT$ have the 
stronger \emph{\qpbmp}: the lack of an interpolant can be witnessed by a pair \qfin{} models 
of polynomial size.
%$\mathfrak N_i$, $i=1,2$, in Definition~\ref{d:bmp} can be chosen so that 
More precisely, suppose $L$ is given by its canonical axioms as
$L = \KFT \oplus \{\alpha(\mathfrak G_j,\mathfrak D_j,\bot) \mid j\in J_L\}$, for some finite set $J_L$ and $\mathfrak G_j=(V_j,S_j)$.
Let $\cbound=\max_{j\in J_L}|V_j|$.
An \atomic{} frame in \eqref{atomicF} is called \emph{\bounded} if 
it is of the form $m^{<}$ or $\cluster{1} \lhd m^{<}$ with $m\le\cbound+1$, or it has one of the three remaining forms with
%\nb{see pp.30--31; also, for cofinal sfr $k\le\kbound$}
%
\[
k\le\pbound:= 2(\kbound-1)\cdot\max\bigl(\cbound+2,\kbound\bigr)+\kbound,
\]
for the polynomial number $\kbound$ defined in \eqref{kbound}. 
In \S~\ref{ss:finax}, we prove:

\begin{theorem}\label{t:finaxstruct}
For any finitely axiomatisable 
logic $L \supseteq \KFT$ and formulas $\varphi_1$, $\varphi_2$ without an interpolant in $L$,  
there are rooted $\delta$-models $\mathfrak N_1,x_1$ and $\mathfrak N_2,x_2$ 
satisfying $(a)$--$(d)$ from Theorem~\ref{t:structmodel},
in which condition $(d).1$ is strengthened to
%
%below, for $\delta = \sig(\varphi_1) \cup \sig(\varphi_2)$ and $\sigma = \sig(\varphi_1) \cap \sig(\varphi_2)$\textup{:}\nz{same as Th 4.5 with $L$-bounded $(d).1$?}
%%
%\begin{itemize}
%\item[$(a)$] 
%$\mathfrak N_1,x_1\models\varphi_1$ and $\mathfrak N_2,x_2\models\neg\varphi_2$;
%
%\item[$(b)$] 
%each $\mathfrak N_i$, $i=1,2$, is based on a frame for $L$;
%
%\item[$(c)$] 
%$\at^\sigma_{\mathfrak N_1}(x_1) = \at^\sigma_{\mathfrak N_2}(x_2)$;
%
%\item[$(d)$]
%there exist $N$, $0<N\leq 2\kbound-1$ and
%$\nn_i^\ell>0$, $\ell<N$, with $\sum_{\ell<N}\nn_i^\ell\le 3\kbound -1$ 
%%there is $N$, $0<N\leq 2\kbound-1$, 
%such that 
%$\mathfrak N_i=\mathfrak N_i^0\lhd\dots\lhd\mathfrak N_i^{N-1}$, $i=1,2$, and, for each $\ell<N$, 
%
%
\begin{enumerate}
\item
each $\mathfrak N_i^\ell$, $i=1,2$, is the ordered sum of 
%$\nn_i^\ell$
$\mathcal{O}\bigl(\max(|\varphi_1|,|\varphi_2|)\bigr)$-many \simple{} $\delta$-models based on \bounded{} \atomic{} frames. 
%\nb{bound on $\mathfrak N_i$ removed}
%$\mathfrak N_i^\ell$ is the ordered sum of $\nn_i^\ell$-many \simple{} $\delta$-models based on \bounded{}  \atomic{} frames.
%for some $\nn_i^\ell$ with $0<\nn_i^\ell\le 2\kbound+1$\textup{;}
%\item
%the pair $(\mathfrak N_1^\ell,\mathfrak N_2^\ell)$ is \match.
\end{enumerate}
%
%\end{itemize}
%
%Thus, $\size{\mathfrak N_i}\le  (3\kbound-1)\cdot\max\bigl(\cbound+2,\pbound\bigr)$, for $i=1,2$. 
%
\end{theorem}

%By Definition~\ref{d:match} and 
In \S\ref{ss:finax}, we also show: 
%\nb{stuff added at end of \S\ref{ss:finax}, have a look}
%By Lemma~\ref{l:gbisall}, this implies:

\begin{theorem}\label{t:finax}
All finitely axiomatisable $L \supseteq\KFT$ have the \qpbmp, with the size of witnessing models bounded by
%\nb{numbering might not be needed, we will see}
%
\[
%\begin{equation}\label{pbound}
 (3\kbound-1)\cdot\max\bigl(\cbound+2,\pbound\bigr).
% (2\kbound-1)\cdot \bigl(2\kbound+1\bigr) \cdot\max\bigl(\cbound+2,\pbound\bigr). 
%\end{equation}
\]
\end{theorem}

\begin{remark}\rm
As a consequence we obtain that each finitely axiomatisable logic $L \supseteq\KFT$ has the \emph{quasi-polysize model property}: $\varphi \in L$ iff $\varphi$ is true in all models $\mathfrak M$ that are $(i)$ ordered sums of simple models and $(ii)$ are based on a frame for $L$ of size $\mathcal{O}(|\varphi|^2)$;
%simple models $\mathfrak M$ based on a basic frame for $L$ of size $\mathcal{O}(|\varphi|^2)$;  
cf.~\cite{DBLP:journals/mlq/ZakharyaschevA95,DBLP:journals/sLogica/LitakW05}.
\end{remark}

In the remainder of \S\ref{ss:qpbmp},  
we show how Theorem~\ref{t:finaxstruct} implies the following:

\begin{theorem}\label{t:iepcoNP}
The IEP for any fixed finitely axiomatisable logic $L \supseteq\KFT$ is \coNP-complete.
%\nz{$L$ given by can-axioms could be input. Do we want to mention this?}
\end{theorem}
\begin{proof}
We describe an NP-algorithm deciding the complement of the IEP for $L$ given by its canonical axioms~\eqref{canon}. 
Given $\varphi_1$ and  $\varphi_2$, let $\delta=\sig(\varphi_1)\cup\sig(\varphi_2)$. 
We guess polynomial-size $N$. Then, for each $\ell<N$, we guess $z_\ell\in\{a,b,c\}$, and if $z_\ell=a$, we let $\nn_1^\ell=\nn_2^\ell=1$; otherwise, we guess polynomial-size $\nn_i^\ell$, for $i=1,2$; 
we also  guess \simple{} $\delta$-models $\mathfrak N_i^{j,\ell}$, for $\ell<N$, $i=1,2$,  $j<\nn_i^\ell$, based on \bounded{} \atomic{} frames that are either of the form
$\cluster{k}$, $\chain{k}{\bullet}$, or $\chain{k}{\circ}$, 
%or $\chain{k}{\bullet}\lhd\cluster{1}$, 
for some  $k\le\pbound$, 
%polynomial-size $k$, 
or of the form 
$m^{<}$ or $\cluster{1} \lhd m^{<}$, for some $m\le\cbound+1$, 
% $L$-bounded,}  
and respective roots $x_i$ in $\mathfrak N_i^{0,0}$. 
%polynomially-many polynomial-size \simple{} $\delta$-models $\mathfrak N_i^\ell$, for all $\ell<N$ and $i=1,2$, based on $L$-bounded \atomic{} frames, and respective roots $x_i$ in $\mathfrak N_i^{0}$. 
We then let $\mathfrak{N_i^\ell}=\mathfrak N_i^{0,\ell}\lhd\dots\lhd\mathfrak N_i^{\nn_i^\ell-1,\ell}$, for $\ell<N$, $i=1,2$, and 
$\mathfrak N_i=\mathfrak N_i^0\lhd\dots\lhd\mathfrak N_i^{N-1}$. 
Checking $(c)$ and $(d).2$ in Theorem~\ref{t:finaxstruct}
%Definition~\ref{d:bmp} 
can clearly be done in time polynomial in $\size{\mathfrak N_i}$
(which is polynomial in  $\max(|\varphi_1|,|\varphi_2|)$).
For $(a)$, we  
%in Definition~\ref{d:bmp} 
use the following:

\begin{lemma}\label{l:mpoly}
Checking whether $\mathfrak M_0\lhd\dots\lhd\mathfrak M_{n-1},x\models\varphi$, for \simple{} $\sig(\varphi)$-models $\mathfrak M_j$, $j<n$, based on \atomic{} frames with root $x$ in $\mathfrak M_0$, can be done in time polynomial in 
$|\varphi|$ and $\size{\mathfrak M_0}+\dots +\size{\mathfrak M_{n-1}}$.
\end{lemma}
\begin{proof}
Let $\mathfrak M = \mathfrak M_0\lhd\dots\lhd\mathfrak M_{n-1}$. Suppose $\mathfrak M_j$ is based on the frame $\chain{k}{\ast}$ defined in Example~\ref{k-omega} with points $a_s$, $s < k$, and $\yy_\ell$, $\ell < \omega$. Using the definition of a \simple{} model, it is readily shown by structural induction that  
any formula $\psi\in\sub(\varphi)$ is satisfiable in $\mathfrak M_j$
iff there is $\ell<k+\md(\psi)$ with $\mathfrak M_j,\yy_\ell \models \psi$, where $\md(\psi)$, the \emph{modal depth} of $\psi$, is the maximal number of nested modal operators in $\psi$. 
%$a_n$, $n < k$, and $b_l$, $l < \omega$. Using the definition of a \simple{} model, it is readily shown by structural induction that $\mathfrak M,a_n \models \varphi$ iff $\mathfrak M,b_{n\cdot\md(\varphi)} \models \varphi$, where $\md(\varphi)$, the \emph{modal depth} of $\varphi$, is the maximal number of nested modal operators in $\varphi$. 
The required algorithm is now obvious.
\end{proof}

%Condition $(b)$ for the fixed $L = \KFT \oplus \gamma$ is checked using 
Suppose $L = \KFT \oplus \{\alpha(\mathfrak G_j,\mathfrak D_j,\bot) \mid j \in J_L\}$ with  finite $J_L$  and $\mathfrak G_j=(V_j,S_j)$. To check condition $(b)$ in Theorem~\ref{t:finaxstruct},
%Definition~\ref{d:bmp}, 
we require the following:
%Lemma~\ref{can-test}:
%
\begin{lemma}\label{l:fpoly}
%\textcolor{red}{Suppose $\mathfrak F=\mathfrak F_0\lhd\dots\lhd\mathfrak F_{n-1}$ is an $L$-bounded \basic{} frame, for 
 %\atomic{} frames $\mathfrak F_j$, $j<n$. Checking whether $\mathfrak F\models\gamma$ can be done in time polynomial in $n$.}\nb{?? there are $2^{n\cdot c_L}$ subsets of $\mathfrak F$}
If $\mathfrak F=\mathfrak F_0\lhd\dots\lhd\mathfrak F_{n-1}$ with \atomic{} frames $\mathfrak F_\ell$, $\ell<n$, then checking whether $\mathfrak F\models \alpha(\mathfrak G_j, \mathfrak D_j,\bot)$, for all $j \in J_L$, can be done \mbox{in time polynomial in}
$$
n _{\mathfrak F,j} = n \cdot \max\bigl(\size{\mathfrak F_0},\dots,\size{\mathfrak F_{n-1}},|V_j|\bigr).
$$
\end{lemma} 
\begin{proof}
Let $\mathfrak F=(W,R,\INT)$. Given any $\alpha(\mathfrak G_j, \mathfrak D_j,\bot)$, we construct the 
Kripke frame $\mathfrak H_j = (W_j,R_j)$, where $R_j = \rest{R}{W_j}$ and $W_j\subseteq W$ comprises 
\begin{itemize}
\item[--]
the underlying sets of all finite $\lhd$-components $\mathfrak F_\ell$ of $\mathfrak F$;

\item[--]
the last $|V_j|+1$-many points $\yy_{|V_j|},\dots,\yy_0$ in $\mathfrak F_\ell = \chain{k}{\ast}$, where $\ast\in\{\bullet,\circ\}$ and $\yy_{|V_j|}$ is `painted' blue (see Example~\ref{k-omega} for the notation).
\end{itemize}
Then $\mathfrak H_j$ is a subframe of $\mathfrak F$ because all finite subsets of $\{\yy_n\mid n<\omega\}$ are internal in  $\chain{k}{\ast}$.
We show below that there is an injection $f \colon V_j \to W$ satisfying \textbf{(cf$_1$)}--\textbf{(cf$_4$)} 
%for $\alpha(\mathfrak G_j, \mathfrak D_j,\bot)$ 
in $\mathfrak F$ iff there is an injection $h \colon V_j \to W_j$ satisfying \textbf{(cf$_1$)}--\textbf{(cf$_4$)} in $\mathfrak H_j$ and having no blue points in $h(V_j)$. 
Note that the latter is checkable in time polynomial in $n _{\mathfrak F,j}$: just enumerate all 
(at most $n _{\mathfrak F,j}^{|V_j|}$-many) injections $V_j \to W_j$ 
and verify that at least one of them meets the required conditions.

$(\Leftarrow)$ Suppose $h \colon V_j \to W_j$ is an injection satisfying \textbf{(cf$_1$)}--\textbf{(cf$_4$)} in $\mathfrak H_j$ and having no blue points in $h(V_j)$. 
We claim that $h$ satisfies \textbf{(cf$_1$)}--\textbf{(cf$_4$)} in $\mathfrak F$. Indeed, \textbf{(cf$_1$)} and \textbf{(cf$_4$)} hold since $\mathfrak H_j$ is a subframe of $\mathfrak F$, and \textbf{(cf$_2$)} holds  
because the final cluster of $\mathfrak H_j$ is the final cluster of $\mathfrak F$ by definition.
To show that $h$ also meets \textbf{(cf$_3$)}, observe that, as $h(x)$ is not blue for any $x\in V_j$,
the immediate predecessor cluster of $C\bigl(h(x)\bigr)$ in $\mathfrak H_j$ is also the 
immediate predecessor of $C\bigl(h(x)\bigr)$ in $\mathfrak F$. 
%Let $f \colon V_j \to W$ coincide with $h$ on $V_j$. 
%By definition, $f$ satisfies \textbf{(cf$_1$)}, \textbf{(cf$_2$)}, \textbf{(cf$_4$)} in $\mathfrak F$. To show that it also meets \textbf{(cf$_3$)}, suppose $x \in \mathfrak D_j$ and $C(y)$ is the immediate predecessor of $C(x)$ in $\mathfrak G_j$. As $h$ satisfies \textbf{(cf$_3$)}, $C(h(y))$ is the immediate predecessor of $C(h(x))$ in $\mathfrak H_j$. By the definition of $W_j$ and since $h(y)$ cannot be blue, $C(h(y))$ is the immediate predecessor of $C(h(x))$ in $\mathfrak F$, too.
%
%For example, let $\mathfrak F = \circ \lhd \chain{1}{\bullet} \lhd \circ$ and $\alpha(\mathfrak G_j, \mathfrak D_j,\bot)$ have $\mathfrak G_j = \circ \lhd \bullet \lhd \circ$ and $\mathfrak D_j$ consisting of the irreflexive point in $\mathfrak G_j$. Clearly, $\mathfrak F \models \mathfrak \alpha(\mathfrak G_j, \mathfrak D_j,\bot)$. In the frame $\mathfrak H_j= \circ \lhd \bullet  \lhd \bullet \lhd \bullet \lhd \bullet \lhd \circ$, the leftmost $\bullet$ is blue and it belongs to the range of the only injective function from $\mathfrak G_j$ to $\mathfrak H_j$ satisfying \textbf{(cf$_1$)}--\textbf{(cf$_4$)}.\nz{drop example?}  

$(\Rightarrow)$ Let $f \colon V_j \to W$ be an injection satisfying \textbf{(cf$_1$)}--\textbf{(cf$_4$)}.  
To obtain $h$, we modify those $f(x)$ that belong to infinite $\lhd$-components $\mathfrak F_\ell = \chain{k}{\ast}$.
Suppose the intersection of $f(V_j)$ with such an $\mathfrak F_\ell$ is not empty.
%and recall the notation for $\chain{k}{\ast}$ from Example~\ref{k-omega}.
By \eqref{infiniteincl} in Example~\ref{k-omega} and \textbf{(cf$_4$)},
$f(V_j) \cap \{a_0,\dots, a_{k-1}\} = \emptyset$, and so
the intersection of $f(V_j)$ with $\mathfrak F_\ell$ is $\{\yy_{i_0},\dots,\yy_{i_{m_\ell-1}}\}$, for some $m_\ell\le |V_j|$. 
It is readily seen that by taking $h(x)=\yy_z$ if $f(x)=\yy_{i_z}$, for $z<m_\ell$, and $h(y) = f(y)$, for $f(y)$ in finite $\lhd$-components, we obtain an injection $h \colon V_j \to W_j$ with \textbf{(cf$_1$)}--\textbf{(cf$_4$)} and no blue points in $h(V_j)$.
\end{proof}

If all checks are positive, then, by Lemma~\ref{l:gbisall}, $\mathfrak N_1,x_1$ and $\mathfrak N_2,x_2$ satisfy the conditions of  Theorem~\ref{criterion}, and so $\varphi_1$ and $\varphi_2$ do not have an interpolant in $L$. 
\end{proof}

%***********

\subsection{Partitioning models into \globally{} $\sigma$-bisimilar intervals}\label{ss:modular}

In this section, we start proving Theorem~\ref{t:structmodel}.
In a nutshell, our plan is as follows.
%We now proceed with the plan for proving Theorem~\ref{t:structmodel}, outlined at the start of \S\ref{section:K4.3}.
Given $\varphi_1$ and $\varphi_2$ without an interpolant in $L \supseteq \KFT$, the criterion of Theorem~\ref{criterion} supplies 
models $\mathfrak M_i$, $i=1,2$, based on finitely $\mathfrak M_i$-generated
descriptive frames $\mathfrak{F}_i = (W_{i},R_{i},\INT_{i})$ with roots $x_i \in W_i$ such that
\begin{itemize}
\item[--] $\mathfrak{M}_{1},x_{1} \models \varphi_1$ and $\mathfrak{M}_{2},x_{2} \models \neg\varphi_2$;

\item[--]
each $\mathfrak M_i$, $i=1,2$, is based on a frame for $L$;

\item[--] $\mathfrak{M}_{1},x_{1} \sim_{\sigma} \mathfrak{M}_{2},x_{2}$, where $\sigma = \sig(\varphi_1) \cap \sig(\varphi_2)$.
\end{itemize}
%
%$\sigma$-bisimilar models $\mathfrak M_i$, for $i=1,2$ and $\sigma = \sig(\varphi_1) \cap \sig(\varphi_2)$,  such that the $\mathfrak M_i$ are based on  finitely $\mathfrak M_i$-generated descriptive frames $\mathfrak F_i$ for $L$ with roots $x_i$ and witness the lack of an interpolant in $L$.
To prove Theorem~\ref{t:structmodel},
we need to turn the $\mathfrak M_i,x_i$ to some $\mathfrak N_i,x_i$ with the required structure and still satisfying these three conditions. In view of  Example~\ref{ex:GL.3},  
extracting the roots $x_i$ and the sets $\mset{i}$, $\sset{i}$ of maximal points from $\mathfrak M_i$ 
(similarly to the proof of Theorem~\ref{dperscofinal}~$(a)$) is not enough now, 
so we need to develop a more involved construction. We proceed in two steps:
\begin{itemize}
%\item[--]
%First, in Section~\ref{ss:structure}, we establish a few general facts about the structure of finitely generated descriptive frames for $\KFT$ that are needed for our construction. 
%
\item[--]
First, we analyse the $\sigma$-types in the $\mathfrak M_i$ 
and partition them into  
\emph{internal} closed intervals $\mathcal I_i=\{I_i^\ell\mid \ell<\partN\}$, for 
the same $N=\mathcal{O}\bigl(\max(|\varphi_1|,|\varphi_2|)\bigr)$, 
%$\partN\le 2\kbound-1$,\nz{$-1$? do we care?} 
such that $\rest{\mathfrak M_1}{I_1^\ell}$ and $\rest{\mathfrak M_2}{I_2^\ell}$ are \globally{} $\sigma$-bisimilar, for every $\ell<\partN$.  
%
%\end{itemize}
%
%\begin{itemize}
%\item[\plan]
%$\bigl\{(y_1,y_2)\in I_1^\ell\times I_2^\ell\mid t^\sigma_{\mathfrak M_1}(y_1)=t^\sigma_{\mathfrak M_2}(y_2)\bigr\}$
%is a \globalb{} $\sigma$-bisimulation between $\rest{\mathfrak M_1}{I_1^\ell}$ and $\rest{\mathfrak M_2}{I_2^\ell}$, for every $\ell<\partN$.
%\end{itemize}
%
%\begin{itemize}
%\item[] 
%As all intervals $I_i^\ell\in\mathcal I_i$ are internal in $\mathfrak F_i$, 
By Lemma~\ref{intpart}, 
$\mathfrak M_i = (\rest{\mathfrak M_i}{I_i^0})\lhd\dots\lhd(\rest{\mathfrak M_i}{I_i^{\partN-1}})$, $i=1,2$.

\item[--]
Then, in \S\ref{ss:small}, we complete the proof of Theorem~\ref{t:structmodel} by transforming each pair $\bigl(\rest{\mathfrak M_1}{I_1^\ell},\rest{\mathfrak M_2}{I_2^\ell}\bigr)$, $\ell<N$, into a pair 
$(\mathfrak N_1^\ell,\mathfrak N_2^\ell)$ of models with the required structure.
\end{itemize}

We begin with a simple observation on definable closed intervals:
\begin{lemma}\label{l:closedintdef}
Suppose $\mathfrak{M}$ is a model based on a rooted  finitely $\mathfrak M$-generated descriptive frame $\mathfrak F = (W, R,\INT)$ for $\KFT$. Then
every closed interval $[C,C']$ in $\mathfrak F$ with a non-limit cluster $C'$ is definable in $\mathfrak M$. 
\end{lemma}
\begin{proof}
By Lemma~\ref{lem:descr'}~$(a)$--$(c)$, the non-limit $C'$ is defined in $\mathfrak M$ by some formula $\gamma$. Let $\delta = \bot$ if $C$ is the root cluster, and let $\delta$ define the immediate predecessor of $C$ in $\mathfrak M$ otherwise, which exists by Lemma~\ref{lem:descr0}~$(a)$ and is definable by Lemma~\ref{lem:descr'}~$(a)$--$(c)$. Then $[C,C']$ is defined in $\mathfrak M$ by $\neg \Diamond^+ \delta \land \Diamond^+ \gamma$. 
\end{proof}

Next, we look into the structure of $\sigma$-types in any model $\mathfrak M$ based on a rooted finitely $\mathfrak M$-generated descriptive frame $\mathfrak F=(W,R,\INT)$ for $\KFT$.
Given $x\in W$ and a signature $\sigma$, we define the \emph{$\sigma$-block $\bl_{\mathfrak{M}}^{\sigma}(x)$ of $x$ in} $\mathfrak{M}$ by taking
\[
\bl_{\mathfrak{M}}^{\sigma}(x) = 
\begin{cases}
\{ y\in W \mid \Diamond t_{\mathfrak{M}}^{\sigma}(y) \subseteq t_{\mathfrak{M}}^{\sigma}(x), \  \Diamond t_{\mathfrak{M}}^{\sigma}(x) \subseteq t_{\mathfrak{M}}^{\sigma}(y)\}, & \text{if $\Diamond t_{\mathfrak{M}}^{\sigma}(x) \subseteq t_{\mathfrak{M}}^{\sigma}(x)$};\\
\{x\}, & \text{otherwise};
\end{cases}
\]
in the latter case---when 
%$\{x\}$ is irreflexive---
$x$ must be an irreflexive point---the $\sigma$-block $\bl_{\mathfrak{M}}^{\sigma}(x)$ is called \emph{degenerate\/}.
(It can happen that 
$\{x\}$ is a degenerate cluster but $\bl_{\mathfrak{M}}^{\sigma}(x)$ is not a degenerate $\sigma$-block.)
%$\Diamond t_{\mathfrak{M}}^{\sigma}(x) \subseteq t_{\mathfrak{M}}^{\sigma}(x)$ and not $xRx$.)
%
We call a set $\bl \subseteq W$ a \emph{$\sigma$-block in} $\mathfrak M$ if $\bl = \bl_{\mathfrak{M}}^{\sigma}(x)$, for some $x$. It is readily seen that 
the relation $x\approx y$ iff $\bl_{\mathfrak{M}}^{\sigma}(x)=\bl_{\mathfrak{M}}^{\sigma}(y)$ is an equivalence relation on $W$,
%if $y\in \bl_{\mathfrak{M}}^{\sigma}(x)$ then $\bl_{\mathfrak{M}}^{\sigma}(y)=\bl_{\mathfrak{M}}^{\sigma}(x)$, 
and every $\sigma$-block $\bl$ is an interval in $\mathfrak F$.
%
%\begin{description}
%\item[(block)] every $\sigma$-block $\bl$ is an interval, and if $x \in \bl$, then $t_{\mathfrak{M}}^{\sigma}(x) \ne t_{\mathfrak{M}}^{\sigma}(y)$ for any $y \notin \bl$.
%\end{description}
%
%It follows that any block is generated by any of its points. 
(See Example~\ref{blocksetc} below for an illustration.)  
It follows from the definitions that

\medskip
\noindent
{\bf (block)} for all $\sigma$-blocks $\bl$ in $\mathfrak M$ and $y\in W$, if $y \notin \bl$, then $t_{\mathfrak{M}}^{\sigma}(y) \notin t_{\mathfrak{M}}^{\sigma}(\bl)$.

\medskip
\noindent
%
%\textcolor{blue}{
%For degenerate $\sigma$-blocks this follows from the definability of degenerate clusters (Lemma~\ref{lem:descr'}), and for other $\sigma$-blocks it is straightforward from the definitions.}

%Call a cluster $C$ \emph{maximal} in $\mathfrak{M}$ if there exists a formula 
%$\varphi$ and $x\in C$ such that $\mathfrak{M},x\models \varphi$ but $\mathfrak{M},y\not\models \varphi$ for any proper $R$-successor $y$ of $x$.
%Call $C$ \emph{$\sigma$-maximal} in $\mathfrak{M}$ if there exists a $\sigma$-formula with these properties. \marginpar{We need notation for proper successor, $R^{s}$?, and also for (proper) successor on clusters}
%\begin{lemma}\label{lem:descr}
%	In descriptive frames $\mathfrak{M}$ (in which every internal set is defined)
%	\begin{enumerate}
%		\item every degenerate cluster is maximal;
%		\item a cluster is maximal iff it has a direct proper $R$-successor.
%		\item maybe here: we cannot have an infinite ascending chain followed directly by a infinite descending chain: any interval open on the left is of the form $(C,_)$.
%	\end{enumerate}
%\end{lemma}
%In general, of course, maximal clusters are not $\sigma$-maximal. However, for %within $\sigma$-blocks is is the case. 

\begin{lemma}\label{int-prop}
Suppose $\mathfrak{M}$ is a model based on a rooted  finitely $\mathfrak M$-generated descriptive frame $\mathfrak F = (W, R,\INT)$ for $\KFT$. 
For any $\sigma$-block $\bl$
%$\bl = \bl_{\mathfrak{M}}^{\sigma}(x)$ 
in $\mathfrak{M}$, 
%based on a descriptive finitely-generated rooted $\mathfrak F$, 
there exist clusters $\sC{\bl}$, $\eC{\bl}$ in $\mathfrak F$ such that the following hold\textup{:}
\begin{itemize}
%\item[$(a)$] $\bl = [C_{1},C_{2}]$\textup{;}
\item[$(a)$] $\bl = \bigl [\sC{\bl},\eC{\bl}\bigr]$\textup{;}

%\item[$(b)$] 
%$t_{\mathfrak{M}}^{\sigma}(\bl) = t_{\mathfrak{M}}^{\sigma}(C_2)$\textup{;}

\item[$(b)$] if $\eC{\bl}$ is maximal in $\mathfrak M$, then it is $\sigma$-maximal in $\mathfrak M$\textup{;}

\item[$(c)$] if $\eC{\bl}$ is degenerate, then $\bl = \eC{\bl}$\textup{;}
%and $C_2$ is $\sigma$-maximal\textup{;}

%\item[$(d)$] if $C_{2}$ is not $\sigma$-maximal, then any $(C_2,C(z)]$ with $C_2R^sC(z)$ contains a $\sigma$-maximal cluster\textup{;}\nb{not needed}

\item[$(d)$] 
$\bl$ is definable in $\mathfrak M$ iff $\eC{\bl}$ is not a limit cluster\textup{;}
%$B \notin \INT$ iff $C_2$ is a limit cluster\textup{;}

\item[$(e)$] 
$t_{\mathfrak{M}}^{\sigma}(\bl) = t_{\mathfrak{M}}^{\sigma}\bigl(\eC{\bl}\bigr)$.
\end{itemize} 
\end{lemma} 
\begin{proof}
$(a)$  Let 
$\mathcal{X}_{\bl}=\{C\mid\mbox{$C$ be a cluster with $\bl\cap C\ne\emptyset$}\}$. As $\bl$ is an interval,
$\bl=\bigcup\mathcal{X}_{\bl}$. By Lemma~\ref{lem:descr0}, there is a $<_R$-largest cluster $\eC{\bl}$ in 
$\mathcal{X}_{\bl}$. Also, there is a $<_R$-largest cluster $D$ with $D<_R C$, for every $C\in \mathcal{X}_{\bl}$.
Suppose there is no $<_R$-smallest cluster in $\mathcal{X}_{\bl}$. Then $D$ is a limit cluster and
if $y\in D$, then $t_{\mathfrak{M}}^{\sigma}(y) \notin t_{\mathfrak{M}}^{\sigma}(\bl)$ by {\bf (block)}.
So there is a $\sigma$-formula $\mu$ such that $\mu\in t_{\mathfrak{M}}^{\sigma}(y)$ and
$\Diamond\mu\notin t_{\mathfrak{M}}^{\sigma}(x)$ for any $x\in \bl$, and so for any $x$ with $yR^s x$.
As $D$ is non-degenerate by Lemma~\ref{lem:descr'}, it follows that $D$ is $\Diamond\mu$-maximal in $\mathfrak M$, contrary to Lemma~\ref{lem:descr'}~$(b)$.
Therefore, there is a $<_R$-smallest cluster $\sC{\bl}$ in $\mathcal{X}_{\bl}$, and so $\bl = \bigl [\sC{\bl},\eC{\bl}\bigr]$.

$(b)$
If $\eC{\bl}$ is maximal in $\mathfrak M$, then either it is final or has an immediate successor, by Lemma~\ref{lem:descr'}~$(b)$.
If $\eC{\bl}$ is final, then it is $\top$-maximal in $\mathfrak M$.
So suppose that $C(y)$ is an immediate successor of $\eC{\bl}=C(x)$. 
%
%\textcolor{blue}{
%If $\eC{\bl}$ is not degenerate, then $\Diamond t_{\mathfrak{M}}^{\sigma}(x) \not\subseteq t_{\mathfrak{M}}^{\sigma}(y)$ follows from $y\notin \bl$. So there is a $\sigma$-formula $\mu$ such that $\mathfrak M,x\models\mu$ and $\mathfrak M,y\not\models\Diamond\mu$.  If $\mathfrak M,y\models\mu$, then $\eC{\bl}$ is $\Diamond\mu$-maximal in $\mathfrak M$.  And if $\mathfrak M,y\not\models\mu$, then $\eC{\bl}$ is $\mu$-maximal in $\mathfrak M$. 
%And if $\mathfrak M,y\not\models\mu$, then $\eC{\bl}$ is $\mu$-maximal in $\mathfrak M$. 
%If $\eC{\bl}$ is degenerate, we cannot have $\Diamond t_{\mathfrak{M}}^{\sigma}(x) \subseteq t_{\mathfrak{M}}^{\sigma}(x)$, for otherwise $t_{\mathfrak{M}}^{\sigma}(x)\subseteq t_{\mathfrak{M}}^{\sigma}(y)$, contrary to {\bf (block)}. Thus,  $\Diamond t_{\mathfrak{M}}^{\sigma}(x) \not\subseteq t_{\mathfrak{M}}^{\sigma}(x)$, and so 
%}
%
There are two cases: 
$(i)$ If $\Diamond t_{\mathfrak{M}}^{\sigma}(x) \subseteq t_{\mathfrak{M}}^{\sigma}(x)$,
then
$\Diamond t_{\mathfrak{M}}^{\sigma}(x) \not\subseteq t_{\mathfrak{M}}^{\sigma}(y)$ follows from $y\notin \bl$ and $xRy$. So there is a $\sigma$-formula $\mu$ such that $\mathfrak M,x\models\mu$ and $\mathfrak M,y\not\models\Diamond\mu$. 
If $\mathfrak M,y\models\mu$, then $\eC{\bl}$ is $\Diamond\mu$-maximal in $\mathfrak M$. 
And if $\mathfrak M,y\not\models\mu$, then $\eC{\bl}$ is $\mu$-maximal in $\mathfrak M$.
$(ii)$ If $\Diamond t_{\mathfrak{M}}^{\sigma}(x) \not\subseteq t_{\mathfrak{M}}^{\sigma}(x)$, then
there is $\sigma$-formula $\mu$ such that $\mathfrak M,x\models\mu$ and $\mathfrak M,x\not\models\Diamond\mu$. Therefore,  $\eC{\bl}$ is $\mu$-maximal in $\mathfrak M$.

% If $\eC{\bl}$ is degenerate, we cannot have $\Diamond t_{\mathfrak{M}}^{\sigma}(x) \subseteq t_{\mathfrak{M}}^{\sigma}(x)$, for otherwise $t_{\mathfrak{M}}^{\sigma}(x)\subseteq t_{\mathfrak{M}}^{\sigma}(y)$, contrary to {\bf (block)}. Thus,  $\Diamond t_{\mathfrak{M}}^{\sigma}(x) \not\subseteq t_{\mathfrak{M}}^{\sigma}(x)$, and so there is $\sigma$-formula $\mu$ such that $\mathfrak M,x\models\mu$ and $\mathfrak M,x\not\models\Diamond\mu$. Therefore,  $\eC{\bl}$ is $\mu$-maximal in $\mathfrak M$. 

%$(c)$ If $C_2$ is maximal, then there is a formula $\varphi$ that is only true on $(C,\infty]$, and so there must exist $C'$ with $[C',\infty] = (C_2,\infty]$. Now, as $C_2$ is maximal satisfying some $\sigma$-type $t$, there must exist a $\sigma$-formula $\psi\in t$ such that no point in $C'$ satisfies $\Diamond \psi$. But then $\Diamond\psi$ witnesses that $C_2$ is $\sigma$-maximal.
	
$(c)$ 	
Suppose on the contrary that $\eC{\bl}=\{x\}\ne \bl$. Then $|\bl|>1$, and so 
$\Diamond t_{\mathfrak{M}}^{\sigma}(x) \subseteq t_{\mathfrak{M}}^{\sigma}(x)$ follows from $\bl=\bl^\sigma_{\mathfrak M}(x)$. So, for every $\sigma$-formula $\mu$, if $\mathfrak M,x\models\mu$ then $\mathfrak M,x\models\Diamond\mu$.
On the other hand, $\eC{\bl}$ is maximal in $\mathfrak M$ by Lemma~\ref{lem:descr'}~$(a)$, and so $\sigma$-maximal in $\mathfrak M$ by $(b)$, which is 
a contradiction.
	
%$(d)$ is obvious.
	
%$(d)$ suppose $C_2R^sC(z)$. In a $\sigma$-type $t$ satisfied in $C$, we find $\psi\in t$ such that $\mathfrak M, z \models\neg \Diamond \psi$. The maximal cluster $C'$ satisfying $\Diamond\psi$ is as required.

$(d,\Leftarrow)$
This is by $(a)$ and Lemma~\ref{l:closedintdef}.
%Lemma~\ref{lem:descr'}~$(d)$.
%~\ref{l:defint}.

$(d,\Rightarrow)$
Suppose that $\bl$ is defined in $\mathfrak M$ by some $\psi$.
Then $\eC{\bl}$ is $\psi$-maximal in $\mathfrak M$, and so cannot be a limit cluster by Lemma~\ref{lem:descr'}~$(b)$.
%$(e)$ Let $\delta = \bot$ if $C_1$ is the root cluster and let $\delta$ define the immediate predecessor of $C_1$ otherwise. If $C_2$ is not a limit cluster, it is defined by some $\gamma$. Then $B$ is defined by $\neg \Diamond^+ \delta \land \Diamond^+ \gamma$. Conversely, if $B$ is defined by some $\delta$, then $C_2$ cannot be a limit cluster.

$(e)$
If $\eC{\bl}$ is degenerate, then this is obvious by $(c)$. So suppose $\eC{\bl}=C(y)$ is non-degenerate and $x\in\bl$.
Then $\Diamond t_{\mathfrak{M}}^{\sigma}(x) \subseteq t_{\mathfrak{M}}^{\sigma}(y)$, and so
$\Diamond\bigwedge\Gamma\in t_{\mathfrak{M}}^{\sigma}(y)$ for every finite $\Gamma\subseteq t_{\mathfrak{M}}^{\sigma}(x)$. By Lemma~\ref{maxpoints}, there is $z$ such that $yRz$ and $t_{\mathfrak{M}}^{\sigma}(z)=t_{\mathfrak{M}}^{\sigma}(x)$. By {\bf (block)}, we have $z\in\bl$, and so $z\in\eC{\bl}$.
\end{proof}

\begin{example}\label{blocksetc}\em
The model $\mathfrak M_1$ in Fig.~\ref{GL3gf} from Example~\ref{ex:GL.3}~$(a)$ is partitioned into the following $\sigma$-blocks (indicated by the 
%red 
brackets), for three different $\sigma$: 
%
%\centerline{\includegraphics[scale=0.65]{../Pics/b1}}\\
%\centerline{\includegraphics[scale=0.65]{../Pics/b2}}\\
%\centerline{\includegraphics[scale=0.65]{../Pics/b3}}\\
%
\begin{center}
\begin{tikzpicture}[>=latex,line width=0.5pt,xscale = 1.2,yscale = .65]
\node[right]  at (-1,4.5) {$\sigma=\emptyset$};
\node[point,fill=black,scale = 0.7,label=above:{\footnotesize $\varphi_1$},label=below:{\footnotesize $x_1$}] (x1) at (0,3) {};
\node[point,fill=black,scale = 0.7,label=above:{\footnotesize $p_1,\neg q_1$},label=below:{\footnotesize $y_1$}] (y1) at (1,3) {};
\node[point,scale = 0.7,label=below:{\footnotesize $a_1^0$},label=above:{\footnotesize $p_2$}] (a10) at (2.5,3) {};
\node[point,scale = 0.7,label=below:{\footnotesize $a_1^1$},label=above:{\footnotesize $p_1$}] (a11) at (3.5,3) {};
\draw[] (3,2.9) ellipse (1 and .9);
\node[scale = 0.9]  at (3,3.5) {$q_1$};
\node[]  at (4.5,3) {$\dots$};
\node[point,fill=black,scale = 0.7,label=below:{\footnotesize $\yy_1^3$},label=above:{\footnotesize $p_1,q_1$}] (b13) at (5,3) {};
\node[point,fill=black,scale = 0.7,label=below:{\footnotesize $\yy_1^2$},label=above:{\footnotesize $p_2,q_1$}] (b12) at (6,3) {};
\node[point,fill=black,scale = 0.7,label=below:{\footnotesize $\yy_1^1$},label=above:{\footnotesize $p_1,q_1$}] (b11) at (7,3) {};
\node[point,fill=black,scale = 0.7,label=below:{\footnotesize $\yy_1^0$},label=above:{\footnotesize $p_2,q_1$}] (b10) at (8,3) {};
\draw[->] (x1) to (y1);
\draw[->] (y1) to (2,3);
\draw[->] (b13) to (b12);
\draw[->] (b12) to (b11);
\draw[->] (b11) to (b10);
\draw[-,thick] (-.2,3.7) -- (-.2,3.9) -- (4,3.9) -- (4,3.7);
\draw[-,thick] (4.7,3.7) -- (4.7,3.9) -- (5.3,3.9) -- (5.3,3.7);
\draw[-,thick] (5.7,3.7) -- (5.7,3.9) -- (6.3,3.9) -- (6.3,3.7);
\draw[-,thick] (6.7,3.7) -- (6.7,3.9) -- (7.3,3.9) -- (7.3,3.7);
\draw[-,thick] (7.7,3.7) -- (7.7,3.9) -- (8.3,3.9) -- (8.3,3.7);
\end{tikzpicture}
\begin{tikzpicture}[>=latex,line width=0.5pt,xscale = 1.2,yscale = .65]
\node[right]  at (-1,4.5) {$\sigma=\{p_1,p_2\}$};
\node[point,fill=black,scale = 0.7,label=above:{\footnotesize $\varphi_1$},label=below:{\footnotesize $x_1$}] (x1) at (0,3) {};
\node[point,fill=black,scale = 0.7,label=above:{\footnotesize $p_1,\neg q_1$},label=below:{\footnotesize $y_1$}] (y1) at (1,3) {};
\node[point,scale = 0.7,label=below:{\footnotesize $a_1^0$},label=above:{\footnotesize $p_2$}] (a10) at (2.5,3) {};
\node[point,scale = 0.7,label=below:{\footnotesize $a_1^1$},label=above:{\footnotesize $p_1$}] (a11) at (3.5,3) {};
\draw[] (3,2.9) ellipse (1 and .9);
\node[scale = 0.9]  at (3,3.5) {$q_1$};
\node[]  at (4.5,3) {$\dots$};
\node[point,fill=black,scale = 0.7,label=below:{\footnotesize $\yy_1^3$},label=above:{\footnotesize $p_1,q_1$}] (b13) at (5,3) {};
\node[point,fill=black,scale = 0.7,label=below:{\footnotesize $\yy_1^2$},label=above:{\footnotesize $p_2,q_1$}] (b12) at (6,3) {};
\node[point,fill=black,scale = 0.7,label=below:{\footnotesize $\yy_1^1$},label=above:{\footnotesize $p_1,q_1$}] (b11) at (7,3) {};
\node[point,fill=black,scale = 0.7,label=below:{\footnotesize $\yy_1^0$},label=above:{\footnotesize $p_2,q_1$}] (b10) at (8,3) {};
\draw[->] (x1) to (y1);
\draw[->] (y1) to (2,3);
\draw[->] (b13) to (b12);
\draw[->] (b12) to (b11);
\draw[->] (b11) to (b10);
\draw[-,thick] (-.2,3.7) -- (-.2,3.9) -- (.2,3.9) -- (.2,3.7);
\draw[-,thick] (.6,3.7) -- (.6,3.9) -- (4,3.9) -- (4,3.7);
\draw[-,thick] (4.7,3.7) -- (4.7,3.9) -- (5.3,3.9) -- (5.3,3.7);
\draw[-,thick] (5.7,3.7) -- (5.7,3.9) -- (6.3,3.9) -- (6.3,3.7);
\draw[-,thick] (6.7,3.7) -- (6.7,3.9) -- (7.3,3.9) -- (7.3,3.7);
\draw[-,thick] (7.7,3.7) -- (7.7,3.9) -- (8.3,3.9) -- (8.3,3.7);
\end{tikzpicture}
\begin{tikzpicture}[>=latex,line width=0.5pt,xscale = 1.2,yscale = .65]
\node[right]  at (-1,4.5) {$\sigma=\{p_1,p_2,q_1,q_2\}$};
\node[point,fill=black,scale = 0.7,label=above:{\footnotesize $\varphi_1$},label=below:{\footnotesize $x_1$}] (x1) at (0,3) {};
\node[point,fill=black,scale = 0.7,label=above:{\footnotesize $p_1,\neg q_1$},label=below:{\footnotesize $y_1$}] (y1) at (1,3) {};
\node[point,scale = 0.7,label=below:{\footnotesize $a_1^0$},label=above:{\footnotesize $p_2$}] (a10) at (2.5,3) {};
\node[point,scale = 0.7,label=below:{\footnotesize $a_1^1$},label=above:{\footnotesize $p_1$}] (a11) at (3.5,3) {};
\draw[] (3,2.9) ellipse (1 and .9);
\node[scale = 0.9]  at (3,3.5) {$q_1$};
\node[]  at (4.5,3) {$\dots$};
\node[point,fill=black,scale = 0.7,label=below:{\footnotesize $\yy_1^3$},label=above:{\footnotesize $p_1,q_1$}] (b13) at (5,3) {};
\node[point,fill=black,scale = 0.7,label=below:{\footnotesize $\yy_1^2$},label=above:{\footnotesize $p_2,q_1$}] (b12) at (6,3) {};
\node[point,fill=black,scale = 0.7,label=below:{\footnotesize $\yy_1^1$},label=above:{\footnotesize $p_1,q_1$}] (b11) at (7,3) {};
\node[point,fill=black,scale = 0.7,label=below:{\footnotesize $\yy_1^0$},label=above:{\footnotesize $p_2,q_1$}] (b10) at (8,3) {};
\draw[->] (x1) to (y1);
\draw[->] (y1) to (2,3);
\draw[->] (b13) to (b12);
\draw[->] (b12) to (b11);
\draw[->] (b11) to (b10);
\draw[-,thick] (-.2,3.7) -- (-.2,3.9) -- (.2,3.9) -- (.2,3.7);
\draw[-,thick] (.6,3.7) -- (.6,3.9) -- (1.4,3.9) -- (1.4,3.7);
\draw[-,thick] (2,3.7) -- (2,3.9) -- (4,3.9) -- (4,3.7);
\draw[-,thick] (4.7,3.7) -- (4.7,3.9) -- (5.3,3.9) -- (5.3,3.7);
\draw[-,thick] (5.7,3.7) -- (5.7,3.9) -- (6.3,3.9) -- (6.3,3.7);
\draw[-,thick] (6.7,3.7) -- (6.7,3.9) -- (7.3,3.9) -- (7.3,3.7);
\draw[-,thick] (7.7,3.7) -- (7.7,3.9) -- (8.3,3.9) -- (8.3,3.7);
\end{tikzpicture}
\end{center}
To show this for $\sigma=\emptyset$, observe that, for every $n >0$, we have $\Diamond^n \top \in t_{\mathfrak{M}_1}^{\sigma}(\yy_1^n)$, $\Diamond^{n+1} \top \notin t_{\mathfrak{M}_1}^{\sigma}(\yy_1^n)$, $\neg\Diamond \top \in t_{\mathfrak{M}_1}^{\sigma}(\yy_1^0)$, and $\Diamond^n \top \in t_{\mathfrak{M}_1}^{\sigma}(a_1^0)$. The cluster $C(a_1^0)$ is not maximal in $\mathfrak M_1$ as any formula $\alpha$ that is true at $a_1^0$ or $a_1^1$ is also true at $\yy_1^n$, for some $n < \omega$ (which is seen by induction on the structure of $\alpha$). 
%
%Therefore, the $\sigma$-block of each $n < \omega$ is the singleton $\{n\}$. The remaining points comprise another $\sigma$-block. For $\sigma = \{p_1,p_2\}$ and $\mathfrak M_1$, the root forms a separate $\sigma$-block, followed by the block comprising the next two clusters.\nb{+ each of the remaining points, pic, explain that $\clustert$ isn't maximal} Finally, for $\sigma = \{p_1,p_2,q_1,q_2\}$, each cluster in the underlying frame is a separate $\sigma$-block.
The model $\mathfrak M_1$ in Example~\ref{ex:GL.3}~$(b)$
 has only one $\emptyset$-block comprising all of its points.
%the whole frame. For $\sigma = \{p_1,p_2\}$, we have the following $\sigma$-blocks (from left to right): the root, the interval between the next point and $\clustert$, each of the remaining irreflexive points in ${\mathfrak  C}(\clustert,0)$, and the final reflexive point in ${\mathfrak C}(\clustero)$. 
\hfill $\dashv$ 
\end{example}

We now return to our models $\mathfrak M_i$, $i=1,2$, witnessing the lack of interpolants for $\varphi_1$ 
and $\varphi_2$. By Lemma~\ref{int-prop}~$(a)$, $\sigma$-blocks in each $\mathfrak{M}_i$
%that is based on some rooted finitely $\mathfrak M$-generated descriptive frame $\mathfrak F = (W,R,\INT)$ for $\KFT$ 
are closed intervals that form a partition of $W_i$ (with not all of them being necessarily definable in $\mathfrak M_i$). 
%We index them by an obvious linear order $(F,<)$, obtaining a family $(I_{\mathfrak{M},i})_{i\in F}$ of disjoint $\sigma$-blocks covering $W$. We drop $\mathfrak M$ and simply write $I_i$, $i\in F$, if understood. Then we also obtain the family $(T_{\mathfrak{M},i})_{i\in F}$, where $T_{\mathfrak{M},i}$ is the set of $\sigma$-types satisfied in $I_{\mathfrak{M},i}$. Again, if $\mathfrak{M}$ is clear from context, we write $T_{i}$ for $T_{\mathfrak{M},i}$, $i\in F$.
We show that  
there is a $\intord{\mathfrak F_i}$-respecting bijection between the $\sigma$-blocks of the two models.
%we have the same number of $\sigma$-blocks in both models
Indeed, suppose that $W_1$ is partitioned as $\{\bl^j\mid j\in F\}$ into $\sigma$-blocks in $\mathfrak M_1$, for some countable set $F$. 
For each $j\in F$, we let
%
%\begin{equation}\label{bisint}
\[
\bis(\bl^j)=\{y\in W_2\mid t^\sigma_{\mathfrak M_2}(y)\in t^\sigma_{\mathfrak M_1}(\bl^j)\}.
\]
%\end{equation}
% 
\begin{lemma}\label{bis-blocks}
For all $j\in F$, the following hold\textup{:}
\begin{itemize}
\item[$(a)$]
$t^\sigma_{\mathfrak M_1}(\bl^j)=t^\sigma_{\mathfrak M_2}\bigl(\bis(\bl^j)\bigr)$\textup{;}
%for every $y_1\in\bl^j$, there is $y_2\in\bis(\bl^j)$ with $t^\sigma_{\mathfrak M_1}(y_1)=t^\sigma_{\mathfrak M_2}(y_2)$, and,
%for every $y_2\in\bis(\bl^j)$, there is $y_1\in\bl^j$ with $t^\sigma_{\mathfrak M_1}(y_1)=t^\sigma_{\mathfrak M_2}(y_2)$\textup{;}

\item[$(b)$]
$\bis(\bl^j)$ is a $\sigma$-block in $\mathfrak M_2$, and $\bl^j$ is degenerate iff $\bis(\bl^j)$ is degenerate\textup{;}

\item[$(c)$]
$\{\bis(\bl^j)\mid j\in F\}$ is a partition of $W_2$\textup{;}

\item[$(d)$]
$\bl^j\intord{\mathfrak F_1} \bl^k$ iff $\bis(\bl^j)\intord{\mathfrak F_2}\bis(\bl^k)$, for $j,k\in F$\textup{;}

\item[$(e)$]
$\bl^j$ is definable in $\mathfrak M_1$ iff $\bis(\bl^j)$ is definable in $\mathfrak M_2$.
%\nb{$(f)$ omitted}
%if $\bl^j$ is definable in $\mathfrak M_1$, then 
%$\{(y_1,y_2)\in \bl^j\times\bis(\bl^j)\mid t^\sigma_{\mathfrak M_1}(y_1)=t^\sigma_{\mathfrak M_2}(y_2)\}$
% is a $\sigma$-bisimulation between
%$\mathfrak M_1\restriction \bl^j$ and  $\mathfrak M_2\restriction \bis(\bl^j)$.
%\item[$(f)$]
%for every $\ast\in\{\circ,\bullet\}$,
%$\eC{\bl^j}$ is a limit cluster of type $\ast$ iff $\eC{\bis(\bl^j)}$ is a limit cluster of type $\ast$. 
\end{itemize}
\end{lemma}
\begin{proof}
$(a)$
This follows from $\mathfrak M_1,x_1\sim_\sigma\mathfrak M_2,x_2$ and Lemma~\ref{bisim-lemma}.

$(b)$ Let $j\in F$. As $\mathfrak M_1,x_1\sim_\sigma\mathfrak M_2,x_2$, $\bis(\bl^j) \ne \emptyset$.
Take some $y\in \bis(\bl^j)$. We show that $\bis(\bl^j)=\bl^\sigma_{\mathfrak M_2}(y)$. Indeed, this is straightforward from the definitions if $\Diamond t^\sigma_{\mathfrak M_2}(y)\subseteq t^\sigma_{\mathfrak M_2}(y)$.
If $\Diamond t^\sigma_{\mathfrak M_2}(y)\not\subseteq t^\sigma_{\mathfrak M_2}(y)$, then 
$\bl^\sigma_{\mathfrak M_2}(y)=\{y\}$. Take some $x\in\bl^j$ with $t^\sigma_{\mathfrak M_1}(x)=t^\sigma_{\mathfrak M_2}(y)$. Then $\Diamond t^\sigma_{\mathfrak M_1}(x)\not\subseteq t^\sigma_{\mathfrak M_1}(x)$, and so $\bl^j=\{x\}$. Thus, $\bis(\bl^j)=\{z\in W_2\mid t^\sigma_{\mathfrak M_2}(z)=t^\sigma_{\mathfrak M_2}(y)\}$, and so there is
a $\sigma$-formula $\mu$ such that $\mu\in  t^\sigma_{\mathfrak M_2}(z)=t^\sigma_{\mathfrak M_2}(y)$ and
$\Diamond\mu\notin  t^\sigma_{\mathfrak M_2}(z)=t^\sigma_{\mathfrak M_2}(y)$.
Suppose there is $z\in\bis(\bl^j)$, $z\ne y$. Then either $zR_2y$ or $yR_2z$, which is a contradiction.

$(c)$
As $\bis(\bl^j)$ and $\bis(\bl^{k})$ are disjoint for $j\ne k$ by $(a)$ and {\bf (block)}, the relation
`$y\approx y'$ iff there is $j\in F$ with $y,y'\in \bis(\bl^j)$' is an equivalence relation on $W_2$. 

$(d)$
This follows from $\mathfrak M_1,x_1\sim_\sigma\mathfrak M_2,x_2$, $(a)$ and {\bf (block)}.

$(e)$ This follows from $(b)$--$(d)$ and Lemma~\ref{int-prop} $(a)$ and $(d)$.
%
%$(f)$ 
%As every degenerate cluster is a degenerate $\sigma$-block by Lemma~\ref{int-prop}~$(c)$, it follows from $(b)$--$(e)$ that $\eC{\bl^j}$ is a limit cluster of type $\bullet$ iff $\eC{\bis(\bl^j)}$ is a limit cluster of type $\bullet$.
%Now $(f)$ follows from Lemma~\ref{l:cofinal}.
\end{proof}

So, from now on we assume that we have a strict linear order $(F,\prec)$ such that
each $W_i$, $i=1,2$, is partitioned as $\{\bl_i^j\mid j\in F\}$ into $\sigma$-blocks in $\mathfrak M_i$ 
with $j\prec k$ iff $\bl_1^j\intord{\mathfrak F_1}\bl_1^k$ iff $\bl_2^j\intord{\mathfrak F_2}\bl_2^k$, for $j,k\in F$. 
(We write $j\preceq k$ whenever $j\prec k$ or $j=k$.)
Observe that, by Lemmas~\ref{lem:descr0}~$(a)$ and \ref{l:countable},
$(F,\succ)$ is isomorphic to a countable ordinal. We say that $j\in F$ is a $\succ$-\emph{limit} iff it corresponds to a limit ordinal under this isomorphism. Thus, 
every $j\in F$ has an immediate $\prec$-predecessor, and  if $j$ is not a $\succ$-limit, then it also has an immediate $\prec$-successor. Also, $j$ is a $\succ$-limit iff $\eC{\bl_i^j}$ is a limit cluster, for $i=1,2$. 
%
%\begin{equation}\label{precpred}
%\mbox{every $j\in F$ has an immediate $\prec$-predecessor $j^-\in F$.}
%\end{equation}
%
%Further, by Lemma~\ref{bis-blocks}~$(a)$, we have  $t^\sigma_{\mathfrak M_1}(\bl_1^j)=t^\sigma_{\mathfrak M_2}(\bl_2^j)$, for every $j\in F$.
%in which case we write $\bl_1^j\sim_\sigma \bl_2^j$.\nb{notation needed? if so, we need it for intervals}

%We partition the $W_i$ into $\sigma$-blocks $\bl^i_j$, for $j \in F_i$, $i = 1,2$. Denote by $\bis(\bl^i_j)$ the set of all points in $\mathfrak M_{1-i}$ that are bisimilar to some points in $\bl^i_j$. It follows from \textbf{(block)} that if $\bl$ is a $\sigma$-block in $\mathfrak M_i$, then $\bis(\bl)$ is a $\sigma$-block in $\mathfrak M_{1-i}$; so we write $\bl \sim_\sigma \bis(\bl)$. Moreover, if $\bl$ is followed by $\bl'$ in $\mathfrak M_i$, then $\bis(\bl)$ is followed by $\bis(\bl')$ in $\mathfrak M_{1-i}$, and so we assume that $F_1 = F_2 = F$.

Next, we analyse some properties of special $\sigma$-blocks. Recall that 
%
%To begin with, 
\textbf{Steps 1} and \textbf{2} in the proof of Theorem~\ref{dperscofinal}~$(a)$
%Section~\ref{sec:warming} 
give us the sets $\mset{i}$ containing 
%$x_i$ and 
the $\{\psi\}$-maximal points in $\mathfrak M_i$ that satisfy each formula $\psi$ in $\sub(\varphi_i)$ that is satisfiable in $\mathfrak M_i$;
%for $i=1$ and in $\sub(\psi)$ for $i=2$; 
the set $T$ of the $\sigma$-types of points in $\{x_1,x_2\}\cup\mset{1} \cup\mset{2}$ (cf.\ \eqref{typesT}); and also the sets $\sset{i} \subseteq W_i$ of $t$-maximal points in $\mathfrak M_i$ satisfying the $\sigma$-types $t$ from $T$.
Points in $\{x_i\}\cup \mset{i}\cup \sset{i}$ are called \emph{relevant in} $\mathfrak M_i$.
%if $t^\sigma_{\mathfrak M_i}(x) \in T$.
%there is $y \in \bl_{\mathfrak{M}}^{\sigma}(x)$ such that $t^\sigma_{\mathfrak M_i}(y) \in T$. 
%A $\sigma$-block $\bl_i^j$ is \emph{relevant} in $\mathfrak M_i$ if it contains a relevant point, which by \textbf{(block)} is equivalent to $\bl_i^j \cap S_i \ne \emptyset$. 
A cluster or an interval is \emph{relevant in} $\mathfrak M_i$ if it contains a relevant point,
and \emph{irrelevant} otherwise.
The number of relevant clusters (and of relevant $\sigma$-blocks) in $\mathfrak M_i$ is 
clearly bounded by the number of relevant points, that is, by $\kbound$ (defined in \eqref{kbound}).
%\textcolor{blue}{defined on p.~\pageref{kbound}.}\nb{not bad, but maybe numbered ref would be better}
Note that the root  and final clusters of $\mathfrak M_i$ are always relevant (the latter because $\sub(\varphi_i)$ is closed under negation, so the final cluster always  intersects with $\mset{i}$).

%$\mathcal O(|\varphi_1|+|\varphi_2|)$. 
%$\mathcal{O}\bigl(\max(|\varphi_1|,|\varphi_2|)\bigr)$.
%Note that if $\bl_i^j$ is a relevant block in $\mathfrak M_i$, then $S_i\cap\bl_i^j\subseteq \eC{\bl_i^j}$ by Lemma~\ref{int-prop}~$(e)$.
%
%An interval $I_i \in \{ I^i_j \mid j=  1,\dots,n\}$, $i = 1,2$, is \emph{relevant} if it contains a relevant point, and so a relevant $\sigma$-block; otherwise $I_i$ is \emph{irrelevant}.

\begin{example}\label{e:relevant}\em 
For the models $\mathfrak M_i$ shown in Fig.~\ref{GL3gf} from Example~\ref{ex:GL.3}~$(a)$ and $\sigma = \{p_1,p_2\}$, we have $\mset{i} = \{x_i,y_i,\yy_i^1,\yy_i^0\}$, 
$\sset{1} = \{x_1,y_1,a_1^1,\yy_1^1,\yy_1^0\}$, and $\sset{2} = \{x_2,y_2,a_2^0,\yy_2^1,\yy_2^0\}$,
so only the first two and the last two $\sigma$-blocks in the $\mathfrak M_i$ are relevant (cf.\ Example~\ref{blocksetc} for the $\sigma$-blocks). 
\hfill $\dashv$
\end{example}

%********************************

The next lemma lists a few important properties of relevant $\sigma$-blocks:

\begin{lemma}\label{l:maxbisblock}
For all $j\in F$ and $i=1,2$, the following hold\textup{:}
\begin{itemize}
\item[$(a)$]
$\sset{i}\cap\bl_i^j=\sset{i}\cap \eC{\bl_i^j}$\textup{;}

%{\color{blue}
%\item[$(b)$]
%$\bigl(\{x_i\}\cup \mset{i}\cup\sset{i}\bigr)\cap \eC{\bl_i^j}=\sset{i}\cap  \eC{\bl_i^j}$\textup{;}}

\item[$(b)$]
if $x_i\in \eC{\bl_i^j}$, then $x_i\in\sset{i}$\textup{;}
%\nb{item added, fixed crossrefs to Lemma throughout}

\item[$(c)$]
$t_{\mathfrak M_i}^\sigma\bigl((\{x_i\}\cup \mset{i}\cup\sset{i})\cap \eC{\bl_i^j}\bigr)=t_{\mathfrak M_i}^\sigma\bigl(\sset{i}\cap  \eC{\bl_i^j}\bigr)$\textup{;}

\item[$(d)$]
$\bl_i^j$ is relevant iff $\sset{i}\cap\eC{\bl_i^j}\ne\emptyset$\textup{;}

\item[$(e)$]
there is a bijection $\iso^-\colon\bigl(\sset{1}\cap\eC{\bl_1^j}\bigr)\to\bigl(\sset{2}\cap\eC{\bl_2^j}\bigr)$ with $t_{\mathfrak M_1}^\sigma(y)=t_{\mathfrak M_2}^\sigma\bigl(\iso^-(y)\bigr)$, for every $y\in\sset{1}\cap\eC{\bl_1^j}$\textup{;}
%$t_{\mathfrak M_1}^\sigma\bigl(\sset{1}\cap\eC{\bl_1^j}\bigr)= t_{\mathfrak M_2}^\sigma\bigl(\sset{2}\cap\eC{\bl_2^j}\bigr)$\textup{;}

%\item[$(e)$]
%$|\sset{1}\cap\eC{\bl_1^j}|=|\sset{2}\cap\eC{\bl_2^j}|$\textup{;}\nb{only here because of temporal stuff on p.46}

\item[$(f)$]
$\bl_1^j$ is relevant iff $\bl_2^j$ is relevant.
\end{itemize}
\end{lemma}
\begin{proof}
Recall the following properties of the $\sset{i}$ defined in \textbf{Step 2} of the proof of Theorem~\ref{dperscofinal}\textup{:}
\begin{enumerate}
%[label=\arabic*)]
\item
if $x\in\sset{i}$, then $x$ is $t_{\mathfrak M_i}^\sigma(x)$-maximal in $\mathfrak M_i$;

%{\color{blue}
%\item
%if $x\in\{x_i\}\cup \mset{i}$ and $x$ is $t_{\mathfrak M_i}^\sigma(x)$-maximal in $\mathfrak M_i$, then $x\in\sset{i}$;}

\item
if $x_i$ is $t_{\mathfrak M_i}^\sigma(x_i)$-maximal, then $x_i\in\sset{i}$;
%\nb{item added, fixed crossrefs below}

\item
if $x\in\mset{i}$ and $x$ is $t_{\mathfrak M_i}^\sigma(x)$-maximal in $\mathfrak M_i$,
then there is $z\in C(x)\cap\sset{i}$ with $t_{\mathfrak M_i}^\sigma(z)=t_{\mathfrak M_i}^\sigma(x)$;

\item
$t_{\mathfrak M_i}^\sigma\bigl(\{x_i\}\cup \mset{i}\cup\sset{i}\bigr)\subseteq 
t_{\mathfrak M_i}^\sigma(\sset{i})$;

\item
there is a bijection $\iso\colon\sset{1}\to\sset{2}$ with $t_{\mathfrak M_1}^\sigma(y)=t_{\mathfrak M_2}^\sigma\bigl(\iso(y)\bigr)$, for every $y\in\sset{1}$.
%$t_{\mathfrak M_1}^\sigma(\sset{1})=t_{\mathfrak M_2}^\sigma(\sset{2})$;

%\item[$(v)$]
%if $x,y\in\sset{i}$, $x\ne y$ then 
%$t_{\mathfrak M_i}^\sigma(x)\ne t_{\mathfrak M_i}^\sigma(y)$;
\end{enumerate}

$(a)$
Let $x\in\sset{i}\cap\bl_i^j$. By Lemma~\ref{int-prop}~$(e)$, there is $y\in \eC{\bl_i^j}$ with  
$t_{\mathfrak M_i}^\sigma(y)=t_{\mathfrak M_i}^\sigma(x)$. Then $C(x)=C(y)$ follows from 1.,
and so $x\in \eC{\bl_i^j}$.

$(b)$
If $x_i\in \eC{\bl_i^j}$ then $x_i$ is $t_{\mathfrak M_i}^\sigma(x_i)$-maximal by {\bf (block)}, and so 
$x_i\in\sset{i}$ by 2.

$(c)$
Take $x\in \bigl(\{x_i\}\cup \mset{i}\bigr)\cap \eC{\bl_i^j}$.
By 4., there is $y\in\sset{i}$ with $t_{\mathfrak M_i}^\sigma(y)=t_{\mathfrak M_i}^\sigma(x)$. 
By 1., $y$ is $t_{\mathfrak M_i}^\sigma(y)$-maximal in $\mathfrak M_i$. Thus, by {\bf (block)} 
and Lemma~\ref{int-prop}~$(e)$, $y\in \eC{\bl_i^j}$.
It follows that $x$ is $t_{\mathfrak M_i}^\sigma(x)$-maximal in $\mathfrak M_i$,
and so 
%\textcolor{blue}{$x\in\sset{i}$} 
either $x=x_i\in\sset{i}$ by 2., or
there is $z\in C(x)\cap\sset{i}$ with $t_{\mathfrak M_i}^\sigma(z)=t_{\mathfrak M_i}^\sigma(x)$
by 3.

$(d)$
We show that 
$t_{\mathfrak M_i}^\sigma\bigl(\bigl(\{x_i\}\cup \mset{i}\cup\sset{i}\bigr)\cap\bl_i^j\bigr)\subseteq
t_{\mathfrak M_i}^\sigma\bigl(\bigl(\{x_i\}\cup \mset{i}\cup\sset{i}\bigr)\cap \eC{\bl_i^j}\bigr)$.
Then $(d)$ follows from $(c)$.
To this end, take $x\in \bigl(\{x_i\}\cup \mset{i}\cup\sset{i}\bigr)\cap\bl_i^j$.
By 4., there is $y\in\sset{i}$ with $t_{\mathfrak M_i}^\sigma(y)=t_{\mathfrak M_i}^\sigma(x)$. 
By 1., $y$ is $t_{\mathfrak M_i}^\sigma(y)$-maximal in $\mathfrak M_i$. Thus, by {\bf (block)} 
and Lemma~\ref{int-prop}~$(e)$, $y\in \eC{\bl_i^j}$.

$(e)$
Let $\iso^-=\rest{\iso}{\sset{1}\cap\eC{\bl_1^j}}$ for the bijection $f$ provided by 5. Then, for every
$x\in \sset{1}\cap\eC{\bl_1^j}$,
$\iso^-(x)=\iso(x)\in\sset{2}$ with $t_{\mathfrak M_2}^\sigma(\iso(x))=t_{\mathfrak M_1}^\sigma(x)$. 
By Lemma~\ref{bis-blocks}~$(a)$, $t_{\mathfrak M_2}^\sigma(\iso(x))\in t_{\mathfrak M_2}^\sigma(\bl_2^j)$, so $\iso(x)\in\bl_2^j$ follows by {\bf (block)}. Thus, $\iso(x)\in\eC{\bl_2^j}$ by $(a)$.  

$(f)$ follows from $(d)$ and $(e)$.
\end{proof}

We are now in a position to partition each of the $\mathfrak M_i$  into the same polynomial number $N$ of closed intervals  
$\mathcal I_i=\{I_i^\ell\in\INT_i\mid \ell<\partN\}$ 
%such that \plan{} above holds,
such that $\rest{\mathfrak M_1}{I_1^\ell}$ and $\rest{\mathfrak M_2}{I_2^\ell}$ are \globally{} $\sigma$-bisimilar, for every $\ell<\partN$, 
even if there are infinitely many $\sigma$-blocks in each $\mathfrak M_i$ and
not all of them are definable in $\mathfrak M_i$.

\begin{definition}\label{d:ints}
\rm
We define the partitions $\mathcal I_i$ of $\mathfrak M_i$, $i=1,2$, in three steps.
In each step, we add interval-pairs $(I_1,I_2)$ to $\mathcal I_1\times\mathcal I_2$ in such a way that
\begin{itemize} 
\item[$(a)$]
$I_i$ is a closed interval whose final cluster is a non-limit cluster, for $i=1,2$;
\item[$(b)$]
there are $j,j'\in F$ such that $I_i=\bigcup_{j\preceq k\preceq j'}\bl_i^k$, for $i=1,2$. 
\end{itemize}
It follows then from $(a)$ and Lemma~\ref{l:closedintdef} that all intervals in $\mathcal I_i$ are definable in 
$\mathfrak M_i$. Also, it follows from $(b)$ and Lemma~\ref{bis-blocks}~$(a)$ that 
%condition \plan{} (at the start of Section~\ref{section:K4.3}) 
%
\begin{multline}\label{Int}
\mbox{$\bigl\{(y_1,y_2)\in I_1^\ell\times I_2^\ell\mid t^\sigma_{\mathfrak M_1}(y_1)=t^\sigma_{\mathfrak M_2}(y_2)\bigr\}$ is a \globalb{} $\sigma$-bisimulation}\\
\mbox{between $\rest{\mathfrak M_1}{I_1^\ell}$ and $\rest{\mathfrak M_2}{I_2^\ell}$, for every $\ell<\partN$.}
\end{multline}
The three steps are as follows:
\begin{itemize}
\item[\step{1}]
First, suppose $\bl_1^j$, $j\in F$, is a relevant $\sigma$-block that is definable in $\mathfrak M_1$.
By Lemmas~\ref{l:maxbisblock}~$(f)$ and \ref{bis-blocks}~$(e)$,
$\bl_2^j$ is also a relevant $\sigma$-block definable in $\mathfrak M_2$.
We put into $\mathcal{I}_i$ all those 
relevant $\sigma$-blocks $\bl_i^j$ that are definable in $\mathfrak M_i$, for $i=1,2$.
Then $(b)$ clearly holds, and $(a)$ holds by Lemma~\ref{int-prop}~$(a),(d)$.

\item[\step{2}]
Next, suppose $\bl_1^j$, $j\in F$, is a relevant $\sigma$-block that is not definable in $\mathfrak M_1$.
By Lemmas~\ref{l:maxbisblock}~$(f)$ and \ref{bis-blocks}~$(e)$,
$\bl_2^j$ is also a relevant $\sigma$-block that is not definable in $\mathfrak M_2$.
By Lemma~\ref{int-prop}~$(d)$, each $\eC{\bl_i^j}$ is a limit cluster in $\mathfrak F_i$,
and so $j$ is a $\succ$-limit.
We pick some $\ell\succ j$ such that the $\sigma$-blocks $\bl_i^k$, for $j\prec k\preceq \ell$, are all irrelevant, for $i=1,2$.
Such an $\ell$ must exist as $j$ is a $\succ$-limit 
and the number of relevant points is finite, but this $\ell$ is not unique. 
Let $F^-=\{k\in F\mid j\preceq k\preceq\ell\}$ and $\succ^-=\rest{\succ}{F^-}$.
By Lemmas~\ref{lem:descr0}~$(a)$ and \ref{l:countable}, there is an isomorphism $\iso$ from
some countable ordinal $\gamma$ to $(F^-,\succ^-)$.
As $j$ is a $\succ$-limit, $\gamma\ge\omega$. Take $\iso(n)$, $n<\omega$.
There are two cases:
\begin{enumerate}
\item
%\item[$(i)$]
There exists $m$, $0<m<\omega$, such that $\bl_1^{\iso(n)}$ is a degenerate $\sigma$-block
for every $n$ with $m\le n<\omega$.
Then, by Lemma~\ref{bis-blocks}~$(b)$, $\bl_2^{\iso(n)}$ is a degenerate $\sigma$-block,
for every $n$ with $m\le n<\omega$. We set $j'=\iso(m)$.

\item
%\item[$(ii)$]
For every $n<\omega$, there is $m_n$, $n\le m_n<\omega$, such that 
$\bl_1^{\iso(m_n)}$ is a non-degenerate $\sigma$-block.
Then, by Lemma~\ref{bis-blocks}~$(b)$, $\bl_2^{\iso(m_n)}$ is a non-degenerate $\sigma$-block as well. Note that if $n\ge 1$, then $\iso(m_n)$ is not a $\succ$-limit. Thus, $\eC{\bl_i^{\iso(m_n)}}$ is not a limit cluster, and so it is definable in $\mathfrak M_i$ by Lemma~\ref{lem:descr'}.
%Thus, $\eC{\bl_i^{j'}}$ is not a limit cluster, for $i=1,2$.
We set $j'=\iso(m_1)$.
%\end{itemize}
\end{enumerate}
In both cases, we 
put the intervals $\bigcup_{j\preceq k\preceq j'}\bl_i^{k}$ into $\mathcal I_i$,  $i=1,2$, and 
say that they \emph{extend} the relevant non-definable $\sigma$-blocks $\bl_i^j$.
Then $(a)$, $(b)$ hold.
% and $(a)$ holds because of the following: For every $0<n<\omega$, $\iso(n)$ is not a $\succ$-limit, and so $\eC{\bl_i^{\iso(n)}}$ is not a limit cluster. Thus, $\eC{\bl_i^{j'}}$ is not a limit cluster, for $i=1,2$.

\item[\step{3}]
Finally, suppose that, for $i=1,2$, the intervals $I_i=\bigcup_{n_1\preceq k\preceq n_2}\bl_i^k$ and $J_i=\bigcup_{j_1\preceq k\preceq j_2}\bl_i^k$ are such that there is $k$ with $n_2\prec k\prec j_1$, $I_i,J_i\in\mathcal I_i$, and 
 there is no interval in $\mathcal I_i$ intersecting the `gap' between $I_i$ and $J_i$ (that is,
 any $\bl_i^k$ with $n_2\prec k\prec j_1$). By $I_i\in\mathcal I_i$ and $(i)$, $n_2$ is not a $\succ$-limit.
Let  $n_2^+$ be the immediate $\prec$-successor of $n_2$ and $j_1^-$ the immediate $\prec$-predecessor of $j_1$.
 Then we put the (irrelevant) interval $\bigcup_{n_2^+\preceq k\preceq j_1^-}\bl_i^{k}$ into $\mathcal I_i$, for $i=1,2$.
 Then $(b)$ clearly holds, and $(a)$ holds as $j_1^-$ is not a $\succ$-limit.
 By doing this for all the gaps, we end up with the required partition $\mathcal I_i$ of $\mathfrak M_i$. 
\end{itemize}
The number of intervals added in steps \step{1} and \step{2} 
together cannot exceed the number of relevant $\sigma$-blocks, and so it is bounded by $\kbound$. 
As the $\intord{\mathfrak F_i}$-smallest and $\intord{\mathfrak F_i}$-largest $\sigma$-blocks are relevant, the number of intervals added in step \step{3}
is bounded by $\kbound-1$, so altogether the (same) number $N$ of intervals in each $\mathcal I_i$ does not exceed $2\kbound$. 
\end{definition}

%If $\bl^i_j = [C_1^i,C_2^i]$ is not $\sigma$-definable, then $C_2^i$ is a limit cluster by Lemma~\ref{int-prop} $(e)$. We pick some irrelevant $\sigma$-block $[D_1^i,D_2^i]$ such that $D_2^i$ is a  non-limit cluster, there are no relevant points in $(C_2^i,D_2^i]$, and $[D_1^1,D_2^1] \sim_\sigma [D_1^2,D_2^2]$, which must exist because each block is a closed interval, the number of relevant blocks is finite, and there is an infinite strictly descending chain of clusters with limit $C_2^i$. We then set $\bl^i_j := [C_1^i,D_2^i]$. Clearly each new $\bl^i_j$ is $\sigma$-definable and $\bl^1_j \sim_\sigma \bl^2_j$. Finally, if the gap between $\bl^i_j$ and $\bl^i_{j+1}$ is not empty, it takes the form $[D_1^i,D_2^i]$ with non-limit $D_2^i$ and is $\sigma$-definable.  Moreover, $[D_1^1,D_2^1] \sim_\sigma [D_1^2,D_2^2]$. 
%This gives us the required $\sigma$-\emph{definable} intervals $I^i_1,\dots, I^i_n$, for $n = O(\varphi_i)$, and decompositions $\mathfrak M_i = \mathfrak M^i_1 \lhd \dots \lhd \mathfrak M^i_n$. Each $I_j^i$ is either an original relevant $\sigma$-block or an extended relevant non-definable $\sigma$-block, or a union of consecutive irrelevant $\sigma$-blocks. 

The following example illustrates Definition~\ref{d:ints}.

\begin{example}\label{ex:inter:GL.3}\em
For models $\mathfrak M_i$, $i=1,2$, from Example~\ref{ex:GL.3}~$(a)$ and 
$\sigma$-blocks from Example~\ref{blocksetc} for
$\sigma = \{p_1,p_2\}$, we can pick the intervals $I_i^j$, $j \le 4$, shown below, where $I_i^2$ are irrelevant and all other intervals are relevant (cf.\ Example~\ref{e:relevant}). 
The choice of the infinite intervals $I_i^1$ extending the non-definable $\sigma$-blocks till $\yy_i^4$ is arbitrary. We could make them  shorter or, on the contrary, extend until $\yy_i^2$, in which case there would be no gap between these intervals (extending relevant non-definable $\sigma$-blocks) and the next relevant interval. \hfill $\dashv$
\end{example}

\begin{center}
\begin{tikzpicture}[>=latex,line width=0.5pt,xscale = 1.1,yscale = .65]
\node[]  at (-.8,0) {$\mathfrak M_2$};
\node[point,fill=black,scale = 0.7,label=below:{\footnotesize $\neg\varphi_2$},label=above:{\footnotesize $x_2$}] (x2) at (0,0) {};
\node[point,fill=black,scale = 0.7,label=below:{\footnotesize $p_2,\neg q_2$},label=above:{\footnotesize $y_2$}] (y2) at (1,0) {};
\node[point,scale = 0.7,label=above:{\footnotesize $a_2^0$},label=below:{\footnotesize $p_2$}] (a20) at (2.5,0) {};
\node[point,scale = 0.7,label=above:{\footnotesize $a_2^1$},label=below:{\footnotesize $p_1$}] (a21) at (3.5,0) {};
\draw[] (3,.1) ellipse (1 and .9);
\node[scale = 0.9]  at (3,-.45) {$q_2$};
\node[]  at (4.5,0) {$\dots$};
\node[point,fill=black,scale = 0.7,label=above right:{\footnotesize $\yy_2^4$},label=below:{\footnotesize $p_1,q_2$}] (b24) at (5,0) {};
\node[point,fill=black,scale = 0.7,label=above right:{\footnotesize $\yy_2^3$},label=below:{\footnotesize $p_2,q_2$}] (b23) at (6,0) {};
\node[point,fill=black,scale = 0.7,label=above right:{\footnotesize $\yy_2^2$},label=below:{\footnotesize $p_1,q_2$}] (b22) at (7,0) {};
\node[point,fill=black,scale = 0.7,label=above right:{\footnotesize $\yy_2^1$},label=below:{\footnotesize $p_2,q_2$}] (b21) at (8,0) {};
\node[point,fill=black,scale = 0.7,label=above right:{\footnotesize $\yy_2^0$},label=below:{\footnotesize $p_2,q_2$}] (b20) at (9,0) {};
\draw[->] (x2) to (y2);
\draw[->] (y2) to (2,0);
\draw[->] (b24) to (b23);
\draw[->] (b23) to (b22);
\draw[->] (b22) to (b21);
\draw[->] (b21) to (b20);
\draw[-,thick] (-.2,-.7) -- (-.2,-.9) -- (.2,-.9) -- (.2,-.7);
\draw[-,thick] (.6,-.7) -- (.6,-.9) -- (4,-.9) -- (4,-.7);
\draw[-,thick] (4.7,-.7) -- (4.7,-.9) -- (5.3,-.9) -- (5.3,-.7);
\draw[-,thick] (5.7,-.7) -- (5.7,-.9) -- (6.3,-.9) -- (6.3,-.7);
\draw[-,thick] (6.7,-.7) -- (6.7,-.9) -- (7.3,-.9) -- (7.3,-.7);
\draw[-,thick] (7.7,-.7) -- (7.7,-.9) -- (8.3,-.9) -- (8.3,-.7);
\draw[-,thick] (8.7,-.7) -- (8.7,-.9) -- (9.3,-.9) -- (9.3,-.7);
\draw[-,gray] (-.2,-1.1) -- (-.2,-1.3) -- (.2,-1.3) -- (.2,-1.1);
\node[gray] at (0,-1.7) {\footnotesize $I_2^0$};
\draw[-,gray] (.6,-1.1) -- (.6,-1.3) -- (5.3,-1.3) -- (5.3,-1.1);
\node[gray] at (3,-1.7) {\footnotesize $I_2^1$};
\draw[-,gray] (5.7,-1.1) -- (5.7,-1.3) -- (7.3,-1.3) -- (7.3,-1.1);
\node[gray] at (6.5,-1.7) {\footnotesize $I_2^2$};
\draw[-,gray] (7.7,-1.1) -- (7.7,-1.3) -- (8.3,-1.3) -- (8.3,-1.1);
\node[gray] at (8,-1.7) {\footnotesize $I_2^3$};
\draw[-,gray] (8.7,-1.1) -- (8.7,-1.3) -- (9.3,-1.3) -- (9.3,-1.1);
\node[gray] at (9,-1.7) {\footnotesize $I_2^4$};
\node[]  at (-.8,3) {$\mathfrak M_1$};
\node[point,fill=black,scale = 0.7,label=above:{\footnotesize $\varphi_1$},label=below:{\footnotesize $x_1$}] (x1) at (0,3) {};
\node[point,fill=black,scale = 0.7,label=above:{\footnotesize $p_1,\neg q_1$},label=below:{\footnotesize $y_1$}] (y1) at (1,3) {};
\node[point,scale = 0.7,label=below:{\footnotesize $a_1^0$},label=above:{\footnotesize $p_2$}] (a10) at (2.5,3) {};
\node[point,scale = 0.7,label=below:{\footnotesize $a_1^1$},label=above:{\footnotesize $p_1$}] (a11) at (3.5,3) {};
\draw[] (3,2.9) ellipse (1 and .9);
\node[scale = 0.9]  at (3,3.5) {$q_1$};
\node[]  at (4.5,3) {$\dots$};
\node[point,fill=black,scale = 0.7,label=below right:{\footnotesize $\yy_1^4$},label=above:{\footnotesize $p_2,q_1$}] (b14) at (5,3) {};
\node[point,fill=black,scale = 0.7,label=below right:{\footnotesize $\yy_1^3$},label=above:{\footnotesize $p_1,q_1$}] (b13) at (6,3) {};
\node[point,fill=black,scale = 0.7,label=below right:{\footnotesize $\yy_1^2$},label=above:{\footnotesize $p_2,q_1$}] (b12) at (7,3) {};
\node[point,fill=black,scale = 0.7,label=below right:{\footnotesize $\yy_1^1$},label=above:{\footnotesize $p_1,q_1$}] (b11) at (8,3) {};
\node[point,fill=black,scale = 0.7,label=below right:{\footnotesize $\yy_1^0$},label=above:{\footnotesize $p_2,q_1$}] (b10) at (9,3) {};
\draw[->] (x1) to (y1);
\draw[->] (y1) to (2,3);
\draw[->] (b14) to (b13);
\draw[->] (b13) to (b12);
\draw[->] (b12) to (b11);
\draw[->] (b11) to (b10);
\draw[-,thick] (-.2,3.7) -- (-.2,3.9) -- (.2,3.9) -- (.2,3.7);
\draw[-,thick] (.6,3.7) -- (.6,3.9) -- (4,3.9) -- (4,3.7);
\draw[-,thick] (4.7,3.7) -- (4.7,3.9) -- (5.3,3.9) -- (5.3,3.7);
\draw[-,thick] (5.7,3.7) -- (5.7,3.9) -- (6.3,3.9) -- (6.3,3.7);
\draw[-,thick] (6.7,3.7) -- (6.7,3.9) -- (7.3,3.9) -- (7.3,3.7);
\draw[-,thick] (7.7,3.7) -- (7.7,3.9) -- (8.3,3.9) -- (8.3,3.7);
\draw[-,thick] (8.7,3.7) -- (8.7,3.9) -- (9.3,3.9) -- (9.3,3.7);
\draw[-,gray] (-.2,4.1) -- (-.2,4.3) -- (.2,4.3) -- (.2,4.1);
\node[gray] at (0,4.7) {\footnotesize $I_1^0$};
\draw[-,gray] (.6,4.1) -- (.6,4.3) -- (5.3,4.3) -- (5.3,4.1);
\node[gray] at (3,4.7) {\footnotesize $I_1^1$};
\draw[-,gray] (5.7,4.1) -- (5.7,4.3) -- (7.3,4.3) -- (7.3,4.1);
\node[gray] at (6.5,4.7) {\footnotesize $I_1^2$};
\draw[-,gray] (7.7,4.1) -- (7.7,4.3) -- (8.3,4.3) -- (8.3,4.1);
\node[gray] at (8,4.7) {\footnotesize $I_1^3$};
\draw[-,gray] (8.7,4.1) -- (8.7,4.3) -- (9.3,4.3) -- (9.3,4.1);
\node[gray] at (9,4.7) {\footnotesize $I_1^4$};
\node[gray]  at (-.5,1.5) {$\bis$};
\draw[gray,thick,dotted] (x1) to (x2);
\draw[gray,thick,dotted] (y1) to (a21);
\draw[gray,thick,dotted] (y2) to (a10);
\draw[gray,thick,dotted] (a10) to (a20);
\draw[gray,thick,dotted] (a11) to (a21);
\draw[gray,thick,dotted] (b14) to (b24);
\draw[gray,thick,dotted] (b13) to (b23);
\draw[gray,thick,dotted] (b12) to (b22);
\draw[gray,thick,dotted] (b11) to (b21);
\draw[gray,thick,dotted] (b10) to (b20);
\end{tikzpicture}
\end{center}
%

%*************************************************************************************

\subsection{Simplifying interval-based models}\label{ss:small}

Consider again our $\sigma$-bisimilar $\delta$-models $\mathfrak M_i$, $i=1,2$, that are based on finitely $\mathfrak M_i$-generated
descriptive frames $\mathfrak{F}_i = (W_{i},R_{i},\INT_{i})$ for $L$ with roots $x_i \in W_i$ 
and witness the lack of an interpolant for $\varphi_1$ and $\varphi_2$, where
$\delta = \sig(\varphi_1) \cup \sig(\varphi_2)$ and $\sigma = \sig(\varphi_1) \cap \sig(\varphi_2)$.
In Definition~\ref{d:ints}, we determined 
%$N=\mathcal{O}\bigl(\max(|\varphi_1|,|\varphi_2|)\bigr)$ 
$N<2\kbound$, for the polynomial number $\kbound$ from \eqref{kbound}, 
and constructed the partitions $\mathcal I_i=\{I_i^\ell\in\INT_i\mid \ell<\partN\}$ of $\mathfrak M_i$
with $I_i^0\intord{\mathfrak F_i}\cdots\intord{\mathfrak F_i}I_i^{N-1}$  satisfying \eqref{Int}. 
%%
%\begin{itemize}
%\item[\plan]
%$\bigl\{(y_1,y_2)\in I_1^\ell\times I_2^\ell\mid t^\sigma_{\mathfrak M_1}(y_1)=t^\sigma_{\mathfrak M_2}(y_2)\bigr\}$ is a \globalb{} $\sigma$-bisimulation between $\rest{\mathfrak M_1}{I_1^\ell}$ and $\rest{\mathfrak M_2}{I_2^\ell}$, for $\ell<N$.
%\end{itemize}
%
We now use these partitions to prove Theorem~\ref{t:structmodel}.
First, in Lemma~\ref{l:replacement},
we transform each pair $(\rest{\mathfrak M_1}{I_1^\ell},\rest{\mathfrak M_2}{I_2^\ell})$, $\ell<N$, 
into a pair $(\mathfrak N_1^\ell,\mathfrak N_2^\ell)$ of  models 
meeting the list of requirements in Definition~\ref{d:nice}. 
%such that $\mathfrak N_1,x_1$ and $\mathfrak N_2,x_2$ meet a certain list of requirements, for
%$\mathfrak N_i=\mathfrak N_i^0\lhd\dots\lhd\mathfrak N_i^{N-1}$, $i=1,2$.
Then, in Lemma~\ref{l:final},
we show that these requirements ensure that 
$\mathfrak N_i=\mathfrak N_i^0\lhd\dots\lhd\mathfrak N_i^{N-1}$, $i=1,2$,
%the pair $(\mathfrak N_1,\mathfrak N_2)$
satisfy all conditions in Theorem~\ref{t:structmodel}.

For all $i=1,2$ and $\ell<N$, the  
frame $\mathfrak H^\ell_i=(H^\ell_i,R_i^\ell,\INT^\ell_i)$ underlying $\mathfrak N_i^\ell$ is such that
$H^\ell_i\subseteq I_i^\ell$ is definable in $\mathfrak M_i$ and $R_i^\ell=\rest{R_i}{H^\ell_i}$, but $\mathfrak H_i^\ell$ is not necessarily a subframe of $\rest{\mathfrak F_i}{I_i^\ell}$. However, we require  
%the frames $\mathfrak H_i=\mathfrak H_i^0\lhd\dots\lhd\mathfrak H_i^{N-1}$ underlying $\mathfrak N_i$, $i=1,2$, 
each $\mathfrak H^\ell_i$ to meet some conditions making sure that the canonical formulas refuted in $\mathfrak H_i=\mathfrak H_i^0\lhd\dots\lhd\mathfrak H_i^{N-1}$
are also refuted in $\mathfrak F_i$ (and so $L\subseteq \Log(\mathfrak F_i) \subseteq \Log(\mathfrak H_i)$).
Another feature of the construction is that the atomic type of some points in $\mathfrak N_i$ could be different 
from their atomic type in $\mathfrak M_i$.
We prove $\mathfrak N_1,x_1\models\varphi_1$ and $\mathfrak N_2,x_2\models\neg\varphi_2$ by
ensuring that no new (compared to $\mathfrak M_i$) atomic types are introduced in $\mathfrak N_i$, and the distribution of old atomic types in $\mathfrak N_i$ properly matches their distribution in $\mathfrak M_i$. We achieve this
by introducing functions $\parent_i^\ell$ that assign to each point $x$ in $\mathfrak N_i^\ell$ a unique `parent' point in $\rest{\mathfrak M_i}{I_i^\ell}$ whose $\mathfrak M_i$-behaviour $x$ is intended to mimic in $\mathfrak N_i$.
%Then, we generalise the proof of Lemma~\ref{cofinsubfr}~$(a)$ and show that, for any $\tau\in\sub(\varphi_i)$ and any $x\in H_i$, $\mathfrak N_i,x\models\tau$ iff $\mathfrak M_i,\parent_i(x)\models\tau$.

\begin{definition}\label{d:nice}\em
Suppose $I\in\INT_i$, $i=1,2$, is an interval in $\mathfrak F_i$.
We say that a model $\mathfrak N$ based on a frame $\mathfrak H=(H,S,\INT')$ is 
$(I,i)$-\emph{\nice} if the following hold: 
\begin{align}
\label{sub}
& \mbox{$H\subseteq I$ and $S=\rest{R_i}{H}$}\textup{;}\\
 \label{relsin}
 & \bigl(\{x_i\}\cup \mset{i}\cup \sset{i}\bigr)\cap I\subseteq H\textup{;}\\
 \label{fincluster}
& \mbox{the final cluster in $\mathfrak H$ is a subset of the final cluster in $\rest{\mathfrak F_i}{I}$}\textup{;}\\
\label{rootcluster}
& \mbox{if the root cluster $C$ in $\mathfrak H$ is degenerate,}\\[-2pt]
\nonumber
& \hspace*{1cm}\mbox{then $C$ is the root cluster in $\rest{\mathfrak F_i}{I}$\textup{;}}
\\
 \label{inpred}
 & \mbox{for every $x\in H$, if $\{x\}$ is a degenerate non-root cluster in $\mathfrak H$}\\[-2pt]
 \nonumber
 & \hspace*{1cm}\mbox{and $C\subseteq I$ is the immediate predecessor of $\{x\}$ in $\mathfrak F_i$, then}\\[-2pt]
 \nonumber
 & \hspace*{1cm}\mbox{$C\cap H$ is the immediate predecessor of $\{x\}$ in $\mathfrak H$;}
 %\\
  \end{align}
\begin{align}
 \label{defdef}
 & \mbox{for every $x\in H$, if $\{x\}\in\INT'$, then $\{x\}\in\INT_i$;}
\end{align}
there is a function $\parent\colon H\to H$ such that
\begin{align}
 \label{relh}
 & \mbox{$\parent(x)=x$ for all $x\in \bigl(\{x_i\}\cup \mset{i}\cup \sset{i}\bigr)\cap H$;}\\
 \label{hatfix}
 & \mbox{$\at_{\mathfrak N}(x)=\at_{\mathfrak N}\bigl(\parent(x)\bigr)=\at_{\mathfrak M_i}\bigl(\parent(x)\bigr)$ for all $x\in H$;}\\
 \label{hhom}
 & \mbox{if $xR_i y$, then $\parent(x)R_i\parent(y)$ for all $x,y\in H$;}\\
 \label{hmax}
 & \mbox{if $\parent(x)R_i y$, then $xR_i y$ for all $x\in H$,  $y\in \mset{i}\cap H$.}
 \end{align}
\end{definition}

%The above notion is modular in the following sense:
%
%\begin{lemma}\label{l:nicesum}
%Suppose $\{J^j\in\INT_i\mid j<n\}$ is partition of some interval $I\in\INT_i$ into finitely many intervals
%with $J^0\intord{\mathfrak F_i}\cdots\intord{\mathfrak F_i}J^{n-1}$.
%If $\mathfrak N^j$ is a $(J^j,i)$-\nice{} model, for $j<n$, then 
%$\mathfrak N=\mathfrak N^0\lhd\cdots\lhd\mathfrak N^{n-1}$ is $(I,i)$-\nice.
%\end{lemma}
%
%\begin{proof}
%..\nb{TBC}

%\end{proof}

%******** replacement lemma

%\{The next lemma uses the partitions $\mathcal I_i=\{I_i^\ell\in\INT_i\mid \ell<\partN\}$ of $\mathfrak M_i$ with $I_i^0\intord{\mathfrak F_i}\cdots\intord{\mathfrak F_i}I_i^{N-1}$, $i=1,2$, defined above.}

%Now recall the partitions $\mathcal I_i=\{I_i^\ell\in\INT_i\mid \ell<\partN\}$ of $\mathfrak M_i$
%\textcolor{red}{with $I_i^0\intord{\mathfrak F_i}\cdots\intord{\mathfrak F_i}I_i^{N-1}$,} $i=1,2$, defined above. 
%For $i=1,2$, $\ell<N$, let $\rn_i^\ell$ be the number of relevant clusters in $I_i^\ell$.
%Then $\sum_{\ell<N}\rn_i^\ell\le \kbound$.

\begin{lemma}\label{l:replacement}
For all $i=1,2$ and $\ell<N$, there exist 
%sets $H_i^\ell\subseteq I_i^\ell$,
models $\mathfrak N_i^\ell$ based on frames $\mathfrak H_i^\ell=(H_i^\ell,S_i^\ell,\INT_i^\ell)$, 
%functions $\parent_i^\ell\colon H_i^\ell\to H_i^\ell$, 
and numbers $\nn_i^\ell>0$ with $\sum_{\ell<N}\nn_i^\ell\le3\kbound-1$ 
such that the following hold\textup{:}
%
%For $i=1,2$, $\ell<N$, there exist numbers $\nn_i^\ell>0$ with $\sum_{\ell<N}\nn_i^\ell\le3\kbound-1$,
%sets $H_i^\ell\subseteq I_i^\ell$,
%models $\mathfrak N_i^\ell$ based on frames $\mathfrak H_i^\ell=(H_i^\ell,S_i^\ell,\INT_i^\ell)$ 
%and functions $\parent_i^\ell\colon H_i^\ell\to H_i^\ell$ 
%such that the following hold, for $i=1,2$, $\ell<N$\textup{:}
 %
 \begin{itemize}
 \item[$(a)$]
$\mathfrak N_i^\ell$ is $(I_i^\ell,i)$-\nice\textup{;}

\item[$(b)$]
$\mathfrak N_i^\ell$ is the ordered sum of $\nn_i^\ell$-many \simple{} $\delta$-models based on \atomic{} frames\textup{;}

\item[$(c)$]
the pair $(\mathfrak N_1^\ell,\mathfrak N_2^\ell)$ is \match.
\end{itemize}
\end{lemma}
\begin{proof}
 We consider three \textbf{\emph{Cases}} I--III, depending on the step the pair $(I_1^\ell,I_2^\ell)$
 is added to $\mathcal I_1\times\mathcal I_2$ in Definition~\ref{d:ints}.
 %We begin with the technically `gentlest' case:
 
 \textbf{\emph{Case}} I: $(I_1^\ell,I_2^\ell)$ is added in step \step{3}, so $I_i^\ell$ are irrelevant intervals.
 %and so $\rn_i=0$.
 We let $\nn_1^\ell=\nn_2^\ell=1$ and define $\mathfrak N_1^\ell$ and $\mathfrak N_2^\ell$
 as follows. Let $Z_i^\ell=\{z_i^j\mid j< m_i\}$, for $i=1,2$, be the \tail{} of $\rest{\mathfrak F_i}{I_i^\ell}$, for some $m_i\leq\omega$, with $z_{i}^jR_i^s z_i^{j-1}$,  $0<j<m_i$. 
By \eqref{Int},
%\plan,
$\bigl\{(y_1,y_2)\in I_1^\ell\times I_2^\ell\mid t^\sigma_{\mathfrak M_1}(y_1)=t^\sigma_{\mathfrak M_2}(y_2)\bigr\}$
is a \globalb{} $\sigma$-bisimulation between $\rest{\mathfrak M_1}{I_1^\ell}$ and $\rest{\mathfrak M_2}{I_2^\ell}$.
%So by Lemma~\ref{l:tails}, 
It is straightforward to see that
because of this
we must have $|Z_1^\ell|=|Z_2^\ell|=m$, for some $m\leq\omega$, and $Z_1^\ell=I_1^\ell$ iff $Z_2^\ell=I_2^\ell$.
Also, if $Z_i^\ell\ne I_i^\ell$, then there exist $w_i^\ell$ in the \source{} of $Z_i^\ell$ with $t^\sigma_{\mathfrak M_1}(w_1^\ell)=t^\sigma_{\mathfrak M_2}(w_2^\ell)$. For $i=1,2$, let 
\[
H_i^\ell=\left\{
\begin{array}{ll}
Z_i^\ell, & \mbox{ if $Z_i^\ell=I_i^\ell$},\\[3pt]
\{w_i^\ell\}\cup Z_i^\ell, & \mbox{ otherwise},
\end{array} 
\right.
\]
$S_i^\ell=\rest{R_i}{H_i^\ell}$,
and let $\INT_i^\ell$  consist of all finite subsets of $Z_i^\ell$ and their complements in $H_i^\ell$.
Then $\iso\colon H_1^\ell\to H_2^\ell$ defined by $\iso(z_1^j)=z_2^j$, $j<m$, and $\iso(w_1^\ell)=w_2^\ell$ is an isomorphism between the resulting frames $\mathfrak H_1^\ell$ and $\mathfrak H_2^\ell$, which  are isomorphic to 
\begin{itemize}
\item[$(i)$]
$m^<$, when $Z_i^\ell=I_i^\ell$\textup{;}

\item[$(ii)$]
$\cluster{1}\lhd m^<$, when $Z_i^\ell\ne I_i^\ell$ and $m<\omega$\textup{;}

\item[$(iii)$]
$\chain{1}{\bullet}$, when $Z_i^\ell$ is infinite (as $w_i^\ell R_i w_i^\ell$ by \eqref{facenondeg}).

\end{itemize}
This gives \eqref{sub}--\eqref{defdef} for $\mathfrak H=\mathfrak H_i^\ell$ and $I=I_i^\ell$ (we have \eqref{defdef} because of \eqref{infiniteincl} and Lemma~\ref{l:defpoints}).
%
%We clearly have that $\rest{\mathfrak{F}_i}{I_i}\rightarrowtail\mathfrak{H_i}$ for $(i)$ and $(ii)$, and for $(iii)$, this is because \eqref{fincofinZ} ensures that
%the function mapping $I_i\setminus Z_i$ to $u_i$ and being identical on $Z_i$ is a \rp{} p-morphism from
%$\rest{\mathfrak F_i}{I_i}$ onto $\mathfrak H_i$. So we have $(b)$ by Lemma ...
%
For $p\in\delta$, let $\mathfrak{w}_i^\ell(p)=\mathfrak{v}_i(p)\cap H_i^\ell$ in cases $(i)$ and $(ii)$, and
\[
\mathfrak w_i^\ell(p)=\left\{
\begin{array}{ll}
H_i^\ell, & \mbox{ if $w_i^\ell\in\mathfrak v_i(p)$},\\[3pt]
\emptyset, & \mbox{ otherwise}
\end{array} 
\right.
\]
in case $(iii)$. 
In all cases, 
$\mathfrak w_i^\ell(p)\in\INT_i^\ell$ and $(b)$ holds
for $\mathfrak N_i^\ell= (\mathfrak H_i^\ell,\mathfrak w_i^\ell)$.
 Also, the pair $(\mathfrak N_1^\ell,\mathfrak N_2^\ell)$ is of type $(a)$ in 
Definition~\ref{d:match}, and so $(c)$ of the lemma holds.
Finally, for $x\in H_i^\ell$, we let $\parent_i^\ell(x)=x$ in cases $(i)$ and $(ii)$, and $\parent_i^\ell(x)=w_i^\ell$ in case $(iii)$.
It is straightforward to check that \eqref{hatfix} and \eqref{hhom} hold for $\parent=\parent_i^\ell$.
Note that
%
%Then we also have $(d).2$ and $(d).3$.
%
\eqref{relh} and \eqref{hmax} hold vacuously, as $H_i^\ell\subseteq I_i^\ell$ and
%each $I_i^\ell$ is an irrelevant interval in $\mathfrak M_i$ when $(I_1^\ell,I_2^\ell)$ is added in step \step{3}, and so 
$\bigl (\{x_i\}\cup \mset{i}\cup \sset{i}\bigr) \cap I_i^\ell=\emptyset$, $i=1,2$. Thus, we have $(a)$ of the lemma.

\smallskip
\textbf{\emph{Case}} II: $(I_1^\ell,I_2^\ell)$ is added in step \step{1}. 
For $i=1,2$, let $\bl_i$ be the relevant $\sigma$-blocks such that 
$t^\sigma_{\mathfrak M_1}(\bl_1)=t^\sigma_{\mathfrak M_2}(\bl_2)$ and
$I_i^\ell=\bl_i$ is definable in $\mathfrak M_i$.
%or $I_i$ is extending $\bl_i$ that is not definable in $\mathfrak M_i$, for $i=1,2$.
For $\ell<N$, let $\rn_i^\ell$ denote the number of relevant clusters in $I_i^\ell$,  and
let $C_i^{\ell,j}$, $j < \rn_i^\ell$, be the sequence (ordered by $<_{R_i}$) of all relevant clusters in $\bl_i$  (that intersect with $\{x_i\}\cup \mset{i}\cup \sset{i}$). Then $C_i^{\ell,\rn_i^\ell-1}$ is the final cluster $\eC{\bl_i}$ of $\bl_i$.

\textbf{\emph{Case}} II.1: 
Observe that, by Lemmas~\ref{int-prop}~$(c)$ and \ref{bis-blocks}~$(b)$,
$C_1^{\ell,\rn_1^\ell-1}=\eC{\bl_1}$ is degenerate iff
$C_2^{\ell,\rn_2^\ell-1}=\eC{\bl_2}$ is degenerate iff
both $\bl_1=\eC{\bl_1}$ and $\bl_2=\eC{\bl_2}$ are degenerate $\sigma$-blocks, and 
so $\rn_i^\ell=1$.
So, in this case, we just set $\nn_i^\ell=1$, $\mathfrak H_i^\ell=\rest{\mathfrak F_i}{\bl_i}$,
$\mathfrak N_i^\ell=\rest{\mathfrak M_i}{\bl_i}$ and $\parent_i(z_i)=z_i$ for the only point $z_i$
in $\bl_i$, $i=1,2$.
It is straightforward to check that $(a)$--$(c)$ of the lemma hold. In particular,
$(c)$ holds because
the pair $(\mathfrak N_1^\ell,\mathfrak N_2^\ell)$ is of type $(a)$ in Definition~\ref{d:match}.
%Then $(a)$, $(b)$, $(d)$ and $(c).1$ of the lemma clearly hold. Condition $(c).2$ also holds because $(\mathfrak N_1^\ell,\mathfrak N_2^\ell)$ meets condition $(a)$ in Definition~\ref{d:match}.

\textbf{\emph{Case}} II.2: 
So, let $C_i^{\ell,\rn_i^\ell-1}=\eC{\bl_i}$ be non-degenerate, for $i=1,2$.
We may assume that, for any $j< \rn_i^\ell-1$, $C_i^{\ell,j}$ is a non-limit cluster.
(For $j>0$, this follows from Lemma~\ref{lem:descr'}, as 
$C_i^{\ell,j}\cap\mset{i}\ne\emptyset$ by Lemma~\ref{l:maxbisblock}~$(a)$.
However, if $C_i^0$ is the root cluster in $\mathfrak F_i$, it can happen that 
$(\{x_i\}\cup\mset{i})\cap C_i^0=\{x_i\}$, $x_i\notin\mset{i}$ and $C_i^0$ is a limit cluster.
We may exclude this case by Lemma~\ref{l:rootnolimit}.)
%
%\nb{$C(x_i)$ limit cluster TBA!}
Also, as $\bl_i$ is definable in $\mathfrak M_i$, $C_i^{\ell,\rn_i^\ell-1}=\eC{\bl_i}$  is a non-limit cluster by Lemma~\ref{int-prop}~$(d)$.
Below, we define sets $A_i^\ell\subseteq\eC{\bl_i}$,
intervals $J_i^{\ell,j}\subseteq I_i^\ell$, and models $\mathfrak N_i^{\ell,j}=(\mathfrak H_i^{\ell,j},\mathfrak w_i^{\ell,j})$ with $\mathfrak H_i^{\ell,j}=(H_i^{\ell,j},\rest{R_i}{H_i^{\ell,j}},\INT_i^{\ell,j})$, for $i=1,2$ and $j<\rn_i^\ell$, such that the following hold:
\begin{align}
\label{nice}
& \mbox{$\mathfrak N_i^{\ell,j}$ is $\bigl(J_i^{\ell,j},i\bigr)$-\nice, for $j<\rn_i^\ell$\textup{;}}\\
\label{simplenumber}
& \mbox{$\mathfrak N_i^{\ell,j}$ is the ordered sum of at most two \simple{} $\delta$-models}\\[-2pt]
\nonumber
& \hspace*{1cm}\mbox{based on \atomic{} frames, for $j<\rn_i^\ell$}\textup{;}\\
\label{parti}
 & \mbox{$\{J_i^{\ell,j}\mid j<\rn_i^\ell\}$ is a partition of $I_i^\ell$ with $J_i^{\ell,0}\intord{\mathfrak F_i}\cdots\intord{\mathfrak F_i}J_i^{\ell,\rn_i^\ell-1}$;}\\
  \label{match1}
 & \mbox{$\displaystyle t_{\mathfrak M_1}^\sigma\Bigl(\bigcup_{j<\rn_1^\ell-1}H_1^{\ell,j}\Bigr)\subseteq
t_{\mathfrak M_2}^\sigma\bigl(A_2^\ell\bigr)$ and $\displaystyle t_{\mathfrak M_2}^\sigma\Bigl(\bigcup_{j<\rn_2^\ell-1}H_2^{\ell,j}\Bigr)\subseteq t_{\mathfrak M_1}^\sigma\bigl(A_1^\ell\bigr)$\textup{;}}\\
  \label{match2}
% & \mbox{\textcolor{blue}{there is a $\sigma$-type preserving bijection between $A_1^\ell$ and $A_2^\ell$\textup{;}}}\\
& t_{\mathfrak M_1}^\sigma\bigl(A_1^\ell\bigr)=t_{\mathfrak M_2}^\sigma\bigl(A_2^\ell\bigr).
\end{align}
Then we  show that \eqref{nice}--\eqref{match2} imply $(a)$--$(c)$ for 
$\mathfrak N_i^\ell=\mathfrak N_i^{\ell,0}\lhd\dots\lhd\mathfrak N_i^{\ell,\rn_i^\ell-1}$
% (and the underlying $ \mathfrak H_i^\ell=\mathfrak H_i^{\ell,0}\lhd\dots\lhd\mathfrak H_i^{\ell,\rn_i^\ell-1}$ based on $ H_i^\ell=\bigcup_{j<\rn_i^\ell}H_i^{\ell,j}$),
 %$\parent_i^\ell=\bigcup_{j<\rn_i^\ell}\parent_i^{\ell,j}$,
 and some $\nn_i^\ell\le 2\rn_i^\ell$.
 In particular, $(c)$ because $(\mathfrak N_1^\ell,\mathfrak N_2^\ell)$  is of type $(b)$ in Definition~\ref{d:match}.

 %These clearly imply that conditions $(a)$, $(b).1$, $(b).2$, $(c)$, $(d).1$, and $(d).2$ of the lemma hold for $\mathfrak N_i=\mathfrak N_i^0\lhd\dots\lhd\mathfrak N_i^{N-1}$, $\mathfrak H_i=\mathfrak H_i^0\lhd\dots\lhd\mathfrak H_i^{N-1}$ and $\parent_i=\bigcup_{\ell<N}\parent_i^\ell$. Finally, we will also show that the remaining `non-local'  conditions $(b).3$, $(d).3$, and $(d).4$ of the lemma also hold for $\mathfrak H_i$ and $\parent_i$.
% 
%We aim to achieve that \eqref{relsin}--\eqref{hhom} will hold for $\mathfrak H_i^\ell=\mathfrak H_i^{\ell,0}\lhd\dots\lhd\mathfrak H_i^{\ell,\rn_i^\ell-1}$  and $\mathfrak N_i^\ell=\mathfrak N_i^{\ell,0}\lhd\dots\lhd\mathfrak N_i^{\ell,\rn_i^\ell-1}$. 

To this end,
we cover first the cases when $j<\rn_i^\ell-1$ and then, separately, the case $j=\rn_i^\ell-1$.
So suppose first that $j<\rn_i^\ell-1$, and let $J_i^{\ell,j}=[D_i^{\ell,j},C_i^{\ell,j}]$, where $D_i^{\ell,0}$ is the root cluster in $\rest{\mathfrak F_i}{I_i^\ell}$ and $D_i^{\ell,j}$ is the immediate successor of the non-limit cluster $C_i^{\ell,j-1}$,  $0<j<\rn_i^\ell-1$.
Observe that
$\bigl (\{x_i\}\cup \mset{i}\cup \sset{i}\bigr)\cap J_i^{\ell,j}\subseteq C_i^{\ell,j}$.
We consider four subcases $(i)$--$(iv)$, depending on the \tail{} $Z_i^{\ell,j}$ of $\rest{\mathfrak F_i}{J_i^{\ell,j}}$.
\begin{itemize}
\item[$(i)$]
%\item[\rm\emph{Case} 1.1] 
$Z_i^{\ell,j}=\emptyset$, so $C_i^{\ell,j}$ is non-degenerate. Let 
$H_i^{\ell,j}=\bigl(\{x_i\}\cup \mset{i}\cup \sset{i}\bigr)\cap C_i^{\ell,j}$ and $\INT_i^{\ell,j}=2^{H_i^{\ell,j}}$. Then $\mathfrak H_i^j$ is isomorphic to $\cluster{k}$, for $k=|H_i^{\ell,j}|$.
%$k=\bigl|\bigl(\{x_i\}\cup \mset{i}\cup \sset{i}\bigr)\cap C_i^{\ell,j}\bigr|\le\kbound$.\nz{$k=|H_i^{\ell,j}|\le\kbound$}
%\textcolor{red}{\eqref{sub}--\eqref{defdef} clearly holds for $\mathfrak H=\mathfrak H_i^{\ell,j}$ and $I=J_i^{\ell,j}$. In particular, 
%we have \eqref{defdef} by Lemma~\ref{l:defpoints}.
We set $\parent_i^{\ell,j}(x)=x$, for $x\in H_i^{\ell,j}$, and
%we clearly have \eqref{relh}--\eqref{hmax}. For $p\in\delta$, we set 
$\mathfrak{w}_i^{\ell,j}(p)=\mathfrak{v}_i(p)\cap H_i^{\ell,j}$, for $p\in\delta$.

\item[$(ii)$]
If $0<|Z_i^{\ell,j}|=m<\omega$ and $Z_i^{\ell,j}=J_i^{\ell,j}$, then
by taking $H_i^{\ell,j}=J_i^{\ell,j}$ and $\INT_i^{\ell,j}=2^{H_i^{\ell,j}}$ we obtain $\mathfrak H_i^{\ell,j}$ isomorphic to $m^<$.
% and meeting \eqref{relsin}--\eqref{inpred}. We have \eqref{defdef} by Lemma~\ref{l:defpoints}.  By taking 
We set
$\parent_i^{\ell,j}(x)=x$, for $x\in H_i^{\ell,j}$,
%we clearly have \eqref{relh}--\eqref{hmax}.
and $\mathfrak{w}_i^{\ell,j}(p)=\mathfrak{v}_i(p)\cap H_i^{\ell,j}$, for $p\in\delta$.

\item[$(iii)$]
If $0<|Z_i^{\ell,j}|=m<\omega$
and $Z_i^{\ell,j}\ne J_i^{\ell,j}$, then setting $H_i^{\ell,j}=\{w_i^{\ell,j}\}\cup Z_i^{\ell,j}$, for any $w_i^{\ell,j}$ in the \source{} of $Z_i^{\ell,j}$, and $\INT_i^{\ell,j}=2^{H_i^{\ell,j}}$ gives $\mathfrak H_i^{\ell,j}$ isomorphic to $\cluster{1}\lhd m^<$.
% and meeting \eqref{relsin}--\eqref{inpred}. We have \eqref{defdef} by Lemma~\ref{l:defpoints}.  By taking 
Let $\parent_i^{\ell,j}(x)=x$, for $x\in H_i^{\ell,j}$\!, 
%we clearly have \eqref{relh}--\eqref{hmax}.
and $\mathfrak{w}_i^{\ell,j}(p)=\mathfrak{v}_i(p)\cap H_i^{\ell,j}$\!, for $p\in\delta$.

\item[$(iv)$]
If $Z_i^{\ell,j}$ is infinite, then let $H_i^{\ell,j}=\{w_i^{\ell,j}\}\cup Z_i^{\ell,j}$, for any $w_i^{\ell,j}$ in the \source{} of $Z_i^{\ell,j}$, and 
$\INT_i^{\ell,j}$ consist of all finite subsets of $H_i^{\ell,j}$ and their complements in $H_i^{\ell,j}$.
By \eqref{facenondeg},
the resulting $\mathfrak H_i^{\ell,j}$ is isomorphic to $\chain{1}{\bullet}\lhd 1^<$.
% and meets \eqref{relsin}--\eqref{inpred}. We have \eqref{defdef} by \eqref{infiniteincl} and Lemma~\ref{l:defpoints}. 
In this case, $C_i^{\ell,j}=\{y_i^{\ell,j}\}$ is a degenerate cluster for some 
$y_i^{\ell,j}\in \{x_i\}\cup \mset{i}\cup \sset{i}$.
We set
$\parent_i^{\ell,j}(y_i^{\ell,j})=y_i^{\ell,j}$ and
 $\parent_i^{\ell,j}(x)=w_i^{\ell,j}$ for all $x\in H_i^{\ell,j}\setminus\{y_i^{\ell,j}\}$.
% we have \eqref{relh}--\eqref{hhom}. As $\mset{i}\cap H_i^{\ell,j}=\{y_i^{\ell,j}\}$, we have \eqref{hmax}.}
%
%and so $\bigl (\{x_i\}\cup \mset{i}\cup \sset{i}\bigr)\cap J_i^{\ell,j}=\{z_i^{\ell,j}\}$.
%
For $p\in\delta$, let
\[
\mathfrak w_i^{\ell,j}(p)=\left\{
\begin{array}{ll}
\bigl(H_i^{\ell,j}\setminus\{y_i^{\ell,j}\}\bigr)\cup\bigl(\mathfrak v_i(p)\cap\{y_i^{\ell,j}\}\bigr),
 & \mbox{ if $w_i^{\ell,j}\in\mathfrak v_i(p)$},\\
 \mathfrak v_i(p)\cap\{y_i^{\ell,j}\},  
 & \mbox{ otherwise}.
\end{array} 
\right.
\]
\end{itemize}
Then it is not hard to check that, in all $(i)$--$(iv)$, we have $\mathfrak w_i^{\ell,j}(p)\in\INT_i^{\ell,j}$, \eqref{simplenumber} for $\mathfrak N_i^{\ell,j}= (\mathfrak H_i^{\ell,j},\mathfrak w_i^{\ell,j})$, and \eqref{sub}--\eqref{hmax} hold for $\mathfrak H=\mathfrak H_i^{\ell,j}$, $\mathfrak N=\mathfrak N_i^{\ell,j}$, $\parent=\parent_i^{\ell,j}$, and $I=J_i^{\ell,j}$. In particular, in $(i)$--$(iii)$, we have \eqref{defdef} by Lemma~\ref{l:defpoints}. In $(iv)$, we also need \eqref{infiniteincl} to obtain \eqref{defdef},
and the fact that $\mset{i}\cap H_i^{\ell,j}=\{y_i^{\ell,j}\}$ to obtain \eqref{hmax}.
 Therefore, we have \eqref{nice} for $j<\rn_i^\ell-1$. 

%********

Now, consider $j=\rn_i^\ell-1$.
%$\mathfrak N_i^{\ell,\rn_i^\ell-1}$. 
First, we let $J_i^{\ell,\rn_i^\ell-1}=[D_i^{\ell,\rn_i^\ell-1},C_i^{\ell,\rn_i^\ell-1}]$, where
$D_i^{\ell,\rn_i^\ell-1}=C_i^{\ell,\rn_i^\ell-1}$ if $\rn_i^\ell=1$ and $D_i^{\ell,\rn_i^\ell-1}$ 
is the immediate successor of the non-limit cluster $C_i^{\ell,\rn_i^\ell-2}$ otherwise.
Then we have \eqref{parti}.
%Recall that $C_i^{\ell,\rn_i^\ell-1}=\eC{\bl_i}$. 
We let $Y_i^\ell=\bigcup_{j<\rn_1^{\ell}-1}Y_i^{\ell,j}$, where
$Y_i^{\ell,j}=H_i^{\ell,j}$ in
cases $(i)$--$(iii)$ above, and $Y_i^{\ell,j}=\{w_i^{\ell,j},y_i^{\ell,j}\}$ in case $(iv)$. 
%Y_i^\ell=\Bigl(\bigl(\{x_i\}\cup \mset{i}\cup \sset{i}\bigr)\cap C_i^{\ell,\rn_i^\ell-1}\Bigr)\cup
%\bigcup_{j<\rn_1^{\ell}-1}Y_i^{\ell,j},
So $Y_i^\ell$ is finite.
%
%\begin{equation}\label{hfixonY}
%\parent_i^{\ell,j}(y)=y,\quad\mbox{for all $y\in Y_i^\ell$.}
%\end{equation}
%
Set $\Theta=\bigl\{ t_{\mathfrak{M}_{1}}^{\sigma}(x) \mid x\in Y_1^\ell\bigr\} 
\cup
\bigl\{ t_{\mathfrak{M}_{2}}^{\sigma}(x) \mid x\in Y_2^\ell\bigr\}$.
%\begin{equation}\label{stypes}
%\Theta=\bigl\{ t_{\mathfrak{M}_{1}}^{\sigma}(x) \mid x\in Y_1^\ell\bigr\} 
%\cup \bigl\{ t_{\mathfrak{M}_{2}}^{\sigma}(x) \mid x\in Y_2^\ell\bigr\}.
%\end{equation}
%
Let $A_i^\ell$ be the smallest set such that
$\bigl(\{x_i\}\cup \mset{i}\cup \sset{i}\bigr)\cap C_i^{\ell,\rn_i^\ell-1}\subseteq A_i^\ell\subseteq C_i^{\ell,\rn_i^\ell-1}$ and $A_i^\ell$ contains
a point $z_t$ with $t^\sigma_{\mathfrak M_i}(z_t)=t$,  for each $t\in \Theta$.
As $Y_i^\ell\subseteq I_i^\ell=\bl_i$, $C_i^{\ell,\rn_i^\ell-1}=\eC{\bl_i}$, and $t^\sigma_{\mathfrak M_1}(\bl_1)=t^\sigma_{\mathfrak M_2}(\bl_2)$, such $A_i^\ell$ exist by Lemma~\ref{int-prop}~$(e)$
%
%{\color{blue} Observe that not only $t^\sigma_{\mathfrak M_1}(A_1^\ell)=t^\sigma_{\mathfrak M_2}(A_2^\ell)$ but,  by Lemma~\ref{l:maxbisblock}~$(b)$, $(d)$, we actually have \eqref{match2}. Then $k:=\bigl|A_1^\ell\bigr|=\bigl|A_2^\ell\bigr| \le 2^{|\delta|}$, by Lemma~\ref{lem:descr0}~$(b)$, and also
%
%\begin{equation}\label{clusterbound}
%k\le |Y_1^\ell|+|Y_2^\ell|+\kbound.
%\end{equation}
%
%By taking $H_i^{\ell,\rn_i^\ell-1}=A_i^\ell$ and $\INT_i^{\ell,\rn_i^\ell-1}=2^{A_i^\ell}$, $i=1,2$, we obtain   $\mathfrak H_1^{\ell,\rn_1^\ell-1}$ and $\mathfrak H_2^{\ell,\rn_2^\ell-1}$ both isomorphic to $\cluster{k}$. (The sets $A_i^\ell$ are used differently in Case III.) Then we have \eqref{match1}.}
%
and clearly satisfy \eqref{match1}.
By Lemma~\ref{l:maxbisblock}~$(c)$ and $(e)$, we have $t^\sigma_{\mathfrak M_1}(A_1^\ell)=t^\sigma_{\mathfrak M_2}(A_2^\ell)$, and so \eqref{match2} holds.
 Let $k_i=\bigl|A_i^\ell\bigr|$. Then $k_i\le 2^{|\delta|}$ by Lemma~\ref{lem:descr0}~$(b)$ and 
\begin{equation}\label{clusterbound}
k_i\le |Y_1^\ell|+|Y_2^\ell|+\kbound.
\end{equation}
By taking $H_i^{\ell,\rn_i^\ell-1}=A_i^\ell$ and $\INT_i^{\ell,\rn_i^\ell-1}=2^{A_i^\ell}$, we obtain  
$\mathfrak H_i^{\ell,\rn_i^\ell-1}$ isomorphic to $\clusterki$.
(The sets $A_i^\ell$ are used differently in Case III.)
For $p\in\delta$, set $\mathfrak{w}_i^{\ell,\rn_i^\ell-1}(p)=\mathfrak{v}_i(p)\cap H_i^{\ell,\rn_i^\ell-1}$ and 
\mbox{$\parent_i^{\ell,\rn_i^\ell-1}(x)=x$} for all $x\in H_i^{\ell,\rn_i^\ell-1}$.
Then we clearly have $\mathfrak w_i^{\ell,\rn_i^\ell-1}(p)\in\INT_i^{\ell,\rn_i^\ell-1}$, 
\eqref{simplenumber} for $\mathfrak N_i^{\ell,\rn_i^\ell-1}= (\mathfrak H_i^{\ell,\rn_i^\ell-1},\mathfrak w_i^{\ell,\rn_i^\ell-1})$, and \eqref{sub}--\eqref{hmax} hold for $\mathfrak H=\mathfrak H_i^{\ell,\rn_i^\ell-1}$, $\mathfrak N=\mathfrak N_i^{\ell,\rn_i^\ell-1}$, $\parent=\parent_i^{\ell,\rn_i^\ell-1}$ and $I=J_i^{\ell,\rn_i^\ell-1}$ (\eqref{defdef} is by Lemma~\ref{l:defpoints}). This gives \eqref{nice} for $j=\rn_i^\ell-1$.

%Finally, an inspection of the definitions shows that \eqref{relsin}--\eqref{hhom} hold for $\mathfrak H_i^\ell=\mathfrak H_i^{\ell,0}\lhd\dots\lhd\mathfrak H_i^{\ell,\rn_i^\ell-1}$ and $\mathfrak N_i^\ell=\mathfrak N_i^{\ell,0}\lhd\dots\lhd\mathfrak N_i^{\ell,\rn_i^\ell-1}$. In particular, \eqref{match} holds because $(\mathfrak N_1^\ell,\mathfrak N_2^\ell)$ meets condition $(b)$ in Definition~\ref{d:match}.

Finally, we claim that $(a)$--$(c)$ hold for 
$\mathfrak N_i^\ell=\mathfrak N_i^{\ell,0}\lhd\dots\lhd\mathfrak N_i^{\ell,\rn_i^\ell-1}$ 
%and $\parent_i^\ell=\bigcup_{j<\rn_i^\ell}\parent_i^{\ell,j}$,
%the $H_i^\ell$, $\mathfrak N_i^\ell$, $\mathfrak H_i^\ell$ and $\parent_i^\ell$ defined in \eqref{elldef}
 and some $\nn_i^\ell$ with $0<\nn_i^\ell\le 2\rn_i^\ell$. Indeed,
 $(b)$ is by the definition of $\lhd$ and \eqref{simplenumber}.
 For $(c)$: The final cluster in $\mathfrak N_i^\ell=$ final cluster in
 $\mathfrak N_i^{\ell,\rn_i^\ell-1}=$ the non-degenerate cluster $A_i^\ell$. 
 So the requirements in Definition~\ref{d:match}~$(b)$ follow from \eqref{match1} and \eqref{match2}.
 %\nb{ok with new \eqref{match2}} 
 %
 For $(a)$: By \eqref{nice}, each $\mathfrak N_i^{\ell,j}$ is $\bigl(J_i^{\ell,j},i\bigr)$-\nice, for $j<\rn_i^\ell$, that is, conditions \eqref{sub}--\eqref{hmax} are  satisfied for $\mathfrak N=\mathfrak N_i^{\ell,j}$,
 $\mathfrak H=\mathfrak H_i^{\ell,j}$,
$I=J_i^{\ell,j}$, and $\parent=\parent_i^{\ell,j}$ (as defined above).
We claim that \eqref{sub}--\eqref{hmax} are  satisfied for $\mathfrak N=\mathfrak N_i^\ell$,
$\mathfrak H=$ the frame $\mathfrak H_i^\ell$ underlying $\mathfrak N_i^\ell$,
$I=I_i^\ell$ and $\parent_i^\ell=\bigcup_{j<\rn_i^\ell}\parent_i^{\ell,j}$.
Indeed, 
\eqref{sub},  \eqref{relsin},  \eqref{defdef}--\eqref{hatfix},
clearly follow from \eqref{parti}, the definition of $\lhd$, and the corresponding properties for 
$\mathfrak N_i^{\ell,j}$, $\mathfrak H_i^{\ell,j}$, $J_i^{\ell,j}$, and $\parent_i^{\ell,j}$, $j<\rn_i^\ell$;
%
%\eqref{sub} is by \eqref{parti} and \eqref{sub} for $j<\rn_i^\ell-1$.
 %
% \eqref{relsin} is by the definition of $\lhd$.
 %
 \eqref{fincluster} follows from \eqref{fincluster} for $\mathfrak H_i^{\ell,\rn_i^\ell-1}$ and $J^{\ell,\rn_i^\ell-1}$ ;
 %$j=\rn_i^\ell-1$.
 and \eqref{rootcluster} follows from \eqref{rootcluster} for $\mathfrak H_i^{\ell,0}$ and $J^{\ell,0}$.
For \eqref{inpred}:
 Suppose $x\in H_i^\ell$, $\{x\}$ is a degenerate non-root cluster in $\mathfrak H_i^\ell$ and $C\subseteq I_i^\ell$ is the immediate predecessor of $\{x\}$ in $\mathfrak F_i$.  Let $j<\rn_i^\ell$ be such that $x\in H_i^{\ell,j}$.
 If $\{x\}$ is the root cluster in $\mathfrak H_i^{\ell,j}$, then $j>0$ and $\{x\}$ is
 the root cluster in $\rest{\mathfrak F_i}{J_i^{\ell,j}}$ by \eqref{rootcluster} for 
 $\mathfrak H_i^{\ell,j}$ and $J_i^{\ell,j}$.
 Thus, the final cluster  $C^-\subseteq H_i^{\ell,j-1}\subseteq H_i^\ell$ of $\mathfrak H_i^{\ell,j-1}$ is a subset of $C$ by \eqref{fincluster} for $\mathfrak H_i^{\ell,j-1}$ and $J_i^{\ell,j-1}$.
 If $\{x\}$ is a non-root cluster in $\mathfrak H_i^{\ell,j}$, then $C\subseteq J_i^{\ell,j}$ and \eqref{inpred}  
 for $\mathfrak H_i^\ell$ and $I_i^\ell$ follows from $H_i^{\ell,j}\subseteq H_i^\ell$ and
 \eqref{inpred} for $\mathfrak H_i^{\ell,j}$ and $J_i^{\ell,j}$.
% 
% \eqref{defdef} is by the definition of $\lhd$ and \eqref{defdef} for $j<\rn_i^\ell$.
 %
%\eqref{relh} is by the definition of $\lhd$ and \eqref{relh} for $j<\rn_i^\ell$.
%
%\eqref{hatfix} is by the definition of $\lhd$ and \eqref{hatfix} for $j<\rn_i^\ell$.
%
For \eqref{hhom}:
Suppose $x,y\in H_i^\ell$, $xR_i y$ and let $j\le j'<\rn_i^\ell$ be such that $x\in H_i^{\ell,j}$ and 
$y\in H_i^{\ell,j'}$. Then  $\parent_i^\ell(x)R_i\parent_i^\ell(y)$ follows by \eqref{hhom} for 
$\parent_i^{\ell,j}$
when $j=j'$,
and by the definition of $\lhd$ when $j<j'$.
For \eqref{hmax}:
Suppose $x,y\in H_i^\ell$, $y\in\mset{i}$, $\parent_i^\ell(x)R_i y$,  and let $j\le j'<\rn_i^\ell$ be such that $x\in H_i^{\ell,j}$ and 
$y\in H_i^{\ell,j'}$. Then  $xR_i y$ follows by \eqref{hmax} for $\parent_i^{\ell,j}$ when $j=j'$, and by the definition of $\lhd$ when $j<j'$.

\textbf{\emph{Case}} III: $(I_1^\ell,I_2^\ell)$ is added in step \step{2}. 
For $i=1,2$, let $\bl_i$ be the relevant $\sigma$-blocks such that 
$t^\sigma_{\mathfrak M_1}(\bl_1)=t^\sigma_{\mathfrak M_2}(\bl_2)$ and
$I_i^\ell$ is extending $\bl_i$ that is not definable in $\mathfrak M_i$.
We use the notation from Case II.
As explained in Case II, we may again assume that,
for every $j< \rn_i^\ell-1$,
%\nb{$C(x_i)$ limit cluster TBA!}
$C_i^{\ell,j}$ is a non-limit cluster. However, as now
$\bl_i$ is not definable in $\mathfrak M_i$, $C_i^{\ell,\rn_i^\ell-1}=\eC{\bl_i}$  is a limit cluster by
Lemma~\ref{int-prop}~$(d)$.
We again define sets $A_i^\ell\subseteq\eC{\bl_i}$,
intervals $J_i^{\ell,j}\subseteq I_i^\ell$, and models $\mathfrak N_i^{\ell,j}=(\mathfrak H_i^{\ell,j},\mathfrak w_i^{\ell,j})$ with $\mathfrak H_i^{\ell,j}=(H_i^{\ell,j},\rest{R_i}{H_i^{\ell,j}},\INT_i^{\ell,j})$ such that 
\eqref{nice}--\eqref{match1} 
and the strengthening \eqref{match2s} of \eqref{match2} (to be defined below)
hold.
%sets $A_i^\ell\subseteq\eC{\bl_i}$,
%$H_i^{\ell,j}\subseteq J_i^{\ell,j}\subseteq I_i^\ell$, frames $\mathfrak H_i^{\ell,j}=(H_i^{\ell,j},\rest{R_i}{H_i^{\ell,j}},\INT_i^{\ell,j})$, models $\mathfrak N_i^{\ell,j}=(\mathfrak H_i^{\ell,j},\mathfrak w_i^{\ell,j})$, and functions $\parent_i^\ell\colon H_i^\ell\to H_i^\ell$, for $i=1,2$, $j<\rn_i^\ell$.
%We aim to achieve that \eqref{parti}--\eqref{match2} hold, and also \eqref{relsin}--\eqref{hhom}.\nz{why separately?}
Then we show that $(a)$--$(c)$ of the lemma hold for 
$\mathfrak N_i^\ell=\mathfrak N_i^{\ell,0}\lhd\dots\lhd\mathfrak N_i^{\ell,\rn_i^\ell-1}$.
%and $\parent_i^\ell=\bigcup_{j<\rn_i^\ell}\parent_i^{\ell,j}$.
%
This time, $(\mathfrak N_1^\ell,\mathfrak N_2^\ell)$ is \match{} because it is of type
$(c)$ in Definition~\ref{d:match}. 

To this end, for any $i=1,2$ and $j<\rn_i^\ell-1$, we define everything like in II.2.
For $j=\rn_i^\ell-1$, we set $J_i^{\ell,\rn_i^\ell-1}=[D_i^\ell,E_i^\ell]$, where 
$D_i^\ell$ is the root cluster in $\rest{\mathfrak F_i}{I_i^\ell}$ if $\rn_i^\ell=1$
%\nb{$C(x_i)$ limit cluster TBA!}
and the immediate successor of the non-limit cluster $C_i^{\ell,\rn_i^\ell-2}$ if $\rn_i^\ell>1$,
and $E_i^\ell$ is the final cluster in $I_i^\ell$. 
We clearly have \eqref{parti} and can define the sets $Y_i^\ell$ and $A_i^\ell$ in  the same way as in II.2. 
However, for property \eqref{defdef} to hold for $\mathfrak H=\mathfrak H_i^{\ell,\rn_i^\ell-1}$, 
we need to define $\mathfrak H_i^{\ell,\rn_i^\ell-1}$ differently.
We consider the two cases in step \step{2} of Definition~\ref{d:ints}:
%(and refer to the notation $f$, $m$ and $m_n$ introduced there):
%
\begin{enumerate}
\item
The \tail{} of $\rest{\mathfrak M_i}{I_i^\ell}$ is $\{\yy_i^n\in I_i^\ell\setminus\bl_i\mid n<\omega\}$ with
$\yy_i^n R_i \yy_i^{n-1}$, $0<n<\omega$.
(Using the notation of Definition~\ref{d:ints}: $\{\yy_i^n\}=\bl_i^{\iso(m+n)}$, $n<\omega$.)
%There exists $0<m<\omega$ such that $\bl_i^{\iso(n)}$, $i=1,2$, is a degenerate $\sigma$-block for every $n$ with $m\le n<\omega$. Suppose $\bl_i^{\iso(m+n)}=\{\yy_i^n\}$, for $n<\omega$.
%Then $\{\yy_i^n\mid n<\omega\}$ is the \tail{} of $\rest{\mathfrak M_i}{I_i^\ell}$. Also, each $\{\yy_i^n\}$ is a degenerate cluster, and so it is definable in $\mathfrak M_i$ by Lemma~\ref{lem:descr'}.

 \item
 There is a sequence of non-degenerate clusters $D_i^n\subseteq I_i^\ell\setminus\bl_i$ definable in $\mathfrak M_i$, $n<\omega$, with $D_i^0$ being the final cluster in $\rest{\mathfrak M_i}{I_i^\ell}$ and $D_i^n<_{R_i} D_i^{n-1}$, $0<n<\omega$.
 (Using the notation of Definition~\ref{d:ints}:  $D_i^n=\eC{\bl_i^{\iso(m_{n+1})}}$.)
 For $n<\omega$, we pick some $\yy_i^n\in D_i^n$.
 %For every $n<\omega$ there is $m_n$, $n\le m_n<\omega$ such that 
%$\bl_i^{\iso(m_n)}$, $i=1,2$, is a non-degenerate $\sigma$-block with its final cluster
%$\eC{\bl_i^{\iso(m_{n+1})}}$ being definable in $\mathfrak M_i$, $i=1,2$, whenever $n\ge 1$.
%For every $n<\omega$, pick some $\yy_i^n$ from  $\eC{\bl_i^{\iso(m_{n+1})}}$.
% Take the sequence $\bl_i^{\iso(m_{n+1})}$, $n<\omega$, of non-degenerate $\sigma$-blocks that are definable in $\mathfrak M_i$, and pick some $\yy_i^n$ from the last cluster $\eC{\bl_i^{\iso(m_{n+1})}}$ of $\bl_i^{\iso(m_{n+1})}$. 
%Then, for each $n<\omega$ and $i=1,2$, $\{\yy_i^n\}$ is a non-degenerate cluster that is definable $\mathfrak M_i$ by Lemma~\ref{l:defpoints}.
\end{enumerate}
In both cases, we set $H_i^{\ell,\rn_i^\ell-1}=A_i^\ell\cup\{\yy_i^m\mid m<\omega\}$.
Then \eqref{match1} continues to hold. Moreover,
observe that now we not only have \eqref{match2}, but the stronger property 
\begin{equation}
\label{match2s}
\mbox{there is a $\sigma$-type preserving bijection between $A_1^\ell$ and $A_2^\ell$.}
\end{equation}
Indeed, as $C_i^{\ell,\rn_i^\ell-1}=\eC{\bl_i}$  is a limit cluster, we have 
$\mset{i}\cap\eC{\bl_i}=\emptyset$ by Lemma~\ref{lem:descr'}~$(b)$, and so
$\bigl(\{x_i\}\cup \mset{i}\cup\sset{i}\bigr)\cap \eC{\bl_i^j}=\sset{i}\cap  \eC{\bl_i^j}$ by
Lemma~\ref{l:maxbisblock}~$(b)$. Thus,
% Lemmas~\ref{l:rootnolimit} and \ref{lem:descr'}~$(b)$, so 
\eqref{match2s} holds by the definition of $A_i^\ell$ and Lemma~\ref{l:maxbisblock}~$(e)$. Now
take the $k<\omega$ with  $|A_1^\ell|=|A_2^\ell|=k$, and
suppose $A_i^\ell=\{a_i^0,\dots,a_i^{k-1}\}$, $i=1,2$.
We let $\INT_i^{\ell,\rn_i^\ell-1}$ be generated in $(H_i^{\ell,\rn_i^\ell-1},\rest{R_i}{H_i^{\ell,\rn_i^\ell-1}})$ by the sets $\{\yy_i^n\}$, $n<\omega$, and $X_i^s$, $s<k$,
where $X_i^s=\{a_i^s\} \cup \{\yy_i^n\mid n < \omega,\ n \equiv s \ (\text{mod}\ k)\}$
(see Example~\ref{k-omega}).
The resulting $\mathfrak H_i^{\ell,\rn_i^\ell-1}$ are both isomorphic to 
$\chain{k}{\bullet}$ in case 1., and to $\chain{k}{\circ}$ in case 2.
%Then we clearly have \eqref{match1}--\eqref{inpred}, and \eqref{defdef} follows from \eqref{infiniteincl} and Lemma~\ref{l:defpoints}. 
%
For $p\in\delta$, we set $\mathfrak w_i^{\ell,\rn_i^\ell-1}(p)=\bigcup_{a_i^s\in\mathfrak v_i(p)}X_i^s$.
For every $x\in H_i^{\ell,\rn_i^\ell-1}$, we set 
\[
\parent_i^{\ell,\rn_i^\ell-1}(x)=\left\{
\begin{array}{ll}
x, & \mbox{if $x=a_i^s$, for $s<k$},\\[3pt]
a_i^s & \mbox{if $x=\yy_i^n$, $n<\omega$ and $n\equiv s$ (mod $k$)}.
\end{array} 
\right.
\]
Then clearly $\mathfrak w_i^{\ell,\rn_i^\ell-1}(p)\in\INT_i^{\ell,\rn_i^\ell-1}$ and \eqref{simplenumber} holds for $\mathfrak N_i^{\ell,\rn_i^\ell-1}= (\mathfrak H_i^{\ell,\rn_i^\ell-1},\mathfrak w_i^{\ell,\rn_i^\ell-1})$. It is not hard to check that \eqref{nice} holds for $\mathfrak N_i^{\ell,\rn_i^\ell-1}$ as well. In particular, \eqref{defdef} for $\mathfrak H=\mathfrak H_i^{\ell,\rn_i^\ell-1}$ follows from \eqref{infiniteincl} and Lemma~\ref{l:defpoints}.
%and we have \eqref{relh}--\eqref{hhom} for $\parent_i^{\ell,\rn_i^\ell-1}$. 
Also, as $A_i^\ell\subseteq\eC{\bl_i}$ and 
$\eC{\bl_i}$ is a limit cluster, $\mset{i}\cap H_i^{\ell,\rn_i^\ell-1}=\mset{i}\cap A_i^\ell=\emptyset$ follows by Lemma~\ref{lem:descr'}, and so
we also have \eqref{hmax} for $H=H_i^{\ell,\rn_i^\ell-1}$ and $\parent=\parent_i^{\ell,\rn_i^\ell-1}$.
%Now an inspection of the definitions shows that \eqref{relsin}--\eqref{hhom} hold for $\mathfrak H_i^\ell=\mathfrak H_i^{\ell,0}\lhd\dots\lhd\mathfrak H_i^{\ell,\rn_i^\ell-1}$ and $\mathfrak N_i^\ell=\mathfrak N_i^{\ell,0}\lhd\dots\lhd\mathfrak N_i^{\ell,\rn_i^\ell-1}$. In particular, \eqref{match} holds because $(\mathfrak N_1^\ell,\mathfrak N_2^\ell)$ meets condition $(c)$ in Definition~\ref{d:match}.\nb{TBA}

Next, the arguments showing that $(a)$ and $(b)$ of the lemma hold for the models
$\mathfrak N_i^\ell=\mathfrak N_i^{\ell,0}\lhd\dots\lhd\mathfrak N_i^{\ell,\rn_i^\ell-1}$
%$\parent_i^\ell=\bigcup_{j<\rn_i^\ell}\parent_i^{\ell,j}$,
%the $H_i^\ell$, $\mathfrak N_i^\ell$, $\mathfrak H_i^\ell$ and $\parent_i^\ell$ defined in \eqref{elldef}
 and some $\nn_i^\ell$ with $0<\nn_i^\ell\le 2\rn_i^\ell$
 are the same as in Case II.2.
 %apart from condition $(c).2$.
 To establish $(c)$, we show that the pair $(\mathfrak N_1^\ell,\mathfrak N_2^\ell)$ is of type $(c)$ in Definition~\ref{d:match}.
 Indeed, observe that the last $\lhd$-components of $\mathfrak N_i^\ell$ are $\mathfrak N_i^{\ell,\rn_i^\ell-1}$ whose underlying frames $\mathfrak H_i^{\ell,\rn_i^\ell-1}$ are both isomorphic to the same \atomic{} frame of the form $\chain{k}{\bullet}$ or $\chain{k}{\circ}$,
with $0<k\le 2^{|\delta|}$. Also, the $\cluster{k}$-cluster in $\mathfrak H_i^{\ell,\rn_i^\ell-1}$ is
 $A_i^\ell$, and so the requirements in Definition~\ref{d:match}~$(c)$ follow from \eqref{match1} and \eqref{match2s}.
 
 Finally,
 observe that $\nn_i^\ell=1$ if $(I_1^\ell,I_2^\ell)$ is added in step \step{3} of Definition~\ref{d:ints} (see Case I), and $\nn_i^\ell\le 2\rn_i^\ell$ if $(I_1^\ell,I_2^\ell)$ is added in steps \step{1} or \step{2} (see II and III).
 So  $\sum_{\ell<N}\nn_i^\ell\le(\kbound-1) + \sum_{\ell<N}2\rn_i^\ell\le  3\kbound-1$, as required.
%We complete the proof of Lemma~\ref{l:replacement} by showing that the 
% `non-local'  conditions $(b).3$, $(d).3$, and $(d).4$ of the lemma also hold for $\mathfrak H_i$ and $\parent_i$. ...
 \end{proof}

%***** final lemma

We now complete the proof of Theorem~\ref{t:structmodel}. 
In Definition~\ref{d:ints}, we partitioned the models $\mathfrak M_1,x_1$ and $\mathfrak M_2,x_2$ witnessing the lack of interpolants for $\varphi_1$, $\varphi_2$  into the same 
polynomial number $N$ of intervals. 
% for some $N$ with $0<N\le 2\kbound -1$. 
For each $\ell<N$, 
Lemma~\ref{l:replacement} gave us a pair of models $(\mathfrak N_1^\ell,\mathfrak N_2^\ell)$.
Let $\mathfrak N_i=\mathfrak N_i^0\lhd\dots\lhd\mathfrak N_i^{\partN-1}$, for $i=1,2$.
%The next lemma completes the proof of Theorem
%
\begin{lemma}\label{l:final} 
Conditions $(a)$--$(d)$ in Theorem~\ref{t:structmodel}
 hold for $\mathfrak N_1,x_1$ and $\mathfrak N_2,x_2$.
\end{lemma}

%Finally we show that Lemma~\ref{l:replacement} implies the conditions in Theorem~\ref{t:structmodel}:
 %
%\begin{lemma}
%$(a)$ \plan{3}, \plan{6} and \plan{7} imply condition $(a)$ from Definition~\ref{d:bmp}\textup{;}
%$(b)$ \plan{4} implies condition $(b)$\textup{;}
%$(c)$ \plan{6} and \plan{7} imply $(c)$, and 
%$(d)$ $\plan{5} = (d)$.  %\textup{;}
 %\end{lemma}
%
\begin{proof}
%we write `item $(z)$' for `item $(z)$ of Lemma~\ref{l:replacement}.
We use the notation of the proof of Lemma~\ref{l:replacement}.
By Lemma~\ref{l:replacement}~$(a)$, each $\mathfrak N_i^\ell$ is $(I_i^\ell,i)$-\nice, that is, conditions \eqref{sub}--\eqref{hmax} are  satisfied for $\mathfrak N=\mathfrak N_i^\ell$, $\mathfrak H=\mathfrak H_i^\ell$, $I=I_i^\ell$, and $\parent=\parent_i^\ell$.

$(a)$ 
We show by induction that $\mathfrak M_i,\parent_i^\ell(x)\models\tau$ iff $\mathfrak N_i,x\models\tau$, for any $i=1,2$, $\ell<\partN$, $\tau\in\sub(\varphi_i)$, and \mbox{$x\in H_i^\ell$}. Then 
$\mathfrak{N}_{1},x_{1}\models\varphi_1$ and $\mathfrak{N}_{2},x_{2}\models\neg \varphi_2$ follow
from $\mathfrak{M}_{1},x_{1}\models\varphi_1$ and $\mathfrak{M}_{2},x_{2}\models\neg \varphi_2$,
as we have $x_i\in H_i^0$ and $\parent_i^0(x_i)=x_i$ by \eqref{relsin} and \eqref{relh}.
%items (a) and $(d)$.1.
%
For $\tau=p \in\delta$, the statement follows from \eqref{hatfix}.
%item $(d)$.2.
The Boolean cases are straightforward, so suppose $\tau=\Diamond\psi$.

$(\Rightarrow)$
If $\mathfrak M_i,\parent_i^\ell(x)\models\Diamond\psi$, then there are  $k\geq \ell$ and $y_\psi\in \mset{i}\cap I_i^k$ with 
$\parent_i^\ell(x)R_i y_\psi$ and $\mathfrak M_i,y_\psi\models\psi$. 
We have $y_\psi\in H_i^k$ by \eqref{relsin},
%item $(a)$, 
and so $\mathfrak N_i,y_\psi\models\psi$ by \eqref{relh}
%item $(d)$.1 
and IH.
We claim that $xR_i y_\psi$, and so $\mathfrak N_i,x\models\Diamond\psi$. Indeed, for $k>\ell$, this follows 
from the definition of $\lhd$, and for $\ell=k$, by \eqref{hmax}.
%item $(d)$.4.

$(\Leftarrow)$
If $\mathfrak N_i,x\models\Diamond\psi$, then there are $k\geq \ell$ and $y\in H_i^k$ with 
$xR_i y$ and $\mathfrak N_i,y\models\psi$.
We have $\mathfrak M_i,\parent_i^k(y)\models\psi$ by IH, and
$\parent_i^\ell(x)R_i\parent_i^k(y)$ by the definition of $\lhd$ when $k>\ell$, and by \eqref{hhom}
%item $(d)$.3 
when $k=\ell$. Hence  
$\mathfrak M_i,\parent_i(x)\models\Diamond\psi$. 

$(c)$ follows from \eqref{relsin}, \eqref{relh}, \eqref{hatfix}
%items $(a)$, $(d)$.1, $(d)$.2, 
and $\mathfrak M_1,x_1\sim_\sigma\mathfrak M_2,x_2$.

$(d)$ It is shown in Definition~\ref{d:ints} that $0<N< 2\kbound$. The rest of $(d)$ follows from
Lemma~\ref{l:replacement}~$(b)$,$(c)$.
%items $(c)$.1 and $(c)$.2. 

$(b)$ We use the refutation criteria for the canonical formulas to show that the frame $\mathfrak H_i$ underlying $\mathfrak N_i$ is a frame for $L$, $i=1,2$.
To this end, we prove that, 
\begin{align}\label{samef}
& \mbox{for any canonical formula $\alpha(\mathfrak G, \mathfrak D,\bot)$,
if $f$ is an injection}\\[-2pt]
\nonumber
& \hspace*{1cm}\mbox{from $\mathfrak G$ to $\mathfrak H_i$ satisfying \textbf{(cf$_1$)}--\textbf{(cf$_4$)},
then the same $f$}\\[-2pt]
\nonumber
& \hspace*{1cm}\mbox{is an injection from $\mathfrak G$ to $\mathfrak F_i$ satisfying \textbf{(cf$_1$)}--\textbf{(cf$_4$)}.}
\end{align}
Indeed, 
\textbf{(cf$_1$)} holds by \eqref{sub}
%$(b)$.1 
and the definition of $\lhd$; 
\textbf{(cf$_2$)} holds, as
the final cluster in $\mathfrak H_i=$ final cluster in $\mathfrak H_i^{N-1}\subseteq$
final cluster in $\rest{\mathfrak F_i}{I_i^{N-1}}=$ final cluster in $\mathfrak M_i$, by \eqref{fincluster}.
% item $(b)$.2.
%
Condition \textbf{(cf$_4$)} holds by \eqref{defdef}
%$(b)$.5 
and the definition of $\lhd$.
For \textbf{(cf$_3$)}, suppose $x \in \mathfrak D$, $C(y)$ is the immediate predecessor of $C(x) = \{x\}$ in $\mathfrak G$ and $C(f(y))$ is the immediate predecessor of $C(f(x))=\{f(x)\}$ in $\mathfrak H_i$.
Let $\ell<N$ be such that $x\in H_i^\ell$.
 If $\{x\}$ is the root cluster in $\mathfrak H_i^\ell$, then $\ell>0$ and $\{x\}$ is
 the root cluster in $\rest{\mathfrak F_i}{I_i^\ell}$ by \eqref{rootcluster}.
 %$(b)$.3. 
 Thus, 
 $C(f(y))$ in $\mathfrak H_i=$ final cluster in $\mathfrak H_i^{\ell-1}\subseteq $ final cluster in
 $\rest{\mathfrak F_i}{I_i^{\ell-1}}=C(f(y))$ in $\mathfrak F_i$, by \eqref{fincluster}.
 %item $(b)$.2.
 If $\{x\}$ is a non-root cluster in $\mathfrak H_i^\ell$, then 
 $C(f(y))$ in $\mathfrak H_i=C(f(y))$ in $\mathfrak H_i^\ell\subseteq C(f(y))$ in $\rest{\mathfrak F_i}{I_i^\ell}=C(f(y))$ in $\mathfrak F_i$, by \eqref{inpred}.
% item $(b)$.4.
%
Now, \eqref{samef} implies 
that $L\subseteq\Log(\mathfrak F_i)\subseteq\Log(\mathfrak H_i)$,  as required.
%prove that $\Log(\mathfrak F_i) \subseteq \Log(\mathfrak H_i)$ by showing that for any canonical formula ...
 %By items $(b)$ and Lemmas~\ref{pmorphism}, \ref{l:logs}, $\mathfrak N_i$, $i=1,2$, is based on a frame for $L$.
\end{proof}

%****************

\subsection{Proofs of Theorems~\ref{t:finaxstruct} and \ref{t:finax}.}\label{ss:finax}
%
%We now refine the proof of Lemma~\ref{l:replacement} to obtain a proof of Theorem~\ref{t:finaxstruct} (and so of its consequence Theorems~\ref{t:finax}, saying that every finitely axiomatisable logic $L \supseteq \KFT$ has the \qpbmp).
%
Suppose the finitely axiomatisable logic $L$
%$L=\KFT \oplus \gamma_L$, for some formula $\gamma_L$.
is given by its canonical axioms as
$L = \KFT \oplus \{\alpha(\mathfrak G_j,\mathfrak D_j,\bot) \mid j\in J_L\}$, for some finite index set $J_L$ and $\mathfrak G_j=(V_j,S_j)$, $j\in J_L$.
Let $\cbound=\max_{j\in J_L}|V_j|$.
%
%By Theorem~\ref{t:smallmodel}, $L$ has the  \qfbmp. 
Given formulas $\varphi_1$, $\varphi_2$ without an interpolant in $L$,  
let $0<N<2\kbound$ and $\mathfrak N_i=\mathfrak N_i^0\lhd\dots\lhd\mathfrak N_i^{N-1}$ with root $x_i$ $i=1,2$,
be the models 
%with underlying frames $\mathfrak H_i$ 
satisfying the conditions of Theorem~\ref{t:structmodel} and obtained via Lemma~\ref{l:replacement}.
%\ref{l:irrelevant} and \ref{l:relevant}.
In particular, the underlying frame $\mathfrak H_i$ of each $\mathfrak N_i$ is a frame for $L$.
We show in Lemma~\ref{l:replacementL} below that the proof of Lemma~\ref{l:replacement}
%\ref{l:irrelevant} and \ref{l:relevant} 
can be refined to yield polynomial-size models $\mathfrak N_i^{\starr\ell}$,  $\ell<N$.
However, $\mathfrak N_i^{\starr\ell}$ is no longer $(I_i^\ell,i)$-\nice, as 
conditions \eqref{rootcluster} and \eqref{inpred} in Definition~\ref{d:nice} do not necessarily hold
for $\mathfrak H = \mathfrak H_i^{\starr\ell}$ 
underlying $\mathfrak N_i^{\starr\ell}$ and $I=I_i^\ell$. Thus,  
we do not have \eqref{samef} in the proof of Lemma~\ref{l:final}
for the frames $\mathfrak H_i^\starr$ underlying $\mathfrak N_i^\starr=\mathfrak N_i^{\starr 0}\lhd\dots\lhd\mathfrak N_i^{\starr N-1}$.
We prove that $\mathfrak H_i^\starr$, $i=1,2$, are frames for $L$ (as required by 
Theorem~\ref{t:finaxstruct}~$(b)$) by using Lemma~\ref{l:ftof} below instead.

Take the number $N$, $0<N<2\kbound$, 
%partitions $\mathcal I_i=\{I_i^\ell\in\INT_i\mid \ell<\partN\}$ of $\mathfrak M_i$, $i=1,2$, 
provided by Definition~\ref{d:ints}, the numbers $\nn_i^\ell>0$ with $\sum_{\ell<N}\nn_i^\ell\le3\kbound-1$, and sets $H_i^\ell$, $i=1,2$, $\ell<N$,
from Lemma~\ref{l:replacement}.
%Recall that, for $i=1,2$, $\ell<N$, $\rn_i^\ell$ denotes the number of relevant clusters in $I_i^\ell$.
%
\begin{lemma}\label{l:replacementL}
If $L\supseteq\KFT$ is finitely axiomatisable, then,
for $i=1,2$, $\ell<N$, there exist  
%$\nn_i^\ell>0$ with $\sum_{\ell<N}\nn_i^\ell\le3\kbound-1$,
sets $H_i^{\starr\ell}\subseteq H_i^\ell$ and
models $\mathfrak N_i^{\starr\ell}$ based on frames $\mathfrak H_i^{\starr\ell}=(H_i^{\starr\ell},S_i^{\starr\ell},\INT_i^{\starr\ell})$ 
%and functions $\parent_i^{\starr\ell}\colon H_i^{\starr\ell}\to H_i^{\starr\ell}$ 
%such that conditions $(a)$, $(b)$.1, $(b)$.2, $(b)$.5, $(c)$.2 and $(d)$ of Lemma~\ref{l:replacement} hold, and $\mathfrak N_i^{\starr\ell}$ is the ordered sum of $\nn_i^\ell$-many \simple{} $\delta$-models based on \underline{\bounded} \atomic{} frames.
%
such that the following hold\textup{:}
 \begin{itemize}
 \item[$(a)$]
$\mathfrak N_i^{\starr\ell}$ is `almost' $(I_i^\ell,i)$-\nice{} in the sense that \eqref{sub}--\eqref{fincluster}, \eqref{defdef}--\eqref{hmax} hold
for $\mathfrak N=\mathfrak N_i^{\starr\ell}$ and $I=I_i^\ell$\textup{;}

\item[$(b)$]
$\mathfrak N_i^{\starr\ell}$ is the ordered sum of $\nn_i^\ell$-many \simple{} $\delta$-models based on \underline{\bounded} \atomic{} frames\textup{;}

\item[$(c)$]
the pair $(\mathfrak N_1^{\starr\ell},\mathfrak N_2^{\starr\ell})$ is \match.
\end{itemize}
\end{lemma}
\begin{proof}
We go through Cases I--III in the proof of Lemma~\ref{l:replacement} and make the necessary
modifications.

\textbf{\emph{Case}} I: $(I_1^\ell,I_2^\ell)$ is added in step \step{3} of Definition~\ref{d:ints}. 
An inspection of this part of the proof of Lemma~\ref{l:replacement}
reveals that $m^<$ or $\cluster{1}\lhd m^<$ is used in cases $(i)$ and $(ii)$, and in both cases
\emph{all} the $m$ elements of the finite non-empty \tails{} $Z_i^\ell$ of 
$\rest{\mathfrak F_i}{I_i^\ell}$ are put into the chosen subset $H_i^\ell$ of $I_i^\ell$.
Now, we choose a subset $H_i^{\starr\ell}\subseteq H_i^\ell$ with $|H_i^{\starr\ell}|\le\cbound+2$ as follows.
Suppose $Z_i^\ell=\{z_i^a\mid a< m\}$ with $z_i^aR_i^s z_i^{a-1}$,  $0<a<m$,
and let $m'=\min(m,\cbound+1)$. We set $H_i^{\starr\ell}=\{z_i^a\mid a< m'\}$ in case $(i)$,
and $H_i^{\starr\ell}=\{w_i^\ell\}\cup\{z_i^a\mid a< m'\}$, for the chosen $w_i^\ell$ from the
\source{} of $Z_i^\ell$ in case $(ii)$.
In case $(iii)$, we let $H_i^{\starr\ell}=H_i^\ell$.
Then, in all cases $(i)$--$(iii)$, 
we let $\mathfrak H_i^{\starr\ell}=\rest{\mathfrak H_i^\ell}{H_i^{\starr\ell}}$ and
$\mathfrak N_i^{\starr\ell}=\rest{\mathfrak N_i^\ell}{H_i^{\starr\ell}}$.
Observe that we have $\parent_i^\ell(x)\in H_i^{\starr\ell}$, for every $x\in H_i^{\starr\ell}$, and so 
$\parent_i^{\starr\ell}=\rest{\parent_i^\ell}{H_i^{\starr\ell}}$ is a $H_i^{\starr\ell}\to H_i^{\starr\ell}$ function.
It is straightforward to check that \eqref{sub}--\eqref{fincluster}, \eqref{defdef}--\eqref{hmax} hold
for $\mathfrak N=\mathfrak N_i^{\starr\ell}$, $\mathfrak H=\mathfrak H_i^{\starr\ell}$, $\parent=\parent_i^{\starr\ell}$, and $I=I_i^\ell$.
%$\mathfrak H_i^{\starr\ell}$, $\mathfrak N_i^{\starr\ell}$, and $\parent_i^{\starr\ell}$ meet conditions $(a)$, $(b)$.1, $(b)$.2, $(b)$.5, $(c)$.2 and $(d)$ of Lemma~\ref{l:replacement}. 
Note that all non-degenerate clusters in $\mathfrak H_i^{\starr\ell}$ are of the form $\cluster{1}$ this case, and so  $\mathfrak N_i^{\starr\ell}$ is a \simple{} $\delta$-model based on an \bounded{} \atomic{} frame.

\textbf{\emph{Case}} II and III: 
$(I_1^\ell,I_2^\ell)$ is added in steps \step{1} or \step{2}. 
An inspection of these parts of the proof of Lemma~\ref{l:replacement}
reveals that $m^<$ or $\cluster{1}\lhd m^<$ is used only when $\bl_i$ is non-degenerate,
in cases $(ii)$ and $(iii)$ of the definition of
$\mathfrak H_i^{\ell,j}$ for some $j<\rn_i^\ell-1$.
(Recall that $\rn_i^\ell$ denotes the number of relevant clusters in $I_i^\ell$.)
In both cases $(ii)$ and $(iii)$,
\emph{all} the $m$ elements of the finite non-empty \tail{} $Z_i^{\ell,j}$ of 
$\rest{\mathfrak F_i}{J_i^{\ell,j}}$ are put into the chosen subset $H_i^{\ell,j}$ of $J_i^{\ell,j}$,
for some subinterval $J_i^{\ell,j}$ of $I_i^\ell$.
We repeat the trick from Case I above. 
Suppose $Z_i^{\ell,j}=\{z^a\mid a< m\}$ with $z^aR_i^s z^{a-1}$,  $0<a<m$,
and let $m'=\min(m,\cbound+1)$. We set $H_i^{\starr\ell,j}=\{z^a\mid a< m'\}$ in case $(ii)$,
and $H_i^{\starr\ell,j}=\{w_i^{\ell,j}\}\cup\{z^a\mid a< m'\}$, for the chosen $w_i^{\ell,j}$ from the
\source{} of $Z_i^{\ell,j}$ in case $(iii)$.
In cases $(i)$ and $(iv)$ of II and III, we let $H_i^{\starr\ell,j}=H_i^{\ell,j}$.
Then, in all cases $(i)$--$(iv)$, 
we let $\mathfrak H_i^{\starr\ell,j}=\rest{\mathfrak H_i^{\ell,j}}{H_i^{\starr\ell,j}}$ and
$\mathfrak N_i^{\starr\ell,j}=\rest{\mathfrak N_i^{\ell,j}}{H_i^{\starr\ell,j}}$, for $j<\rn_i^\ell-1$.
One can see that $\parent_i^{\ell,j}(x)\in H_i^{\starr\ell,j}$, for every $x\in H_i^{\starr\ell,j}$, and so 
$\parent_i^{\starr\ell,j}=\rest{\parent_i^{\ell,j}}{H_i^{\starr\ell,j}}$ is an $H_i^{\starr\ell,j}\to H_i^{\starr\ell,j}$ function.
It is straightforward to check that, for all $j<\rn_i^\ell-1$, \eqref{sub}--\eqref{fincluster}, \eqref{defdef}--\eqref{hmax} hold
for $\mathfrak N=\mathfrak N_i^{\starr\ell,j}$, $\mathfrak H=\mathfrak H_i^{\starr\ell,j}$, $\parent=\parent_i^{\starr\ell,j}$, and $I=J_i^{\ell,j}$ (but \eqref{rootcluster} and \eqref{inpred} do not necessarily hold). 
%$\mathfrak H_i^{\starr\ell,j}$, $\mathfrak N_i^{\starr\ell,j}$, and $\parent_i^{\starr\ell,j}$ satisfy \eqref{fincluster} and \eqref{defdef}--\eqref{hmax}
%(but not necessarily \eqref{rootcluster} and \eqref{inpred}).
Note that the size of non-degenerate clusters in these $\mathfrak H_i^{\starr\ell,j}$ is bounded by $\kbound$,
and so $\mathfrak N_i^{\starr\ell,j}$ is the ordered sum of at most two \simple{} $\delta$-models based on \bounded{} \atomic{} frames.

We also need to adjust the definitions of $\mathfrak H_i^{\ell,\rn_i^\ell-1}$, $\mathfrak N_i^{\ell,\rn_i^\ell-1}$ and $\parent_i^{\ell,\rn_i^\ell-1}$. 
We define the sets $Y_i^{\starr\ell}\subseteq Y_i^\ell$ and $A_i^{\starr\ell}\subseteq A_i^\ell$ from $H_i^{\starr\ell,j}$, $j<\rn_i^\ell-1$, in 
the same way as $Y_i^\ell$ and $A_i^\ell$ were defined from $H_i^{\ell,j}$, $j<\rn_i^\ell-1$,
%
%{\color{blue} resulting in \eqref{match2} and \eqref{match1}.  Let $k^\starr=\bigl|A_1^{\starr\ell}\bigr|=\bigl|A_2^{\starr\ell}\bigr|$.  By \eqref{clusterbound},
%
%\begin{multline*}
%k^\starr\le |Y_1^{\starr\ell|}+|Y_2^{\starr\ell}|+\kbound\le\\
%2(\kbound-1)\cdot\max\bigl(\cbound+2,\kbound\bigr)+\kbound=\pbound.
%\end{multline*}
%
%In Case II,}
%
resulting in \eqref{match1} and \eqref{match2} in Case II, and in \eqref{match1} and \eqref{match2s} in Case III. In Case II, for $i=1,2$, we let $k_i^\starr=\bigl|A_i^{\starr\ell}\bigr|$. 
By \eqref{clusterbound},
\begin{multline*}
k_i^\starr\le |Y_1^{\starr\ell|}+|Y_2^{\starr\ell}|+\kbound\le\\
2(\kbound-1)\cdot\max\bigl(\cbound+2,\kbound\bigr)+\kbound=\pbound.
\end{multline*}
We set $\mathfrak H_i^{\starr\ell,\rn_i^\ell-1}=\rest{\mathfrak H_i^{\ell,\rn_i-1}}{H_i^{\starr\ell,\rn_i^\ell-1}}$, $\mathfrak N_i^{\starr\ell,\rn_i^\ell-1}=\rest{\mathfrak N_i^{\ell,\rn_i-1}}{H_i^{\starr\ell,\rn_i^\ell-1}}$, and $\parent_i^{\starr\ell,\rn_i^\ell-1}=\rest{\parent_i^{\ell,\rn_i-1}}{H_i^{\starr\ell,\rn_i^\ell-1}}$.
In Case III, the definitions of $\mathfrak H_i^{\ell,\rn_i^\ell-1}$, $\mathfrak N_i^{\ell,\rn_i^\ell-1}$, $\parent_i^{\ell,\rn_i^\ell-1}$
need to be mimicked for 
%\textcolor{blue}{$k^\starr$}
$k^\starr=\bigl|A_1^{\starr\ell}\bigr|=\bigl|A_2^{\starr\ell}\bigr|$
 in place of $k$ to obtain
$\mathfrak H_i^{\starr\ell,\rn_i^\ell-1}$, $\mathfrak N_i^{\starr\ell,\rn_i^\ell-1}$, $\parent_i^{\starr\ell,\rn_i^\ell-1}$.
It is straightforward to check now that \eqref{sub}--\eqref{fincluster}, \eqref{defdef}--\eqref{hmax} hold
for $\mathfrak N=\mathfrak N_i^{\starr\ell,\rn_i^\ell-1}$, $\mathfrak H=\mathfrak H_i^{\starr\ell,\rn_i^\ell-1}$, $\parent=\parent_i^{\starr\ell,\rn_i^\ell-1}$, and $I=J_i^{\ell,\rn_i^\ell-1}$. 
%satisfy \eqref{fincluster}--\eqref{hmax}.
Note that the size $k^\starr$ of the root cluster in $\mathfrak H_i^{\starr\ell,\rn_i^\ell-1}$ is bounded 
by $\pbound$ and every other non-degenerate cluster in it is of the form $\cluster{1}$,
so $\mathfrak N_i^{\starr\ell,\rn_i^\ell-1}$ is a \simple{} $\delta$-model based on an \bounded{} \atomic{} frame. 

Therefore, $\mathfrak N_i^{\starr\ell}=\mathfrak N_i^{\starr\ell,0}\lhd\dots\lhd\mathfrak N_i^{\starr\ell,\rn_i^\ell-1}$, for $i=1,2$, $\ell<N$, 
is the ordered sum of $\nn_i^\ell$-many \simple{} $\delta$-models based on \bounded{} \atomic{} frames, for the same $\nn_i^\ell$ as in Lemma~\ref{l:replacement}, and so we have $(b)$ of the lemma.
Finally, by the same arguments as in the proof of Lemma~\ref{l:replacement}, we obtain $(a)$ and $(c)$.
% for $\mathfrak N_i^{\starr\ell}$, $i=1,2$.}
%it follows that  $\mathfrak N_i^{\starr\ell}$ and $\parent_i^{\starr\ell}=\bigcup_{j<\rn_i^\ell}\parent_i^{\starr\ell,j}$ meet conditions $(a)$, $(b)$.1, $(b)$.2, $(b)$.5, $(c)$.2 and $(d)$ of Lemma~\ref{l:replacement}, as required.
\end{proof}

\begin{lemma}\label{l:ftof}
For $i=1,2$, $\ell<N$, take the frames $\mathfrak H_i^{\ell}$ and $\mathfrak H_i^{\starr\ell}$ provided by Lemmas~\ref{l:replacement} and  \ref{l:replacementL}. Let 
$\mathfrak H_i=\mathfrak H_i^{0}\lhd\dots\lhd\mathfrak H_i^{\partN-1}$ and
\mbox{$\mathfrak H_i^\starr=\mathfrak H_i^{\starr 0}\lhd\dots\lhd\mathfrak H_i^{\starr\partN-1}$}.
Then, for any $j\in J_L$, if there is an injection $f$ from $\mathfrak G_j$ to $\mathfrak H_i^\starr$ satisfying
{\rm \textbf{(cf$_1$)}--\textbf{(cf$_4$)}} for $\alpha(\mathfrak G_j,\mathfrak D_j,\bot)$, then there is an injection $f^\dag$ from $\mathfrak G_j$ to $\mathfrak H_i$ also 
satisfying {\rm \textbf{(cf$_1$)}--\textbf{(cf$_4$)}} for $\alpha(\mathfrak G_j,\mathfrak D_j,\bot)$. Thus, $\mathfrak H_i\models L$ implies $\mathfrak H_i^\starr\models L$.
\end{lemma}

\begin{proof}
Fix some $i \in \{1,2\}$ and $j\in J_L$. Suppose that $f$ is an injection from $\mathfrak G_j=(V_j,S_j)$ to $\mathfrak H_i^\starr$ satisfying \textbf{(cf$_1$)}--\textbf{(cf$_4$)}  for $\alpha(\mathfrak G_j,\mathfrak D_j,\bot)$.
For every \atomic{} $\lhd$-component $\mathfrak F^\starr=(H^\starr,\rest{R_i}{H^\starr})$ in $\mathfrak H_i^\starr$ such that 
\begin{enumerate}
\item
$\mathfrak F^\starr$ is obtained from the \atomic{}  $\lhd$-component $\mathfrak F=(H,\rest{R_i}{H})$ in $\mathfrak H_i$ of the form $m^<$ or $\cluster{1}\lhd m^<$, and

\item
there is $v\in V_j$ such that $f(v_j)$ is an irreflexive point in $\mathfrak F^\starr$,
\end{enumerate}
we proceed as follows.
Suppose $H=\{y_0,\dots,y_{m-1}\}$ or $H=\{y,y_0,\dots,y_{m-1}\}$ with
$yR_i y R_i y_{m-1}R_i\dots R_i y_0$, and so 
$H^\starr=\{y_0,\dots,y_{\cbound}\}$ or $H^\starr=\{y,y_0,\dots,y_{\cbound}\}$.
Let $V_j^-=\bigl\{v\in V_j\mid f(v)\in H^\starr\setminus\{y\}\bigr\}$. Then $|V_j^-|\le |V_j|\le\cbound$.
Thus, by the pigeonhole principle, there is $i\le\cbound$ with $y_i\notin f(V_i)\cap  (H^\starr\setminus\{y\})$. Suppose $V_j^-=\{v_0,\dots,v_{s-1}\}$, for some $s\le\cbound$  with
$v_{s-1}S_j\dots S_j v_0$.
Let $a$ be the largest $k< s$ with $y_iR_i f(v_k)$. As $f$ satisfies \textbf{(cf$_3$)},
$v_a\notin\mathfrak D_j$. Now, for $k<s$, we set
\[
f^\dag(v_k)=\left\{
\begin{array}{ll}
y_k, & \mbox{ if $k\le a$},\\[3pt]
y_{m-(s-k)}, & \mbox{ if $a+1\le k<s$.}
\end{array} 
\right.
\]
\begin{center}
\begin{tikzpicture}[>=latex,line width=0.5pt,xscale = 1,yscale = .6]
\node[point,scale = 0.7,label=below:{\footnotesize $y$}] (h0) at (0,0) {};
\node[point,fill,scale = 0.6,label=below:{\footnotesize $y_{m-1}$}] (h1) at (1,0) {};
\node[point,fill,scale = 0.6] (h2) at (2,0) {};
\node[point,fill,scale = 0.6] (h3) at (3,0) {};
\node[point,fill,scale = 0.6] (h4) at (4,0) {};
\node[point,fill,scale = 0.6] (h5) at (5,0) {};
\node[point,fill,scale = 0.6] (h6) at (6,0) {};
\node[point,fill,scale = 0.6] (h7) at (7,0) {};
\node[point,fill,scale = 0.6] (h8) at (8,0) {};
\node[point,fill,scale = 0.6] (h9) at (9,0) {};
\node[point,fill,scale = 0.6,label=below:{\footnotesize $y_0$},label=right:{$\ \ H$}] (h10) at (10,0) {};
\node[point,fill,scale = 0.6] (m0) at (2.5,1.5) {};
\node[point,fill,scale = 0.6] (m1) at (4,1.5) {};
\node[point,fill,scale = 0.6,label=below:{\footnotesize $v_a\ \ $}] (m2) at (5.5,1.5) {};
\node[point,fill,scale = 0.6] (m3) at (7,1.5) {};
\node[point,fill,scale = 0.6,label=right:{$\ \ V_j^-$}] (m4) at (8.5,1.5) {};
\node[point,scale = 0.7,label=above:{\footnotesize $y$}] (hh0) at (2,3) {};
\node[point,fill,scale = 0.6,label=above:{\footnotesize $y_{\cbound}$}] (hh3) at (3,3) {};
\node[point,fill,scale = 0.6] (hh4) at (4,3) {};
\node[point,fill,scale = 0.6,label=above:{\footnotesize $y_i$}] (hh5) at (5,3) {};
\node[point,fill,scale = 0.6] (hh6) at (6,3) {};
\node[point,fill,scale = 0.6] (hh7) at (7,3) {};
\node[point,fill,scale = 0.6] (hh8) at (8,3) {};
\node[point,fill,scale = 0.6] (hh9) at (9,3) {};
\node[point,fill,scale = 0.6,label=above:{\footnotesize $y_0$},label=right:{$\ \ H^\starr$}] (hh10) at (10,3) {};
\draw[->,shorten >=1pt,shorten <=2pt] (h0) to (h1);
\draw[->,shorten >=1pt,shorten <=2pt] (h1) to (h2);
\draw[->,shorten >=1pt,shorten <=2pt] (h2) to (h3);
\draw[->,shorten >=1pt,shorten <=2pt] (h3) to (h4);
\draw[->,shorten >=1pt,shorten <=2pt] (h4) to (h5);
\draw[->,shorten >=1pt,shorten <=2pt] (h5) to (h6);
\draw[->,shorten >=1pt,shorten <=2pt] (h6) to (h7);
\draw[->,shorten >=1pt,shorten <=2pt] (h7) to (h8);
\draw[->,shorten >=1pt,shorten <=2pt] (h8) to (h9);
\draw[->,shorten >=1pt,shorten <=2pt] (h9) to (h10);
\draw[->,shorten >=1pt,shorten <=2pt] (m0) to (m1);
\draw[->,shorten >=1pt,shorten <=2pt] (m1) to (m2);
\draw[->,shorten >=1pt,shorten <=2pt] (m2) to (m3);
\draw[->,shorten >=1pt,shorten <=2pt] (m3) to (m4);
\draw[->,shorten >=1pt,shorten <=2pt] (hh0) to (hh3);
\draw[->,shorten >=1pt,shorten <=2pt] (hh3) to (hh4);
\draw[->,shorten >=1pt,shorten <=2pt] (hh4) to (hh5);
\draw[->,shorten >=1pt,shorten <=2pt] (hh5) to (hh6);
\draw[->,shorten >=1pt,shorten <=2pt] (hh6) to (hh7);
\draw[->,shorten >=1pt,shorten <=2pt] (hh7) to (hh8);
\draw[->,shorten >=1pt,shorten <=2pt] (hh8) to (hh9);
\draw[->,shorten >=1pt,shorten <=2pt] (hh9) to (hh10);
\draw[dashed,gray,->,shorten >=1pt,shorten <=2pt] (m4) to (hh9);
\draw[dashed,gray,->,shorten >=1pt,shorten <=2pt] (m4) to (h10);
\draw[dashed,gray,->,shorten >=1pt,shorten <=2pt] (m3) to (hh7);
\draw[dashed,gray,->,shorten >=1pt,shorten <=2pt] (m3) to (h9);
\draw[dashed,gray,->,shorten >=1pt,shorten <=2pt] (m2) to (hh6);
\draw[dashed,gray,->,shorten >=1pt,shorten <=2pt] (m2) to (h8);
\draw[dashed,gray,->,shorten >=1pt,shorten <=2pt] (m1) to (hh4);
\draw[dashed,gray,->,shorten >=1pt,shorten <=2pt] (m1) to (h2);
\draw[dashed,gray,->,shorten >=1pt,shorten <=2pt] (m0) to (hh3);
\draw[dashed,gray,->,shorten >=1pt,shorten <=2pt] (m0) to (h1);
\node[gray]  at (2.2,2.2) {$f$};
\node[gray]  at (1.2,.8) {$f^\dag$};
\end{tikzpicture}
\end{center}
We do this for every $\mathfrak F^\starr$ having 1.\ and 2.\ above, and set $f^\dag(x)=f(x)$, for any
other $x\in V_j$. It is straightforward to check that the resulting $f^\dag$ is an injection from $\mathfrak G_j$ to $\mathfrak H_i$ satisfying {\rm \textbf{(cf$_1$)}--\textbf{(cf$_4$)}} for $\alpha(\mathfrak G_j,\mathfrak D_j,\bot)$. 
\end{proof}
%However, we no longer have items $(b)$ of Lemmas~\ref{l:irrelevant} and \ref{l:relevant}. 
% \ref{l:relevantblock} and \ref{l:extendedblock} (and so \plan{4}). 
%
%That the resulting smaller frames $\mathfrak G_i'=\mathfrak{G}_{i}^1\lhd \cbound^{<} \lhd \mathfrak{G}_{i}^2$ are frames for $L$ 
%(as required by Definition~\ref{d:bmp} $(b)$) follows from 
%$\mathfrak G_i \models L$ and  $\mathfrak G_i \models \gamma_L$ and Lemma~\ref{lem:pump}.
%he observation below: 

%\textcolor{red}{But is this enough? Don't we need something stronger like this:
%
%\begin{lemma}\label{lem:pumpp}
%Assume that we have $\mathfrak M,x$ such that $\mathfrak M,x\models\varphi$, the underlying frame $\mathfrak H$ of $\mathfrak M$ is a frame for $\gamma$ and it is of the form $\mathfrak{H}_{1}\lhd m^{<} \lhd \mathfrak{H}_{2}$ (with $\mathfrak{H}_{i}$ possibly empty) for some $m<\omega$, and $\mathfrak M,x\sim_\sigma\mathfrak N,y$ for some $\mathfrak N,y,\sigma$. (We may assume that $\mathfrak H$ is \basic{} and $\mathfrak M$ is \simple{} if that help.)
%
%Then there exist $n\le 2^{|\gamma|}+1$ and $\mathfrak M'$ based on $\mathfrak H'=\mathfrak{H}_{1}\lhd n^{<} \lhd \mathfrak{H}_{2}$ such that 
%$\mathfrak M',x'\models\varphi$, $\mathfrak H'$ is a frame for $\gamma$, and 
%$\mathfrak M',x'\sim_\sigma\mathfrak N,y$ for some $x'$.
%\end{lemma}}

This completes the proof of Theorem~\ref{t:finaxstruct}. We obtain Theorem~\ref{t:finax}
using Lemma~\ref{l:gbisall} as we have, for $i=1,2$:
\begin{multline*}
\size{\mathfrak N_i^\starr}=\size{\mathfrak N_i^{\starr 0}}+\cdots +\size{\mathfrak N_i^{\starr N-1}}
\le\sum_{\ell<N}\nn_i^\ell\cdot\max\bigl(\cbound+2,\pbound\bigr)\le\\
(3\kbound-1)\cdot\max\bigl(\cbound+2,\pbound\bigr).
%\size{\mathfrak N_i}\le  (3\kbound-1)\cdot\max\bigl(\cbound+2,\pbound\bigr),
\end{multline*}

%****************

\subsection{Cofinal subframe logics}\label{ss:cofinalsfr}

By Theorem~\ref{dperscofinal} $(a)$, all d-persistent cofinal subframe logics $L \supseteq \KFT$ have the polysize bisimilar model property, with the polynomial $\kbound$ (defined in \eqref{kbound})
not dependent on $L$. We show now that, for arbitrary, not necessarily d-persistent cofinal subframe $L$, it is enough to replace polysize in Theorem~\ref{dperscofinal} $(a)$ by quasi-polysize:

\begin{theorem}\label{t:cofinalsfr}
All cofinal subframe logics $L \supseteq\KFT$ have the \qpbmp, with the size of witnessing models bounded by $\kbound$.
\end{theorem}

This follows from the following special case of the `structural' Theorem~\ref{t:structmodel}:

\begin{theorem}\label{cof-str}
For any cofinal subframe logic $L \supseteq \KFT$ and formulas $\varphi_1$, $\varphi_2$ without an interpolant in $L$,  
there are rooted $\delta$-models $\mathfrak N_1,x_1$ and $\mathfrak N_2,x_2$ 
satisfying $(a)$--$(c)$ from Theorem~\ref{t:structmodel} as well as conditions $(d)$ and $(e)$ below\textup{:}
%, for $\delta = \sig(\varphi_1) \cup \sig(\varphi_2)$ and $\sigma = \sig(\varphi_1) \cap \sig(\varphi_2)$
%
\begin{itemize}
%\item[$(a)$] 
%$\mathfrak N_1,x_1\models\varphi_1$ and $\mathfrak N_2,x_2\models\neg\varphi_2$\textup{;}
%
%\item[$(b)$] 
%each $\mathfrak N_i$, $i=1,2$, is based on a frame for $L$\textup{;}
%
%\item[$(c)$] 
%$\at^\sigma_{\mathfrak N_1}(x_1) = \at^\sigma_{\mathfrak N_2}(x_2)$\textup{;}
%
\item[$(d)$]
there is $M \leq\kbound$ such that 
$\mathfrak N_i=\mathfrak N_i^0\lhd\dots\lhd\mathfrak N_i^{M-1}$ and, \mbox{for all $j<M$,}
\begin{enumerate}
\item
$\mathfrak N_i^j$ is the ordered sum of \simple{} $\delta$-models based on \atomic{} frames\textup{;}
\item
the pair $(\mathfrak N_1^j,\mathfrak N_2^j)$ is \match\textup{;}
\end{enumerate}

\item[$(e)$]
$\{x_i\}\cup\mset{i}\cup\sset{i}$ coincides with  
the set of points in $\mathfrak N_i$, $i=1,2$ that are not in the $\{\yy_i^n\mid n<\omega\}$-part of some
$\lhd$-component based on a $\chain{k}{\ast}$.
\end{itemize}
It follows from $(e)$ that $\size{\mathfrak N_i}=|\{x_i\}\cup\mset{i}\cup\sset{i}|\le\kbound$.
\end{theorem}
\begin{proof}
As in the proof of Theorem~\ref{t:structmodel}, we take any $\sigma$-bisimilar witness models $\mathfrak M_i,x_i$, $i=1,2$, based on frames $\mathfrak{F}_i = (W_{i},R_{i},\INT_{i})$ for $L$.
Let $M$ be the number of relevant $\sigma$-blocks in $\mathfrak M_1$ (or $\mathfrak M_2$, by Lemma~\ref{l:maxbisblock}~$(f)$). For $i=1,2$,
consider the partitions $\mathcal I_i=\{I_i^\ell\in\INT_i\mid \ell<\partN\}$ of $\mathfrak M_i$ given by Definition~\ref{d:ints}, and let $0=\ell_0<\dots <\ell_{M-1}=N-1$ be the list of indices such that the pair $(I_1^{\ell_j},I_2^{\ell_j})$ is added to $\mathcal I_1\times\mathcal I_2$ in 
steps \step{1} or \step{2},
and $I_i^{\ell_0}\intord{\mathfrak F_i}\cdots\intord{\mathfrak F_i}I_i^{\ell_{M-1}}$. 
 We define $\mathfrak N_i^{\ell_z}$, $z<M$, by choosing \emph{fewer} points from $I_i^{\ell_z}$
than in Cases II.2 and III in the proof of Lemma~\ref{l:replacement},
and we also define functions $\parent_i^{\ell_z}$.
Let $\ell=\ell_z$, for $z<M$, let $C_i^{\ell,j}$, $j < \rn_i^\ell$, be the sequence (ordered by $<_{R_i}$) of all relevant clusters in $I_i^\ell$, and $D_i^{\ell,j}=C_i^{\ell,j}\cap\bigl(\{x_i\}\cup\mset{i}\cup\sset{i}\bigr)$. Three cases are possible now, the first of which coincides with Case II.1, while the other two select fewer points for  $\mathfrak N_i^{\ell_z}$ than Cases II.2 and III:
\begin{itemize}
\item[$(i)$]
If  $(I_1^{\ell},I_2^{\ell})$ is added in step \step{1} and $I_i^{\ell}$ consists of a dege\-nerate cluster, then, like in Case II.1, we let $\mathfrak N_i^\ell=\rest{\mathfrak M_i}{I_i^\ell}$
and $\parent_i^\ell$ be the identity \mbox{function on $\mathfrak N_i^\ell$.}

\item[$(ii)$]
If $C_i^{\ell,\rn_i^\ell-1}$ is non-degenerate and $(I_1^{\ell},I_2^{\ell})$ is added in step \step{1} as in Case II.2, then $C_i^{\ell,\rn_i^\ell-1}$ is definable in $\mathfrak M_i$. 
We let 
$\mathfrak N_i^\ell=\rest{\mathfrak M_i}{D_i^{\ell,0}}\lhd\dots\lhd \rest{\mathfrak M_i}{D_i^{\ell,\rn_i^\ell-1}}$ and 
$\parent_i^\ell$ be the identity function on $\mathfrak N_i^\ell$.

\item[$(iii)$]
If $(I_1^{\ell},I_2^{\ell})$ is added in \step{2} like in Case III,  
then $C_i^{\ell,\rn_i^\ell-1}$ is a not definable in $\mathfrak M_i$. 
As shown in Case III, there is an infinite sequence of irrelevant points 
$\{\yy_i^n\in I_i^\ell\mid n<\omega\}$ such that  $\yy_i^n R_i \yy_i^{n-1}$, 
$C_i^{\ell,\rn_i^\ell-1}<_{R_i} C(\yy_n^i)$ and $C(\yy_n^i)\in\INT_i$, $n<\omega$, and  
the $\yy_i^n$ are either 1)  all irreflexive 
or 2) all reflexive.
%
%\textcolor{blue}{By Lemma~\ref{l:maxbisblock},}
%
Also, as $C_i^{\ell,\rn_i^\ell-1}$  is a limit cluster by Lemma~\ref{int-prop}~$(d)$, we have $D_i^{\ell,\rn_i^\ell-1}=C_i^{\ell,\rn_i^\ell-1}\cap\sset{i}$ by
Lemmas~\ref{lem:descr'}~$(b)$ and \ref{l:maxbisblock}~$(b)$.
%Lemmas~\ref{l:rootnolimit} and \ref{lem:descr'}~$(b)$.
By Lemma~\ref{l:maxbisblock}~$(e)$, there is a $\sigma$-type preserving bijection between $D_1^{\ell,\rn_1^\ell-1}$ and $D_2^{\ell,\rn_2^\ell-1}$, and so
there is $k\le\kbound$ with $|D_1^{\ell,\rn_1^\ell-1}|=|D_2^{\ell,\rn_2^\ell-1}|=k$.
Suppose $D_i^{\ell,\rn_i^\ell-1}=\{a_i^0,\dots,a_i^{k-1}\}$.
We let \mbox{$H_i^{\ell,\rn_i^\ell-1}=D_i^\ell\cup\{\yy_i^n\mid n<\omega\}$} and 
$\INT_i^{\ell,\rn_i^\ell-1}$ be generated in $(H_i^{\ell,\rn_i^\ell-1},\rest{R_i}{H_i^{\ell,\rn_i^\ell-1}})$ by the sets $\{\yy_i^n\}$, $n<\omega$, and 
$X_i^s=\{a_i^s\} \cup \{\yy_i^n\mid n < \omega,\ n \equiv s \ (\text{mod}\ k)\}$, $s<k$ 
(see Example~\ref{k-omega}).
The resulting $\mathfrak H_i^{\ell,\rn_i^\ell-1}$ are both isomorphic to 
$\chain{k}{\bullet}$ in case 1), and to $\chain{k}{\circ}$ in case 2).
We then set $\mathfrak w_i^{\ell,\rn_i^\ell-1}(p)=\bigcup_{a_i^s\in\mathfrak v_i(p)}X_i^s$ 
and also
$
\mathfrak N_i^{\ell,\rn_i^\ell-1}=\bigl((H_i^{\ell,\rn_i^\ell-1},\rest{R_i}{H_i^{\ell,\rn_i^\ell-1}},
\INT^{\ell,\rn_i^\ell-1}),\mathfrak w_i^{\ell,\rn_i^\ell-1}\bigr).
$
Finally, we set
\[
\mbox{$
\mathfrak N_i^\ell=\rest{\mathfrak M_i}{D_i^{\ell,0}}\lhd\dots\lhd \rest{\mathfrak M_i}{D_i^{\ell,\rn_i^\ell-2}}\mathop{\lhd}\mathfrak N_i^{\ell,\rn_i^\ell-1}$}
\]
 and
define $\parent_i^\ell$ as the identity on relevant points in $\mathfrak N_i^\ell$ and
$\parent_i^\ell(\yy_i^n)=a_i^s$, for $n<\omega$ with $n\equiv s$ (mod $k$).
\end{itemize}
Clearly, $(c)$, $(d).1$ and $(e)$ hold for 
$\mathfrak N_i=\mathfrak N_i^{\ell_0}\lhd\dots\lhd\mathfrak N_i^{\ell_{M-1}}$. 
Condition $(a)$ is shown like in Lemma~\ref{l:final} using that 
\eqref{relh}--\eqref{hmax} hold for $\parent=\parent_i^{\ell_z}$ and  
$H=H_i^{\ell_z}$, $z<M$. 
%the $\parent_i^{\ell_z}$ satisfy $(d)$.1--4 in Lemma~\ref{l:replacement}.
Condition $(b)$ is proved via \eqref{samef}: 
\textbf{(cf$_1$)} clearly holds; \textbf{(cf$_2$)} holds as the final cluster in $\mathfrak F_i$ is always relevant; and \textbf{(cf$_4$)} holds, as  $\{x\}$ being definable in $\mathfrak N_i^{\ell_z}$  implies $\{x\}\in\INT_i$,
for all $z<M$ and $x$ in $\mathfrak N_i^{\ell_z}$.
As $L$ is a cofinal subframe logic, $\mathfrak D=\emptyset$, so \textbf{(cf$_3$)} holds vacuously.
Finally, to show $(d).2$, observe that $(\mathfrak N_1^{\ell_z},\mathfrak N_2^{\ell_z})$ is \match{} as it always meets one of the conditions in Definition~\ref{d:match}:
in case $(i)$, it meets $(a)$;
in case $(ii)$, it meets $(b)$; and
in case $(iii)$, it meets $(c)$.
\end{proof}

\begin{example}\label{ex:GL.3full}\em 
By Example~\ref{e:canform}~$(a)$,
given any formulas $\varphi_1$ and $\varphi_2$ without an interpolant in $\GLT$, one can always find witnessing models $\mathfrak N_i$, $i=1,2$, of size $\le\kbound$ that are ordered sums of \simple{} models based on $m^<$ or $\chain{k}{\bullet}$ (see e.g.\ the models depicted in Fig.~\ref{GL3gf} in Example~\ref{ex:GL.3}~$(a)$).
\hfill $\dashv$
\end{example}

We emphasise that the construction in the proof of Theorem~\ref{cof-str} does not work for non-cofinal subframe logics, in which case $\mathfrak D \ne \emptyset$; see also the special treatment of the density axiom in the proof of Theorem~\ref{temp-selection} below.

%**********************************************************************************************************************

\section{The IEP for standard Priorean temporal logics}\label{sec:temporal}
\emph{Priorean temporal logics}~\cite{Prior1968-PRIPOT-2} deal
with the operators `sometime in the future' denoted $\Df$, `sometime in the past' denoted $\Dp$, and their duals `always in the future' $\Bf$ and `always in the past' $\Bp$. \emph{Temporal formulas}---propositional bimodal formulas with these operators---are interpreted over general \emph{temporal  frames} of the form  $\mathfrak F = (W,R,R^-,\INT)$ representing various \emph{flows of time} in such a way that $(W,R)$ is transitive and connected~\eqref{connected},
 %\emph{connected}, i.e.,
%
%\begin{equation*}
%\forall x,y \in W \, \big(xRy \lor x=y \lor xR^-y \big),
%\end{equation*}
%
$R$ is the `future-time' accessibility relation for $\Df$, $\Bf$, its inverse $R^-$ is the `past-time'  accessibility relation for $\Dp$, $\Bp$, and the internal sets $\INT \subseteq 2^W$ are closed under the Booleans and the operators 
\begin{equation*}
\Df^{\mathfrak F} X = \{ x \in W \mid \exists y\in X \, x R y  \}, \quad 
\Dp^{\mathfrak F} X = \{ x \in W \mid \exists y\in X \, x R^- y  \}.
\end{equation*} 
To simplify notation, we omit $R^-$ and write $\mathfrak F = (W,R,\INT)$. Also, as before, if $\INT = 2^W$, we call $\mathfrak F$ a \emph{Kripke frame} and write $\mathfrak F = (W,R)$. 
%
%`sometime' $\Diamond$ and 
The \emph{universal modality} `always'  can be introduced as an abbreviation
%
%\begin{equation*}
%\Diamond \varphi = \varphi \lor \Df \varphi \lor \Dp \varphi,\qquad
$\Box \varphi = \varphi \land \Bf \varphi \land \Bp \varphi$. 
%\end{equation*}
%
Descriptive temporal frames are defined in the same way as in \S\ref{prelims}. Note that tightness condition \textbf{(\tight)} for $R^-$ actually follows from \textbf{(\tight)} for $R$.

In fact, many results from \S\ref{prelims}, \ref{sec:warming} straightforwardly generalise to the temporal setting. 
%In particular, we require the following extension of Lemma~\ref{maxpoints}. 
%
Let $\mathfrak{M}$ be a \emph{temporal model\/}---that is, a model based on some temporal frame $\mathfrak{F} = (W,R,\INT)$---and let $\Gamma$ be a set of temporal formulas. A point $x\in W$ is $\Gamma$-\emph{minimal in} $\mathfrak{M}$ if $\mathfrak{M},x \models \Gamma$ and whenever $x'R x$ and $\mathfrak{M},x'\models \Gamma$, then $x R x'$. Denote by $\min_{\mathfrak{M}} \Gamma$ the set of all $\Gamma$-minimal points in $\mathfrak{M}$. (The definition of $\max_{\mathfrak{M}} \Gamma$ remains the same.) 
In the temporal case, Lemma~\ref{maxpoints} generalises to

\begin{lemma}\label{minpoints}
Suppose $\Gamma$ is a set of temporal formulas and $\mathfrak M$ a model based on a 
descriptive temporal frame $\mathfrak{F}= (W,R,\INT)$.
Then the following hold\textup{:}

\medskip
\noindent
{\rm\bf (temporal saturation)}
If $\mathfrak M, x\models \Df \bigwedge\Gamma'$ for every finite $\Gamma' \subseteq \Gamma$, then there is $y$ with $xRy$ and $\mathfrak M, y \models \Gamma$.
If $\mathfrak M, x\models \Dp \bigwedge\Gamma'$ for every finite $\Gamma' \subseteq \Gamma$, then there is $y$ with $xR^-y$ and $\mathfrak M, y \models \Gamma$.

\medskip
\noindent
{\rm\bf (maximal and minimal points)}
If there is $x$ with $\mathfrak M, x \models \Gamma$, then $\max_{\mathfrak{M}}\Gamma \ne \emptyset$ and $\min_{\mathfrak{M}}\Gamma \ne \emptyset$.
\end{lemma}

%\begin{lemma}\label{minpoints}\nz{make Lemma 5$'$?}
%Suppose $\mathfrak{M} =(\mathfrak{F},\mathfrak{v})$ is a model based on a descriptive temporal frame $\mathfrak{F}= (W,R,\INT)$ and $\Gamma$ a set of formulas. Then the following hold\textup{:}
%
%\begin{itemize}
%\item[$(a)$] {\bf (temporal saturation)}
%If $\mathfrak M, x\models \Df \bigwedge\Gamma'$ for every finite $\Gamma' \subseteq \Gamma$, then there is $y$ with $xRy$ and $\mathfrak M, y \models \Gamma$. 
%
%If $\mathfrak M, x\models \Dp \bigwedge\Gamma'$ for every finite $\Gamma' \subseteq \Gamma$, then there is $y$ with $xR^-y$ and $\mathfrak M, y \models \Gamma$.
%
%\item[$(b)$] {\bf (maximal and minimal points)}
%If $\mathfrak M, x \models \Gamma$, for some $x \in W$, then $\max_{\mathfrak{M}}\Gamma \ne \emptyset$ and $\min_{\mathfrak{M}}\Gamma \ne \emptyset$.
%\end{itemize}
%\end{lemma}
%The other notions in Section~\ref{sec:warming} can also be adjusted to the temporal setting.

A relation $\bis \subseteq W_1 \times W_2$ is a \emph{temporal $\sigma$-bisimulation} between temporal models $\mathfrak M_1$ and $\mathfrak M_2$ based on respective frames $\mathfrak F_i = (W_i,R_i,\INT_i)$, $i=1,2$, if it satisfies {\bf (\atom)}, {\bf (\move)} and its past-time couterpart: 
whenever $x_1\bis x_2$, then  

%\item[(\move$^+$)] $x_1 R_1 y_1$ implies $(y_1,y_2)\in \bis$, for some  $y_2 \in W_2$ with $x_2 R_2 y_2$, and the other way round;
%

\medskip
\noindent
{\bf (\move$^-$)} $x_1 R^-_1 y_1$ implies $y_1\bis y_2$, for some $y_2 \in W_2$ with $x_2 R^-_2 y_2$; conversely,\\
\mbox{}\hspace*{0.7cm}$x_2 R^-_2 y_2$ implies $y_1\bis y_2$, for some $y_1 \in W_1$ with $x_1 R^-_1 y_1$.

\medskip
\noindent
The relation $\mathfrak{M}_1,x_1 \equiv_{\sigma} \mathfrak{M}_2,x_2$, saying that temporal models $\mathfrak{M}_1$ and $\mathfrak{M}_2$ satisfy the same temporal $\sigma$-formulas at $x_1$ and $x_2$, respectively, is characterised in terms of temporal $\sigma$-bisimulations: it is readily seen that, with this modification, Lemma~\ref{bisim-lemma} and Theorem~\ref{criterion} continue to hold for all Priorean temporal logics.
(As temporal frames are transitive and 
connected, any of their points can be regarded as a root with respect to the relation $R\cup R^-$.)

%(Due\nz{I'd drop this} to connectedness, all temporal frames are rooted with respect to $R\cup R^-$.)}

%The relation $\mathfrak{M}_1,x_1 \equiv_{\sigma} \mathfrak{M}_2,x_2$, saying that temporal models $\mathfrak{M}_1$ and $\mathfrak{M}_2$ satisfy the same $\sigma$-formulas at $x_1$ and $x_2$, respectively, can again be characterised in terms of bisimulations, $\bis$, which are defined by adding the following condition to the definition in Section~\ref{sec:warming}: whenever $(x_1,x_2)\in \bis$, then  
%
%\begin{description}
%\item[(\move$^-$)] $x_1 R^-_1 y_1$ implies $(y_1,y_2)\in \bis$, for some $y_2 \in W_2$ with $x_2 R^-_2 y_2$, and the other way round.
%\end{description}
%
%It is readily seen that, with this modification, Lemma~\ref{bisim-lemma} and Theorem~\ref{criterion}\nz{root TBD!} continue to hold for all Priorean temporal logics. 

In this article, we consider the Priorean temporal logics of five most popular classes of temporal Kripke frames~\cite{DBLP:journals/jsyml/Bull68}:
\begin{itemize}
\item[] $\Lin = \{ \varphi \mid \mathfrak F \models \varphi, \ \mathfrak F = (W,R) \text{ is any temporal Kripke frame} \}$\\
\hspace*{0.34cm} ${} = \KF_2 \oplus p \to \Bf\Dp p \oplus p \to \Bp\Df p \oplus \Df\Dp p \lor \Dp\Df p \to p \lor \Df p \lor \Dp p$;

\item[] $\LinQ = \{ \varphi \mid (\mathbb Q,<) \models \varphi\}$\\ 
\hspace*{0.55cm} ${}= \Lin \oplus \Df \top \oplus \Dp \top \oplus \Df p \to \Df\Df p$;

\item[] $\LinR = \{ \varphi \mid (\mathbb R,<) \models \varphi\}$\\
\hspace*{0.55cm} ${}= \LinQ \oplus \Box(\Bp p \to \Df\Bp p) \to (\Bp p \to \Bf p)$;

\item[] $\Linf = \{ \varphi \mid \mathfrak{F} \models \varphi, \  \mathfrak{F} = (W,<)\text{ any \emph{finite} strict linear order} \}$\\
\hspace*{0.78cm} ${} = \Lin \oplus \bigl\{ \Bt (\Bt p \to p) \to \Bt p \mid \mathsf{X} \in \{\mathsf{F}, \mathsf{P}\}\bigr\}$;

\item[] $\LinZ = \{ \varphi \mid (\mathbb Z,<) \models \varphi\}$\\ 
\hspace*{0.53cm} ${} = \Lin \oplus \Df \top \oplus \Dp \top \oplus \bigl\{ \Bt(\Bt p \to p) \to (\Dt\Bt p \to \Bt p) \mid \mathsf{X} \in \{\mathsf{F}, \mathsf{P}\}\bigr\}$,
\end{itemize}
where $\KF_2$ is the bimodal version of $\KF$ (with $\Df$ and $\Dp$). 
None of these five logics (and any other temporal logic with frames of  unbounded depth) has the CIP~\cite{MGabbay2005-MGAIAD,DBLP:books/daglib/0030819}, and our aim in this section is to prove that the IEP for each of them is decidable in \coNP. The following example illustrates the new semantic phenomena of temporal logics compared to modal logics containing $\KFT$ that we need to address in order to achieve this aim.

\begin{example}\label{ex:GL.3-temporal}\em
$(a)$ Consider the formulas $\varphi_1$ and $\varphi_2$ from Example~\ref{ex:GL.3}~$(a)$ in the context of $\Linf$ in place of $\GLT$, reading $\Diamond$ as $\Df$ and $\Box$ as $\Bf$:
\begin{align*}
& \varphi_1 = \Df( p_{1} \wedge  \Df^+ \neg q_{1}) \wedge \Bf (p_{2} \rightarrow \Bf^+ q_{1}) \wedge  \Bf(p_{1} \rightarrow \neg p_{2}),\\ 
& \varphi_2 = \neg [ \Df( p_{2} \wedge \Df^+\neg q_{2}) \land \Bf (p_{1} \rightarrow \Bf^+ q_{2}) ].
\end{align*}
We clearly have $(\varphi_1 \rightarrow \varphi_2) \in \Linf$. Using Theorem~\ref{criterion}, we show that $\varphi_1$ and $\varphi_2$ have no interpolant in $\Linf$. The argument from Example~\ref{ex:GL.3}~$(a)$ shows that any models $\mathfrak M_i$ meeting the criterion of Theorem~\ref{criterion} cannot be based on a Kripke frame for $\Linf$. 
%
%On the other hand, 
However, the descriptive frame $\bullet \lhd \bullet \lhd \mathfrak C(\clustert, \bullet)$ we employed for $\GLT$ in Example~\ref{ex:GL.3}~$(a)$ does not help now, because it refutes $\Bp (\Bp p \to p) \to \Bp p$ at any point save the first two under the valuation below:
%
%\centerline{
%\includegraphics[scale=0.6]{../Pics/GL0-}}\\
%
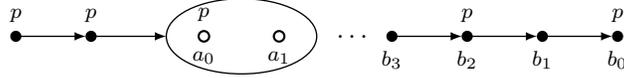
\begin{center}
\begin{tikzpicture}[>=latex,line width=0.5pt,xscale = 1,yscale = 1]
\node[point,fill=black,scale = 0.6,label=above:{\footnotesize $p$}] (x2) at (0,0) {};
\node[point,fill=black,scale = 0.6,label=above:{\footnotesize $p$}] (y2) at (1,0) {};
\node[point,scale = 0.7,label=above:{\footnotesize $p$},label=below:{\footnotesize $a_0$}] (a20) at (2.5,0) {};
\node[point,scale = 0.7,label=below:{\footnotesize $a_1$}] (a21) at (3.5,0) {};
\draw[] (3,0) ellipse (1 and .5);
\node[]  at (4.5,0) {$\dots$};
\node[point,fill=black,scale = 0.6,label=below:{\footnotesize $\yy_3$}] (b23) at (5,0) {};
\node[point,fill=black,scale = 0.6,label=above:{\footnotesize $p$},label=below:{\footnotesize $\yy_2$}] (b22) at (6,0) {};
\node[point,fill=black,scale = 0.6,label=below:{\footnotesize $\yy_1$}] (b21) at (7,0) {};
\node[point,fill=black,scale = 0.6,label=above:{\footnotesize $p$},label=below:{\footnotesize $\yy_0$}] (b20) at (8,0) {};
\draw[->] (x2) to (y2);
\draw[->] (y2) to (2,0);
\draw[->] (b23) to (b22);
\draw[->] (b22) to (b21);
\draw[->] (b21) to (b20);
\end{tikzpicture}
\end{center}
To fix this issue, we modify $\mathfrak C(\clustert, \bullet)$ by making it symmetric in both directions. Consider the frame $\mathfrak F_k = (W_k',R_{\bullet k\bullet },\INT_k')$, $k > 0$, in which the points in 
$$
W_k' = \{a_0,\dots, a_{k-1} \} \cup \{\yy_n^{L},\yy_n^{R} \mid n < \omega \}
$$
are ordered as shown in the picture below
%
%\centerline{
%\includegraphics[scale=0.6]{../Pics/GL1-}}
%\\
%
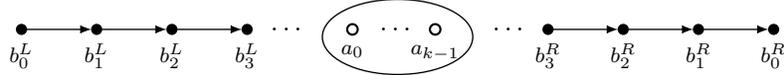
\begin{center}
\begin{tikzpicture}[>=latex,line width=0.5pt,xscale = 1,yscale = 1]
\node[point,fill=black,scale = 0.6,label=below:{\footnotesize $\yy_3^{L}$}]  (y3) at (1,0) {};
\node[point,fill=black,scale = 0.6,label=below:{\footnotesize $\yy_2^{L}$}]  (y2) at (0,0) {};
\node[point,fill=black,scale = 0.6,label=below:{\footnotesize $\yy_1^{L}$}]  (y1) at (-1,0) {};
\node[point,fill=black,scale = 0.6,label=below:{\footnotesize $\yy_0^{L}$}]  (y0) at (-2,0) {};
\node[]  at (1.55,0) {$\dots$};
\node[point,scale = 0.7,label=below:{\footnotesize $a_0$}] (a20) at (2.4,0) {};
\node[point,scale = 0.7,label=below:{\footnotesize $a_{k-1}$}] (a21) at (3.5,0) {};
\node[]  at (3,0) {$\dots$};
\draw[] (3,-.1) ellipse (1 and .5);
\node[]  at (4.5,0) {$\dots$};
\node[point,fill=black,scale = 0.6,label=below:{\footnotesize $\yy_3^{R}$}] (b23) at (5,0) {};
\node[point,fill=black,scale = 0.6,label=below:{\footnotesize $\yy_2^{R}$}] (b22) at (6,0) {};
\node[point,fill=black,scale = 0.6,label=below:{\footnotesize $\yy_1^{R}$}] (b21) at (7,0) {};
\node[point,fill=black,scale = 0.6,label=below:{\footnotesize $\yy_0^{R}$}] (b20) at (8,0) {};
\draw[->] (y0) to (y1);
\draw[->] (y1) to (y2);
\draw[->] (y2) to (y3);
\draw[->] (b23) to (b22);
\draw[->] (b22) to (b21);
\draw[->] (b21) to (b20);
\end{tikzpicture}
\end{center}
or, more formally, $xR_{\bullet k\bullet} y$ iff ($x=\yy_n^{L}$, $y=\yy_m^{L}$ for $n<m$), 
or ($x=\yy_n^{L}$, $y=a_{i}$),
or ($x=\yy_n^{L}$, $y=\yy_m^{R}$),
or ($x=a_{i}$, $y=a_{j}$) 
or $(x=a_{i}$, $y= \yy_n^{R}$), 
or ($x = \yy_n^{R}$, $y = \yy_m^{R}$, for $n>m$). 
The internal sets in $\INT_k$ are generated by 
\begin{equation}\label{internalS}
X_i = \{a_i\} \cup \{\yy_n^{L},\yy_n^{R} \mid n < \omega,\ n \equiv i \ (\text{mod}\ k)  \}, \quad \text{for $i<k$.}
\end{equation}
Observe that $\{\yy_n^{L}\},\{\yy_n^{R}\}\in\INT_k'$, for all $n<\omega$.
It is not hard to see that $\mathfrak F_k$ is a descriptive frame; we denote it by $\mathfrak{C}(\bullet,\clusterk,\bullet)$. 
As an exercise, the reader can check that, for any natural numbers $k, l, \dots, m, n > 0$,
\begin{align}\label{gff}
& \mathfrak{C}(\bullet,\clusterk,\bullet)\lhd \cdots \lhd \mathfrak{C}(\bullet,\clustern,\bullet) \models \Linf,\\ \label{gfz}
& \mathfrak{C}(\clusterk,\bullet)\lhd \mathfrak{C}(\bullet,\clusterl,\bullet)\lhd \cdots \lhd \mathfrak{C}(\bullet,\clusterm,\bullet)\lhd \mathfrak{C}(\bullet,\clustern)
\models \LinZ,
\end{align}
where $\mathfrak{C}(\bullet,\clustern)$ is the mirror image of $\mathfrak{C}(\clustern,\bullet)$; see also Lemma~\ref{l:tempdframes}.

The picture below shows models $\mathfrak M_1$ and $\mathfrak M_2$ based on $\mathfrak{C}(\bullet,\clustert,\bullet)$ and satisfying the conditions of Theorem~\ref{criterion} for $\varphi_1$ and $\varphi_2$:
%
%\centerline{\includegraphics[scale=0.7]{../Pics/GL2-}}
%\\
%
\begin{center}
\begin{tikzpicture}[>=latex,line width=0.5pt,xscale = 1.05,yscale = .5]
\node[]  at (-2.6,0) {{\small $\mathfrak M_2$}};
\node[point,fill=black,scale = 0.6,label=below:{\footnotesize $p_2,q_2$}] (y23) at (1,0) {};
\node[point,fill=black,scale = 0.6,label=below:{\footnotesize $p_1,q_2$}] (y22) at (0,0) {};
\node[point,fill=black,scale = 0.6,label=below:{\footnotesize $p_2$}] (y21) at (-1,0) {};
\node[point,fill=black,scale = 0.6,label=below:{\footnotesize $\neg \varphi_2$}] (y20) at (-2,0) {};
\node[]  at (1.55,0) {$\dots$};
\node[point,scale = 0.7,label=below:{\footnotesize $p_2$}] (a20) at (2.5,0) {};
\node[point,scale = 0.7,label=below:{\footnotesize $p_1$}] (a21) at (3.5,0) {};
\draw[] (3,-.1) ellipse (1 and .9);
\node[scale = 0.9]  at (3,.3) {$q_2$};
\node[]  at (4.5,0) {$\dots$};
\node[point,fill=black,scale = 0.6,label=below:{\footnotesize $p_1,q_2$}] (b23) at (5,0) {};
\node[point,fill=black,scale = 0.6,label=below:{\footnotesize $p_2,q_2$}] (b22) at (6,0) {};
\node[point,fill=black,scale = 0.6,label=below:{\footnotesize $p_1,q_2$}] (b21) at (7,0) {};
\node[point,fill=black,scale = 0.6,label=below:{\footnotesize $p_2,q_2$}] (b20) at (8,0) {};
\draw[->] (y20) to (y21);
\draw[->] (y21) to (y22);
\draw[->] (y22) to (y23);
\draw[->] (b23) to (b22);
\draw[->] (b22) to (b21);
\draw[->] (b21) to (b20);
\node[]  at (-2.6,3) {{\small $\mathfrak M_1$}};
\node[point,fill=black,scale = 0.6,label=above:{\footnotesize $p_1,q_1$}] (y13) at (1,3) {};
\node[point,fill=black,scale = 0.6,label=above:{\footnotesize $p_2,q_1$}] (y12) at (0,3) {};
\node[point,fill=black,scale = 0.6,label=above:{\footnotesize $p_1$}] (y11) at (-1,3) {};
\node[point,fill=black,scale = 0.6,label=above:{\footnotesize $\varphi_1$}] (y10) at (-2,3) {};
\node[]  at (1.55,3) {$\dots$};
\node[point,scale = 0.7,label=above:{\footnotesize $p_2$}] (a10) at (2.5,3) {};
\node[point,scale = 0.7,label=above:{\footnotesize $p_1$}] (a11) at (3.5,3) {};
\draw[] (3,3.1) ellipse (1 and .9);
\node[scale = 0.9]  at (3,2.7) {$q_1$};
\node[]  at (4.5,3) {$\dots$};
\node[point,fill=black,scale = 0.6,label=above:{\footnotesize $p_1,q_1$}] (b13) at (5,3) {};
\node[point,fill=black,scale = 0.6,label=above:{\footnotesize $p_2,q_1$}] (b12) at (6,3) {};
\node[point,fill=black,scale = 0.6,label=above:{\footnotesize $p_1,q_1$}] (b11) at (7,3) {};
\node[point,fill=black,scale = 0.6,label=above:{\footnotesize $p_2,q_1$}] (b10) at (8,3) {};
\draw[->] (y10) to (y11);
\draw[->] (y11) to (y12);
\draw[->] (y12) to (y13);
\draw[->] (b13) to (b12);
\draw[->] (b12) to (b11);
\draw[->] (b11) to (b10);
\node[gray]  at (-2.3,1.5) {$\bis$};
\draw[gray,thick,dotted] (y10) to (y20);
\draw[gray,thick,dotted] (y11) to (y22);
\draw[gray,thick,dotted] (y12) to (y21);
\draw[gray,thick,dotted] (y12) to (y23);
\draw[gray,thick,dotted] (y13) to (y22);
\draw[gray,thick,dotted] (a10) to (a20);
\draw[gray,thick,dotted] (a11) to (a21);
\draw[gray,thick,dotted] (b13) to (b23);
\draw[gray,thick,dotted] (b12) to (b22);
\draw[gray,thick,dotted] (b11) to (b21);
\draw[gray,thick,dotted] (b10) to (b20);
\end{tikzpicture}
\end{center}
By~\eqref{gff}, $\mathfrak{C}(\bullet,\clustert,\bullet) \models \Linf$, so $\varphi_1$ and $\varphi_2$ do not have an interpolant in $\Linf$.

\smallskip
$(b)$ Consider next the temporal version of the implication $\varphi_1'\rightarrow\varphi_2$ from Example~\ref{ex:GL.3}~$(b)$, which is clearly valid in $\LinZ$. To demonstrate that $\varphi_1'$ and $\varphi_2$ have no interpolant in $\LinZ$, we can use  
$\mathfrak{C}(\clustero,\bullet) \lhd \mathfrak{C}(\bullet,\clustert,\bullet)\lhd \mathfrak{C}(\bullet,\clustero)$, which is a frame for $\LinZ$ by~\eqref{gfz}. The models $\mathfrak M_1$ and $\mathfrak M_2$ depicted below
%
%\centerline{\includegraphics[scale=0.6]{../Pics/GL3-}}
%\\
%
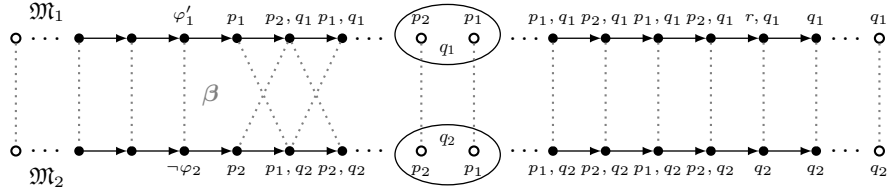
\begin{center}
\begin{tikzpicture}[>=latex,line width=0.5pt,xscale = .7,yscale = .4]
\node[]  at (-4.6,-.7) {{\small $\mathfrak M_2$}};
\node[point,fill=black,scale = 0.5,label=below:{\tiny $p_2,q_2$}] (y23) at (1,0) {};
\node[point,fill=black,scale = 0.5,label=below:{\tiny $p_1,q_2$}] (y22) at (0,0) {};
\node[point,fill=black,scale = 0.5,label=below:{\tiny $p_2$}] (y21) at (-1,0) {};
\node[point,fill=black,scale = 0.5,label=below:{\tiny $\neg \varphi_2$}] (y20) at (-2,0) {};
\node[point,fill=black,scale = 0.5] (w20) at (-3,0) {};
\node[point,fill=black,scale = 0.5] (w21) at (-4,0) {};
\node[]  at (-4.6,0) {$\dots$};
\node[point,scale = 0.6] (ww2) at (-5.2,0) {};
\node[]  at (1.55,0) {$\dots$};
\node[point,scale = 0.6,label=below:{\tiny $p_2$}] (a20) at (2.5,0) {};
\node[point,scale = 0.6,label=below:{\tiny $p_1$}] (a21) at (3.5,0) {};
\draw[] (3,-.1) ellipse (1 and 1);
\node[scale = 0.7]  at (3,.3) {$q_2$};
\node[]  at (4.5,0) {$\dots$};
\node[point,fill=black,scale = 0.5,label=below:{\tiny $p_1,q_2$}] (b23) at (5,0) {};
\node[point,fill=black,scale = 0.5,label=below:{\tiny $p_2,q_2$}] (b22) at (6,0) {};
\node[point,fill=black,scale = 0.5,label=below:{\tiny $p_1,q_2$}] (b21) at (7,0) {};
\node[point,fill=black,scale = 0.5,label=below:{\tiny $p_2,q_2$}] (b20) at (8,0) {};
\node[point,fill=black,scale = 0.5,label=below:{\tiny $q_2$}] (z20) at (9,0) {};
\node[point,fill=black,scale = 0.5,label=below:{\tiny $q_2$}] (z21) at (10,0) {};
\node[]  at (10.6,0) {$\dots$};
\node[point,scale = 0.6,label=below:{\tiny $q_2$}] (zz2) at (11.2,0) {};
\draw[->] (w21) to (w20);
\draw[->] (w20) to (y20);
\draw[->] (y20) to (y21);
\draw[->] (y21) to (y22);
\draw[->] (y22) to (y23);
\draw[->] (b23) to (b22);
\draw[->] (b22) to (b21);
\draw[->] (b21) to (b20);
\draw[->] (b20) to (z20);
\draw[->] (z20) to (z21);
\node[]  at (-4.6,3.7) {{\small $\mathfrak M_1$}};
\node[point,fill=black,scale = 0.5,label=above:{\tiny $p_1,q_1$}] (y13) at (1,3) {};
\node[point,fill=black,scale = 0.5,label=above:{\tiny $p_2,q_1$}] (y12) at (0,3) {};
\node[point,fill=black,scale = 0.5,label=above:{\tiny $p_1$}] (y11) at (-1,3) {};
\node[point,fill=black,scale = 0.5,label=above:{\tiny $\varphi_1'$}] (y10) at (-2,3) {};
\node[point,fill=black,scale = 0.5] (w10) at (-3,3) {};
\node[point,fill=black,scale = 0.5] (w11) at (-4,3) {};
\node[]  at (-4.6,3) {$\dots$};
\node[point,scale = 0.6] (ww1) at (-5.2,3) {};
\node[]  at (1.55,3) {$\dots$};
\node[point,scale = 0.6,label=above:{\tiny $p_2$}] (a10) at (2.5,3) {};
\node[point,scale = 0.6,label=above:{\tiny $p_1$}] (a11) at (3.5,3) {};
\draw[] (3,3.1) ellipse (1 and 1);
\node[scale = 0.7]  at (3,2.7) {$q_1$};
\node[]  at (4.5,3) {$\dots$};
\node[point,fill=black,scale = 0.5,label=above:{\tiny $p_1,q_1$}] (b13) at (5,3) {};
\node[point,fill=black,scale = 0.5,label=above:{\tiny $p_2,q_1$}] (b12) at (6,3) {};
\node[point,fill=black,scale = 0.5,label=above:{\tiny $p_1,q_1$}] (b11) at (7,3) {};
\node[point,fill=black,scale = 0.5,label=above:{\tiny $p_2,q_1$}] (b10) at (8,3) {};
\node[point,fill=black,scale = 0.5,label=above:{\tiny $r,q_1$}] (z10) at (9,3) {};
\node[point,fill=black,scale = 0.5,label=above:{\tiny $q_1$}] (z11) at (10,3) {};
\node[]  at (10.6,3) {$\dots$};
\node[point,scale = 0.6,label=above:{\tiny $q_1$}] (zz1) at (11.2,3) {};
\draw[->] (w11) to (w10);
\draw[->] (w10) to (y10);
\draw[->] (y10) to (y11);
\draw[->] (y11) to (y12);
\draw[->] (y12) to (y13);
\draw[->] (b13) to (b12);
\draw[->] (b12) to (b11);
\draw[->] (b11) to (b10);
\draw[->] (b10) to (z10);
\draw[->] (z10) to (z11);
\node[gray]  at (-1.5,1.5) {{\small $\bis$}};
\draw[gray,thick,dotted] (ww1) to (ww2);
\draw[gray,thick,dotted] (w10) to (w20);
\draw[gray,thick,dotted] (w11) to (w21);
\draw[gray,thick,dotted] (y10) to (y20);
\draw[gray,thick,dotted] (y11) to (y22);
\draw[gray,thick,dotted] (y12) to (y21);
\draw[gray,thick,dotted] (y12) to (y23);
\draw[gray,thick,dotted] (y13) to (y22);
\draw[gray,thick,dotted] (a10) to (a20);
\draw[gray,thick,dotted] (a11) to (a21);
\draw[gray,thick,dotted] (b13) to (b23);
\draw[gray,thick,dotted] (b12) to (b22);
\draw[gray,thick,dotted] (b11) to (b21);
\draw[gray,thick,dotted] (b10) to (b20);
\draw[gray,thick,dotted] (zz1) to (zz2);
\draw[gray,thick,dotted] (z10) to (z20);
\draw[gray,thick,dotted] (z11) to (z21);
\end{tikzpicture}
\end{center}
satisfy the conditions of Theorem~\ref{criterion} for $\varphi_1'$ and $\varphi_2$. 
\hfill $\dashv$
\end{example}

As illustrated by Example~\ref{ex:GL.3-temporal}, the temporal frames $\mathfrak F = (W,R,\INT)$ we need for checking the criterion of Theorem~\ref{criterion}  may contain both infinite descending and ascending chains of clusters (and so the $\FC^{-1}$ are not necessarily isomorphic to ordinals). Accordingly, we now have $R$-\emph{final} and $R^-$-\emph{final clusters} as well as two types of limit clusters: 
an $R$-\emph{limit cluster} is a non-$R^-$-final cluster without an immediate $R^-$-successor; an $R^-$-\emph{limit cluster} is a non-$R$-final cluster without an immediate $R$-successor. Some clusters can be both $R$- and $R^-$-limit clusters. 

%an $R$-\emph{limit cluster} and an $R^-$-\emph{limit cluster}. The former is the limit of some ascending chain $\langle C_n \mid n < \omega \rangle$ with $C_n <_R C_{n+1}$, for $n < \omega$; the latter is the limit of some descending chain $\langle C_n \mid n < \omega \rangle$ with $C_n <_{R^-} C_{n+1}$, for $n < \omega$; and some clusters are both $R$- and $R^-$-limits. By definition, $R$-limit ($R^-$-limit) clusters do not have an immediate $R^-$-successor (respectively, $R$-successor). 

We say that a set $S \ne \emptyset$ of clusters in $\mathfrak F$ is \emph{$R$-unbounded} (\emph{$R^-$-unbounded}) if there is no $C \in S$ such that $C' \le_R C$ (respectively, $C \le_R  C'$), for all $C' \in S$. 
A cluster $C$ is the $R$-\emph{limit of} an $R$-unbounded set $S$ if $C' <_R C$ for all $C' \in S$ and there is no cluster $C''$ with $C' <_R C''<_R C$ for all $C' \in S$; the $R^-$-\emph{limit of} an $R^-$-unbounded set $S$ is defined symmetrically by replacing $R$ with $R^-$. 
It is straightforward to see that each $R$-limit cluster $C$ is the $R$-limit of the $R$-unbounded set
$\{C'\mid C'<_R C\}$, and each $R^-$-limit cluster $D$ is the $R^-$-limit of the $R^-$-unbounded set
$\{D'\mid D<_R D'\}$. 
For any cluster $C$, we let
$(C,+\infty)=\{x\mid C<_R C(x)\}$ and $(-\infty,C)=\{x\mid C(x)<_R C\}$.

%The following statement is the temporal analogue of Lemma~\ref{lem:descr0} {without $(c)$:} 
%and~\ref{lem:descr'} for temporal frames, which we require in the sequel: 
%
\begin{lemma}\label{lem:descr00}
Suppose $\mathfrak F=(W,R,\INT)$ is a temporal $n$-generated descriptive frame, for some $n<\omega$. Then
%Suppose $\mathfrak{M}$ is a model based on an $n$-generated temporal descriptive frame $\mathfrak F$, for some $n < \omega$, and $\mathfrak F$ is $\mathfrak M$-generated. Then
%
\begin{itemize}
\item[$(a)$] every cluster in $\mathfrak F$ has at most $2^n$ points\textup{;}
\item[$(b)$] every $R$-unbounded \textup{(}$R^-$-unbounded\textup{)} set of clusters in $\mathfrak F$ has an $R$-limit \textup{(}$R^-$-limit\textup{)} in $\mathfrak F$, and so $\mathfrak F$ contains both $R$- and $R^-$-final clusters.
\end{itemize}
\end{lemma}
\begin{proof}
$(a)$ is proved similarly to Lemma~\ref{lem:descr0}~$(b)$. 

$(b)$ 
Suppose $\mathfrak F$ is $\mathfrak M$-generated, for some model $\mathfrak M$. 
Let $S$ be an $R$-unbounded set of clusters in $\mathfrak F$ with $y_C\in C$, $C\in S$, and let
\begin{equation*}
\Gamma = \bigcup_{C \in S}\Dp t_{\mathfrak M}(y_C) \cup  \bigcup_{C \in S} \{\psi \mid \Bf\psi \in t_{\mathfrak M}(y_C)\}.
\end{equation*}
Clearly, $\Gamma$ is finitely satisfiable in $\mathfrak M$, and so by {\bf (com)} and Lemma~\ref{minpoints}, there is a $\Gamma$-minimal point $x$ in $\mathfrak M$.
By {\bf (\tight)}, $y_C R x$ for all $C\in S$. Now suppose that $y$ is such that $y_CRy$, for all $C\in S$, and
$yRx$. Then $\Gamma\subseteq t_{\mathfrak M}(y)$, and so $xRy$ by the $\Gamma$-minimality of $x$.
Thus, $C(x)$ is the $R$-limit of $S$.
The existence of $R^-$-limits of $R^-$-unbounded $S$ is symmetric. 
%
%{Suppose $\mathfrak F$ is $\mathfrak M$-generated, for some model $\mathfrak M$. 
%Let $S$ be a non-empty $R$-unbounded set of clusters $C_i$, $i \in J$,  in $\mathfrak F$ with $y_i\in C_i$ and let
%
%\begin{equation*}
%\Gamma = \bigcup_{i \in J}\Dp t_{\mathfrak M}(y_i) \cup  \bigcup_{i \in J} \{\psi \mid \Bf\psi \in t_{\mathfrak M}(y_i)\}.
%\end{equation*}
%
%Clearly, $\Gamma$ is finitely satisfiable in $\mathfrak M$, and so by {\bf (com)} and Lemma~\ref{minpoints}, there is a $\Gamma$-minimal point $x$ in $\mathfrak M$.
%By {\bf (ref)}, $y_i R x$ for all $i\in J$. Now suppose that $y$ is such that $y_iRy$, for all $i\in J$, and $yRx$. Then $\Gamma\subseteq t_{\mathfrak M}(y)$, and so $xRy$ by the $\Gamma$-minimality of $x$. Thus, $C(x)$ is the $R$-limit of $S$.
%The existence of $R^-$-limits of $R^-$-unbounded $S$ is symmetric.}
\end{proof}

%As a consequence of Lemma~\ref{lem:descr00}~$(a)$, we obtain the temporal analogue of {\bf (limit)}: \{each $R$-limit cluster is the $R$-limit of some $R$-unbounded set of clusters, and each $R^-$-limit cluster is the $R^-$-limit of some $R^-$-unbounded set of clusters.}

A cluster $C$ is called \emph{minimal} (\emph{maximal}) in a temporal model $\mathfrak{M}$ if there is a formula $\mu$ such that $C\cap\min_{\mathfrak M}\{\mu\}\ne\emptyset$ ($C\cap\max_{\mathfrak M}\{\mu\}\ne\emptyset$). If there is such a $\sigma$-formula $\mu$, for some signature $\sigma$, we call $C$ \emph{$\sigma$-minimal} (\emph{$\sigma$-maximal}) \emph{in} $\mathfrak{M}$.

\begin{lemma}\label{lem:descr''}
Suppose $\mathfrak{M}$ is a model based on a finitely $\mathfrak M$-generated temporal descriptive frame $\mathfrak F$.  
%for some $n < \omega$, and $\mathfrak F$ is $\mathfrak M$-generated. 
Then
\begin{itemize}
%\item[$(a)$] $\mathfrak F$ is countable and every cluster in $\mathfrak F$ has at most $2^n$ points\textup{;}

%\item[$(b)$] every infinite ascending \textup{(}descending\textup{)} chain of clusters in $\mathfrak F$ has an $R$-limit \textup{(}respectively, $R^-$-limit\textup{)} in $\mathfrak F$, and so $\mathfrak F$ contains both $R$- and $R^-$-final clusters\textup{;}

\item[$(a)$] every degenerate cluster in $\mathfrak F$ is both maximal and minimal in $\mathfrak{M}$\textup{;}

\item[$(b)$] a cluster is maximal \textup{(}minimal\textup{)} in $\mathfrak{M}$ iff either it is $R$-final \textup{(}respectively, $R^-$-final\textup{)} or has an immediate $R$-successor \textup{(}respectively, $R^-$-successor\textup{)}\textup{;}

\item[$(c)$] a cluster is definable in $\mathfrak{M}$ iff it is both maximal and minimal in $\mathfrak{M}$.
\end{itemize}
It follows that the $R$- and $R^-$-limit clusters are not definable and not degenerate\textup{;} all other clusters are definable in $\mathfrak M$. We also have that 
\begin{itemize}
\item[$(d)$] for any clusters $C <_R C'$ in $\mathfrak F$, the interval $[C, C']$ contains a
 maximal cluster and also a minimal one\textup{;} 

%\item[$(d)$] {for any clusters $C(x) <_R C(y)$ in $\mathfrak F$, the interval $[C(x), C(y)]$ contains a maximal and a minimal clusters.}

\item[$(e)$]
if $C$ is not an $R$-limit cluster and $C'$ is not an $R^-$-limit cluster, then the closed interval $[C,C']$  is definable in $\mathfrak M$. 
\end{itemize}

\end{lemma}
\begin{proof}
Items $(a)$--$(c)$ are proved in the same way as  Lemma~\ref{lem:descr'}. Item $(d)$ follows from {\bf (\tight)}, which gives formulas $\varphi$ and $\psi$ with $\mathfrak M,x \not\models \Bf\varphi$, $\mathfrak M,y \models \Bf\varphi$ and $\mathfrak M,x \models \Bp\psi$, $\mathfrak M,y \not\models \Bp\psi$, and so $[C(x), C(y)]$ contains a $\Bf\varphi$-minimal cluster and a $\Bp\psi$-maximal one.
Item $(e)$: by $(b)$, $C$ is $\lambda$-minimal and $C'$ is $\mu$-maximal for some $\lambda,\mu$.
Then $[C,C']$ is defined in $\mathfrak M$ by $\Dp^+\lambda\land\Df^+\mu$. 
\end{proof}

%{\color{red}
%We also have the following temporal analogue of Lemma~\ref{l:defint}:
%
%\begin{lemma}\label{l:defint''}
%Suppose $\mathfrak{M}$ is a model based on a finitely $\mathfrak M$-generated temporal descriptive frame $\mathfrak F$.   If $C$ is a not an $R$-limit cluster and $C'$ is a not an $R^-$-limit cluster, then the closed interval $[C,C']$ in $\mathfrak F$ is definable in $\mathfrak M$.
%\end{lemma}
%
%\begin{proof}
%By Lemma~\ref{lem:descr''}, $C$ is $\lambda$-minimal and $C'$ is $\mu$-maximal for some formulas $\lambda,\mu$. Then $[C,C']$ is defined in $\mathfrak M$ by $\Dp^+\lambda\land\Df^+\mu$.
%\end{proof}
%}

%\begin{theorem}\label{tem-count}
%If $\mathfrak F = (W,R,\INT)$ is a finitely generated temporal descriptive frame, then $|W| \le \aleph_0$.
%\end{theorem}

The following temporal analogue is harder to prove than Lemma~\ref{l:countable}:

\begin{lemma}\label{tem-count}
If $\mathfrak F = (W,R,\INT)$ is a finitely generated temporal descriptive frame, then 
$W$ is countable.
\end{lemma}
\begin{proof}
By Lemma~\ref{lem:descr00}~$(a)$, it suffices to show that $\mathfrak F_c = (W_c, <_R)$ is countable. 
Suppose $\mathfrak F$ is $\mathfrak{M}$-generated, for some $\delta$-model $\mathfrak M=(\mathfrak F,\mathfrak v)$ and finite signature $\delta$. 
First, observe that, by Lemma~\ref{lem:descr''}~$(b)$, each non-$R$-limit cluster $C$ is $\mu_C$-minimal in $\mathfrak M$ for some $\mu_C$. Thus,
%a proof similar to that of Lemma~\ref{l:countable} (using Lemma~\ref{lem:descr''} in place of Lemma~\ref{lem:descr'}) shows 
the internal set $X_C=\mathfrak v(\Dp^+\mu_C)$ distinguishes $C$ from every $D$ with $D<_R C$, and so $X_C\ne X_D$ whenever $C\ne D$. As $\INT$ is countable, the number of non-$R$-limit clusters in $\mathfrak F_c$ is countable. Similarly, there are countably-many non-$R^-$-limit clusters in $\mathfrak F_c$.
So  it is enough to show that the number of clusters in $\mathfrak F_c$ that are both $R$- and $R^-$-limits is countable. We refer to such clusters as simply \emph{limit clusters\/}. 
Call an interval $[C^-,C^+]$ a \emph{neighbourhood} of a limit cluster $C$ if $C^-<_{R}C <_R C^+$. 
By Lemma~\ref{lem:descr''}, every limit cluster $C$ has a \emph{nice} neighbourhood $N_C=[C^-,C^+]$ with non-limit clusters 
$C^-$ and $C^+$.
As the number of different nice $N_C$ is countable, it follows that  %for every interval $[D,D']$,
%in $\mathfrak F_c$,
%
\begin{multline}
\label{uncountint}
\mbox{every uncountable interval $[D,D']$ contains a limit cluster $C$}\\
\mbox{all of whose neighbourhoods are uncountable}
\end{multline}
(otherwise all limit clusters in $[D,D']$ would belong to the countable union of the countable intervals $N_C$, and so $[D,D']$ were countable).

By an \emph{atomic type} we mean any $\at_{\mathfrak M}^\delta(x)$ with  $x \in W$.
For any cluster $C$, we set $\at(C)=\{\at_{\mathfrak M}^\delta(x)\mid x\in C\}$.
Let $C$ be an $R$-limit cluster. We say that an atomic type $\att$ 
\emph{occurs infinitely $R$-close to} $C$ if, for every $C'<_R C$, there is $C''$ such that $C'<_R C''<_R C$ and 
$\att\in\at(C'')$.
Similarly, $\att$ \emph{occurs infinitely $R^-$-close to} an $R^-$-limit cluster $C$ if whenever $C<_R C'$, then there is $C''$ such that $C<_R C''<_R C'$ and $\att\in\at(C'')$.
We claim that 
\begin{equation}
\label{infcloseto}
\mbox{if $\att$ occurs infinitely $R$-close to an $R$-limit cluster $C$, then $\att\in\at(C)$.}
\end{equation}
Indeed, let $S$ be an $R$-unbounded set of clusters with $R$-limit $C$ and $y_D\in D$, $D\in S$, and let
\begin{equation*}
\Gamma_{\att} = \att\cup\bigcup_{D \in S}\Dp t_{\mathfrak M}(y_D) \cup  \bigcup_{D \in S} \{\psi \mid \Bf\psi \in t_{\mathfrak M}(y_D)\}.
\end{equation*}
If $\att$ occurs infinitely $R$-close to $C$,
it can be shown similarly to the proof of Lemma~\ref{lem:descr00}~$(b)$ that there is a $\Gamma_{\att}$-minimal point $x\in C$, so $\att=\at_{\mathfrak M}^\delta(x)\in\at(C)$.

The converse of~\eqref{infcloseto} also holds:
\begin{equation}
\label{infclosefrom}
\mbox{if $\att\in\at(C)$, for an $R$-limit $C$, then $\att$ occurs infinitely $R$-close to $C$.}
\end{equation}
Indeed, suppose there is $C'<_R C$ with $\att\notin\at(C'')$, for any $C''$ in the interval $C'<_R C''<_R C$.
By Lemma~\ref{lem:descr''}~$(d)$, there is a cluster $C''$ in $[C',C]$ that is $\mu$-minimal in $\mathfrak M$ for some formula $\mu$. But then
$C$ is $\Dp\mu\land\bigwedge \att$-minimal, contrary to Lemma~\ref{lem:descr''}~$(b)$.
Symmetric variants of \eqref{infcloseto} and \eqref{infclosefrom} hold for $R^-$-limit clusters.

Call non-degenerate clusters $C' <_R C''$ \emph{twins} if $\at(C')=\at(C'')$ and, for every $C$ in $[C',C'']$, we have 
$\at(C)\subseteq \at(C')=\at(C'')$. We claim that
\begin{equation}
\label{notwins}
\mbox{there are no twins.}
\end{equation}
Indeed, suppose $C'$, $C''$ are twins. By induction on the construction of a $\delta$-formula $\alpha$, we see that if $x,y\in [C',C'']$ with $xRy$ and 
$\at_{\mathfrak M}^\delta(x) = \at_{\mathfrak M}^\delta(y)$, then $\mathfrak M,x \models \alpha$ iff $\mathfrak M,y \models \alpha$. 
We only consider one of the nontrivial cases.
Let $\mathfrak M,x \models \Df\alpha$. Then there is $z$ with $xRz$ and $\mathfrak M,z \models \alpha$. If $yRz$, then clearly $\mathfrak M,y \models \Df\alpha$. Otherwise, $z\in [C',C'']$,
so $\at_{\mathfrak M}^\delta(z) = \at_{\mathfrak M}^\delta(z')$, for some $z' \in C''$. Thus, by IH, $\mathfrak M,z' \models \alpha$, which implies $\mathfrak M,y \models \Df\alpha$ as $C''$ is non-degenerate. It follows that there are $x \in C'$ and $y \in C''$ with $t_{\mathfrak M}(x) = t_{\mathfrak M}(y)$, contrary to {\bf (dif)}.

We can now prove that $\mathfrak F_c$ is countable. Suppose $\mathfrak F_c$ is uncountable. By \eqref{uncountint} and Lemma~\ref{lem:descr00}~$(b)$,  $\mathfrak F_c$  contains a limit cluster $C$ whose neighbourhoods are all uncountable. Let $C$ be such a cluster with a minimal $\at(C)$.
As $\delta$ is finite, $C$ has a neighbourhood $N$ such that, 
for any $D\in N$ with $D<_R C$, every $\att\in\at(D)$ occurs infinitely $R$-close to $C$, and, 
for any $D\in N$ with $C<_R D$, every $\att\in\at(D)$ occurs infinitely $R^-$-close to $C$.
We call such $N$ a \emph{close proximity of} $C$.
As $N$ is uncountable, either $[C^-,C)$ or $(C,C^+]$ is uncountable. We only consider the former case, as the latter is similar.
We claim that
\begin{equation}
\label{limcountable}
\mbox{for every cluster $C'$ in $[C^-,C)$, the interval $[C^-,C']$ is countable.}
\end{equation}
Indeed, take such $C'$. As $[C^-,C']$ is contained in the close proximity $N$, for every limit cluster $D$ in $[C^-,C']$, we have $\at(D)\subsetneq \at(C)$, by \eqref{infcloseto} and \eqref{notwins}.
So by the $\at(C)$-minimality of $C$ among limit clusters with only uncountable neighbourhoods, every limit cluster $D$ in
$[C^-,C']$ has a countable neighbourhood. Thus, $[C^-,C']$ is countable by \eqref{uncountint}.

By \eqref{infclosefrom}, there is a countably infinite ascending chain $C_1 <_R C_2 <_R \dots$ of clusters in $[C^-,C)$ such that, 
for every $\att\in\at(C)$ and every $n<\omega$, there is $m$ with $n<m<\omega$ and $\att\in\at(C_m)$.
Let $C'$ be the $R$-limit of the $R$-unbounded set $\{C_n\mid n<\omega\}$ (which exists by Lemma~\ref{lem:descr00}~$(b)$). Then $C'\leq_R C$.
Also, every $\att\in\at(C)$ occurs infinitely $R$-close to $C'$, and so $\at(C)\subseteq\at(C')$ by \eqref{infcloseto}.
We cannot have $C'<_R C$ since otherwise (as $C'$ belongs to the close proximity $N$ of $C$) 
every $\att\in\at(C')$ occurred infinitely $R$-close to $C$, resulting in $\at(C)=\at(C')$ by \eqref{infcloseto}, and so
$C'$ and $C$ were twins, contrary to \eqref{notwins}.
It follows that $C'=C$, and so $[C^-,C)=\bigcup_{n<\omega} [C^-,C_n]$. As each $[C^-,C_n]$ is countable by \eqref{limcountable}, $[C^-,C)$ is also countable, which is a contradiction.
\end{proof}

Using Lemmas~\ref{lem:descr00} and \ref{lem:descr''}, 
we can also obtain elegant characterisations of descriptive frames for $\LinQ$, $\LinR$, $\Linf$ and $\LinZ$ (cf.~\cite{Burgess84,Gol,DBLP:journals/mlq/Wolter96a,DBLP:books/el/07/WolterZ07}):

\begin{lemma}\label{l:tempdframes}
Let $\mathfrak F = (W,R,\INT)$ be any finitely generated temporal descriptive frame. Then
\begin{description}
\item[$\LinQ$] $\mathfrak F \models \LinQ$ iff $\mathfrak F$ is \emph{serial} in both directions---i.e., the $R$- and $R^-$-final clusters in $\mathfrak F$ are both non-degenerate, 
%i.e., $\forall x  \exists y,z \, (xRy \land x R^- z)$, so $\mathfrak F$ has non-degenerate $R$-final and $R^-$-final clusters---
and $\mathfrak F$ is \emph{dense}---i.e., there is a non-degenerate cluster between any two distinct degenerate ones\textup{;}  

\item[$\LinR$] 
%The axiom of $\LinR$ is refuted in $\mathfrak F$ iff $\mathfrak F$ contains a \emph{Dedekind cut} $X \in \INT$ such that 
%%
%\begin{multline}\label{cut}
%X \ne \emptyset, \quad  W \setminus X \ne \emptyset, \quad \forall x,y \in W \, (x \in X \land xRy \to y \in X), \\ 
%\forall x \in X \exists y \in X \, xR^-y, \quad \forall x \in W \setminus X \exists y \in W \setminus X \, xRy.
%\end{multline}
%
%To illustrate, the set $X = \{x \in \mathbb Q \mid x > \sqrt{2}\}$ satisfies~\eqref{cut} in $(\mathbb Q,<)$; however, there exists no $X$ with~\eqref{cut} in $(\mathbb R,<)$. 
$\mathfrak F \models \LinR$ iff $\mathfrak F$ is serial, dense, and Dedekind-complete in the sense that there is a degenerate cluster between any two distinct non-degenerate ones\textup{;} 

\item[$\Linf$] $\mathfrak F \models \Linf$ iff $\mathfrak F$ does not contain a non-degenerate cluster $C$ such that $(-\infty,C) \in \INT$ or $(C, +\infty) \in \INT$ $($in particular, the $R$- and $R^-$-final clusters in $\mathfrak F$ are degenerate$)$\textup{;}

\item[$\LinZ$] $\mathfrak F \models \LinZ$ iff $\mathfrak F$ is serial and does not contain a non-degenerate cluster $C$ with $\emptyset \ne (-\infty,C) \in \INT$ or $\emptyset \ne (C, +\infty) \in \INT$ $($a single non-degenerate cluster is a frame for $\LinZ$ but not for $\Linf$$)$.
\end{description}
\end{lemma}
\begin{proof}
We only show the $(\Rightarrow)$-directions, leaving the converses to the reader. Suppose $\mathfrak F$ is $\mathfrak M$-generated, for some model $\mathfrak M=(\mathfrak F,\mathfrak v)$. 

$\LinQ$: As $\mathfrak F\models\Df \top$ ($\mathfrak F\models\Dp\top$), Lemma~\ref{minpoints} gives a $\{\Df\top\}$-maximal ($\{\Dp\top\}$-minimal) point $x$ in $\mathfrak M$ with   
$R$-final ($R^-$-final) and non-degenerate $C(x)$. Thus, $\mathfrak F$ is serial. 
Suppose $\{x\}$, $\{y\}$ are degenerate clusters with $xRy$. Lemma~\ref{lem:descr''} gives formulas $\psi_x$ and $\psi_y$ defining $\{x\}$ and $\{y\}$ in 
%some model 
$\mathfrak M$.
%\nb{the $\mathfrak M$ that generates $\mathfrak F$} 
As $\mathfrak M,x\models\Df\psi_y$ and 
$\mathfrak F\models\Df\psi_y\to\Df\Df\psi_y$, the formula
$\Df\psi_y\land\Dp\psi_x$ is satisfiable in $\mathfrak M$. Let $z$ be $\{\Df\psi_y\land\Dp\psi_x\}$-maximal in $\mathfrak M$. 
Then $xRzRy$. As $\mathfrak M,z\models\Df(\Df\psi_y\land\Dp\psi_x)$ by $\mathfrak F\models\Df\psi_y\to\Df\Df\psi_y$, the cluster $C(z)$ is  non-degenerate.

$\LinR$: 
Non-degenerate $C(x) <_R C(y)$ cannot be $<_R$-consecutive because otherwise, by
Lemma~\ref{lem:descr''}, $C(x)$ were $\psi$-maximal in $\mathfrak M$ for some formula $\psi$, 
and so $\mathfrak M,x\not\models\Box(\Bp\Df\psi \to \Df\Bp \Df\psi) \to (\Bp \Df\psi \to \Bf \Df\psi)$, contrary to $\mathfrak F\models\LinR$.
Thus, there is $z$ with $C(x) <_R C(z) <_R C(y)$. If $z$ is irreflexive, we are done. Otherwise, by {\bf (\tight)}, 
there is some formula $\chi$ with $\Bf\chi\in t_{\mathfrak M}(y)$ and $\chi\notin t_{\mathfrak M}(z)$, and so
$\mathfrak M,z\models\Df\neg\chi$.
 Let $z'$ be a $\Df\neg\chi$-maximal point in $\mathfrak M$. Clearly, $C(x) <_R C(z') <_R C(y)$. If $z'$ is irreflexive, we are done. Otherwise, we take the immediate $R$-successor $z''$ of $z'$, which  exists by Lemma~\ref{lem:descr''}. 
 As $\mathfrak M,z''\models\Bf\chi\land\neg\chi$, point $z''$ is irreflexive and $C(z'')<_R C(y)$.

%Suppose $C(x)$, $C(y)$ are non-degenerate with $C(x) <_R C(y)$. The clusters $C(x)$ and $C(y)$ cannot be $R$-consecutive in $\mathfrak F$ because otherwise Lemma~\ref{lem:descr''} would imply that $(-\infty,C(y)) \ne \emptyset$ and $(C(x),\infty) \ne \emptyset$ are both in $\INT$, and so $\mathfrak F$ refutes the axiom $\Box(\Bp p \to \Df\Bp p) \to (\Bp p \to \Bf p)$ at $x$ under the valuation $\mathfrak v(p) = (-\infty,C(y))$.  
%
%Thus, there is $z'$ with $C(x) <_R C(z') <_R C(y)$. By {\bf (ref)}, we have $\mathfrak M, y \models \Bf \chi$ and $\mathfrak M, z' \models \neg \Bf \chi$, for some $\chi$ and $\mathfrak M$. Let $z$ be a maximal point with $\mathfrak M, z \models \neg \Bf \chi$. Clearly, $C(x) <_R C(z) <_R C(y)$ and $(\infty,C(z)] \in \INT$. If $z$ is irreflexive, we are done. Otherwise, we take the immediate $R$-successor of $z$, which  exists by Lemma~\ref{lem:descr''}.  Since $\mathfrak F \models \LinQ$, this $R$-successor is irreflexive. 

$\Linf$: 
If there existed a non-degenerate cluster $C(x)$ and a formula $\psi$ with $(-\infty,C(x))=\mathfrak v(\psi)$, then
$\mathfrak M,x\not\models\Bp (\Bp\psi \to \psi) \to \Bp \psi$, contrary to $\mathfrak F\models\Linf$. 
%$(-\infty,C(x)) \in \INT$, for a non-degenerate $C(x)$. Then $\mathfrak F$ refutes $\Bp (\Bp p \to p) \to \Bp p$ at $x$ under the valuation $\mathfrak v(p) = (-\infty,C(x))$.

$\LinZ$:
If there existed a non-degenerate cluster $C(x)$ and some formula $\psi$ with 
$\emptyset\ne(C(x),+\infty)=\mathfrak v(\psi)$, then
$\mathfrak M,x\not\models\Bf(\Bf \psi \to \psi) \to (\Df\Bf \psi \to \Bf \psi)$, contrary to $\mathfrak F\models\LinZ$.
% If $\mathfrak F$ is serial and $(-\infty,C(x)) \in \INT$, for a non-degenerate $C(x)$, then $\mathfrak F$ refutes $\Bp(\Bp p \to p) \to (\Dp\Bp p \to \Bp p)$ at $x$ under $\mathfrak v(p) = (-\infty,C(x))$. 
\end{proof}

Note that $\Lin$ and $\LinQ$ are d-persistent while the other three logics are not~\cite{DBLP:journals/mlq/Wolter96a}.   
%For more details the reader is referred to~\cite{Burgess84,Gol,Gabbay1994-GABTLM,DBLP:books/cu/Demri2016}.
%Note that $\Lin$ and $\LinQ$ are d-persistent while the other three logics are not~\cite{DBLP:journals/mlq/Wolter96a}.}

\begin{example}\em\label{noncanR}
The descriptive frame $\mathfrak F = (W_2,R_{\circ\bullet},\INT_2)$ with $(W_2,R_{\circ\bullet})$ depicted below and $\INT_2$ defined in Example~\ref{k-omega} is serial, dense and Dedekind-complete, so $\mathfrak F \models \LinR$.
%
%\centerline{
%\includegraphics[scale=0.6]{../Pics/R}}
%\\
%
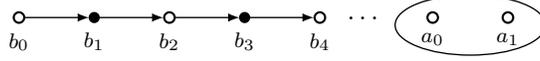
\begin{center}
\begin{tikzpicture}[>=latex,line width=0.5pt,xscale = 1,yscale = 1]
\node[point,scale = 0.7,label=below:{\footnotesize $\yy_4$}] (y4) at (1,0) {};
\node[point,fill=black,scale = 0.6,label=below:{\footnotesize $\yy_3$}] (y3) at (0,0) {};
\node[point,scale = 0.7,label=below:{\footnotesize $\yy_2$}] (y2) at (-1,0) {};
\node[point,fill=black,scale = 0.6,label=below:{\footnotesize $\yy_1$}] (y1) at (-2,0) {};
\node[point,scale = 0.7,label=below:{\footnotesize $\yy_0$}] (y0) at (-3,0) {};
\node[point,scale = 0.7,label=below:{\footnotesize $a_0$}] (a20) at (2.5,0) {};
\node[point,scale = 0.7,label=below:{\footnotesize $a_1$}] (a21) at (3.5,0) {};
\draw[] (3,-.1) ellipse (1 and .4);
\node[]  at (1.6,0) {$\dots$};
\draw[->] (y0) to (y1);
\draw[->] (y1) to (y2);
\draw[->] (y2) to (y3);
\draw[->] (y3) to (y4);
\end{tikzpicture}
\end{center}
It is readily seen, however, that $(W_2,R_{\circ\bullet})\not\models \LinR$, so $\LinR$ is not d-persistent. \hfill $\dashv$
\end{example}

The notion of $\sigma$-block from \S\ref{ss:modular} also needs a modification for temporal models. Namely, a set $\bl \subseteq W$ is a \emph{$\sigma$-block} in a temporal model $\mathfrak M$ based on $\mathfrak F = (W,R,\INT)$ if $\bl = \bl_{\mathfrak{M}}^{\sigma}(x)$, for some $x \in W$, where
\begin{multline*}
\bl_{\mathfrak{M}}^{\sigma}(x) = 
\big\{ y\in W \mid \Dt t_{\mathfrak{M}}^{\sigma}(y) \subseteq t_{\mathfrak{M}}^{\sigma}(x)\ \& \ \Dt t_{\mathfrak{M}}^{\sigma}(x) \subseteq t_{\mathfrak{M}}^{\sigma}(y), \text{ for }\mathsf{X} \in \{\mathsf{F}, \mathsf{P}\} \big\},
\end{multline*}
if both $\Df t_{\mathfrak{M}}^{\sigma}(x) \subseteq t_{\mathfrak{M}}^{\sigma}(x)$ and $\Dp t_{\mathfrak{M}}^{\sigma}(x) \subseteq t_{\mathfrak{M}}^{\sigma}(x)$ hold; otherwise $\bl_{\mathfrak{M}}^{\sigma}(x) = \{x\}$. 
Then we have the following temporal analogue of Lemma~\ref{int-prop}:

\begin{lemma}\label{int-prop'}
Suppose $\mathfrak{M}$ is a model based on a finitely $\mathfrak M$-generated temporal descriptive frame $\mathfrak F = (W, R,\INT)$. Then, for any $\sigma$-block $\bl$ in $\mathfrak{M}$,  
there exist clusters $\sC{\bl}$ and $\eC{\bl}$ in $\mathfrak F$ such that the following hold\textup{:}
\begin{itemize}
\item[$(a)$] $\bl = \bigl [\sC{\bl}, \eC{\bl}\bigr]$\textup{;}

%\item[$(b)$] 
%$t_{\mathfrak{M}}^{\sigma}(\bl) = t_{\mathfrak{M}}^{\sigma}(C_2)$\textup{;}

\item[$(b)$] if cluster $\sC{\bl}$ \textup{(}cluster $\eC{\bl}$\textup{)} is minimal \textup{(}respectively, maximal\textup{)} in $\mathfrak M$, then it is $\sigma$-minimal \textup{(}respectively, $\sigma$-maximal\textup{)} in $\mathfrak M$\textup{;} 

\item[$(c)$] if $\bl$ is non-degenerate, then both $\sC{\bl}$ and $\eC{\bl}$ are non-degenerate\textup{;}

\item[$(d)$] 
$\bl$ is definable in $\mathfrak M$ iff $\sC{\bl}$ is not an $R$-limit cluster and $\eC{\bl}$ is not an $R^-$-limit cluster\textup{;}
%both $\sC{\bl}$ and $\eC{\bl}$ are not limit\nz{both $R$- and $R^-$?} clusters\textup{;} 

\item[$(e)$] 
$t_{\mathfrak{M}}^{\sigma}(\bl) = t_{\mathfrak{M}}^{\sigma}\bigl(\sC{\bl}) = t_{\mathfrak{M}}^{\sigma}\bigl(\eC{\bl}\bigr)$.
\end{itemize} 
\end{lemma}
\begin{proof}
This can be proved similarly to Lemma~\ref{int-prop}, using Lemmas~\ref{lem:descr00}, \ref{lem:descr''}, 
%\textcolor{red}{\ref{l:defint''}} 
and \ref{minpoints}, in place of Lemmas~\ref{lem:descr0}, \ref{lem:descr'},  
%\textcolor{red}{\ref{l:defint}}  
and \ref{maxpoints}, respectively. 
\end{proof}

Given $\sigma$-bisimilar models $\mathfrak M_i$, $i=1,2$, based on finitely $\mathfrak M_i$-generated temporal frames,
%such that $\mathfrak M_1,x_1\sim_\sigma\mathfrak M_2,x_2$, 
we can adapt Lemma~\ref{bis-blocks} to the temporal setting to show that
$\sigma$-blocks in $\mathfrak M_1$ and $\mathfrak M_2$ always come in $\sigma$-bisimilar pairs $\bl$, $\bis(\bl)$.
%(see \eqref{bisint}). 
%
Being equipped with these modifications, we show first how to extend the selection procedure from the proof of Theorem~\ref{dperscofinal} 
%~$(a)$ and obtain a \coNP-algorithm deciding interpolant existence in 
to $\Lin$, $\LinQ$ and $\LinR$.

\begin{theorem}\label{temp-selection}
Each $L \in \{\Lin, \LinQ, \LinR\}$ has the polysize bisimilar model property, and the IEP for $L$ is \coNP-complete.
\end{theorem}
\begin{proof}
Suppose $\varphi_1$ and $\varphi_2$ have no interpolant in $L$, $\sigma = \sig(\varphi_1) \cap \sig(\varphi_2)$, and $\delta= \sig(\varphi_1) \cup \sig(\varphi_2)$. By Theorem~\ref{criterion}, there are $\delta$-models $\mathfrak M_i$, for $i=1,2$, based on $\mathfrak M_i$-generated temporal descriptive frames $\mathfrak{F}_i = (W_{i},R_{i},\INT_{i})$ for $L$ with $\mathfrak{M}_{1},x_{1} \sim_{\sigma} \mathfrak{M}_{2},x_{2}$, $\mathfrak{M}_{1},x_{1} \models \varphi_1$ and $\mathfrak{M}_{2}, x_{2} \models \neg \varphi_2$. 
Let $\bis$ be the largest $\sigma$-bisimulation between $\mathfrak M_1$ and $\mathfrak M_2$, that is, 
$y_1\bis y_2$ iff $t_{\mathfrak{M}_1}^{\sigma}(y_1)=t_{\mathfrak{M}_2}^{\sigma}(y_2)$, for all $y_i\in W_i$.
We show that there exist such $\mathfrak M_i$ of polynomial size in $\max(|\varphi_1|,|\varphi_2|)$.

For any $i=1,2$ and $\tau\in \sub(\varphi_i)$ satisfied in $\mathfrak{M}_i$, we take one $\{\tau\}$-maximal and one $\{\tau\}$-minimal point in $\mathfrak M_i$. Let $\mset{i}$ be the set of all selected points and let
\begin{equation*}
T=\bigl\{ t_{\mathfrak{M}_{1}}^{\sigma}(x) \mid x\in \{x_{1}\}\cup\mset{1}\bigr\} \cup\bigl\{ t_{\mathfrak{M}_{2}}^{\sigma}(x) \mid x\in \{x_{2}\}\cup\mset{2}\bigr\}.
\end{equation*}
For each $t\in T$, we take a smallest set $\sset{i} \subseteq W_i$ containing one $t$-maximal and one $t$-minimal point in $\mathfrak M_i$.
Also, we define $\sset{i}$ in such a way that if $\{x_i\}\cup \mset{i}$ already contained some $t$-maximal or -minimal points, then we choose from those, always selecting $x_i$ when $x_i$ is $t=t_{\mathfrak M_i}^\sigma(x_i)$-maximal or -minimal.

%that $x_i$ is chosen to be in $\sset{i}$ whenever $x_i$ is $t_{\mathfrak M_i}^\sigma(x_i)$-minimal or -maximal, and }

Let $W_{i}'=\{x_{i}\}\cup \mset{i} \cup \sset{i}$, $R'_i = \rest{R_i}{W'_i}$, $\mathfrak F'_i = (W'_i,R'_i)$, let $\mathfrak{M}_{i}'$ be the restriction of $\mathfrak{M}_{i}$ to $\mathfrak{F}_{i}'$, and let $x_1'\bis' x_2'$ iff $t_{\mathfrak{M}_1}^{\sigma}(x_1')=t_{\mathfrak{M}_2}^{\sigma}(x_2')$, for all $x_1'\in W_1'$, $x_2'\in W_2'$.
Following the proof of Lemma~\ref{cofinsubfr}, we see that $\mathfrak{M}_{1}',x_{1}\models\varphi_1$, $\mathfrak{M}_{2}',x_{2}\models\neg \varphi_2$, and $\bis'$ is a $\sigma$-bisimulation between $\mathfrak{M}_{1}'$ and $\mathfrak{M}_{2}'$ with 
$x_1\bis' x_2$. Clearly, $\mathfrak F_i' \models \Lin$ and the $\mathfrak M_i$ are of polynomial size in $\max(|\varphi_1|,|\varphi_2|)$.

\smallskip
For $L = \LinQ$, we do not necessarily have $\mathfrak{F}_{i}' \models L$. 
To fix this, we add some extra points from $W_i$ to $W_i'$. As $\mathfrak F_i\models\LinQ$, the $R$- and $R^-$-final clusters in $\mathfrak F_i$ are non-degenerate and, as observed in the selection procedure from \S\ref{sec:warming}, $W_i'$ contains some points from these final clusters.
Thus, $\mathfrak{F}_{i}' \not\models \LinQ$ iff $\mathfrak{F}_{i}'$ contains an irreflexive point $x$ with an immediate irreflexive $R'_i$-successor $y$.  We call such  pair $x,y$ an \emph{irr-defect in} $\mathfrak{F}_{i}'$. We are going to `cure' one irr-defect after the other without introducing new irr-defects in either frame.

Given an irr-defect $u_1,v_1$ in $\mathfrak{F}_{1}'$, we find an $R_1$-reflexive $z_1$ with $u_1 R_1 z_1 R_1 v_1$, which exists by $\mathfrak F_1\models\LinQ$ and Lemma~\ref{l:tempdframes}. Let $t=t_{\mathfrak{M}_{1}}^{\sigma}(z_1)$ and $\bl=\bl_{\mathfrak{M}_{1}}^{\sigma}(z_1)$. 
As $\Diamond_{F}t\subseteq t$ and  $\Diamond_{P}t\subseteq t$, $\bl$ is a non-degenerate $\sigma$-block in $\mathfrak M_1$. By Lemma~\ref{int-prop'}, there are $t$-minimal and $t$-maximal points $z_1^-$ and $z_1^+$ in the non-degenerate clusters $\sC{\bl}$ and $\eC{\bl}$. 
As $\bis(\bl)$ is a non-degenerate $\sigma$-block in $\mathfrak M_2$ by Lemma~\ref{bis-blocks}, there are $t$-minimal and $t$-maximal points $z_2^-$ and $z_2^+$ in the non-degenerate clusters $\sC{\bis(\bl)}$ and $\eC{\bis(\bl)}$.
By adding $z_1$, $z_1^-$, $z_1^+$ to $W_1'$ and $z_2^-$, $z_2^+$ to $W_2'$ we cure the irr-defect $u_1,v_1$ without creating a new
irr-defect in either frame.
%in view of $\mathfrak F_1\models\LinQ$. 
%Let $t_1 = t_{\mathfrak{M}_{1}}(z_1)$.\nz{full type} Take a $t_1$-minimal and $t_1$-maximal points $z_1^-$ and $z_1^+$ in $\mathfrak M_1$. As $\Diamond_{F}t_1 \subseteq t_1$ and  $\Diamond_{P}t_1 \subseteq t_1$, these points are reflexive and $u_1 R_1 z_1^- R_1 z_1 R_1 z_1^+ R_1 v_1$ because $u_1$ and $v_1$ are definable by Lemma~\ref{lem:descr''}. Then we consider the irr-defect $u_2 \in \bis(u_1)$ and $v_2 \in \bis(v_1)$ in $\mathfrak{F}_2'$, find a reflexive point $z_2$ in $\mathfrak M_2$ with $z_1 \bis z_2$, let $t_2 = t_{\mathfrak{M}_{2}}(z_2)$, and take a $t_2$-minimal and $t_2$-maximal reflexive points $z_2^-$ and $z_2^+$ in $\mathfrak M_2$ with $u_2 R_2 z_2^- R_2 z_2 R_2 z_2^+ R_2 v_2$. By the construction, $z_1^- \bis z_2^-$ and $z_1^+ \bis z_2^+$. 
%
%We add $z_1$, $z_1^-$, $z_1^+$ to $W_1'$ and $z_2$, $z_2^-$, $z_2^+$ to $W_2'$ thereby fixing the irr-defects $u_1,v_1$ and $u_2,v_2$ without creating a new irr-defect in either frame. 
Let $W_i''$, $i=1,2$, be the sets we obtain after curing all irr-defects in both frames in this way,  
$R''_i = \rest{R_i}{W''_i}$, $\mathfrak F''_i = (W''_i,R''_i)$, let $\mathfrak{M}_{i}''$ the restriction of $\mathfrak{M}_{i}$ to $\mathfrak{F}_{i}''$, and $x_1'\bis'' x_2'$ iff $t_{\mathfrak{M}_1}^{\sigma}(x_1')=t_{\mathfrak{M}_2}^{\sigma}(x_2')$, for all $x_1'\in W_1''$, $x_2'\in W_2''$. Then $\mathfrak F_i''\models\LinQ$, by Lemma~\ref{l:tempdframes},  
and

\medskip
\noindent
{\bf (\minmax)} 
for all $x\in W_1''\cup W_2''$ and $i=1,2$, the set $W_1''$ contains $t_{\mathfrak M_i}^\sigma(x)$-minimal and $t_{\mathfrak M_i}^\sigma(x)$-maximal points in $\mathfrak M_1$, and $W_2''$ contains $t_{\mathfrak M_i}^\sigma(x)$-minimal and $t_{\mathfrak M_i}^\sigma(x)$-maximal points in $\mathfrak M_2$. 

\medskip
\noindent
So it is readily seen (similarly to the proof of Lemma~\ref{cofinsubfr}) that 
$\mathfrak{M}_{1}'',x_{1}\models\varphi_1$, $\mathfrak{M}_{2}'',x_{2}\models\neg \varphi_2$, and $\bis''$ is a $\sigma$-bisimulation between $\mathfrak{M}_{1}''$ and $\mathfrak{M}_{2}''$ with 
$x_1\bis'' x_2$.  
%
%consider the points $x$ and $y$ in the original model $\mathfrak{M}_{i}$. By Lemma~\ref{lem:descr''} $(c,e)$, $\{x\},\{y\}\in \mathcal{P}_i$. As $\mathfrak{F}_{i} \models L$, there is a reflexive point $z$ with $x R_i z R_i y$. Indeed, $(C(x),C(y)) \ne \emptyset$ and cannot only contain irreflexive points because, by Lemma~\ref{lem:descr''} $(b)$, any infinite chain of irreflexive points has a non-degenerate limit within $(C(x),C(y))$. Then we add to $W_{1}'$ and $W_{2}'$ a maximal and a minimal point satisfying $t=t_{\mathfrak{M}_{i}}^{\sigma}(z)$. Observe that these points are reflexive since $\Diamond_{F}t\subseteq t$ and  $\Diamond_{P}t\subseteq t$. We also add to $W_{i}'$ the point $z$ itself. The resulting frame has fewer irr-defects than the one we started with. models we started with.models still satisfy $\varphi_{1}$, $\varphi_{2}$ (Lemma~\ref{cofinsubfr}), are $\sigma$-bisimilar, and have fewer irr-defects than the models we started with.

\smallskip
Finally, let $L = \LinR$. Since $\LinQ \subseteq \LinR$, we first cure the irr-defects in the $\mathfrak F_i'$, $i=1,2$, as described above.
Let $\mathfrak F_i''$, $i=1,2$, be the resulting serial and dense frames. Thus, $\mathfrak F_i'' \not\models L$ iff $\mathfrak{F}_{i}''$ contains two $<_{R''_i}$-consecutive non-degenerate clusters $C(x) \ne C(y)$. 
We call such $x,y$ a \emph{ref-defect in} $\mathfrak{F}_{i}''$.
We show that the ref-defects can also be cured in a step-by-step manner without introducing new defects of either type, while maintaining {\bf (\minmax)}. 

%Indeed, observe first that $C(x)$ and $C(y)$ cannot be $R_i$-consecutive in $\mathfrak F_i$ because otherwise Lemma~\ref{lem:descr''} $(d)$ would imply that $(-\infty,C(y)) \ne \emptyset$ and $(C(x),\infty) \ne \emptyset$ are both in $\INT_i$.  
%%
%Thus, there is $z'$ with $x R^s_i z' R^s_i y$. By {\bf (ref)}, we have $\mathfrak M_i, y \models \Bf \chi$ and $\mathfrak M_i, z' \models \neg \Bf \chi$, for some $\chi$. Let $z$ be a maximal point with $\mathfrak M_i, z \models \neg \Bf \chi$. Clearly, $x R^s_i z R^s_i y$ and $(\infty,C(z)] \in \INT_i$. If $z$ is irreflexive, we are done. Otherwise, we take the immediate $R_i$-successor of $z$, which must exist by Lemma~\ref{lem:descr''} $(d)$.  Since $\mathfrak F_i \models L$, this $R_i$-successor is irreflexive. 
%

If $u_1,v_1$ is a ref-defect in $\mathfrak F''_1$, Lemma~\ref{l:tempdframes} provides an irreflexive $z_1 \in W_1$ with $u_1 R_1 z_1 R_1 v_1$. Let $t=t_{\mathfrak M_1}^\sigma(z_1)$. The insertion of extra points 
%between $u_1$ and $v_1$ 
into $W_1''$ depends on whether $u_1$ and $v_1$ are in the same $\sigma$-block in $\mathfrak M_1$ or not.

\emph{Case} 1: $u_1,v_1 \in \bl$, for some $\sigma$-block $\bl$ in $\mathfrak M_1$. By Lemma~\ref{int-prop'}, $\bl$ is non-degenerate, and there are $t$-minimal and $t$-maximal points $z_1^-$ and $z_1^+$ in the non-degenerate clusters $\sC{\bl}$ and $\eC{\bl}$. 
By Lemma~\ref{bis-blocks}, $\bis(\bl)$ is a non-degenerate $\sigma$-block in $\mathfrak M_2$, so 
there are $t$-minimal and $t$-maximal points $z_2^-$ and $z_2^+$ in the non-degenerate clusters $\sC{\bis(\bl)}$ and $\eC{\bis(\bl)}$. 
By adding $z_1$, $z_1^-$, $z_1^+$ to $W_1''$ and $z_2^-$, $z_2^+$ to $W_2''$ we cure the ref-defect $u_1,v_1$ in $\mathfrak F_1''$ and maintain {\bf (\minmax)}.
Also, as {\bf (\minmax)} held in $\mathfrak F_i''$, by Lemma~\ref{int-prop'} we already had some points from $\sC{\bl}$ and $\eC{\bl}$ in $W_1''$ and some points from $\sC{\bis(\bl)}$ and $\eC{\bis(\bl)}$ in $W_2''$. So we did not create new defects in either frame,
and the property {\bf (\minmax)} is maintained.

%$u_1,v_1 \in \bl_1$, for some $\sigma$-block $\bl_1$ in $\mathfrak M_1$. Let $z_1^-$ and $z_1^+$ be $t_{\mathfrak{M}_{i}}^{\sigma}(z_1)$-minimal and $t_{\mathfrak{M}_{i}}^{\sigma}(z_1)$-maximal points in $\mathfrak{M}_{1}$, respectively. Clearly, $z_1^- \in C^-_{\bl_1}$ and $z_1^+ \in C^+_{\bl_1}$. 
%Then we take $u_2 \in \bis(u_1)$ and $v_2 \in \bis(v_1)$ in $\mathfrak{F}_2'$ that belong to some $\sigma$-block $\bl_2$ in $\mathfrak M_2$ by Lemma~\ref{bis-blocks}. We find an irreflexive $z_2$ in $\mathfrak M_2$ with $u_2 R_2 z_2 R_2 v_2$ and $z_1 \bis z_2$, a $t^\sigma_{\mathfrak{M}_{2}}(z_2)$-minimal $z_2^-\in C^-_{\bl_2}$ and a $t^\sigma_{\mathfrak{M}_{2}}(z_2)$-maximal $z_2^+ \in C^+_{\bl_2}$ in $\mathfrak M_2$ with $z_1^- \bis z_2^-$ and $z_1^+ \bis z_2^+$. 
%
%We add $z_1$, $z_1^-$, $z_1^+$ to $W_1'$ and $z_2$, $z_2^-$, $z_2^+$ to $W_2'$ thereby fixing the ref-defects $u_1,v_1$ and $u_2,v_2$ without creating new defects. The resulting models still satisfy $\varphi_1$, $\neg\varphi_2$ and are $\sigma$-bisimilar.

\emph{Case} 2: $u_1 \in \bl^{u_1}$, $v_1 \in \bl^{v_1}$, for $\sigma$-blocks $\bl^{u_1} \ne \bl^{v_1}$ in $\mathfrak M_1$. By the definition of $W_1''$ and $C(u_1)$, $C(v_1)$ being $<_{R_1''}$-consecutive, 
%we have $u_1 \in C^+_{\bl^{u_1}}$ and $v_1 \in C^-_{\bl^{v_1}}$. 
$C(u_1) = C^+_{\bl^{u_1}}$ and $C(v_1) = C^-_{\bl^{v_1}}$, so $z_1\notin\bl^{u_1}$.
We claim that there is an irreflexive $z \in W_1$ such that $u_1 R_1 z R_1 v_1$ and
$z$ is either $t_{\mathfrak{M}_{1}}^{\sigma}(z)$-maximal or $t_{\mathfrak{M}_{1}}^{\sigma}(z)$-minimal.
Indeed, as $u_1 R_1 z_1$, we have $\Df t_{\mathfrak{M}_{1}}^{\sigma}(z_1) \subseteq t_{\mathfrak{M}_{1}}^{\sigma}(u_1)$
and $\Dp t_{\mathfrak{M}_{1}}^{\sigma}(u_1) \subseteq t_{\mathfrak{M}_{1}}^{\sigma}(z_1)$. As $z_1\notin\bl^{u_1}$,
there can be two cases: 
either $(i)$ $\Df t_{\mathfrak{M}_{1}}^{\sigma}(u_1) \not\subseteq t_{\mathfrak{M}_{1}}^{\sigma}(z_1)$ 
or $(ii)$ $\Dp t_{\mathfrak{M}_{1}}^{\sigma}(u_1) \not\subseteq t_{\mathfrak{M}_{1}}^{\sigma}(z_1)$.
In case $(i)$, there is a $\sigma$-formula $\chi$ with $\mathfrak{M}_1,u_1\models \Df\chi$ but $\mathfrak{M}_1,z_1\not\models \Df\chi$. Take a $\{\Df\chi\}$-maximal point $z'$. Clearly, $u_1 R_1 z' R_1 v_1$. If $z'$ is irreflexive, we set $z=z'$ as it is $t_{\mathfrak{M}_{1}}^{\sigma}(z')$-maximal. Otherwise, Lemma~\ref{lem:descr''} gives an immediate degenerate $<_{R_1}$-successor $C(z)$ of $C(z')$ such that $z$ is $t_{\mathfrak{M}_{1}}^{\sigma}(z)$-maximal.
In case $(ii)$, there is a $\sigma$-formula $\chi$ with $\mathfrak{M}_1,u_1\not\models \Dp\chi$ but $\mathfrak{M}_1,z_1\models\chi$,
and so $\mathfrak{M}_1,v_1\models \Dp\chi$. Take a $\{\Dp\chi\}$-minimal point $z'$. Clearly, $u_1 R_1 z' R_1 v_1$. If $z'$ is irreflexive, we set $z=z'$ as it is $t_{\mathfrak{M}_{1}}^{\sigma}(z')$-minimal. Otherwise, Lemma~\ref{lem:descr''} gives an immediate degenerate $<_{R_1}$-predecessor $C(z)$ of $C(z')$ such that $z$ is $t_{\mathfrak{M}_{1}}^{\sigma}(z)$-minimal. 

Let $\bl=\bl_{\mathfrak M_1}^\sigma(z)$. Then $\bl$ is a degenerate $\sigma$-block in $\mathfrak M_1$ by Lemma~\ref{int-prop'}. By Lemma~\ref{bis-blocks}, $\bis(\bl)$ is a degenerate $\sigma$-block in $\mathfrak M_2$ with
$\bis(\bl^{u_1})\intord{\mathfrak F_2}\bis(\bl)\intord{\mathfrak F_2}\bis(\bl^{v_1})$.
Also, by {\bf (\minmax)} in $\mathfrak F_i''$, $\eC{\bis(\bl^{u_1})}$ and $\sC{\bis(\bl^{v_1})}$ are $<_{R_2''}$-consecutive non-degenerate clusters. Therefore, by adding $z$ to $W_1''$ and $z_2$ with  $C(z_2)=\bis(\bl)$ to $W_2''$, we cured the ref-defect $u_1,v_1$ in $\mathfrak F_1''$ and we did not create new defects of either kind in either frame while maintaining {\bf (\minmax)}.
So again it is readily seen (similarly to the proof of Lemma~\ref{cofinsubfr}) that, after fixing all defects, we end up with a pair of models as required that are based on frames for $\LinR$ by Lemma~\ref{l:tempdframes}.

This establishes the polysize bisimilar model property of $L \in \{\Lin, \LinQ, \LinR\}$. We show that the IEP for $L$ is in \coNP{} using the description of finite frames for $L$ in Lemma~\ref{l:tempdframes}. 
\end{proof}

%***********

The finitary selection construction in the proof above does not work for logics $L\in\{\Linf,\LinZ\}$.
In fact, these logics do not have the polysize bisimilar model property.
However, below we show that they still have a kind of \qfbmp{} similar to Definition~\ref{d:bmp} in the following sense. We can always witness the lack of an interpolant for $\varphi_1$, $\varphi_2$ in $L$
by a pair of temporal models that are based on frames for $L$, and assembled from
$\mathcal{O}\bigl(\max(|\varphi_1|,|\varphi_2|)\bigr)$-many
`\simple' models (like those in Example~\ref{ex:GL.3-temporal}) that are based 
\emph{\atomic} descriptive frames of the forms $m^<$, $\cluster{k}$, $\mathfrak{C}(\clusterk,\bullet)$, $\mathfrak{C}(\bullet,\clusterk,\bullet)$, and $\mathfrak{C}(\bullet,\clusterk)$, for $m,k=\mathcal{O}\bigl(\max(|\varphi_1|,|\varphi_2|)\bigr)$, $k>0$.
%(see Example~\ref{ex:GL.3-temporal}).
%\nb{and $\cluster{k}$? F: depends on whether we want to use it in the final construction, see marginpar}

\begin{theorem}\label{Fin-and-Z}
The IEPs for $\Linf$ and $\LinZ$ are both \coNP-complete. 
\end{theorem}
\begin{proof}
Given
%\nb{moved here, previously some of this was before Thm.\ref{Fin-and-Z}}
$\varphi_i$, $\mathfrak M_i$, $x_i$, for $i=1,2$, as above, define sets $\mset{i}$, $\sset{i}$, and 
$W'_i=\{x_i\}\cup \mset{i}\cup \sset{i}$ as in the proof of Theorem~\ref{temp-selection}, calling the points from $W'_i$ \emph{relevant in} $\mathfrak M_i$. A cluster or a $\sigma$-block in $\mathfrak M_i$ is \emph{relevant} if it contains a relevant point in $\mathfrak M_i$.
Given any pair $\bl$, $\bis(\bl)$ of $\sigma$-bisimilar $\sigma$-blocks in $\mathfrak M_1$ and $\mathfrak M_2$,
we can now have the temporal analogue of Lemma~\ref{l:maxbisblock}, dealing not only with $\sset{1}\cap\eC{\bl}$ and $\sset{2}\cap\eC{\bis(\bl)}$ but also
with $\sset{1}\cap\sC{\bl}$ and $\sset{2}\cap\sC{\bis(\bl)}$. 
In particular,
%We also have the following temporal analogues of ...\nb{!!}
%\eqref{maxbisblock1} and \eqref{maxbisblock2}: 
%
\begin{align}
& \label{maxbisblockroot}
\mbox{If $x_1\in \sC{\bl}\cup\eC{\bl}$, then $x_1\in\sset{1}$,}\\
 \nonumber
& \hspace*{4.5cm}\mbox{and if $x_2\in \sC{\bis(\bl)}\cup\eC{\bis(\bl)}$, then $x_2\in\sset{2}$}\textup{;}\\
 \label{maxbisblock2t}
& \mbox{there are $\sigma$-type preserving bijections 
$f^-\colon \sset{1}\cap\sC{\bl}\to \sset{2}\cap\sC{\bis(\bl)}$}\\ 
\nonumber
& \hspace*{5.5cm}\mbox{and $f^+\colon \sset{1}\cap\eC{\bl}\to \sset{2}\cap\eC{\bis(\bl)}$}\textup{;}\\
\label{maxbisblock3t}
& |\sset{1}\cap\sC{\bl}|=|\sset{1}\cap\eC{\bl}|\ne 0\ \mbox{ and }\ |\sset{2}\cap\sC{\bis(\bl)}|=|\sset{2}\cap\eC{\bis(\bl)}|\ne 0\textup{;}\\
\label{maxbisblock1t}
& \mbox{$\bl$ is relevant in $\mathfrak M_1$ iff $\bis(\bl)$ is relevant in $\mathfrak M_2$}.
 \end{align}
Let $\bl_1^0,\dots,\bl_1^\partN$ be all the relevant $\sigma$-blocks in $\mathfrak M_1$ ordered by $\intord{\mathfrak F_1}$,
for some $N=\mathcal{O}(\bigl(\max(|\varphi_1|,|\varphi_2|)\bigr)$.
By \eqref{maxbisblock1t} and Lemma~\ref{bis-blocks}, the $\intord{\mathfrak F_2}$-ordered list of all relevant $\sigma$-blocks in $\mathfrak M_2$ is $\bl_2^0,\dots,\bl_2^\partN$, where $\bl_2^j=\bis(\bl_1^j)$, for $j\leq\partN$.
By Lemma~\ref{bis-blocks}, $\bl_1^j$ is degenerate iff $\bl_2^j$ is degenerate, for $j\leq\partN$.

\emph{Case} $L = \Linf$: 
By Lemmas~\ref{l:tempdframes} and \ref{int-prop'}, $\bl_i^0$ and $\bl_i^\partN$, $i=1,2$, are degenerate. By Lemmas~\ref{lem:descr''}, \ref{l:tempdframes}, \ref{int-prop'}, if $\bl_i^j$ is non-degenerate, then $\sC{\bl_i^j}$ and $\eC{\bl_i^j}$ are non-degenerate
$R^-$- and $R$-limit clusters, and $C\cap \mset{i}=\emptyset$, for every non-degenerate cluster $C$ in $\bl_i^j$. Thus, by \eqref{maxbisblockroot}, we have $\bigl(\{x_i\}\cup\mset{i}\cup\sset{i}\bigr)\cap \sC{\bl_i^j}=\sset{i}\cap \sC{\bl_i^j}$ and $\bigl(\{x_i\}\cup\mset{i}\cup\sset{i}\bigr)\cap \eC{\bl_i^j}=\sset{i}\cap \eC{\bl_i^j}$, for $i=1,2$ and $j\leq\partN$.
Therefore, by \eqref{maxbisblock2t} and \eqref{maxbisblock3t}, 
for every $j\leq\partN$ there is $k^j>0$ such that 
$k^j=|\sset{1}\cap\sC{\bl_1^j}|=|\sset{1}\cap\eC{\bl_1^j}|=|\sset{2}\cap\sC{\bl_2^j}|=|\sset{2}\cap\eC{\bl_2^j}|$.

For all $i=1,2$ and $j\leq\partN$, we let $m_i^j=\bigl|\bigl((\{x_i\}\cup \mset{i})\cap \bl_i^j\bigr)\setminus(\sC{\bl_i^j}\cup\eC{\bl_i^j})\bigr|$ and define an \atomic{} frame $\mathfrak H_i^j=(H_i^j,R_i^j,\INT_i^j)$ by taking 
\[
\mathfrak H_i^j=\left\{
\begin{array}{ll}
\bullet, & \mbox{if $\bl_i^j$ is degenerate;}\\
\mathfrak{C}\bigl(\bullet,\clusterkj,\bullet\bigr), & \mbox{if $\sC{\bl_i^j}=\eC{\bl_i^j}$ is non-degenerate;}\\
\mathfrak{C}\bigl(\bullet,\clusterkj,\bullet\bigr) \lhd (m_i^{j})^< \lhd \mathfrak{C}\bigl(\bullet,\clusterkj,\bullet\bigr), & \mbox{otherwise.}
\end{array}
\right.
\]
Note that $m_1^j$ and $m_2^j$ might be different, and $(\{x_i\}\cup \mset{i})\cap \bl_i^j=\emptyset$ (and so $m_i^j=0$) can happen even when $\sC{\bl_i^j}\ne\eC{\bl_i^j}$.
Let $\mathfrak H_i=(H_i,R_i',\INT_i')= \mathfrak H_i^0\lhd\dots\lhd\mathfrak H_i^\partN$.
It is readily seen that $\mathfrak H_i$ is a frame for $\Linf$, for $i=1,2$.
Next, we define 
%a model $\mathfrak N_i^j$ based on $\mathfrak H_i^j$ with the help of 
a `parent' function  $\parent_i\colon H_i\to W_i'$ such that, for all $x\in H_i$,
\begin{align}
\label{parentfirst}
& \mbox{for all $j\leq \partN$, if $x\in H_i^j$ then $\parent_i(x)\in W_i'\cap\bl_i^j$,}\\
& \mbox{for all $y\in H_i$, if $xR_i' y$ then $\parent_i(x)R_i\parent_i(y)$,}\\
%\end{align}
%\begin{align}
\label{parentlast}
& \mbox{for all $y\in \mset{i}$, if $\parent_i(x)R_i y$ then $xR_i' z$ and $\parent_i(z)=y$ for some $z$.}
\end{align}
Finally, for $j\leq\partN$, we define a model $\mathfrak N_i^j$ based on $\mathfrak H_i^j$ by taking, for all $x\in H_i^j$,
\begin{equation}
\label{modeldef}
\at_{\mathfrak N_i^j}(x)=\at_{\mathfrak M_i}\bigl(\parent_i(x)\bigr),
\end{equation}
and let $\mathfrak N_i= \mathfrak N_i^0\lhd\dots\lhd\mathfrak N_i^\partN$.

Instead of giving the general definitions of $\parent_i$ and $\mathfrak N_i$, we illustrate
the construction in the picture below, where $\mathfrak M_i$
has three degenerate $\sigma$-blocks $\bl_i^0,\bl_i^2$ and $\bl_i^3$ and one non-definable non-degenerate $\sigma$-block $\bl_i^1$;
%which are indicated by brackets $[\dots]$; 
the relevant points in $\mathfrak M_i$ are underlined; $k^0=k^2=k^3=1$, $k^1=2$ and $m_i^1=3$.%\nb{slightly new pic (how it was in 1st submission ;-)}
%
%The picture also illustrates the definition of the `parent' function $\parent_i$, indicated by the dashed arrows. 
%\\
%
%\centerline{\includegraphics[scale=0.7]{../Pics/TGL}}
%\\
%
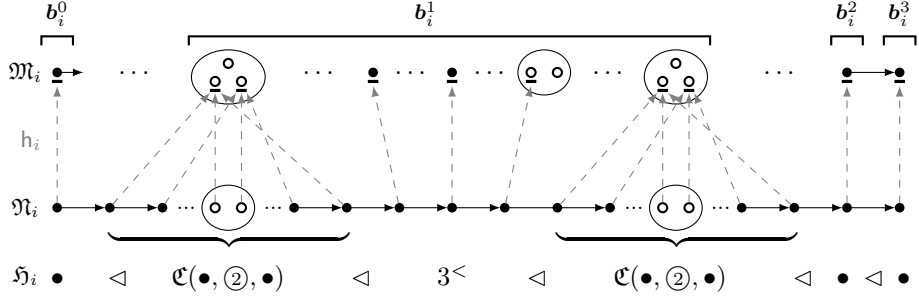
\begin{center}
\begin{tikzpicture}[>=latex,line width=0.4pt,xscale = .7,yscale = .6]
\node[]  at (-.6,0) {{\small $\mathfrak N_i$}};
\node[point,fill=black,scale = 0.5] (d1) at (0,0) {};
\node[point,fill=black,scale = 0.5] (y0l) at (1,0) {};
\node[point,fill=black,scale = 0.5] (y1l) at (2,0) {};
\node[]  at (2.45,0) {{\footnotesize ...}};
\node[point,scale = 0.6] (a0) at (3,0) {};
\node[point,scale = 0.6] (a1) at (3.5,0) {};
\draw[] (3.25,0) ellipse (.5 and .5);
\node[]  at (4.1,0) {{\footnotesize ...}};
\node[point,fill=black,scale = 0.5] (y1r) at (4.5,0) {};
\node[point,fill=black,scale = 0.5] (y0r) at (5.5,0) {};
\node[point,fill=black,scale = 0.5] (m0) at (6.5,0) {};
\node[point,fill=black,scale = 0.5] (m1) at (7.5,0) {};
\node[point,fill=black,scale = 0.5] (m2) at (8.5,0) {};
\node[point,fill=black,scale = 0.5] (z0l) at (9.5,0) {};
\node[point,fill=black,scale = 0.5] (z1l) at (10.5,0) {};
\node[]  at (10.95,0) {{\footnotesize ...}};
\node[point,scale = 0.6] (b0) at (11.5,0) {};
\node[point,scale = 0.6] (b1) at (12,0) {};
\draw[] (11.75,0) ellipse (.5 and .5);
\node[]  at (12.6,0) {{\footnotesize ...}};
\node[point,fill=black,scale = 0.5] (z1r) at (13,0) {};
\node[point,fill=black,scale = 0.5] (z0r) at (14,0) {};
\node[point,fill=black,scale = 0.5] (d2) at (15,0) {};
\node[point,fill=black,scale = 0.5] (d3) at (16,0) {};
\draw[->] (d1) to (y0l);
\draw[->] (y0l) to (y1l);
\draw[->] (y1r) to (y0r);
\draw[->] (y0r) to (m0);
\draw[->] (m0) to (m1);
\draw[->] (m1) to (m2);
\draw[->] (m2) to (z1l);
\draw[->] (z1r) to (z0r);
\draw[->] (z0r) to (d2);
\draw[->] (d2) to (d3);
\node[]  at (-.6,3) {{\small $\mathfrak M_i$}};
\node[point,fill=black,scale = 0.5] (dd1) at (0,3) {};
\draw[-,very thick] (-.1,2.8) to (.1,2.8);
\node[]  at (1.5,3) {$\dots$};
\node[point,scale = 0.6] (aa0) at (3,2.8) {};
\draw[-,very thick] (2.9,2.6) to (3.1,2.6);
\node[point,scale = 0.6] (aa1) at (3.5,2.8) {};
\draw[-,very thick] (3.4,2.6) to (3.6,2.6);
\node[point,scale = 0.6] (aa2) at (3.25,3.2) {};
\draw[] (3.25,2.9) ellipse (.7 and .6);
\node[]  at (5,3) {$\dots$};
\node[point,fill=black,scale = 0.5] (mm0) at (6,3) {};
\draw[-,very thick] (5.9,2.8) to (6.1,2.8);
\node[]  at (6.75,3) {$\dots$};
\node[point,fill=black,scale = 0.5] (mm1) at (7.5,3) {};
\draw[-,very thick] (7.4,2.8) to (7.6,2.8);
\node[]  at (8.25,3) {$\dots$};
\node[point,scale = 0.6] (mm2) at (9,3) {};
\draw[-,very thick] (8.9,2.8) to (9.1,2.8);
\node[point,scale = 0.6] (mmm2) at (9.5,3) {};
\draw[] (9.25,3) ellipse (.5 and .5);
\node[]  at (10.5,3) {$\dots$};
\node[point,scale = 0.6] (bb0) at (11.5,2.8) {};
\draw[-,very thick] (11.4,2.6) to (11.6,2.6);
\node[point,scale = 0.6] (bb1) at (12,2.8) {};
\draw[-,very thick] (11.9,2.6) to (12.1,2.6);
\node[point,scale = 0.6] (bb2) at (11.75,3.2) {};
\draw[] (11.75,2.9) ellipse (.6 and .6);
\node[]  at (13.75,3) {$\dots$};
\node[point,fill=black,scale = 0.5] (dd2) at (15,3) {};
\draw[-,very thick] (14.9,2.8) to (15.1,2.8);
\node[point,fill=black,scale = 0.5] (dd3) at (16,3) {};
\draw[-,very thick] (15.9,2.8) to (16.1,2.8);
\draw[->] (dd1) to (.5,3);
\draw[->] (dd2) to (dd3);
\node[gray]  at (-.5,1.5) {{\small $\parent_i$}};
\draw[->,gray,thin,dashed] (d1) to (0,2.7);
\draw[->,gray,thin,dashed] (y0l) to (2.9,2.55);
\draw[->,gray,thin,dashed] (y1l) to (3.4,2.55);
\draw[->,gray,thin,dashed] (a0) to (3,2.55);
\draw[->,gray,thin,dashed] (a1) to (3.5,2.55);
\draw[->,gray,thin,dashed] (y1r) to (3.6,2.55);
\draw[->,gray,thin,dashed] (y0r) to (3.1,2.55);
\draw[->,gray,thin,dashed] (m0) to (6,2.7);
\draw[->,gray,thin,dashed] (m1) to (7.5,2.7);
\draw[->,gray,thin,dashed] (m2) to (9,2.7);
\draw[->,gray,thin,dashed] (z0l) to (11.4,2.55);
\draw[->,gray,thin,dashed] (z1l) to (11.9,2.55);
\draw[->,gray,thin,dashed] (b0) to (11.5,2.55);
\draw[->,gray,thin,dashed] (b1) to (12,2.55);
\draw[->,gray,thin,dashed] (z1r) to (12.1,2.55);
\draw[->,gray,thin,dashed] (z0r) to (11.6,2.55);
\draw[->,gray,thin,dashed] (d2) to (15,2.7);
\draw[->,gray,thin,dashed] (d3) to (16,2.7);
\draw[-,thick] (-.3,3.6) -- (-.3,3.8) -- (.3,3.8) -- (.3,3.6);
\node[] at (0,4.3) {\footnotesize $\bl_i^0$};
\draw[-,thick] (2.5,3.6) -- (2.5,3.8) -- (12.4,3.8) -- (12.4,3.6);
\node[] at (7,4.3) {\footnotesize $\bl_i^1$};
\draw[-,thick] (14.7,3.6) -- (14.7,3.8) -- (15.3,3.8) -- (15.3,3.6);
\node[] at (15,4.3) {\footnotesize $\bl_i^2$};
\draw[-,thick] (15.7,3.6) -- (15.7,3.8) -- (16.3,3.8) -- (16.3,3.6);
\node[] at (16,4.3) {\footnotesize $\bl_i^3$};
\node[]  at (-.6,-1.5) {{\small $\mathfrak H_i$}};
\node[right] at (.8,-.7) {$\underbrace{\hspace{3.2cm}}$};
\node[right] at (9.3,-.7) {$\underbrace{\hspace{3.2cm}}$};
\node[right] at (-.3,-1.5) {$\bullet\hspace{.5cm}\lhd\hspace{.5cm}\mathfrak{C}(\bullet,\clustert,\bullet)\hspace{.8cm}\lhd \hspace{.8cm}3^<\hspace{.7cm}\lhd\hspace{.8cm}\mathfrak{C}(\bullet,\clustert,\bullet)\hspace{.8cm}\lhd\ \ \bullet\;\lhd\;\bullet$};
\end{tikzpicture}
\end{center}

\noindent
(It can happen
%\nb{I think, pls doublecheck} 
that the distinguished point $x_i$ of $\mathfrak M_i$ belongs to a non-degenerate cluster $C$ in some $\sigma$-block $\bl_i^j$, which is different from $\sC{\bl_i^j}$ and $\eC{\bl_i^j}$. However, as $C\cap \mset{i}=\emptyset$ by Lemmas~\ref{lem:descr''}, \ref{l:tempdframes} and \ref{int-prop'}, we  have \eqref{parentlast} even in this case.)

It is readily seen that this way \eqref{parentfirst}--\eqref{parentlast} hold and $\mathfrak N_i^j$ is based on $\mathfrak H_i^j$, for $j\leq\partN$.
Thus, $\mathfrak H_i^0\lhd\dots\lhd\mathfrak H_i^\partN$, $i=1,2$, is a frame for $\Linf$ by Lemma~\ref{l:tempdframes}. 
Using \eqref{parentfirst}--\eqref{parentlast}, a proof similar to that of Lemma~\ref{l:final}~$(a)$ shows that each point $x$ in $\mathfrak N_i$ makes true exactly the same formulas in $\sub(\varphi_i)$ as its parent $\parent_i(x)$ in $\mathfrak M_i$. It follows that  $\mathfrak{N}_{1},x_{1}' \models \varphi_1$ and $\mathfrak{N}_{2},x_{2}' \models \neg \varphi_2$,
where $x_i=\parent_i(x'_i)$.

Further, the construction and \eqref{maxbisblock2t} 
%\eqref{parentfirst}, \eqref{modeldef} and Lemma~\ref{l:maxbisblock} 
guarantee that each pair $(\mathfrak N_1^j,\mathfrak N_2^j)$, for $j\leq \partN$, satisfies a
%\textcolor{blue}{obvious} 
condition similar to 
%Definition~\ref{d:bmp}~$(a)$ 
Definition~\ref{d:match}~$(a)$ or $(c)$.
%\nb{typo fixed, `obvious' deleted}
%$(d)$ in Definition~\ref{d:bmp}: either
%
%\begin{itemize}
%\item[$(d_1)$] both $\mathfrak N_1^j$ and $\mathfrak N_2^j$  are based on the same degenerate frame $\bullet$, 
%
%or 
%\item[$(d_3)$]
%for every point $y_1$ in $\mathfrak N_1^j$, there is a point $y_2$ in each $\clusterkj$-cluster of $\mathfrak N_2^j$ with $\at^\sigma_{\mathfrak N_1^j}(y_1)=\at^\sigma_{\mathfrak N_2^j}(y_2)$, and
%
%for every $y_2$ in $\mathfrak N_2^j$, there is a point $y_1$ in each $\clusterkj$-cluster of $\mathfrak N_1^j$ with $\at^\sigma_{\mathfrak N_1^j}(y_1)=\at^\sigma_{\mathfrak N_2^j}(y_2)$.
%\end{itemize}
%
%Using these, 
Then a proof similar to that of Lemma~\ref{l:gbisall} shows that $\mathfrak N_1^j$ and $\mathfrak N_2^j$ are \globally{}
$\sigma$-bisimilar for every $j\leq\partN$, and so $\mathfrak{N}_{1},x'_{1} \sim_{\sigma} \mathfrak{N}_{2},x'_{2}$.

\emph{Case} $L = \LinZ$: 
While the definitions of $\mathfrak H_i^j$, for $0<j<\partN$, are the same as above, for $j=0,\partN$ we need new ones.
Now, by Lemmas~\ref{l:tempdframes} and \ref{int-prop'}, $\bl_i^0$ and $\bl_i^\partN$ are non-degenerate, for $i=1,2$.
Also, by Lemmas~\ref{lem:descr''}, \ref{l:tempdframes},  and \ref{int-prop'}, 
the $R^-$-final cluster $\sC{\bl_i^0}$ in $\mathfrak F_i$ is an $R^-$-limit cluster, and 
the $R$-final cluster $\eC{\bl_i^\partN}$ in $\mathfrak F_i$ is an $R$-limit cluster, for $i=1,2$.
There are several cases. If $\partN=0$ (that is, $\bl_i^0=W_i$) and $\sC{\bl_i^0}=\eC{\bl_i^0}$, then we let 
$\mathfrak H_i^0=\clusterkz$. 
%\nb{is this ok? F: it is ok, however we could simplify, see }
If $\partN=0$ and $\sC{\bl_i^0}\ne\eC{\bl_i^0}$, then we let 
$\mathfrak H_i^0=\mathfrak{C}\bigl(\clusterkz,\bullet\bigr)\lhd (m_i^0)^<\lhd \mathfrak{C}\bigl(\bullet,\clusterkz\bigr)$.
If $\partN>0$ then
\[
\mathfrak H_i^0=\left\{
\begin{array}{ll}
\mathfrak{C}\bigl(\clusterkz,\bullet\bigr), & \mbox{if $\sC{\bl_i^0}=\eC{\bl_i^0}$;}\\
\mathfrak{C}\bigl(\clusterkz,\bullet\bigr) \lhd (m_i^{0})^< \lhd \mathfrak{C}\bigl(\bullet,\clusterkz,\bullet\bigr), & \mbox{otherwise,}
\end{array}
\right.
\]
and
\[
\mathfrak H_i^\partN=\left\{
\begin{array}{ll}
\mathfrak{C}\bigl(\bullet,\clusterkn\,\bigr), & \mbox{if $\sC{\bl_i^\partN}=\eC{\bl_i^\partN}$;}\\
\mathfrak{C}\bigl(\bullet,\clusterkn\,,\bullet\bigr) \lhd (m_i^{\partN})^< \lhd \mathfrak{C}\bigl(\bullet,\clusterkn\,\bigr), & \mbox{otherwise.}
\end{array}
\right.
\]
%
%{ F: we could simplify and just go for $\mathfrak{C}\bigl(\clusterkz,\bullet\bigr)\lhd (m_i^0)^<\lhd \mathfrak{C}\bigl(\bullet,\clusterkz\bigr)$ if $N=0$. For $N>0$ we could also always take the "otherwise" option. No strong opinion..}
%
This way, by Lemma~\ref{l:tempdframes}, $\mathfrak H_i= \mathfrak H_i^0\lhd\dots\lhd\mathfrak H_i^\partN$ is a frame for $\LinZ$, for $i=1,2$.
Apart from these modifications, everything is similar to the $\Linf$ case.

\smallskip
A \coNP-algorithm deciding interpolant existence in $\Linf$ or $\LinZ$ is an obvious adaptation of the algorithm detailed in the proof of Theorem~\ref{t:iepcoNP}.
%
%Then each $\sigma$-block $\bl^j_i = [C^{j-}_i,C^{j+}_i]$ can be of types (1), (2) and 
%
%\begin{description}
%\item[(3)] $C^-_\bl = C^+_\bl$ are non-degenerate and coincide with $W_i \in \INT_i$.
%\end{description}
%
%In this case, $\bl^1_i$ and $\bl^n_i$ cannot be degenerate because $\mathfrak F_i$ are serial. If $n = 1$ and $\bl^1_i$ is of type (2), we set 
%
%\begin{equation}\label{tadpole1}
%\bar\bl_i^1 = \mathfrak{C}(\clusterk,\bullet) \lhd m^< \lhd \mathfrak{C}(\bullet,\clusterk).
%\end{equation}
% 
%If $n > 1$, then
%
%\begin{align}\label{tadpole2}
%&\bar\bl_i^1 = \mathfrak{C}(\clusterk,\bullet) \lhd m^< \lhd \mathfrak{C}(\bullet,\clusterk,\bullet),\\\label{tadpole3}
%&\bar\bl_i^n = \mathfrak{C}(\bullet,\clusterk,\bullet) \lhd m^< \lhd \mathfrak{C}(\bullet,\clusterk).
%\end{align}
%
%Here, $k$ and $m$ are defined in the same way as for~\eqref{tadpole}. Apart from these modifications and using~\eqref{gfz} instead of~\eqref{gff}, the construction is similar to the previous case.
%
%\smallskip
%
%A \coNP-algorithm deciding interpolant existence in $\Linf$ or $\LinZ$ is an obvious adaptation of the algorithm detailed in the previous section.
\end{proof}

We conjecture that the IEP for every consistent finitely axiomatisable Priorean temporal logic is \coNP-complete.

%\nb{move sentence to start of section? or before Thm.\ref{Fin-and-Z}?}

%{\color {red} Working with descriptive frames has again been crucial in our proof. It would nevertheless be of interest to understand for which logics descriptive frames can be replaced by Kripke (or even finite) frames in Theorem~\ref{criterion}. While d-persistence is clearly a sufficient condition, the logic $\LinR$ shows that it is not a necessary one~\cite{frank}.
%It is known, however, that $\LinR$ is strongly complete~\cite{frank} which suggests the conjecture that in Theorem~\ref{criterion} descriptive frames can be replaced by Kripke frames iff the logic is strongly complete. Note that a logic is strongly complete iff the corresponding variety of modal algebras is complex~\cite{goldblatt,frank}.}
%*************************************************************************************************************************

\section{Outlook and open problems}

We have turned the lack of the CIP into a research question by asking whether deciding interpolant existence becomes harder than validity for modal logics without the CIP.
As argued in~\cite{DBLP:journals/lmcs/Place18,DBLP:journals/lmcs/PlaceZ21} for the closely related problem of separability of disjoint regular languages using a smaller language class (such as first-order definable languages), this question can be understood as a generalisation of satisfiability that provides new insights into the expressivity of the logic in question. We have shown that, in contrast to modal logics with nominals, the product modal logic $\mathsf{S5}\times \mathsf{S5}$, and the guarded and two-variable fragments of first-order logic, the complexity of deciding interpolant existence in finitely axiomatisable modal logics of linear frames is in \coNP{} and, therefore, of the same complexity as validity. This appears to be the first general result about Craig interpolants for logics lacking the CIP. It gives rise to many further questions of which we mention only a few:
\begin{description}
\item[Q1] Is there a decidable modal logic above $\mathsf{GL}$, $\KF$, or $\mathsf{K}$ with the undecidable IEP? Currently, the only known example of a decidable logic with the undecidable IEP is the two-variable fragment of first-order logic with two equivalence relations~\cite{DBLP:conf/dlog/WolterZ24}.

\item[Q2] Do all d-persistent (cofinal) subframe logics above $\KF$ have the finite bisimilar model property? Can one show a quasi-finite bisimilar model property for all (cofinal) subframe logics above $\KF$ and use it to prove that interpolant existence is decidable for all finitely axiomatisable ones?

\item[Q3] What is the situation with the IEP for propositional superintuitionistic (aka intermediate) logics and (super)intuitionistic modal logics without the CIP? Note that the G\"odel translation reduces the IEP for propositional superintuitionistic logics to the IEP for (certain fragments of) modal logics above $\SF$; see the proof of~\cite[Theorem 14.9]{DBLP:books/daglib/0030819}. 

%Using the G\"odel translation of intuitionistic formulas into modal ones, for propositional intermediate logics this question can be closely linked to IEP for extensions of $\SF$. In fact, the G\"{o}del translation is known to reflect CIP~\cite{xx} and it follows from the proof of~\cite[Theorem 14.9]{DBLP:books/daglib/0030819} that it also reflects IEP.} 

%\textcolor{blue}{Let $L$ be a modal logic above $\SF$, $\varrho L$ its superintuitionsitic fragment, and $\varphi_i$, $i=1,2$, intuitionistic formulas. It follows from the proof of~\cite[Theorem 14.9]{DBLP:books/daglib/0030819}, we have $\varphi_1 \to \varphi_2$ has an interpolant in $\varrho L$ iff $T(\varphi_1) \to T(\varphi_2)$ has an interpolant in $L$.}

\item[Q4] Our proof is not constructive in the sense that is does not provide a non-trivial algorithm for computing interpolants if they exist (beyond exhaustive search) nor any upper bounds on their size. It would be of great interest to develop such algorithms. First steps towards computing interpolants in description logics without CIP are presented in~\cite{jung2025computation}.
\end{description}

Descriptive frames have been crucial for our proofs. It would therefore  be interesting and in line with the modal logic tradition to characterise logics for which descriptive frames can be replaced by Kripke (or even finite) frames in Theorem~\ref{criterion}. While d-persistence is clearly a sufficient condition, $\LinR$ shows that it is not a necessary one; see Example~\ref{noncanR}. It is known, however, that $\LinR$ is strongly complete~\cite{DBLP:journals/mlq/Wolter96a}, which suggests the conjecture that, in Theorem~\ref{criterion},  descriptive frames for $L$ can be replaced by Kripke frames iff $L$ is strongly Kripke complete (in the sense that every $L$-consistent set of formulas is satisfiable in a Kripke frame for $L$). Note that a logic is strongly Kripke complete iff the corresponding variety of modal algebras is complex~\cite{DBLP:journals/apal/Goldblatt89,DBLP:journals/mlq/Wolter96a}.

\medskip
\noindent
\textbf{Acknowledgements.}  We are grateful to the anonymous reviewer whose comments and  suggestions helped us to improve the presentation and terminology.

%**********************************************************************************************************************

%\bibliographystyle{asl}
%\bibliography{BIB}

\providecommand{\noopsort}[1]{}

\end{document}